%% file: thesis.tex
\documentclass[hidelinks,12pt]{report}
\usepackage{fullpage}
\usepackage{setspace}
\usepackage{amsthm,amsfonts,amssymb,amscd,amsmath}
\usepackage{url}
\usepackage{hyperref}
\usepackage{euscript}
\usepackage{verbatim, graphicx, rotating}
\usepackage{morefloats}
\usepackage{multirow}
\usepackage[nottoc]{tocbibind}

\hypersetup{
    colorlinks=false,
    pdfborder={0 0 0},
}

\input xy
\xyoption{all}

\newcommand{\field}{\ensuremath{\mathbb}}

\newcommand{\C}{{\field C}}
\newcommand{\R}{{\field R}}
\newcommand{\Z}{{\field Z}}

\newcommand{\N}{{\mathbb N}}
\newcommand{\QU}{{\mathbb H}} 

\newcommand{\ga}{\alpha}

\newcommand{\wt}{\widetilde}

\newtheorem{prop}{Proposition}[section]
\newtheorem{cor}[prop]{Corollary}

\newtheorem{theorem}[prop]{Theorem}

\newtheorem{conj}[prop]{Conjecture}

\newtheorem{corollary}[prop]{Corollary}

\theoremstyle{definition}

\newtheorem{remark}[prop]{Remark}
\newtheorem*{fact}{Fact}

\newtheorem{definition}[prop]{Definition}

\newtheorem{question}{Question}

\newcommand{\frb}{\mathfrak{b}}

\newcommand{\frg}{\mathfrak{g}}

\newcommand{\frk}{\mathfrak{k}}

\newcommand{\frs}{\mathfrak{s}}
\newcommand{\frt}{\mathfrak{t}}

\newcommand{\caI}{\mathcal{I}}

\newcommand{\caL}{\mathcal{L}}

\newcommand{\caO}{\mathcal{O}}

\newcommand{\caS}{\mathcal{S}}
\newcommand{\caT}{\mathcal{T}}

\newcommand{\caV}{\mathcal{V}}

\newcommand{\caX}{\mathcal{X}}

\renewcommand{\tilde}{\widetilde}

\begin{document}

\newpage
\thispagestyle{empty}
\vspace*{18pt}
\begin{center}
\textsc{Symmetric Subgroup Orbit Closures on Flag Varieties:\\Their Equivariant Geometry, Combinatorics, and Connections With Degeneracy Loci}\\[18pt]
by\\[18pt]
\textsc{Benjamin Joe Wyser}\\[12pt]
(Under the direction of William A. Graham)\\[12pt]
\textsc{Abstract}
\end{center}
We give explicit formulas for torus-equivariant fundamental classes of closed $K$-orbits on the flag variety $G/B$ when $G$ is one of the classical groups $SL(n,\C)$, $SO(n,\C)$, or $Sp(2n,\C)$, and $K$ is a symmetric subgroup of $G$.  We describe parametrizations of each orbit set and the combinatorics of its weak order, allowing us to compute formulas for the equivariant classes of all remaining orbit closures using divided difference operators.  In each of the cases in type $A$, we realize $K$-orbit closures as universal degeneracy loci of a certain type, involving a vector bundle $V$ over a scheme $X$ equipped with a flag of subbundles and a further structure determined by $K$.  We describe how our equivariant formulas can be interpreted as formulas for such loci in the Chern classes of the various bundles on $X$.

\begin{list}{\sc Index words:\hfill}{\labelwidth 1.2in\leftmargin 1.4in\labelsep 0.2in}
\item 
\begin{flushleft}\singlespacing
Equivariant cohomology,
Flag variety,
Symmetric subgroup,
Degeneracy loci,
Equivariant,
Localization,
Orbit
\end{flushleft}
\end{list}

\newpage
\thispagestyle{empty}
\vspace*{18pt}
\begin{center}
\textsc{Symmetric Subgroup Orbit Closures on Flag Varieties:\\Their Equivariant Geometry, Combinatorics, and Connections With Degeneracy Loci}\\[18pt]
by\\[18pt]
\textsc{Benjamin Joe Wyser}\\[12pt]
B.S., Mississippi State University, 2002\\
\vfill
A Dissertation Submitted to the Graduate Faculty \\
of The University of Georgia in Partial Fulfillment \\
of the \\
Requirements for the Degree \\[10pt]
\textsc{Doctor of Philosophy}\\[36pt]
\textsc{Athens, Georgia}\\[18pt]
2012
\end{center}

\newpage
\thispagestyle{empty}
\vspace*{5.5in}
\begin{center}
\copyright 2012 \\
Benjamin Joe Wyser \\
All Rights Reserved
\end{center}

\newpage
\thispagestyle{empty}
\vspace*{18pt}
\begin{center}
\textsc{Symmetric Subgroup Orbit Closures on Flag Varieties:\\Their Equivariant Geometry, Combinatorics, and Connections With Degeneracy Loci}\\[18pt]
by\\[18pt]
\textsc{Benjamin Joe Wyser}
\end{center}
\vfill
\begin{flushleft}\singlespacing
\hspace*{200pt}\makebox[100pt][l]{Major Professor:}William A. Graham\\
\vskip 12pt
\hspace*{200pt}\makebox[100pt][l]{Committee:       }Brian D. Boe\\
\hspace*{200pt}\makebox[100pt][l]{~                }Daniel Krashen\\
\hspace*{200pt}\makebox[100pt][l]{~                }Daniel K. Nakano\\
\hspace*{200pt}\makebox[100pt][l]{~                }Robert Varley\\
\vfill
Electronic Version Approved:\\[12pt]
Dean Maureen Grasso\\
Dean of the Graduate School\\
The University of Georgia\\
May 2012
\end{flushleft}

\cleardoublepage
\phantomsection

\chapter*{Acknowledgments}
The author wholeheartedly thanks his research advisor, Professor William A. Graham, for his help in conceiving this thesis project, as well as for his great generosity with both his time and expertise throughout the project.  Additionally, the author acknowledges and thanks Peter Trapa, Monty McGovern, and Michel Brion for helpful discussions and/or email correspondence.
\pagenumbering{roman}
\setcounter{page}{4}
\addcontentsline{toc}{chapter}{Acknowledgments}

\newpage
\pagenumbering{arabic}
\setcounter{page}{1}
\tableofcontents
\listoffigures
\listoftables

\include{ch1-Intro}
\include{ch2-Type_A}
\include{ch3-Type_B}
\include{ch4-Type_C}
\include{ch5-Type_D}
\include{ch6-Degeneracy_Loci}

\appendix
\include{Appendix-A}
\include{Appendix-B}

\bibliographystyle{alpha}
\bibliography{sourceDatabase}

\end{document}

%% file: ch1-Intro.tex
\chapter{Introduction and Preliminaries}
Let $G$ be a connected, complex, simple algebraic group of classical type.  Let $\theta$ be a (holomorphic) involution
of $G$ - that is, $\theta$ is an automorphism of $G$ whose square is the identity.  Fix $T \subseteq B$, a $\theta$-stable maximal torus and Borel subgroup of $G$, respectively.  Let $K = G^{\theta}$ be the subgroup of elements of $G$ which are fixed by $\theta$.  Such a subgroup of $G$ is referred to as a \textit{symmetric subgroup}.

$K$ acts on the flag variety $G/B$ with finitely many orbits (\cite{Matsuki-79}), and the geometry of these orbits and their closures plays an important role in the theory of Harish-Chandra modules for a certain real form $G_{\R}$ of the group $G$ --- namely, one containing a maximal compact subgroup $K_{\R}$ whose complexification is $K$.  For this reason, the geometry of $K$-orbits and their closures have been studied extensively, primarily in representation-theoretic contexts.

Their role in the representation theory of real groups aside, $K$-orbit closures can be thought of as generalizations of Schubert varieties, and, in principle, any question one has about Schubert varieties may also be posed about $K$-orbit closures.  With this in mind, we note here that our work is motivated by earlier work of Fulton (\cite{Fulton-92,Fulton-96_1,Fulton-96_2}) on Schubert loci in flag bundles, their role as universal degeneracy loci of maps of flagged vector bundles, and by connections between this work and the $T$-equivariant cohomology of the flag variety, $H_T^*(G/B)$, discovered by Graham (\cite{Graham-97}).  We briefly describe this earlier work.  Suppose $V$ is a vector bundle over a variety $X$, and suppose that $E_{\bullet}$ and $F_{\bullet}$ are two complete flags of subbundles of $V$.  Let $w \in S_n$ be given, and consider the locus

\[ \Omega_w = \{ x \in X \ \vert \ \text{rank}(E_i(x) \cap F_j(x)) \geq r_w(i,j) \text{ for all } i,j\}, \]
where $r_w(i,j)$ is a non-negative integer depending on $w$, $i$, and $j$.  Fulton considered the problem of finding a formula for the fundamental class $[\Omega_w] \in H^*(X)$ in terms of the Chern classes of the bundles involved.  Assuming that the flags $E_{\bullet},F_{\bullet}$ are ``sufficiently generic" (in a sense that can be made precise), the problem reduces to the universal case of finding formulas for the fundamental classes of Schubert loci in the flag bundle $Fl(V)$.  Moreover, it is enough to find a formula for the smallest Schubert locus (that corresponding to a point in every fiber).  One may then deduce formulas for larger loci from this formula by applying ``divided difference operators", moving inductively up the (weak) Bruhat order.

Graham considered this problem in a more universal and Lie-theoretic setting.  Let $G$ be a reductive algebraic group over $\C$, with $T \subseteq B \subseteq G$ a maximal torus and Borel subgroup, respectively.  Denote by $E$ the total space of a universal principal $G$-bundle.  This is a contractible space with a free action of $G$ (hence also a free action of $B$, by restriction).  Let $BB$ and $BG$ denote the spaces $E/B$ and $E/G$, respectively.  Then $BB$ and $BG$ are classifying spaces for the groups $B$ and $G$.    In the setting of \cite{Graham-97}, the primary object of interest is the diagonal $\Delta \subseteq BB \times_{BG} BB$.  After a translation between $H^*(BB \times_{BG} BB)$ and the $T$-equivariant cohomology $H_T^*(G/B)$ of $G/B$, one sees that the problem of describing $[\Delta] \in H^*(BB \times_{BG} BB)$ is equivalent to that of describing the $T$-equivariant class of a point.  In the setting of $T$-equivariant cohomology, one has use of the localization theorem, which allows one to verify the correctness of a formula for the class of a point simply by checking that it restricts correctly at all of the $T$-fixed points.  The observation is that had a formula for this class not already been discovered by Fulton using other methods, it might have been determined simply by identifying how it should restrict at each fixed point and attempting to guess a class which restricts as required.

With this in mind, we return now to the setting of symmetric subgroups.  Let $S = K \cap T$, a maximal torus of $K$ contained in $T$.  (When $\text{rank}(K) = \text{rank}(G)$, we have $S = T$, but in general, $S$ may be strictly smaller than $T$.)  In the present paper, we apply equivariant localization and divided difference operators as described above to discover previously unknown formulas for the $S$-equivariant fundamental classes of $K$-orbit closures on $G/B$.  We do so for all symmetric pairs $(G,K)$ when $G=SL(n,\C)$, $SO(n,\C)$, or $Sp(2n,\C)$.  As a means to this end, along the way we also partially handle some pairs $(G,K)$ with $G = Spin(n,\C)$.

In each case, this is done in two steps.  First, we identify the closed orbits and their restrictions at the various $S$-fixed points.  Using this information, we produce polynomials in the generators of $H_S^*(G/B)$ which restrict at the $S$-fixed points as required.  We then conclude by the localization theorem that these polynomials represent the equivariant fundamental classes of the closed $K$-orbits.  (As an interesting aside, we remark that the aforementioned work of Fulton turns out to be vital to this step in two cases, namely the cases $(Sp(2n,\C),GL(n,\C))$ and $(SO(2n,\C),GL(n,\C))$.  Indeed, our formulas for the closed orbits in those cases are ``determinantal" in nature, and are very similar to corresponding formulas of Fulton for the smallest Schubert locus in the type $C$ and type $D$ flag bundles.  Moreover, some algebraic properties of these determinants established in \cite{Fulton-96_1} turn out to amount precisely to a proof of the correctness of our formulas.)

Second, we outline how divided difference operators can be used to deduce formulas for the fundamental classes of the remaining orbit closures.  This is analogous to what is done for Schubert varieties.  However, combinatorial parametrizations of $K \backslash G/B$, as well as descriptions of its weak closure order in terms of such parametrizations, are typically more complicated than the weak Bruhat order on Schubert varieties.  We refer to known results on these combinatorics (\cite{Matsuki-Oshima-90,McGovern-Trapa-09,Richardson-Springer-90,Yamamoto-97}).  Although these parametrizations are likely familiar to experts on the representation theory of real groups, proofs of their correctness have not appeared in the literature in all cases.  As such, we have written down the details here for lack of a suitable reference.

One application of our formulas is that they allow one to deduce Chern class formulas for varieties analogous to the degeneracy loci considered by Fulton.  In general, such loci involve a vector bundle $V$ on a scheme $X$ equipped with a complete flag of subbundles and a further structure determined by $K$.  Given such a setup, degeneracy loci can be defined by conditions on the ``relative position" of the flag and the extra structure over various points of $X$.  In the type $A$ cases, this extra structure is either a splitting as a direct sum of subbundles of ranks $p$ and $q$ (corresponding to $K = S(GL(p,\C) \times GL(q,\C))$, or a non-degenerate bilinear form taking values in the trivial bundle.  The form is symmetric in the case of $K = SO(n,\C)$, and skew-symmetric in the case of $K = Sp(2n,\C)$.  $K$-orbit closures are universal cases of such loci, in exactly the same way that Schubert varieties are universal cases of the degeneracy loci studied by Fulton.  We describe the dictionary between these two viewpoints explicitly in the type $A$ cases, and indicate our thoughts on how this should extend to cases where $G$ is of type $BCD$.

After giving preliminary background in Chapter 1, we treat the various examples in types $ABCD$ in Chapters 2, 3, 4, and 5, respectively.  Each of these chapters is organized as follows:  For each symmetric pair $(G,K)$, we realize $K$ explicitly as a subgroup of $G$, and describe the corresponding embeddings of Weyl groups and root systems.  We then identify the closed orbits explicitly --- their number, and the fixed points contained in each.  Using this information, we determine formulas for each closed $K$-orbit using equivariant localization as described above.  The identification of the closed orbits is straightforward in the cases where $K$ is connected, but we do deal with some cases where $K$ is disconnected.  In those cases, we actually compute formulas for the closed orbits of the identity component of $K$ (or, equivalently, for the closed orbits of a corresponding symmetric subgroup of the simply connected cover of $G$, which \textit{is} connected).  In such cases, we must then identify how the closed $K$-orbits break up as unions of these, and add the formulas accordingly.  Once this is sorted out, we describe a parametrization of the orbit set, as well as the combinatorics of the weak order on the level of that parametrization.  (To avoid overly cluttering the exposition, detailed proofs of the correctness of these parametrizations are relegated to an appendix.)  We conclude in each case with an example calculation.

We end by describing the degeneracy locus picture in Chapter 6.  We give the full details for all cases in type $A$, and indicate some brief thoughts on the remaining cases, leaving some details for future work.

\section{Notation}\label{ssec:notation}
Here we define some notations which will be used throughout the paper.

We denote by $I_n$ the $n \times n$ identity matrix, and by $I_{n,m}$ the block matrix which has $I_n$ in the upper-left block, $-I_m$ in the lower-right block, and $0$'s elsewhere.  That is,
\[  I_{n,m} := 
\begin{pmatrix}
I_n & 0 \\
0 & -I_m \end{pmatrix}. \]

$I_{a,b,c}$ will denote the block matrix which has $I_a$ in the upper-left block, $-I_b$ in the middle block, and $I_c$ in the lower-right block, like so:
\[  I_{a,b,c} := 
\begin{pmatrix}
I_a & 0 & 0\\
0 & -I_b & 0 \\
0 & 0 & I_c \end{pmatrix}. \]

We denote by $J_n$ the $n \times n$ matrix with $1$'s on the antidiagonal and $0$'s elsewhere, i.e. the matrix $(e_{i,j}) = \delta_{i,n+1-j}$.  $J_{n,m}$ shall denote the block matrix which has $J_n$ in the upper-right block, $-J_m$ in the lower-left block, and $0$'s elsewhere.  That is,
\[  J_{n,m} := 
\begin{pmatrix}
0 & J_n \\
-J_m & 0 \end{pmatrix}. \]

For any group $G$ with $g \in G$, $\text{int}(g)$ shall denote the inner automorphism ``conjugation by $g$".

We will use both ``one-line" notation and cycle notation for permutations.  When giving a permutation in one-line notation, the sequence of values will be listed with no delimiters, while for cycle notation, parentheses and commas will be used.  Hopefully this will remove any possibility for confusion on the part of the reader.  So, for example, the permutation $\pi \in S_4$ which sends $1$ to $2$, $2$ to $3$, $3$ to $1$, and $4$ to $4$ will be given in one-line notation as $2314$ and in cycle notation as $(1,2,3)$.

We shall often have occasion to consider ``signed permutations", which are bijections $\sigma$ from the set $\{\pm 1,\pm 2\,\hdots,\pm n\}$ to itself having the property that 
\[ \sigma(-i) = -\sigma(i) \]
for all $i$.  We define the \textit{absolute value} of such a permutation, denoted $|\sigma|$, to be the permutation of $\{1,\hdots,n\}$ given by
\[ |\sigma|(i) = |\sigma(i)|. \]

Signed permutations will usually be written in one-line notation with bars over some of the numbers to indicate negative values.  For instance, the signed permutation $\sigma = 13\overline{2}$ is defined by $\sigma(1) = 1$, $\sigma(2) = 3$, and $\sigma(3) = -2$.

We will at times view signed permutations of $\{1,\hdots,n\}$ as being embedded in a larger symmetric group (usually $S_{2n}$ or $S_{2n+1}$).  To avoid any confusion in terminology, an element $\sigma \in S_m$ will be called a ``signed element of $S_m$" if and only if it has the property that
\[ \sigma(m+1-i) = m+1-\sigma(i) \]
for $i=1,\hdots,m$.  Signed permutations of $\{1,\hdots,n\}$ can be embedded as signed elements of $S_{2n}$ as follows:  Given a signed permutation $\pi$, define the first $n$ values of the signed element $\sigma \in S_{2n}$ by
\[ \sigma(i) = 
\begin{cases}
	\pi(i) & \text{ if $\pi(i) > 0$} \\
	2n+1-|\pi(i)| & \text{ if $\pi(i) < 0$,}
\end{cases} \]
and then define the remaining values of $\sigma$ to be what they are required to be:  $\sigma(2n+1-i) = 2n+1-\sigma(i)$.

Embedding signed permutations in $S_{2n+1}$ works very similarly.  Define the first $n$ values of the signed element $\sigma \in S_{2n+1}$ by
\[ \sigma(i) = 
\begin{cases}
	\pi(i) & \text{ if $\pi(i) > 0$} \\
	2n+2-|\pi(i)| & \text{ if $\pi(i) < 0$,}
\end{cases} \]
then insist that $\sigma(2n+2-i) = 2n+2-\sigma(i)$ for $i=1,\hdots,n$.  Note that this forces $\sigma(n+1) = n+1$.

We will also deal often with flags, i.e. chains of subspaces of a given vector space $V$.  A flag
\[ \{0\} \subset F_1 \subset F_2 \subset \hdots \subset F_{n-1} \subset F_n = V \]
will often be denoted by $F_{\bullet}$.  When we wish to specify the components $F_i$ of a given flag $F_{\bullet}$ explicitly, we will typically use the shorthand notation
\[ F_{\bullet} = \left\langle v_1,\hdots,v_n \right\rangle, \]
which shall mean that $F_i$ is the linear span $\C \cdot \left\langle v_1,\hdots,v_i \right\rangle$ for each $i$.

We will always be dealing with characters of tori $S$ (the maximal torus of $K$) and $T$ (the maximal torus of $G$).  To avoid confusing and ambiguous notation, these characters will generally be thought of as living in separate places, even in the event that $S$ and $T$ coincide.  Characters of $S$ will generally be denoted by capital $X$ variables, while characters of $T$ will be denoted by capital $Y$ variables.  Equivariant cohomology classes, on the other hand, will be represented by polynomials in lower-case $x$ and $y$ variables, where the lower-case variable $x_i$ means $X_i \otimes 1$, and where the lower-case variable $y_i$ means $1 \otimes Y_i$.  (See Proposition \ref{prop:eqvt-cohom-flag-var}.)

Unless stated otherwise, $H^*(-)$ shall always mean cohomology with $\C$-coefficients.

Lastly, we note here once and for all that $K \backslash G/B$ should always be taken to mean the set of $K$-orbits on $G/B$, unless explicitly stated otherwise.  (This as opposed to $B$-orbits on $K \backslash G$, or $B \times K$-orbits on $G$.)

\section{Equivariant cohomology (of the flag variety), and the localization theorem}\label{ssec:eqvt_cohom}
Our primary cohomology theory is equivariant cohomology with respect to the action of a maximal torus $S$ of $K$.  The $S$-equivariant cohomology of an $S$-variety $X$ is, by definition,
\[ H_S^*(X) := H^*((ES \times X)/S). \]
Here, $ES$ denotes the total space of a universal principal $S$-bundle (a contractible space with a free $S$-action), as in the introduction.  In the next section, we will also briefly refer to $S$-equivariant homology, which is by definition the \textit{Borel-Moore} homology $H_*((ES \times X)/S)$.  (For information on Borel-Moore homology, see e.g. \cite[\S B.2]{Fulton-YoungTableaux}.)  For smooth $X$, which is all we shall be concerned with here, the two theories are identified via Poincar\'{e} duality, so we work almost exclusively with cohomology.

Note that $H_S^*(X)$ is always an algebra for the ring $\Lambda_S := H_S^*(\{\text{pt.}\})$, the $S$-equivariant cohomology of a 1-point space (equipped with trivial $S$-action).  The algebra structure is given by pullback through the constant map $X \rightarrow \{\text{pt.}\}$.

Taking $X$ to be the flag variety $G/B$, we now describe $H_S^*(X)$ explicitly.  Let $R = S(\frt^*)$, the $\C$-symmetric algebra on the dual to the Lie algebra $\frt$ of a maximal torus $T$ of $G$.  Let $R' = S(\frs^*)$, the $\C$-symmetric algebra on the dual to the Lie algebra $\frs$ of $S$.  It is a standard fact that $R \cong \Lambda_T$, and $R' \cong \Lambda_S$.  Let $n$ be the dimension of $T$, and let $r$ be the dimension of $S$.  (In many cases, $S=T$, and $r=n$, but in some cases, $S$ is a proper subtorus of $T$, and $r < n$.)  Let $Y_1,\hdots,Y_n$ denote coordinates on $\frt^*$, taken as generators for the algebra $R$.  Likewise, let $X_1,\hdots,X_r$ denote coordinates on $\frs^*$, algebra generators for $R'$.

Note that there is a map $R \rightarrow R'$ induced by restriction of characters, whence $R'$ is a module for $R$.  Note also that $W$ acts on $R$, since it acts naturally on the characters $Y_i$.  Then it makes sense to form the tensor product $R' \otimes_{R^W} R$.  As it turns out, this is the $S$-equivariant cohomology of $X$.

\begin{prop}\label{prop:eqvt-cohom-flag-var}
With notation as above, $H_S^*(X) = R' \otimes_{R^W} R$.  Thus elements of $H_S^*(X)$ are represented by polynomials in variables $x_i := X_i \otimes 1$ and $y_i := 1 \otimes Y_i$.
\end{prop}
\begin{proof}
For the case $S=T$, this is the well-known fact that $H_T^*(X) \cong R \otimes_{R^W} R$, for which a proof can be found in \cite{Brion-98_i}.  For lack of a reference in the more general case, when $S$ may be a strict subtorus of $T$, we provide a proof here, which of course applies also to the case $S=T$.

It is easy to see that $H_S^*(X)$ is free over $R'$ of rank $|W|$.  Indeed, we have a flag bundle $E \times^S (G/B) \rightarrow BS$.  This is a locally trivial fibration with fiber isomorphic to $G/B$.  On the space $E \times^S (G/B)$, for any character $\lambda \in \widehat{T}$, we have a line bundle $\caL_{\lambda}$ which restricts to the line bundle $L_{\lambda} = G \times^B \C_{\lambda}$ over the fiber $G/B$.  Express the $|W|$ Schubert classes (a basis for $H^*(G/B)$) as polynomials in the Chern classes of these line bundles.  Then those same polynomials evaluated at the Chern classes of the line bundles $\caL_{\lambda}$ give $|W|$ classes in $H^*(E \times^S G/B)$ which restrict to a basis for the cohomology of $H^*(G/B)$.  The claim now follows from the Leray-Hirsch Theorem.

Now, note that there is a map $R' \otimes_{\C} R \rightarrow H_S^*(G/B)$.  The map is the tensor product of two maps, $p: R' \rightarrow H_S^*(G/B)$ and $q: R \rightarrow H_S^*(G/B)$.  The map $p$ is pullback through the map to a point, as described above.  The map $q$ takes a character $\lambda$ to $c_1(\caL_{\lambda})$.  The map $p \otimes q$ is surjective, since the $S$-equivariant Schubert classes are hit by the map $q$ on the second factor.

Since $R$ is free over $R^W$ of rank $|W|$, $R' \otimes_{R^W} R$ is free over $R'$ of rank $|W|$, hence $H_S^*(G/B)$ and $R' \otimes_{R^W} R$ are both free $R'$-modules of the same rank.  Consider the possibility that $p \otimes q$ factors through $R' \otimes_{R^W} R$ --- that is, suppose that $x \otimes y \mapsto p(x) q(y)$ is a well-defined map $R' \otimes_{R^W} R \rightarrow H_S^*(G/B)$.   If so, then this map is clearly surjective, since $p \otimes q$ is, so it is injective as well, being a map of free $R'$-modules of the same rank.  The map is moreover a ring homomorphism, and so it is in fact an isomorphism of rings.

Thus we need only see that the map $\phi: R' \otimes_{R^W} R \rightarrow H_S^*(G/B)$ given by $\phi(\ga \otimes \beta) = p(\ga)q(\beta)$ is well-defined.  To see this, note first that the space $E \times^S (G/B)$ is isomorphic to the space $BS \times_{BG} BB$.  Indeed, the map $E \times G \rightarrow E \times_{BG} E$ given by $(e,g) \mapsto (e,eg)$ is an isomorphism, since $E \rightarrow BG$ is a principal $G$-bundle.  This map is $S \times B$-equivariant, where $S \times B$ acts on $E \times G$ by $(e,g).(s,b) = (es,s^{-1}gb)$, and on $E \times_{BG} E$ by $(e_1,e_2).(s,b) = (e_1s,e_2b)$.  Thus the isomorphism descends to quotients, and $(E \times G) / (S \times B) \cong E \times^S (G/B)$, while $(E \times_{BG} E) / (S \times B) \cong BS \times_{BG} BB$.

Now, we have a map $\wt{\phi}: H^*(BS) \otimes_{H^*(BG)} H^*(BB) \rightarrow H^*(BS \times_{BG} BB)$ given by $\alpha \otimes \beta \mapsto \pi_1^*(\alpha) \pi_2^*(\beta)$, where $\pi_1, \pi_2$ are the projections from $BS \times_{BG} BB$ onto $BS$ and $BB$, respectively.  There is no question of this map being well-defined; that it is well-defined is immediate given commutativity of the square
\[
\xymatrixcolsep{1pc} 
\xymatrixrowsep{2pc}
\xymatrix
{BS \times_{BG} BB \ar@{>}[r] \ar@{>}[d] & BB \ar@{>}[d] \\
BS \ar@{>}[r] & BG
}
\]
It is well-known that $H^*(BS) \cong R'$, $H^*(BB) \cong R$, and $H^*(BG) \cong R^W$, so clearly $H^*(BS) \otimes_{H^*(BG)} H^*(BB) \cong R' \otimes_{R^W} R$.  Thus to see that $\phi$ is well-defined, we can simply observe that it is precisely the map $\wt{\phi}$ when $H^*(BS) \otimes_{H^*(BG)} H^*(BB)$ is identified with $R' \otimes_{R^W} R$, and $H^*(BS \times_{BG} BB)$ is identified with $H_S^*(G/B) = H^*(E \times^S (G/B))$ via the isomorphism described above.

On the first factor $R'$, the map $\phi$ maps a character $\lambda$ of $S$ to $c_1((E \times^S \C_{\lambda}) \times (G/B))$.  The bundle $(E \times^S \C_{\lambda}) \times G/B$ is the line bundle associated to the principal $S$-bundle $E \times G/B \rightarrow E \times^S (G/B)$ and the $1$-dimensional representation $\C_{\lambda}$ of $S$.  On the other hand, the map $\wt{\phi}$ maps $\lambda$ to $c_1(\pi_1^*(\caL_{\lambda}))$.  The bundle $\pi_1^*\caL_{\lambda} = (E \times^S \C_{\lambda}) \times_{BG} BB$ is the line bundle associated to the principal $S$-bundle $E \times_{BG} BB \rightarrow BS \times_{BG} BB$ and the same $1$-dimensional representation $\C_{\lambda}$ of $S$.  Since these two line bundles are associated to principal $S$-bundles which correspond via our isomorphism, and to the same representation of $S$, they are in fact the same line bundle when the two spaces are identified.  Thus $\phi$ and $\wt{\phi}$ agree on the $R'$ factor.

The story on the second factor is much the same.  The map $\phi$ maps a character $\lambda$ of $T$ to $c_1(E \times^S (G \times^B \C_{\lambda}))$, the first Chern class of the line bundle associated to the principal $B$-bundle $E \times^S G \rightarrow E \times^S (G/B)$ and the $1$-dimensional representation $\C_{\lambda}$ of $B$ (where, as usual, the $T$-action on $\C_{\lambda}$ is extended to $B$ by letting the unipotent radical act trivially).  The map $\wt{\phi}$ maps $\lambda$ to $c_1(\pi_2^*\caL_{\lambda})$, with $\pi_2^*\caL_{\lambda} = BS \times_{BG} (E \times^B \C_{\lambda})$ the line bundle associated to the principal $B$-bundle $BS \times_{BG} E \rightarrow BS \times_{BG} BB$ and the same representation of $B$.  Since these principal bundles correspond via our identification, and since the line bundles are associated to these principal bundles and the same representations of $B$, they are the same line bundle.  Thus $\phi$ and $\wt{\phi}$ agree on the $R$ factor as well.
\end{proof}

As mentioned, the $S$-equivariant cohomology of any $S$-variety $X$ is an algebra for $\Lambda_S$, the $S$-equivariant cohomology of a point.  We have the following standard localization theorem for actions of tori, one reference for which is \cite{Brion-98_i}:

\begin{theorem}
Let $X$ be an $S$-variety, and let $i: X^S \hookrightarrow X$ be the inclusion of the $S$-fixed locus of $X$.  The pullback map of $\Lambda_S$-modules 
\[ i^*: H_S^*(X) \rightarrow H_S^*(X^S) \]
is an isomorphism after a localization which inverts finitely many characters of $S$.  In particular, if $H_S^*(X)$ is free over $\Lambda_S$, then $i^*$ is injective. 
\end{theorem}

The last statement is what is relevant for us, since when $X$ is the flag variety, $H_S^*(X) = R' \otimes_{R^W} R$ is free over $R'$.  Thus in the case of the flag variety, the localization theorem tells us that any equivariant class is entirely determined by its image under $i^*$.  As noted in the next section (cf. Proposition \ref{prop:s-fixed-points}), the locus of $S$-fixed points is finite, and indexed by the Weyl group $W$, even in the event that $S$ is a proper subtorus of the maximal torus $T$ of $G$.  Thus in our setup, 
\[ H_S^*(X^S) \cong \bigoplus_{w \in W} \Lambda_S, \]
so that in fact a class in $H_S^*(X)$ is determined by its image under $i_w^*$ for each $w \in W$, where here $i_w$ denotes the inclusion of the $S$-fixed point $wB$.    Given a class $\beta \in H_S^*(X)$ and an $S$-fixed point $wB$, we will typically denote the restriction $i_w^*(\beta)$ at $wB$ by $\beta|_{wB}$, or simply by $\beta|_w$ if no confusion seems likely to arise.

Suppose that $Y$ is a closed $K$-orbit.  We denote by $[Y] \in H_S^*(X)$ its $S$-equivariant fundamental class.  For the sake of clarity, we explain this abuse of notation.  To be precise, by $[Y]$ we mean the Poincar\'{e} dual to the direct image of the fundamental (equivariant) homology class of $Y$ in $H_*^S(X)$.  This is the unique equivariant cohomology class $\alpha \in H_S^*(X)$ having the property that $\alpha \cap [X] = [Y]$.

We describe in the next section how to compute $[Y]|_{wB}$ for $w \in W$.  Since $[Y]$ is completely determined by these restrictions, the idea is to compute them and then try to ``guess" a formula for $[Y]$ based on them.  For us, a ``formula for $[Y]$" is a polynomial in the variables $x_i$ and $y_i$ (defined in the statement of Proposition \ref{prop:eqvt-cohom-flag-var}) which represents $[Y]$.  Note that such a formula amounts to a particular choice of lift of $[Y]$ from $R' \otimes_{R^W} R$ to $R' \otimes_{\C} R$.

To be able to tell whether a given guess at a formula for $[Y]$ is correct, we must understand how the restriction maps $i_w^*$ work.  That is the content of the next proposition.

\begin{prop}\label{prop:restriction-maps}
Suppose that $\beta \in H_S^*(X)$ is represented by the polynomial $f = f(x,y)$ in the $x_i$ and $y_i$.  Then $\beta|_{wB} \in \Lambda_S$ is the polynomial $f(X,\rho(wY))$.  Here, $\rho$ denotes the restriction $\frt^* \rightarrow \frs^*$.
\end{prop}
\begin{proof}
It suffices to check that 
\[ x_i|_{wB} = X_i, \]
and that
\[ y_i|_{wB} = \rho(wY_i). \]

For the first, recall that the class $x_i \in H_S^*(X)$ is $\pi^*(X_i)$, where $\pi: X \rightarrow \{\text{pt.}\}$ is the map to a point, and $X_i \in \frs^*$ is a coordinate on $\frs$.  Letting $i_w$ denote the inclusion of the fixed point $wB$ into $X$, we have that $\pi \circ i_w = id$, so that $i_w^* \circ \pi^*$ is the identity on $H_S^*(\{wB\})$.  Thus $i_w^*(x_i) = i_w^*(\pi^*(X_i)) = X_i$, which is what is being claimed.

For the second, recall that $y_i$ is the $S$-equivariant Chern class $c_1^S(L_{Y_i}) = c_1(E \times^S L_{Y_i})$, with $Y_i \in \frt^*$.  Thus
\[ i_w^*(y_i) = i_w^*(c_1(E \times^S L_{Y_i})) = c_1(i_w^*(E \times^S L_{Y_i})). \]
The bundle $i_w^*(E \times^S L_{Y_i})$ over $BS$ is pulled back from the bundle $i_w^*(E \times^T L_{Y_i})$ over $BT$ through the natural map $BS \rightarrow BT$.  The bundle $i_w^*(E \times^T L_{Y_i})$ corresponds to a $T$-equivariant bundle over $\{wB\}$ (i.e. a representation of $T$) having weight $wY_i$, as one easily checks.  Thus the bundle $i_w^*(E \times^S L_{Y_i})$ corresponds to an $S$-equivariant bundle over $\{wB\}$ having $S$-weight $\rho(wY_i)$, since the pullback $\Lambda_T \rightarrow \Lambda_S$ through the map $BS \rightarrow BT$ is determined by restriction of characters.
\end{proof}

\section{Closed orbits}\label{ssec:closed_orbits}
Let $G, B, T, K, S, W$ be as in the introduction.  Let $\Phi = \Phi(G,T)$ denote the roots of $G$.  Let $\Phi^+$ denote the positive system of $\Phi$ such that the roots of $B$ are \textit{negative}, and denote $\Phi^- = -\Phi^+ = \Phi(B,T)$.  Let $X = G/B$ be the flag variety.

In our computations of equivariant classes, the closed orbits play a key role.  These are the orbits for whose classes we give formulas explicitly.  We use equivariant localization as described in the previous section to verify the correctness of these formulas.  Taking such formulas as a starting point, formulas for classes of remaining orbit closures can then be computed using divided difference operators, as explained in the next section.

In this subsection, we give the general facts regarding the closed orbits which we use to compute their equivariant classes.  By equivariant localization, to determine a formula for the $S$-equivariant class of a closed orbit, it suffices, at least in principle, to compute the restriction of this class at each $S$-fixed point.  We start by identifying the $S$-fixed points.  We know that the $T$-fixed points are finite, and indexed by $W$.  The question is whether $X^S$ can be larger than this, in the event that $S$ is a proper subtorus of $T$.  In fact, it cannot.  We refer to \cite{Brion-99} for the following result:

\begin{prop}[\cite{Brion-99}]\label{prop:s-fixed-points}
If $K = G^{\theta}$ is a symmetric subgroup of $G$, $T$ is a $\theta$-stable maximal torus of $G$, and $S$ is a maximal torus of $K$ contained in $T$, then $(G/B)^S = (G/B)^T$.
\end{prop}

With the $S$-fixed locus described, we now outline how the restriction of the class of a closed orbit to an $S$-fixed point can be computed explicitly.  The key fact that we use is the self-intersection formula.  To show that the self-intersection formula even applies, we first need the following easy result:

\begin{prop}\label{prop:closed-orbits-smooth}
	Suppose that $K$ is a connected symmetric subgroup of $G$.  Then each closed $K$-orbit is isomorphic to the flag variety for the group $K$.  In particular, any closed $K$-orbit is smooth.
\end{prop}  
\begin{proof}
	Suppose that $K \cdot gB$ is a closed orbit.  Then $K \cdot gB \cong K / \text{Stab}_K(gB)$, and clearly, $\text{Stab}_K(gB) = g^{-1}Bg \cap K$.  Because $K \cdot gB$ is a closed subvariety of $G/B$ and because $G/B$ is complete, $K \cdot gB$ is complete as well.  Thus $g^{-1}Bg \cap K$ is a parabolic subgroup of $K$ (\cite[\S 21.3]{Humphreys-75}).  Since it is contained in the Borel subgroup $g^{-1}Bg$ of $G$, it is solvable, and so it is in fact a Borel subgroup of $K$.  Thus $K \cdot gB$ is isomorphic to a quotient of $K$ by a Borel.
\end{proof}

Let $Y$ be a closed $K$-orbit, with $Y \stackrel{i}{\hookrightarrow} X$ the inclusion.  Recall that what we are trying to compute is a formula for the Poincar\'{e} dual $\alpha$ to the equivariant homology class $i_*([Y]) \in H_*^S(X)$.  (By abuse of notation, we will generally denote the class $\alpha$ by $[Y]$.)  By equivariant localization, this class     is determined by knowing $\alpha|_{wB}$ for each $w \in W$.  Suppose that $wB \in Y$.  Denote by $j_w$ the inclusion of $wB$ into $Y$, and by $i_w$ the inclusion of $wB$ into $X$, so that $i_w = i \circ j_w$.  Then in $H_*^S(X)$, we have the following:
\[ i_w^*(i_*([Y]) = (j_w^* \circ i^*)(i_*([Y])) = j_w^*((i^* \circ i_*)([Y])) =  \]
\[ j_w^*(c_d^S(N_Y X) \cap [Y]) = c_d^S(N_Y X|_{wB}) \cap j_w^*([Y]) = c_d^S(N_Y X|_{wB}) \cap [wB], \]
where $d$ is the codimension of $Y$ in $X$.  Here we have used some basic facts of intersection theory regarding pushforwards and pullbacks, for which the standard reference is \cite{Fulton-IT}.  Note that we are able to use the self-intersection formula because $Y$ is smooth, and hence $E \times^S Y$ is regularly embedded in $E \times^S X$. 

On the other hand, 
\[ i_w^*(i_*([Y])) = i_w^*(\alpha \cap [X]) = \alpha|_{wB} \cap i_w^*([X]) = \alpha|_{wB} \cap [wB]. \]
Then in $H_S^*(X)$, we have
\[ \alpha|_{wB} = c_d^S(N_Y X|_{wB}). \]

Thus computing the restriction of the class $\alpha$ at each $S$-fixed point amounts to computing $c_d^S(N_Y X|_{wB}) \in H_S^*(\{\text{pt.}\}) \cong \C[X_1,\hdots,X_r]$.  We want to compute this Chern class explicitly, as a polynomial in the $X_i$.  Note that the $S$-equivariant bundle $N_Y X|_{wB}$ is simply a representation of the torus $S$, and its top Chern class is the product of the weights of this representation.  We now compute these weights.

The $S$-module $N_Y X|_{wB}$ is simply $T_w X / T_w Y$, so we determine the weights of $S$ on $T_w X$ and $T_w Y$, then subtract the weights of $T_w Y$ from those of $T_w X$.  It is standard that 
\[ T_w X = \frg / \text{Ad}(w)(\frb). \]
Since $B$ has been taken to correspond to the negative roots, the weights of $S$ on $T_w X$ are the restrictions of the following weights of $T$ on $T_w X$:
\[ \Phi \setminus w \Phi^- = w \Phi^+. \]

A similar computation can be made for $T_w Y$.  We know that
\[ T_w Y = \frk / (\frk \cap \text{Ad}(w)(\frb)), \]
so the weights of $S$ on $T_w Y$ are as follows:
\[ \Phi_K \setminus (\Phi_K \cap w\Phi^-), \]
where $\Phi_K$ denotes the roots of $K$.  Subtracting this set of weights from those on $T_w X$, we conclude the following:

\begin{prop}\label{prop:restriction-of-closed-orbit}
The weights of $S$ on $N_Y X|_{wB}$ are $\rho(w\Phi^+) \setminus (\rho(w\Phi^+) \cap \Phi_K)$, where $\rho$ denotes restriction $\frt^* \rightarrow \frs^*$.
\end{prop}

Now that we have explained how the restrictions of the equivariant classes of closed orbits are computed, the next matter which must be dealt with is how to answer the following two questions for a given symmetric pair $(G,K)$:
\begin{enumerate}
	\item How many closed orbits are there?
	\item Which $S$-fixed points are contained in which closed orbits?
\end{enumerate}

It follows from Borel's fixed point theorem that any closed orbit must contain an $S$-fixed point.  Better yet, it follows from Proposition \ref{prop:closed-orbits-smooth} that for connected $K$, each closed orbit must contain $|W_K|$ $S$-fixed points, where $W_K = N_K(S)/S$ is the Weyl group for $K$.  However, for a given $S$-fixed point $wB$, it need not be the case that $K \cdot wB$ is a closed orbit.

To describe precisely which $K \cdot wB$ are closed in the way which will be most useful to us in our examples, we must first define twisted involutions and the Richardson-Springer map.  The references for what follows are \cite{Richardson-Springer-90,Richardson-Springer-92}.

First, observe that because $T$ is a $\theta$-stable torus, $N_G(T)$ is also $\theta$-stable, and hence there is an induced map (which we also call $\theta$) on $W$.

\begin{definition}
A \textbf{twisted involution} is an element $w \in W$ such that $w = \theta(w)^{-1}$.  We shall denote the set of twisted involutions by $\caI$.
\end{definition}

We now describe a map from $K \backslash G/B$ to $\caI$.  First, define the map
\[ \tau: G \rightarrow G \]
by $\tau(g) = g\theta(g)^{-1}$.  Next, define the set 
\[ \caV := \{g \in G \ \vert \ g\theta(g)^{-1} \in N_G(T) \} = \tau^{-1}(N_G(T)). \]
The set $\caV$ has a left $T$-action and a right $K$-action, and the orbit set $V = T \backslash \mathcal{V}/K$ is in bijective correspondence with $K \backslash G/B$.  (One direction of this bijection is given by $TgK \mapsto K \cdot g^{-1}B$.)  Given an element $v = TgK$ of $V$, we denote the corresponding $K$-orbit $K \cdot g^{-1}B$ by $\mathcal{O}(v)$.  The map
\[ \phi: \mathcal{V} \rightarrow W \]
given by $\phi(g) = \pi(\tau(g))$ (where $\pi$ is the natural projection $N_G(T) \rightarrow W$) is constant on $T \times K$ orbits, so we have a map (which we also call $\phi$) $\phi: V \rightarrow W$.    It is easy to check that $\phi$ actually maps $V$ into $\caI$.

\begin{remark}
Obviously, the map $\phi$ can also be thought of as a map $K \backslash G/B \rightarrow \caI$, defined by $\phi(\caO(v)) = \phi(v)$.  From this point forward, we will generally think of $\phi$ in this way, and will use notation such as $\phi(Q)$ for $Q \in K \backslash G/B$ an orbit, without explicitly mentioning a corresponding element of $V$.
\end{remark}

\begin{definition}\label{def:rs-map}
We will refer to the map $\phi: K \backslash G/B \rightarrow \caI$ defined above as the \textbf{Richardson-Springer map}.
\end{definition}

With these definitions made, we now give the following characterization of the closed orbits.

\begin{prop}[{\cite[Proposition 1.4.2]{Richardson-Springer-92}}]\label{prop:num-closed-orbits}
Let $w \in W$ be given.  The $K$-orbit $Q = K \cdot wB$ is closed if and only if $\phi(Q) = 1$.
\end{prop}

Note that if $W_K$ is the Weyl group for $K$, then $W_K$ is naturally a subgroup of $W$.  This is obvious in the event that $\text{rank}(K) = \text{rank}(G)$, so that $S = T$.  Then $N_K(T)$ is obviously a subgroup of $N_G(T)$, and $W_K = N_K(T)/T$ is obviously a subgroup of $W = N_G(T)/T$.  It is less obvious in the event that $S \subsetneq T$, since it is not \textit{a priori} clear that $N_K(S)$ is a subgroup of $N_G(T)$.  That it is follows from that fact that $T$ can be recovered as $Z_G(S)$, the centralizer of $S$ in $G$ (see \cite{Springer-85,Brion-99}).  Since any element of $G$ normalizing $S$ must also normalize $Z_G(S) = T$, we have an inclusion $N_K(S) \subset N_G(T)$.  This gives a map $W_K = N_K(S)/S \rightarrow N_G(T)/T = W$ defined by $nS \mapsto nT$.  The kernel of this map is $\{nS \mid n \in N_K(S) \cap T\}$.  Since $S = K \cap T$, the group $N_K(S) \cap T$ is simply $S$:
\[ N_K(S) \cap T = N_K(S) \cap (T \cap K) = N_K(S) \cap S = S. \]
Thus the kernel of the map $W_K \rightarrow W$ is $\{1\}$, and so it is an inclusion.

With this in mind, note that if $K \cdot wB$ is a closed orbit, the $S$-fixed points it contains correspond to elements of $W$ having the form $w'w$, with $w' \in W_K$ (viewed as an element of $W$ via the inclusion of Weyl groups we have just described).  Thus, by Proposition \ref{prop:num-closed-orbits}, the \textit{number} of closed orbits is $N / |W_K|$, where $N$ is the number of $w \in W$ with $\phi(w) = 1$.

In particular, we have the following easy corollary of Proposition \ref{prop:num-closed-orbits}, which applies to the majority of the cases that we consider in this paper.

\begin{corollary}\label{cor:num-closed-orbits-equal-rank}
Suppose that $\text{rank}(K) = \text{rank}(G)$, so that $S=T$.  Then $K \cdot wB$ is closed for all $w \in W$.  Thus the number of closed $K$-orbits is $|W|/|W_K|$.  Each orbit $K \cdot wB$ contains the $|W_K|$ $S$-fixed points corresponding to the elements of the left coset $W_K w$.
\end{corollary}
\begin{proof}
The equal rank condition is equivalent to the condition that $\theta$ be an inner involution (\cite[1.8]{Springer-87}).  An inner involution acts trivially on $W$, meaning that $\phi(w) = 1$ for all $w \in W$.
\end{proof}

Finally, we mention another characterization of the closed orbits from \cite{Richardson-Springer-92}, which we will make use of in one example.

\begin{prop}[{\cite[Proposition 1.4.3]{Richardson-Springer-92}}]\label{prop:num-closed-orbits-2}
For $w \in W$, the $K$-orbit $K \cdot wB$ is closed if and only if $wBw^{-1}$ is a $\theta$-stable Borel.
\end{prop}

\section{Other orbits}\label{ssec:other_orbits}
As alluded to in the previous section, there is an ordering on $K \backslash G/B$ with respect to which the closed orbits are minimal (\cite[Theorem 4.6]{Richardson-Springer-90}).  We describe this ordering.  Let $\ga \in \Delta$ be a simple root, and let $P_{\ga}$ be the minimal parabolic subgroup of $G$ of type $\ga$ containing $B$.  Consider the canonical map
\[ \pi_{\ga}: G/B \rightarrow G/P_{\ga}. \]
This is a $\mathbb{P}^1$-bundle.  Letting $Q \in K \backslash G/B$ be given, consider the set $Z_{\ga}(Q):=\pi_{\ga}^{-1}(\pi_{\ga}(Q))$.  The map $\pi_{\ga}$ is $K$-equivariant, so $Z_{\ga}(Q)$ is $K$-stable.  Assuming $K$ is connected, $Z_{\ga}(Q)$ is also irreducible, so it has a dense $K$-orbit.  In the event that $K$ is disconnected, one sees that the component group of $K$ acts transitively on the irreducible components of $Z_{\ga}(Q)$, and from this it again follows that $Z_{\ga}(Q)$ has a dense $K$-orbit.

If $\text{dim}(\pi_{\ga}(Q)) < \text{dim}(Q)$, then the dense orbit on $Z_{\ga}(Q)$ is $Q$ itself.  However, if $\text{dim}(\pi_{\ga}(Q)) = \text{dim}(Q)$, the dense $K$-orbit will be another orbit $Q'$ of one dimension higher.  In either event, using notation as in \cite{McGovern-Trapa-09}, we make the following definition:
\begin{definition}\label{def:weak_order_dot_notation}
With notation as above, $s_{\ga} \cdot Q$ shall denote the dense $K$-orbit on $Z_{\ga}(Q)$.
\end{definition}

\begin{definition}
The partial ordering on $K \backslash G/B$ generated by relations of the form $Q < Q'$ if and only if $Q' = s_{\ga} \cdot Q$ (with $\dim(Q') = \dim(Q) + 1$) for some $\ga \in \Delta$ is referred to as the \textbf{weak closure order}, or simply the \textbf{weak order}.
\end{definition}

Let $Y,Y'$ denote the closures of $Q,Q'$, respectively.  Assume that $Q' = s_{\ga} \cdot Q$, and define an operator $\partial_{\ga}$ on $H_S^*(X)$, known as a ``divided difference operator" or a ``Demazure operator",  as follows:
\[ \partial_{\ga}(f) = \frac{f - s_{\ga}(f)}{\ga}. \]

Let $d$ denote the degree of $\pi_{\ga}|_Y$ over its image.  Using standard facts from intersection theory, along with the fact that $\partial_{\ga} = \pi_{\ga}^* \circ (\pi_{\ga})_*$, it is easy to see that $[Y'] = \frac{1}{d} \partial_{\ga}([Y])$.

Putting all of this together, we see that we can recursively determine formulas for the equivariant classes of all orbit closures given the following data:
\begin{enumerate}
	\item Formulas for classes of the closed orbits.
	\item The weak closure order on $K \backslash G/B$.
	\item For any two orbits $Q,Q'$, with closures $Y, Y'$, and with the property that $Q' = s_{\ga} \cdot Q$, the degree $d$ of $\pi_{\ga}|_Y$ over its image.
\end{enumerate}

In fact, the aforementioned degree $d$ is always either 1 or 2, and this can be determined combinatorially based on the orbit $Q$ and the simple root $\ga$, as we now describe.  Before giving the precise statement, we will require some more preliminary definitions and observations.  All of what follows can be found in \cite{Richardson-Springer-90,Richardson-Springer-92}.

Let $T'$ be any $\theta$-stable maximal torus.  Then $\theta$ induces a map on the root system $\Phi(T',G)$ defined by $T'$.  We say that a root $\ga$ of $T'$ is
\begin{itemize}
	\item \textit{Real} if $\theta(\ga) = -\ga$.
	\item \textit{Complex} if $\theta(\ga) \neq \pm \ga$.
	\item \textit{Imaginary} if $\theta(\ga) = \ga$.  Within this case, there are two subcases:  $\ga$ is \textit{compact imaginary} if the root subgroup $G_\ga \subseteq K$, and \textit{non-compact imaginary} otherwise.
\end{itemize}

Let $\ga \in \Delta$ be a simple root of our fixed torus $T$.  Given any $K$-orbit $Q$, it is possible to find a representative $gB \in Q$ such that $T_Q = gTg^{-1}$ is a $\theta$-stable torus.  The automorphism $\text{int}(g)$ defines an isomorphism between the root systems $\Phi(T,G)$ and $\Phi(T_Q,G)$, and we categorize the root $\ga$ as \textit{real}, \textit{complex}, etc. for $Q$ if the root $\ga' = \text{int}(g)(\ga)$ has that property as a root of $T_Q$, as defined above.  (One checks that this is independent of the choice of representative $gB$.)

Next, we define the cross-action of $W$ on $K \backslash G/B$:
\begin{definition}
The \textbf{cross-action} of $W$ on $K \backslash G/B$, denoted $\times$, is defined by
\[ w \times (K \cdot gB) = K \cdot gw^{-1}B. \]
\end{definition}

With the cross-action defined, we can define non-compact imaginary roots of type I and II:
\begin{definition}
Suppose $\ga$ is a non-compact imaginary root for the orbit $Q$.  Then $\ga$ is of \textbf{type I} if $s_{\ga} \times Q \neq Q$, and of \textbf{type II} if $s_{\ga} \times Q = Q$.
\end{definition}

With these initial observations and definitions made, we can now state the following result on the weak closure order on $K \backslash G/B$, which tells us in particular how to determine the degree $d$ of the map $\pi_{\ga}|_Y$.  (Recall that $\phi$ denotes the Richardson-Springer map $K \backslash G/B \rightarrow \caI$, see Definition \ref{def:rs-map}.)

\begin{prop}\label{prop:nsc_for_orbit_lift}
Suppose $Q$ is a $K$-orbit on $G/B$ with closure $Y$.  Let $a = \phi(Q) \in \caI$, and let $\ga \in \Delta$ be given.  Then $s_{\ga} \cdot Q \neq Q$ (and hence $\dim(s_{\ga} \cdot Q) = \dim(Q) + 1$) if and only if one of the two following scenarios occurs:
\begin{enumerate}
	\item $\ga$ is complex for $Q$ and $l(s_{\ga} a \theta(s_{\ga})) = l(a) + 2$; or
	\item $\ga$ is non-compact imaginary for $Q$.
\end{enumerate}

In case (1) above, the map $\pi_{\ga}|_Y$ has degree 1 (i.e. is birational).  In case (2), $\pi_{\ga}|_Y$ is birational if $\ga$ is non-compact imaginary type I, and has degree 2 if $\ga$ is non-compact imaginary type II.
\end{prop}
\begin{proof}
We only briefly sketch what is involved in the proof.  A more detailed exposition can be found in \cite[Section 4]{Richardson-Springer-90}.  

One first establishes a correspondence between $K$-orbits on $Z_{\ga}(Q)$ and the orbits of $K(g,\ga) := K \cap gP_{\ga}g^{-1}$ on the fiber $\pi_{\ga}^{-1}(gP_{\ga})$.  This fiber being isomorphic to $\mathbb{P}^1$, there are only a few possibilities for the orbit structure.  This structure depends on the image of $h: K(g,\ga) \rightarrow \text{Aut}(\mathbb{P}^1)$ determined by the action of $K(g,\ga)$.  In the event that $\ga$ is complex for the orbit $Q$, this image contains a non-trivial unipotent subgroup, and there are two orbits on the $\mathbb{P}^1$ fiber:  one dense orbit and one fixed point.  The dense orbit corresponds to $Q$ itself if $l(s_{\ga} a \theta(s_{\ga})) = l(a) - 2$, but to an orbit one dimension higher in the event that $l(s_{\ga} a \theta(s_{\ga})) = l(a) + 2$.  

If $\ga$ is non-compact imaginary, then one of two cases occurs.  In case 1, the image $h(K(g,\ga))$ is a maximal torus.  In this case, there are three orbits on the fiber - one dense orbit and two fixed points.  In case 2, $h(K(g,\ga))$ is the normalizer of a maximal torus.  In this case, there are two orbits - one dense orbit and a two-point orbit.  Each point of the two-point orbit is fixed by the identity component $K(g,\ga)^0$, and the two points are permuted by $K(g,\ga)$.  Which case we are in depends on whether $\ga$ is type I or type II.  If $\ga$ is type I, we are in case 1, and if $\ga$ is type II, we are in case 2.

These various cases give us information about the $K$-orbits on $Z_{\ga}(Q)$, and we can see what the degree of $\pi_{\ga}|_Y$ over its image is in each case.  Indeed, when $\ga$ is complex and $l(s_{\ga} a \theta(s_{\ga})) = l(a) + 2$, or when $\ga$ is non-compact imaginary type I, $Q$ corresponds to a 1-point orbit on $\mathbb{P}^1$ (the lone fixed point in the former case, and one of the two fixed points in the latter).  When $\ga$ is non-compact imaginary type II, $Q$ corresponds to the two-point orbit on $\mathbb{P}^1$.  The number of points in the $K(g,\ga)$-orbit corresponding to $Q$ is easily seen to be the number of pre-images in $Q$ of any point of $\pi_{\ga}(Q) \subseteq G/P_{\ga}$.  The conclusion regarding the degree of $\pi_{\ga}|_Y$ follows.
\end{proof}

In \cite{Brion-01}, the graph for the weak order on $K$-orbit closures is endowed with additional data, as follows:  If $Y' = s_{\ga} \cdot Y \neq Y$, then the directed edge originating at $Y$ and terminating at $Y'$ is labelled by the simple root $\ga$, or perhaps by an index $i$ if $\ga = \ga_i$ for some predetermined ordering of the simple roots.  Additionally, if the degree of $\pi_{\ga}|_Y$ is $2$, then this edge is double.  (In other cases, the edge is simple.)  We modify this convention as follows:  Rather than use simple and double edges, in our diagrams we distinguish the degree two covers by blue edges, as opposed to the usual black.  (We do this simply because our weak order graphs were created using GraphViz, which does not, as far as the author can ascertain, have a mechanism for creating a reasonable-looking double edge.  On the other hand, coloring the edges is straightforward.)

\section{Symbolic parametrization of orbits}
In each individual case we consider, a symbolic parametrization of the orbit set, as well as a combinatorial description of the weak ordering in terms of this parametrization, is given.  This allows us to determine formulas for the classes of all orbit closures, starting with the closed orbits at the bottom of the ordering, and moving up by applying divided difference operators.

The details of individual cases are given in the corresponding sections, but here we give some general information and definitions which will be relevant when discussing each of the various cases.

\subsection{Twisted involutions} \label{ssec:twisted_involutions}
Recall the set $\caI$ of twisted involutions and the Richardson-Springer map $\phi$ defined in Subsection \ref{ssec:closed_orbits}.  These play an important role in the combinatorial description of $K \backslash G/B$ in some of the cases we consider --- namely, the three non-equal rank cases in type $A$, these being $K = SO(2n+1,\C)$, $K=SO(2n,\C)$, and $K=Sp(2n,\C)$.  (When $K$ is the special orthogonal group, the analysis differs depending on whether the rank of $G$ is even or odd, so we treat these as separate cases.)

At least in these cases, the weak ordering on $K \backslash G/B$ can be deduced combinatorially from an analogous ``weak Bruhat ordering" on $\mathcal{I}$.  To describe this, we must make a few more definitions.  First, define the ``twisted action" of (the group) $W$ on (the set) $W$ by
\[ a * w = aw \theta(a)^{-1}. \]
One checks easily that $\mathcal{I}$ is stable under the twisted action, whereby we have a $W$-action on $\mathcal{I}$.

Next, we define a monoid $M = M(W)$ associated to the Weyl group $W$.  As a set, the elements of $M$ are symbols $m(w)$, one for each $w \in W$.  The multiplication on $M$ is defined as follows:  Given $w \in W$ and $s \in S$ a simple reflection,
\[ m(s)m(w) = 
\begin{cases}
	m(sw) & \text{ if $l(sw) > l(w)$}, \\
	m(w) & \text{ otherwise.}
\end{cases} \]

There is an action of $M$ on the set $\mathcal{I}$.  Given $s \in S$ and $a \in \mathcal{I}$, define 
\[ m(s) * a =
\begin{cases} 
  a & \text{ if $l(sa) < l(a)$}, \\
  sa & \text{ if $l(sa) > l(a)$ and $s*a=a$}, \\
	s*a & \text{ otherwise.}
\end{cases} \]

The weak Bruhat order on twisted involutions can now be defined as follows:  Given twisted involutions $a, b$, we say that $a < b$ if $b \in M * a$ (and $a \neq b$).  In the event that $b = m(s_i) * a \neq a$ for some simple reflection $s_i \in S$, we will use the notation $a <_i b$.

There is also an $M$-action on $K \backslash G/B$ given simply by
\[ m(s_i) * Q = s_i \cdot Q. \]

(Recall Definition \ref{def:weak_order_dot_notation}.)  However, the map Richardson-Springer map $\phi$ (cf. Definition \ref{def:rs-map}) need not be $M$-equivariant for general $G$ and $K$.  The precise statement is as follows:

\begin{prop}[{\cite[Proposition 3.3.3]{Richardson-Springer-92}}]\label{prop:twisted_inv_map_phi_weak_order}
Let $Q,Q' \in K \backslash G/B$ and $s_i \in S$ be given.  Let $a = \phi(Q)$ and $a' = \phi(Q')$.  Then
	\begin{enumerate}
		\item If $Q' = s_i \cdot Q$, then $a <_i a'$.
		\item If $a <_i a'$, then $Q' = s_i \cdot Q$ unless $s_i$ is compact imaginary for the orbit $Q$.
	\end{enumerate}
\end{prop}
In particular, if $m(s_i) * a = a'$, and $s_i$ is compact imaginary for $Q$ (which means that $m(s_i) * Q = Q$), the map $\phi$ does not respect the $M$-action of $m(s_i)$.  The upshot is that in general, the weak order on $K \backslash G/B$ and the weak Bruhat order on $\mathcal{I}$ need not correspond as perfectly as one might hope, since it can occur that two twisted involutions $a$ and $a'$ are related in the weak Bruhat ordering on $\caI$, while orbits $Q$ and $Q'$ (mapping to $a,a'$, respectively) are \textit{not} related in the weak ordering on $K \backslash G/B$.  However, the following result implies that at least when the map $\phi$ is injective, this does not happen.

\begin{prop}[{\cite[Proposition 7.9, Part (i)]{Richardson-Springer-90}}]\label{prop:twisted_inv_paths}
Let $Q \in K \backslash G/B$ be given, with $a = \phi(Q)$.  Suppose that 
\[ id <_{i_1} a_1 <_{i_2} a_2 <_{i_3} \hdots <_{i_n} a_n = a. \]
Then there exists some closed $K$-orbit $Q'$ such that
\[ Q = s_{i_n} \cdot (s_{i_{n-1}} \cdot \hdots (s_{i_2} \cdot (s_{i_1} \cdot Q')) \hdots ). \]
\end{prop}

In particular, when $\phi$ is injective, we have the following:

\begin{corollary}\label{cor:weak_order_for_phi_injective}
Suppose $\phi: K \backslash G/B \rightarrow \caI$ is injective.  Let $Q_1,Q_2$ be any two orbits, with $a_1 = \phi(Q_1)$ and $a_2 = \phi(Q_2)$.  Then $Q_1 < Q_2$ if and only if $a_1 < a_2$.  Thus the weak order on $K \backslash G/B$ corresponds precisely to the order on $\phi(K \backslash G/B)$ induced by the weak order on $\caI$.
\end{corollary}
\begin{proof}
That $Q_1 < Q_2 \Rightarrow a_1 < a_2$ follows from Proposition \ref{prop:twisted_inv_map_phi_weak_order}, part (1).  So we prove only the direction $\Leftarrow$ here.

We first note that for any $i$, it is impossible to have three distinct twisted involutions $a_1,a_2,b$ with $a_1 <_i b$ and $a_2 <_i b$.  Suppose by contradiction that we have such a situation.  Let $s = s_i$.  By definition of the $M(W)$-action on $\caI$, the possible ways this could happen are
\begin{enumerate}
	\item $s * a_1 = a_1$, $s * a_2 = a_2$:  In this case, we would have that 
	\[ m(s) * a_1 = sa_1 = sa_2 = m(s) * a_2, \] 
	and this contradicts $a_1 \neq a_2$.
	\item $s * a_1 \neq a_1$, $s * a_2 \neq a_2$:  In this case, we have 
	\[ m(s) * a_1 = sa_1\theta(s)^{-1} = sa_2\theta(s)^{-1} = m(s) * a_2, \]
	again contradicting $a_1 \neq a_2$.
	\item $s * a_1 = a_1$, $s * a_2 \neq a_2$:  In this case, we would have $b = sa_1 = sa_2\theta(s)^{-1}$, so that $a_1 = a_2\theta(s)^{-1}$.  Recalling that $a_1,a_2$ are twisted involutions, so that $a_1 = \theta(a_1)^{-1}$ and $a_2 = \theta(a_2)^{-1}$, this says
	\[ \theta(a_1^{-1}) = \theta(a_2^{-1})\theta(s^{-1}) = \theta(a_2^{-1}s), \]
	so that
	\[ a_1^{-1} = a_2^{-1}s \Rightarrow sa_1 = a_2. \]
	But this contradicts that $sa_1 = b$.
\end{enumerate}

With this preliminary observation made, the proof is by induction on $\dim(Q_2) - \dim(Q_1)$.  Suppose first that this quantity is 1, and suppose that $a_1 <_i a_2$  for some simple root $\ga_i$.  By Proposition \ref{prop:twisted_inv_paths}, there is \textit{some} orbit $Q$ with $Q_2 = s_i \cdot Q$.  By Proposition \ref{prop:twisted_inv_map_phi_weak_order} part (1), $\phi(Q) <_i a_2$.  By our initial observation, this implies that $\phi(Q) = \phi(Q_1) = a_1$.  Since $\phi$ is injective, $Q = Q_1$, so $Q_2 = s_i \cdot Q_1$, as required.

Now suppose that $\dim(Q_2) - \dim(Q_1) > 1$, and suppose that $a_1 < a_2$.  Then by definition of the weak Bruhat order on $\caI$, there exists $a \in \caI$ with $a_1 < a <_i a_2$ for some simple root $\ga_i$.  By Proposition \ref{prop:twisted_inv_paths}, there is an orbit $Q$ with $Q_2 = s_i \cdot Q$.  Using our preliminary observation again, $\phi(Q) = a$.  By induction, $Q_1 < Q$, and $Q < Q_2$ by definition, so $Q_1 < Q_2$.
\end{proof}

The map $\phi$ is injective for the pairs $(SL(2n+1,\C),SO(2n+1,\C))$ and $(SL(2n,\C),Sp(2n,\C))$, so in these cases, we are able to describe $K \backslash G/B$ and its weak order entirely in terms of the combinatorics of $\caI$.

Next, we describe how to categorize simple roots as complex or non-compact imaginary (type I or type II) for a given orbit on the level of twisted involutions.  \begin{prop}[{\cite[Section 2.4]{Richardson-Springer-92}}]\label{prop:twisted_inv_complex_imaginary}
Let $Q$ be a $K$-orbit on $G/B$, with $a = \phi(Q)$.  Let a simple root $\ga \in \Delta$ be given, with $s \in S$ the corresponding simple reflection.  Then
	\begin{enumerate}
		\item $s$ (or $\ga$) is complex for $Q$ if and only if $a \theta(\ga) \neq \pm \ga$.
		\item $s$ (or $\ga$) is imaginary for $Q$ if and only if $a \theta(\ga) = \ga$.
	\end{enumerate}
\end{prop}

Recall that given a non-compact imaginary root, to determine whether it is type I or type II, we must compute the cross action of the appropriate simple reflection on the orbit in question.  The following proposition relates the cross action of $W$ on the set of orbits to the twisted action of $W$ on $\mathcal{I}$:
\begin{prop}[{\cite[Proposition 1.4.4]{Richardson-Springer-92}}]\label{prop:twisted_inv_phi_equivariant}
	The map $\phi: K \backslash G/B \rightarrow \mathcal{I}$ is $W$-equivariant with respect to the cross action of $W$ on $K \backslash G/B$ and the twisted action of $W$ on $\mathcal{I}$.
\end{prop}

Finally, we record one other fact from \cite{Richardson-Springer-90} regarding the \textit{full} closure order on $K \backslash G/B$ and its relation to the induced Bruhat order on $\mathcal{I}$ in the event that the map $\phi$ is injective:

\begin{prop}[{\cite[Proposition 9.14]{Richardson-Springer-90}}]\label{prop:full_closure_twisted}
Suppose that the map $\phi: K \backslash G/B \rightarrow \mathcal{I}$ is injective.  Then $K \backslash G/B$, equipped with the full closure order, is isomorphic as a partially ordered set to its image under $\phi$, when this image is given the partial order induced by the Bruhat order on $\mathcal{I}$.
\end{prop}

\begin{remark}\label{rmk:bruhat-order-on-I}
It is worth mentioning that the ``Bruhat order on $\caI$" referred to in Proposition \ref{prop:full_closure_twisted} is defined in \cite{Richardson-Springer-90}, not as the order on $\caI$ induced by the \textit{usual} Bruhat order on $W$, but as the weakest order on $\caI$ which is ``compatible," in a certain sense, with the weak order on $\caI$ defined above.  In fact, it is (incorrectly) stated in \cite{Richardson-Springer-90} that these two ``Bruhat orders" are \textit{not} necessarily the same.  This misstatement is corrected in \cite{Richardson-Springer-92}, with a proof that the two Bruhat orders in fact \textit{are} the same appearing in \cite{Richardson-Springer-94}.  Thus we think of the ``Bruhat order" on the image of $\phi$ to be induced by the usual Bruhat order on $W$.  That it is valid to do so will be important in Subsection \ref{ssec:type-a-other-deg-loci}.
\end{remark}

\subsection{Clans}\label{ssec:clans}
The information of the previous section is used primarily in the non-equal rank type $A$ cases, where $K = Sp(2n,\C)$, $SO(2n+1,\C)$, or $SO(2n,\C)$.  In all other cases, the parametrization of $K \backslash G/B$ is in terms of what are called ``clans".  These are first defined for $K$-orbits on $G/B$ where $G = SL(n,\C)$ and $K = S(GL(p,\C) \times GL(q,\C))$ for $p+q=n$.  Then a \textit{$(p,q)$-clan} is a string of $n$ characters, each either a $+$, a $-$, or a natural number.  Each natural number which appears must appear exactly twice, and the difference between the number of $+$'s and the number of $-$'s appearing in the string must be precisely $p-q$.  The set of $(p,q)$-clans parametrizes $K \backslash G/B$.  (See Subsection \ref{ssec:orbits_supq} for more precise details.)

For all other symmetric pairs $(G,K)$ (that is, for all cases where $G$ is of type $BCD$), it turns out that $K$ is a subgroup of $K' = GL(p,\C) \times GL(q,\C)$ for some appropriate choice of $p$ and $q$.  Further, $G$ is a subgroup of $G' = GL(n,\C)$ for some $n$, and the flag variety $X$ for $G$ naturally embeds in the flag variety $X'$ for $G'$.  As such, the intersection of a $K'$-orbit on $X'$ with $X$, if non-empty, is stable under $K$ and hence is a union of $K$-orbits.

In general, such an intersection need not be a single $K$-orbit.  Indeed, it could either be a single $K$-orbit or a union of $2$ $K$-orbits, and this depends critically on the chosen representative of the isogeny class of $G$, which in turn can affect the connectedness of $K$.  In the cases we consider, we choose $G$ (and the corresponding $K$) so that the intersection of a $K'$-orbit on $X'$ is always a single $K$-orbit.  In fact, we take this opportunity to state this now as a theorem:

\begin{theorem}\label{thm:k-orbit-intersections}
For each symmetric pair $(G,K)$ in types $BCD$ considered in this paper, each $K$-orbit on $G/B$ is exactly the intersection of a $GL(p,\C) \times GL(q,\C)$-orbit on the type $A$ flag variety with the smaller flag variety of the appropriate type, for some appropriate choice of $p$ and $q$.
\end{theorem}

The proof of Theorem \ref{thm:k-orbit-intersections} will be given in Appendix A.  The upshot of the theorem is that in each of the cases outside of type $A$, the set of $K$-orbits can be parametrized by a subset of the $(p,q)$-clans (for appropriate $p,q$) possessing some additional combinatorial properties which amount to the corresponding $GL(p,\C) \times GL(q,C)$-orbits meeting the smaller flag variety non-trivially.  Generally, these combinatorial properties involve at least a symmetry condition.  We make the following definitions now for later reference.  Let $\gamma = (c_1,\hdots,c_n)$ be a clan.

\begin{definition}\label{def:symmetric-clan}
We say that $\gamma$ is \textbf{symmetric} if the clan $(c_n,\hdots,c_1)$ obtained from $\gamma$ by reversing its characters is equal to $\gamma$ as a clan.  Specifically, we require
\begin{enumerate}
	\item If $c_i$ is a sign, then $c_{n+1-i}$ is the same sign.
	\item If $c_i$ is a number, then $c_{n+1-i}$ is also a number, and if $c_{n+1-i} = c_j$, then $c_{n+1-j} = c_i$.
\end{enumerate}
\end{definition}

\begin{definition}\label{def:skew-symmetric-clan}
We say that $\gamma$ is \textbf{skew-symmetric} if the clan $(c_n,\hdots,c_1)$ is the ``negative" of $\gamma$, meaning it is the same clan, except with all signs changed.  Specifically,
\begin{enumerate}
	\item If $c_i$ is a sign, then $c_{n+1-i}$ is the opposite sign.
	\item If $c_i$ is a number, then $c_{n+1-i}$ is also a number, and if $c_{n+1-i} = c_j$, then $c_{n+1-j} = c_i$.
\end{enumerate}
\end{definition}

Note that condition (2) of each of the above definitions allows for the possibility that $c_i = c_{n+1-i}$.  However, this is not necessary for a clan to be symmetric or skew-symmetric.  Indeed, the clan (1,2,1,2) is symmetric (and also skew-symmetric), since its reverse (2,1,2,1) is the same clan (this is explained in Subsection \ref{ssec:orbits_supq}), but there are no matching natural numbers in positions $(i,n+1-i)$ for any $i$.

Once it is established that the $K'$-orbit corresponding to $\gamma$ meets $X$ if and only if $\gamma$ possesses one of these properties (and perhaps meets some additional criteria), showing that each non-empty intersection of a $K'$-orbit on $X'$ with $X$ is a single $K$-orbit on $X$ requires a counting argument, the setup for which is a bit involved.  The details of this are described in Appendix A.

%% file: ch2-Type_A.tex
\chapter{Examples in Type $A$}
We consider first the group $G=SL(n,\C)$, the group of $n \times n$ complex matrices of determinant $1$.

For a maximal torus $T$ of $G$, let $Y_i$ denote coordinates on $\frt$, so that 
\[ \Phi = \{Y_i - Y_j \ \vert \ i \neq j \}. \]
We choose the ``standard" positive system 
\[ \Phi^+ = \{Y_i - Y_j \ \vert \ i < j \}, \]
and let $\Phi^- = -\Phi^+$.  Take $B$ to be the Borel subgroup containing $T$ and whose roots are $\Phi^-$.  (Concretely, we may take $T$ to be the diagonal elements of $G$, and $B$ to be the lower-triangular elements of $G$.  Then $\frt$ is the set of all trace-zero diagonal matrices, and $Y_i(\text{diag}(a_1,\hdots,a_n)) = a_i$ for each $i$.)

In this case, the Weyl group $W$ is isomorphic to the symmetric group $S_n$, and elements of $W$ act on the coordinates $Y_i$ by permutation of the indices.

\section{$K \cong S(GL(p,\C) \times GL(q,\C))$}\label{sect:type_a_supq}
Suppose that $p+q=n$.  Consider the involution $\theta$ of $G$ given by
\[ \theta(A) = I_{p,q} A I_{p,q}. \]

(Refer to Subsection \ref{ssec:notation} for this notation.)  Then
\[ K = G^\theta = \left\{
\left[
\begin{array}{cc}
K_{11} & 0 \\
0 & K_{22} \end{array}
\right] \in SL(n,\C)
\ \middle\vert \ 
\begin{array}{c}
K_{11} \in GL(p,\C) \\
K_{22} \in GL(q,\C) \end{array}
\right\} \cong S(GL(p,\C) \times GL(q,\C)). \]

In the notation of the introduction, this choice of $K$ corresponds to the real form $G_\R = SU(p,q)$ of $G$.

This is an equal rank case, so that the maximal torus $S$ of $K$ can be taken to be exactly $T$.  We still refer to the torus of $K$ as $S$, even in the cases where $S=T$.  We also refer to coordinates on $\frs$ by capital $X$ variables, here $X_1,\hdots,X_n$.  In this notation, the restriction map $\rho: \frt^* \rightarrow \frs^*$ is given by $\rho(Y_i) = X_i$ for $i=1,\hdots,n$.  Although this may seem a bit silly in the equal rank examples, the distinction is necessary in the non-equal rank cases, and it is helpful to keep our notational conventions consistent across all of the cases we consider.

The roots of $K$ are as follows:
\[ \Phi_K = \{ X_i-X_j \mid i,j \leq p \text{ or } i,j>p\}. \]

$W_K$ embeds in $W$ as those permutations of $\{1,\hdots,n\}$ which act separately on the subsets $\{1,\hdots,p\}$ and $\{p+1,\hdots,n\}$, i.e. $W_K \cong S_p \times S_q$.

\subsection{Formulas for the closed orbits}\label{ssec:closed-orbits-supq}
By Corollary \ref{cor:num-closed-orbits-equal-rank} and the above observation regarding $W_K$, there are 
\[ \dfrac{|W|}{|W_K|} = \dfrac{n!}{p!q!} = \dbinom{n}{p} \]
closed orbits, each containing $|W_K| = p!q!$ $S$-fixed points.

Given any $w \in W$, denote by $l_p(w)$ the number 
\[ l_p(w) := \# \{ (i,j) \ \vert \ 1 \leq i < j \leq n, w(j) \leq p < w(i) \}. \]

Then we have the following formula for the $S$-equivariant class of the closed orbit $K \cdot wB$:

\begin{prop}\label{prop:closed_orbit_formula_supq}
	Let $Q \in K \backslash G / B$ be a closed $K$-orbit containing the $T$-fixed point $w$.  Then $[Q]$ is represented by the polynomial
	\[ P(x,y) = (-1)^{l_p(w)} \displaystyle\prod_{i \leq p < j}(x_i - y_{w^{-1}(j)}). \]
\end{prop}
\begin{proof}
First, we observe that this formula is independent of the choice of $S$-fixed point $w$ representing the orbit.  Indeed, any other $S$-fixed point $\wt{w} \in Q $ is of the form $\wt{w} = \sigma w$ for some $\sigma \in W_K$.  Since $\sigma$ preserves the sets $\{1,\hdots,p\}$ and $\{p+1,\hdots,n\}$, we see that 
\[ (\sigma w)(j) \leq p < (\sigma w)(i) \Leftrightarrow w(j) \leq p < w(i), \]
and so $l_p(\wt{w}) = l_p(w)$.  Further, the set $\{w^{-1}(j) \ \vert \ j > p \}$ (that is, the set of indices on the $y$'s in our proposed formula) is clearly the same as $\{\wt{w}^{-1}(j) \ \vert \ j > p \} = \{(w^{-1}\sigma^{-1})(j) \ \vert \ j > p\}$, again because $\sigma^{-1}$ permutes those $j$ which are greater than $p$.

With that established, we now use Proposition \ref{prop:restriction-of-closed-orbit} to identify the restriction of $[Q]$ at each $S$-fixed point.  The set $\rho(w\Phi^+)$ is
\[ \rho(\{w \alpha \mid \alpha \in \Phi^+ \}) = \{ X_{w(i)} - X_{w(j)} \mid i < j \}. \]

Subtracting roots of $K$, we are left with precisely one of $\pm (X_i - X_j)$ for each $i,j$ with $i \leq p < j$.  The number of remaining roots of the form $-(X_i-X_j)$ is precisely $l_p(w)$.  Thus 
\[ [Q]|_w = F(X) := (-1)^{l_p(w)} \displaystyle\prod_{i \leq p < j}(X_i - X_j). \]

(We remark that because $l_p$ is constant on cosets $W_K w$, the restriction $[Q]|_w$ is actually the same at every $S$-fixed point $w \in Q$.)

So for any $u \in W$,
\[ [Q]|_u = 
\begin{cases}
	F(X) & \text{ if $u \in Q$}, \\
	0 & \text{ otherwise.}
\end{cases} \]
Recalling the precise definition of the restriction maps $i_u^*$ given in Proposition \ref{prop:restriction-maps}, we see that we are looking for a polynomial $p$ in the $x_i$ and $y_i$ such that $p(X,\sigma w(X)) = F(X)$ for any $\sigma \in W_K$, and such that $p(X,w'(X)) = 0$ for any $w' \in W$ such that $w'w^{-1} \notin W_K$.

It is straightforward to check that $P$ has these properties.  Indeed, for $\sigma \in W_K$, we see that 
\[ P(X,\sigma w(X)) = (-1)^{l_p(w)} \displaystyle\prod_{i \leq p < j}(X_i - X_{\sigma(j)}), \]
and since $\sigma$ permutes $\{p+1,\hdots,n\}$, this is precisely $F(X)$.

On the other hand, given $w'$ with $w'w^{-1} \notin W_K$,
\[ P(X,w'(X)) = (-1)^{l_p(w)} \displaystyle\prod_{i \leq p < j}(X_i - X_{w'w^{-1}(j)}) = 0, \]
since $w'w^{-1}$, not being an element of $W_K$, necessarily sends some $j > p$ to some $i \leq p$.  We conclude that $P(x,y)$ represents $[Q]$.
\end{proof}

\subsection{Parametrization of $K \backslash G/B$ and the weak order}\label{ssec:orbits_supq}
We remark here at the outset that the references we use for this parametrization typically refer to $GL(p,\C) \times GL(q,\C)$-orbits on the flag variety $GL(n,\C)/B$, rather than $S(GL(p,\C) \times GL(q,\C))$-orbits on $SL(n,\C)/B$.  However, these are the same, since any element of $GL(p,\C) \times GL(q,\C)$ differs from an element of $S(GL(p,\C) \times GL(q,\C))$ only by a scalar matrix, which acts trivially on the flag variety.  So if the reader prefers, he may just as well think of $GL(p,\C) \times GL(q,\C)$-orbits on $GL(n,\C)/B$.
  
The results detailed in this subsection appeared for the first time in \cite{Matsuki-Oshima-90}.  Proofs and further details appear in \cite{Yamamoto-97}.  The combinatorics are also given a nice exposition in \cite{McGovern-Trapa-09}, and much of our description is taken from there.

In this case, as mentioned briefly in the Subsection \ref{ssec:clans}, the orbits are parametrized by what are referred to as \textit{clans} of signature $(p,q)$ (or \textit{$(p,q)$-clans}).  

\begin{definition}
A \textbf{$(p,q)$-clan} is a string $(c_1,\hdots,c_n)$ of  $n=p+q$ characters, each of which is a $+$, a $-$, or a natural number, subject to the following conditions:
\begin{enumerate}
	\item Every natural number which appears must appear exactly twice.
	\item The difference in the number of $+$ signs and the number of $-$ signs must be $p-q$.  (If $q > p$, there are $q-p$ more minus signs than plus signs.)
\end{enumerate}
\end{definition}

We consider such strings only up to an equivalence which says, essentially, that it is the position of matching natural numbers, and not the numbers themselves, which determine the clan.  For instance, the clans $(1,2,1,2)$, $(2,1,2,1)$, and $(5,7,5,7)$ are all the same, since they all have matching natural numbers in positions $1$ and $3$, and also in positions $2$ and $4$.  On the other hand, $(1,2,2,1)$ is a different clan, since it has matching natural numbers in positions $1$ and $4$, and in positions $2$ and $3$.

As an example, suppose $n = 4$, and $p = q = 2$.  Then we must consider all clans of length $4$ where the number of $+$'s and the number of $-$'s is the same (since $p-q = 0$).  There are $21$ of these, and they are as follows:  
	\[ (+,+,-,-); (+,-,+,-); (+,-,-,+); (-,+,+,-); (-,+,-,+); (-,-,+,+); \]
	\[ (1,1,+,-); (1,1,-,+); (1,+,1,-); (1,-,1,+); (1,+,-,1); (1,-,+,1); \]
	\[ (+,1,1,-); (-,1,1,+); (+,1,-,1); (-,1,+,1); (+,-,1,1); (-,+,1,1); \]
	\[ (1,1,2,2); (1,2,1,2); (1,2,2,1) \]

We spell out precisely the correspondence between clans and $K$-orbits.  Let $E_p = \C \cdot \left\langle e_1,\hdots,e_p \right\rangle$ be the span of the first $p$ standard basis vectors, and let $\widetilde{E_q} = \C \cdot \left\langle e_{p+1},\hdots,e_n \right\rangle$ be the span of the last $q$ standard basis vectors.  Let $\pi: \C^n \rightarrow E_p$ be the projection onto $E_p$.

For any clan $\gamma=(c_1,\hdots,c_n)$, and for any $i,j$ with $i<j$, define the following quantities:
\begin{enumerate}
	\item $\gamma(i; +) = $ the total number of plus signs and pairs of equal natural numbers occurring among $(c_1,\hdots,c_i)$;
	\item $\gamma(i; -) = $ the total number of minus signs and pairs of equal natural numbers occurring among $(c_1,\hdots,c_i)$; and
	\item $\gamma(i; j) = $ the number of pairs of equal natural numbers $c_s = c_t \in \N$ with $s \leq i < j < t$.
\end{enumerate}

For example, for the $(2,2)$-clan $\gamma=(1,+,1,-)$,
\begin{enumerate}
	\item $\gamma(i; +) = 0,1,2,2$ for $i=1,2,3,4$;
	\item $\gamma(i; -) = 0,0,1,2$ for $i=1,2,3,4$; and
	\item $\gamma(i;j) = 1,0,0,0,0,0$ for $(i,j) = (1,2), (1,3), (1,4), (2,3), (2,4), (3,4)$.
\end{enumerate}

With all of this notation defined, we have the following theorem on $K$-orbits on $G/B$:
\begin{theorem}[\cite{Yamamoto-97}]\label{thm:orbit_description}
Suppose $p+q=n$.  For a $(p,q)$-clan $\gamma$, define $Q_{\gamma}$ to be the set of all flags $F_{\bullet}$ having the following three properties for all $i,j$ ($i<j$):
\begin{enumerate}
	\item $\dim(F_i \cap E_p) = \gamma(i; +)$
	\item $\dim(F_i \cap \widetilde{E_q}) = \gamma(i; -)$
	\item $\dim(\pi(F_i) + F_j) = j + \gamma(i; j)$
\end{enumerate}

For each $(p,q)$-clan $\gamma$, $Q_{\gamma}$ is nonempty and stable under $K$.  In fact, $Q_{\gamma}$ is a single $K$-orbit on $G/B$.

Conversely, every $K$-orbit on $G/B$ is of the form $Q_{\gamma}$ for some $(p,q)$-clan $\gamma$.  Hence the association $\gamma \mapsto Q_{\gamma}$ defines a bijection between the set of all $(p,q)$-clans and the set of $K$-orbits on $G/B$.
\end{theorem}
As in the statement of the above theorem, we will typically denote a clan by $\gamma$, and the corresponding orbit by $Q_{\gamma}$.

Next, we outline an algorithm, described in \cite{Yamamoto-97}, for producing a representative of $Q_{\gamma}$ given the clan $\gamma$.

First, for each pair of matching natural numbers of $\gamma$, assign one number a ``signature" of $+$, and assign the other a signature of $-$.  Each character $c_i$ in $\gamma$ is then said to have a signature of $+$ if $c_i$ is either a $+$ or a natural number which is assigned a signature of $+$, and a signature of $-$ otherwise.  Having done this, choose a permutation $\sigma$ of $1,\hdots,n$ with the following properties for all $i=1,\hdots,n$:

\begin{enumerate}
	\item $1 \leq \sigma(i) \leq p$ if the signature of $c_i$ is $+$.
	\item $p+1 \leq \sigma(i) \leq n$ if the signature of $c_i$ is $-$.
\end{enumerate}

Having determined such a permutation $\sigma$, take $F_{\bullet} = \left\langle v_1, \hdots, v_n \right\rangle$ to be the flag specified as follows:
\[ v_i = 
\begin{cases}
	e_{\sigma(i)} & \text{ if $c_i=\pm$}, \\
	e_{\sigma(i)} + e_{\sigma(j)} & \text{ if $c_i \in \N$, $c_i$ has signature $+$, and $c_i = c_j$}, \\
	-e_{\sigma(i)} + e_{\sigma(j)} & \text{ if $c_i \in \N$, $c_i$ has signature $-$, and $c_i = c_j$.}
\end{cases} \]

For example, for the orbit corresponding to the clan $(+,+,+,-,-,-)$, we could take $\sigma = 1$, which would give the standard flag $\left\langle e_1,\hdots,e_6 \right\rangle$.  For $(1,-,+,1)$, we could assign signatures to the $1$'s as follows:  $(1_+,-,+,1_-)$.  We could then take $\sigma$ to be the permutation $1324$.  This would give the flag 
	\[ F_{\bullet} = \left\langle e_1 + e_4, e_3, e_2, e_1-e_4 \right\rangle. \]

The closed orbits being those whose clans consist only of $+$'s and $-$'s, this algorithm tells us in particular how to determine an $S$-fixed point contained in such an orbit.  Indeed, any representative determined by the algorithm above for an orbit whose clan consists only of $+$'s and $-$'s necessarily produces an $S$-fixed flag corresponding to a permutation which assigns to the positions of the +'s the numbers $1,\hdots,p$, and to the positions of the $-$'s the numbers $p+1,\hdots,n$.

We now give a combinatorial description of the weak ordering on the orbit set in terms of this parametrization.  Let $\gamma=(c_1,\hdots,c_n)$.  The simple root $\ga_i = X_i - X_{i+1}$ is complex for the orbit $\caO_{\gamma}$, with $s_{\ga_i} \cdot \caO_{\gamma} \neq \caO_{\gamma}$, if and only if one of the following occurs:
\begin{enumerate}
	\item $c_i$ and $c_{i+1}$ are unequal natural numbers, and the mate of $c_i$ is to the left of the mate of $c_{i+1}$;
	\item $c_i$ is a sign, $c_{i+1}$ is a natural number, and the mate of $c_{i+1}$ is to the right of $c_{i+1}$;
	\item $c_i$ is a natural number, $c_{i+1}$ is a sign, and the mate of $c_i$ is to the left of $c_i$.
\end{enumerate}

So, for example, taking $p = 3, q = 2$, and letting $i=2$, $(1,1,2,+,2)$ satisfies the first condition; $(+,-,1,1,+)$ satisfies the second; $(1,1,+,-,+)$ satisfies the third; and $(1,+,1,2,2)$ satisfies none of them (since the mate of the 1 in the 3rd slot occurs to its \textit{left}, rather than to its right).

On the other hand, $\ga_i$ is non-compact imaginary for the orbit $Q_{\gamma}$ if and only if $c_i$ and $c_{i+1}$ are opposite signs.

Furthermore, one sees that the clan $\gamma'$ for $s_{\ga_i} \cdot \caO_{\gamma}$ is obtained from $\gamma$ by interchanging $c_i$ and $c_{i+1}$ in the complex case, and by replacing the opposite signs in the $c_i$ and $c_{i+1}$ slots by a pair of equal natural numbers in the non-compact imaginary case.  So, again taking $p=3,q=2,i=2$, we have, for example, 
\begin{itemize}
	\item $s_{\ga_2} \cdot (1,1,2,+,2) = (1,2,1,+,2)$;
	\item $s_{\ga_2} \cdot (+,-,1,1,+) = (+,1,-,1,+)$;
	\item $s_{\ga_2} \cdot (1,1,+,-,+) = (1,+,1,-,+)$;
	\item $s_{\ga_2} \cdot (1,+,-,1,+) = (1,2,2,1,+)$.
\end{itemize}

Finally, we describe the cross action of $W = S_n$ on the orbits in terms of this parametrization.  In fact, the action is the obvious one, given by permuting the symbols of any clan according to the underlying permutation of any $w \in W$.  (The most straightforward way to see this is to note the effect of simple transpositions on the representatives specified in \cite{Yamamoto-97}.)

Thus we see that if $\ga_i$ is non-compact imaginary for the orbit $Q_{\gamma}$, then $\gamma$ has $(c_i,c_{i+1})$ equal to either $(+,-)$ or $(-,+)$, and the action of $s_{\ga_i}$ (the simple transposition $(i,i+1)$) on $\gamma$ is to switch these signs.  In particular, $s_{\ga_i} \times Q_{\gamma} \neq Q_{\gamma}$, and so we see that all non-compact imaginary roots are of type I in this case.  This means that the weak order graph consists only of black edges, and no factors of $\frac{1}{2}$ are required in our divided difference computations.

\subsection{Example}\label{ssec:supq_example}
With this parametrization in hand, we now do an example, giving formulas for the classes of all orbit closures in the case where $n=4$, $p=q=2$.  (The clans parametrizing the orbits in this case are written down in the previous subsection.)

As we have noted, the method of \cite{Yamamoto-97} for producing representatives of the orbits always produces an $S$-fixed point when applied to a closed orbit.  This allows us to easily determine the formulas for the closed orbits using Proposition \ref{prop:closed_orbit_formula_supq}.  The closed orbit corresponding to $(+,+,-,-)$, as noted above, contains the $S$-fixed point corresponding to the identity element of $W$, so by Proposition \ref{prop:closed_orbit_formula_supq}, its class is represented by $(x_1 - y_3)(x_1 - y_4)(x_2 - y_3)(x_2 - y_4)$.  The closed orbit corresponding to $(+,-,+,-)$ contains the $S$-fixed point corresponding to $1324$, so its class is represented by $-(x_1-y_2)(x_1-y_4)(x_2-y_2)(x_2-y_4)$.  Formulas for the remaining closed orbits are obtained similarly.

Consider the orbit closure $Y_{(+,1,1,-)} = \overline{Q_{(+,1,1,-)}}$.  Because we have $s_{\ga_2} \cdot (+,+,-,-) = (+,1,1,-)$, we know that 
\[ [Y_{(+,1,1,-)}] = \partial_{\ga_2}([Q_{(+,+,-,-)}]). \]

Since
\[ \partial_{\ga_2}(f(x,y)) = \dfrac{f(x,y) - f(x,y_1,y_3,y_2,y_4)}{y_2 - y_3}, \]
we see that 
\[ [Y_{(+,1,1,-)}] = (x_1-y_4)(x_2-y_4)(x_1+x_2-y_2-y_3). \]

Similarly, since $s_{\ga_1} \cdot (+,-,+,-) = (1,1,+,-)$, we have
\[ [Y_{(1,1,+,-)}] = \partial_{\ga_1}([Q_{(+,-,+,-)}]), \]
which one computes to be
\[ [(1,1,+,-)] = -(x_1-y_4)(x_2-y_4)(x_1+x_2-y_1-y_2). \]

With knowledge of the weak order in hand, we can compute formulas for the classes of all orbit closures in a similar manner, moving recursively up from the closed orbits.  The weak order graph, along with formulas for all orbit closures, appear as Figure \ref{fig:type-a-2-2} and Table \ref{tab:type-a-2-2} of Appendix B.

\section{$K \cong SO(2n+1,\C)$}
We realize $K = SO(2n+1,\C)$ as the subgroup of $G$ preserving the quadratic form given by the antidiagonal matrix $J = J_{2n+1}$.  That is, $K = G^{\theta}$ where $\theta$ is the involution
\[ \theta(g) = J(g^{-1})^tJ. \]

This choice of $K$ corresponds to the real form $G_\R = SL(2n+1,\R)$ of $G$.

This realization of $K$ is in fact conjugate to the ``usual" one, that being the fixed point set of the involution $\theta'(g) = (g^{-1})^t$.  We prefer our choice of realization because we can take a maximal torus $S = K \cap T$ consisting of diagonal elements, and a Borel subgroup $B$ consisting of lower-triangular elements.

The torus $\frs = \text{Lie}(S)$ has the form $\text{diag}(a_1,\hdots,a_n,0,-a_n,\hdots,-a_1)$.  Thus if $Y_1,\hdots,Y_{2n+1}$ represent coordinates on $\frt$, restricting to $\frs$ we have $\rho(Y_{n+1}) = 0$, and $\rho(Y_i) = X_i$, $\rho(Y_{2n+2-i}) = -X_i$ for $i=1,\hdots,n$.

The roots of $K$ are as follows:  
\begin{itemize}
	\item $\pm X_i$ ($i = 1,\hdots,n$)
	\item $\pm (X_i + X_j)$ ($1 \leq i < j \leq n)$
	\item $\pm (X_i - X_j)$ ($1 \leq i < j \leq n)$
\end{itemize}

The Weyl group $W_K$ of $K$ should be thought of as consisting of signed permutations of $\{1,\hdots,n\}$ (changing any number of signs).  This is the action of $W_K$ on the coordinates $X_i \in \frs^*$.  $W_K$ is embedded in $W$ as signed elements of $S_{2n+1}$, as described in Subsection \ref{ssec:notation}.

\subsection{A formula for the closed orbit}
As it turns out, there is a unique closed orbit in this case.

\begin{prop}
There is precisely one closed $K$-orbit on $G/B$.  In our chosen realization, it is the orbit $K \cdot 1B$, and contains the $S$-fixed points corresponding to elements of $W_K$, embedded as signed elements of $S_{2n+1}$.
\end{prop}
\begin{proof}
We use Proposition \ref{prop:num-closed-orbits}.  Note that given our particular choice of $\theta$, the induced map on $W$ is $w \mapsto w_0 w w_0$.  By Proposition \ref{prop:num-closed-orbits}, $K \cdot wB$ is closed if and only if $\theta(w) = w$.  The condition that
\[ w_0 w w_0 = w \]
is clearly equivalent to the condition that $w(2n+2-i) = 2n+2-w(i)$, since $w_0(i) = 2n+2-i$ by definition.  This is precisely the definition of a signed element of $S_{2n+1}$.  As we have noted, the signed elements of $S_{2n+1}$ are precisely the images of elements of $W_K$ under the embedding $W_K \subseteq W$.
\end{proof}

We now give a formula for the $S$-equivariant class of the lone closed orbit.

\begin{prop}\label{prop:formula_for_SO_odd}
Let $Q=K \cdot 1B$ be the closed $K$-orbit of the previous proposition.  Then $[Q]$ is represented by
\[ P(x,y) := (-2)^n \displaystyle\prod_{i=1}^n (y_i + y_{n+1})(y_{n+1} + y_{2n+2-i})\displaystyle\prod_{1 \leq i < j \leq n}(y_i + y_j)(y_i + y_{2n+2-j}). \]
\end{prop}
\begin{proof}
We once again apply Proposition \ref{prop:restriction-of-closed-orbit} to determine the restriction $[Q]|_w$ at a fixed point $w \in Q$.  To compute the set $\rho(w \Phi^+)$, we determine the restrictions of the positive roots $\Phi^+$ to $\frs$, then apply the signed permutation corresponding to $w$ to that set of weights.  (The result is the same as if we viewed $w$ as a signed element of $S_{2n+1}$, applied that permutation to the elements of $\Phi^+$, and \textit{then} restricted the resulting roots to $\frs$.)

Restricting the positive roots $\{Y_i - Y_j \ \vert \ 1 \leq i < j \leq 2n+1 \}$ to $\frs$, we get the following set of weights:

\begin{enumerate}
	\item $X_i - X_j$, $1 \leq i < j \leq n$, each with multiplicity 2 (one is the restriction of $Y_i - Y_j$, the other the restriction of $Y_{2n+2-j} - Y_{2n+2-i}$)
	\item $X_i + X_j$, $1 \leq i < j \leq n$, each with multiplicity 2 (one is the restriction of $Y_i - Y_{2n+2-j}$, the other the restriction of $Y_j - Y_{2n+2-i}$)
	\item $X_i$, $1 \leq i \leq n$, each with multiplicity 2 (one is the restriction of $Y_i - Y_{n+1}$, the other the restriction of $Y_{n+1} - Y_{2n+2-i}$)
	\item $2X_i$, $1 \leq i \leq n$, each with multiplicity 1 (the restriction of $Y_i- Y_{2n+2-i}$)
\end{enumerate}

Now, consider applying a signed permutation $w$ to this set of weights.  The resulting set of weights will be

\begin{enumerate}
	\item For each $i,j$ ($1 \leq i < j \leq n$), either $X_i - X_j$ or $-(X_i - X_j)$, occurring with multiplicity 2 (these weights come from applying $w$ to weights of either type (1) or (2) above);
	\item For each $i,j$ ($1 \leq i < j \leq n$), either $X_i + X_j$ or $-(X_i+X_j)$, occurring with multiplicity 2 (these weights also come from applying $w$ to weights of either type (1) or (2) above);
	\item For each $i$ ($1 \leq i \leq n$), either $X_i$ or $-X_i$, occurring with multiplicity 2 (these weights come from applying $w$ to weights of type (3) above);
	\item For each $i$ ($1 \leq i \leq n$), either $2X_i$ or $-2X_i$, ocurring with multiplicity 1 (these weights come from applying $w$ to weights of type (4) above).
\end{enumerate}

Discarding roots of $K$, we are left with the following weights:

\begin{enumerate}
	\item For each $i,j$ ($1 \leq i < j \leq n$), either $X_i - X_j$ or $-(X_i-X_j)$, occurring with multiplicity 1;
	\item For each $i,j$ ($1 \leq i < j \leq n$), either $X_i + X_j$ or $-(X_i+X_j)$, occurring with multiplicity 1;
	\item For each $i$ ($1 \leq i \leq n$), either $X_i$ or $-X_i$, occurring with multiplicity 1;
	\item For each $i$ ($1 \leq i \leq n$), either $2X_i$ or $-2X_i$, occurring with multiplicity 1.
\end{enumerate}

It is clear that the number of weights of the form $-X_i$ and the number of weights of the form $-2X_i$ are the same, so weights of those two forms account for an even number of negative signs.  So in computing the restriction, to get the sign right, we need only concern ourselves with the signs of the weights of types (1) and (2) above.

We claim that the number of $X_i \pm X_j$ ($i < j$) occurring with a negative sign is congruent mod $2$ to $l(|w|)$.  (Cf. Subsection \ref{ssec:notation} for this notation.)  Indeed, suppose first that $|w|$ does not invert $i$ and $j$, so that $k = |w(i)| < |w(j)| = l$.  Then there are four possibilities:

\begin{enumerate}
	\item $w(i)$, $w(j)$ are both positive.  In this case, $X_{w(i)} + X_{w(j)} = X_k + X_l$, and $X_{w(i)} - X_{w(j)} = X_k - X_l$.  Neither of these is a negative root.
	\item $w(i)$ is negative, and $w(j)$ is positive.  Then $X_{w(i)} + X_{w(j)} = -(X_k - X_l)$, and $X_{w(i)} - X_{w(j)} = -(X_k + X_l)$.  Both of these are negative roots.
	\item $w(i)$ is positive, and $w(j)$ is negative.  Then $X_{w(i)} + X_{w(j)} = X_k - X_l$, and $X_{w(i)} - X_{w(j)} = X_k + X_l$.  Neither of these is a negative root.
	\item $w(i)$, $w(j)$ are both negative.  Then $X_{w(i)} + X_{w(j)} = -(X_k + X_l)$, and $X_{w(i)} - X_{w(j)} = -(X_k - X_l)$.  Both of these are negative roots.
\end{enumerate}

All this is to say that if $|w|$ does not invert $i$ and $j$, then this accounts for an even number of negative signs occurring in the restriction.  On the other hand, if $|w|$ does invert $i$ and $j$, so that $k = |w(j)| < |w(i)| = l$, then again there are four possibilities:

\begin{enumerate}
	\item $w(i)$, $w(j)$ are both positive.  In this case, $X_{w(i)} + X_{w(j)} = X_k + X_l$, and $X_{w(i)} - X_{w(j)} = -(X_k - X_l)$.  One of these is a negative root.
	\item $w(i)$ is negative, and $w(j)$ is positive.  Then $X_{w(i)} + X_{w(j)} = X_k - X_l$, and $X_{w(i)} - X_{w(j)} = -(X_k + X_l)$.  One of these is a negative root.
	\item $w(i)$ is positive, and $w(j)$ is negative.  Then $X_{w(i)} + X_{w(j)} = -(X_k - X_l)$, and $X_{w(i)} - X_{w(j)} = X_k + X_l$.  One of these is a negative root.
	\item $w(i)$, $w(j)$ are both negative.  Then $X_{w(i)} + X_{w(j)} = -(X_k + X_l)$, and $X_{w(i)} - X_{w(j)} = X_k - X_l$.  One of these is a negative root.
\end{enumerate}

The upshot is that if $w \in Q$ is an $S$-fixed point, then

\[ [Q]|_w = F(X) := (-1)^{l(|w|)} 2^n \displaystyle\prod_{i = 1}^n X_i^2 \displaystyle\prod_{1 \leq i < j \leq n} (X_i + X_j)(X_i-X_j). \]

So we seek a polynomial in $x_1,\hdots,x_n,y_1,\hdots,y_{2n+1}$, say $p$, with the property that

\[ p(X,\rho(wY)) = 
\begin{cases}
F(X) & \text{ if $w \in W_K$} \\
0 & \text{ otherwise.}
\end{cases} \]

It is straightforward to check that $P(x,y)$ has these properties.  Indeed, suppose first that $w \in W_K$.  (We should think of $w$ here as a signed element of $S_{2n+1}$, since this is how $w$ acts on the $Y_i$.)  Consider first the factors $y_i + y_{n+1}$ and $y_{n+1}+y_{2n+2-i}$ for $i=1,\hdots,n$.  Supposing $w(i) \leq n$, $Y_i + Y_{n+1}$ gives $Y_{w(i)} + Y_{n+1}$, which restricts to $X_{w(i)} + 0 = X_{w(i)}$.  On the other hand, $Y_{n+1} + Y_{2n+2-i}$ gives $Y_{n+1} + Y_{w(2n+2-i)}$, which restricts to $0 - X_{w(i)} = -X_{w(i)}$, so the product $(y_i + y_{n+1})(y_{n+1}+y_{2n+2-i})$ restricts to $-X_{w(i)}^2$.  If $w(i) > n+1$, then the product of these two terms restricts to $-X_{2n+2-w(i)}^2$, with the negative term coming from $y_i + y_{n+1}$, and the positive term coming from the $y_{n+1} + y_{2n+2-i}$.  As $i$ runs from $1$ to $n$, the product of all these terms restricts to $(-1)^n \prod_{i=1}^n X_i^2$.  This explains the factor of $(-2)^n$ in our formula, as opposed to just $2^n$.  The $(-1)^n$ is to account for a possible sign flip coming from terms of this type.  So the terms $(-2)^n \prod_{i=1}^n (y_i + y_{n+1})(y_{n+1} + y_{2n+2-i})$ of our putative formula contribute the $2^n \prod_{i=1}^n X_i^2$ portion of the required restriction.

Next, consider the terms $y_i + y_j$ and $y_i + y_{2n+2-j}$.  Applying $w$ and restricting, these give (up to sign) all required terms of the form $X_i + X_j$ and $X_i - X_j$ ($i<j$).  Writing each such term as either $+1$ or $-1$ times a positive root by factoring out negative signs as necessary, we effectively introduce the sign of $(-1)^{l(|w|)}$, as required.

On the other hand, given any $w \notin W_K$ (i.e. a \textit{non-signed} element of $S_{2n+1}$), there are two possibilities:

\textit{Case 1:  $w$ does not fix $n+1$}

In this case, $w$ moves $n+1$ to some $i$ such that $1 \leq i \leq 2n+1$, and $i \neq n+1$.  Let $j = w^{-1}(2n+2-i)$.  (Note, of course, that $j \neq n+1$.)  Applying $w$ to $y_j + y_{n+1}$, we get $Y_{2n+2-i} + Y_i$, which restricts to $0$.

\textit{Case 2:  $w$ fixes $n+1$}

In this case, $w(2n+2-i) \neq 2n+2-w(i)$ for some $1 \leq i \leq n$.  Let $j = 2n+2-w(i)$, and let $k = w^{-1}(j)$.  Clearly, $k \neq i$, $2n+2-i$, or $n+1$, so the factor $y_i + y_k$ appears in $P$.  Applying $w$ to this factor gives $Y_{w(i)} + Y_{2n+2-w(i)}$, which restricts to zero.

We see that in either case, applying $w$ then restricting kills one of the factors appearing in $P$, and so the result is zero for any $w \notin W_K$, as desired.  This completes the proof.
\end{proof}

\begin{remark}
An alternate representative of $[Q]$ is
\[ P(x,y) := 2^n \displaystyle\prod_{i=1}^n (x_i + y_{n+1})(x_i - y_{n+1})\displaystyle\prod_{1 \leq i < j \leq n}(y_i + y_j)(y_i + y_{2n+2-j}). \]
Indeed, this was the first representative discovered by the author.  However, the representative of the previous proposition is preferable from our perspective, essentially because a formula involving only the $y_i$ will pull back to a Chern class formula for the class of a certain degeneracy locus.  It is not clear that the representative involving the $x_i$ should have such an interpretation.
\end{remark}

\subsection{Parametrization of $K \backslash G/B$ and the weak order}\label{ssec:kgb_param_so_odd}
We shall refer freely to the definitions and notation of Subsection \ref{ssec:twisted_involutions}.

We can consider either the symmetric pair $(SL(2n+1,\C),SO(2n+1,\C))$ or $(GL(2n+1,\C),O(2n+1,\C))$.  These orbits coincide, since in the odd case, one can pass from one component of $O(2n+1,\C)$ to the other by multiplication by $-1$, which acts trivially on the flag variety.  If one thinks of the case $(GL(2n+1,\C),O(2n+1,\C))$, then the parametrization we describe here applies equally well to the even case $(GL(2n,\C),O(2n,\C))$.  However, in the even case, $SO(2n,\C)$-orbits on the flag variety no longer coincide with $O(2n,\C)$-orbits, so parametrizing those orbits is slightly more complicated.  (See Subsection \ref{ssec:so2n_param}.)

Whichever $(G,K)$ one prefers, in this case the Richardson-Springer map $\phi: K \backslash G/B \rightarrow \caI$ is a bijection (\cite[Examples 10.2,10.3]{Richardson-Springer-90}).  By Corollary \ref{cor:weak_order_for_phi_injective}, then, the weak ordering can be determined solely in terms of the twisted action of $M(W)$ on $\mathcal{I}$.  Since the map induced by $\theta$ on $W$ (which we also call $\theta$) is given by $\theta(w) = w_0 w w_0^{-1}$ ($w_0$ the longest element of $W$), we have

\[ \mathcal{I} = \{ w \in W \ \vert \ w_0 w w_0^{-1} = w^{-1} \}. \]

As for the weak ordering, we have the following:

\begin{prop}\label{prop:weak_order_so_odd}
Suppose that $s_i \in S$, and that $a \in \mathcal{I}$ is such that $m(s_i) * a \neq a$ (i.e. $l(s_ia) > l(a)$).  Let $v = \phi^{-1}(a)$.  If $s_i * a \neq a$, then $s_i$ is complex for $\mathcal{O}(v)$, and if $s_i * a = a$, $s_i$ is non-compact imaginary type II for $\mathcal{O}(v)$.
\end{prop}
\begin{proof}
Note first that the map induced by $\theta$ on the roots $\Phi(G,T)$ (which we also denote by $\theta$) is determined by $\theta(Y_i) = -Y_{2n+2-i}$.  Also note that the twisted action of $W$ on $\mathcal{I}$ is given by $w * a = w a w_0 w^{-1} w_0^{-1} = w a w_0 w^{-1} w_0$.
	
We first show that $s_i$ is complex for $\mathcal{O}(v)$ if and only if $s_i * a \neq a$.  This follows from Proposition \ref{prop:twisted_inv_complex_imaginary} and elementary combinatorics.  We know by Proposition \ref{prop:twisted_inv_complex_imaginary} that $s_i$ is complex if and only if $a\theta(\ga_i) \neq \pm \ga_i$.  Well, $\theta(\ga_i) = \theta(Y_i - Y_{i+1}) = Y_{2n+1-i} - Y_{2n+2-i}$.  Then $a\theta(\ga_i)$ is equal to $\pm \ga_i$ if and only if $\{a(2n+1-i),a(2n+2-i)\} = \{i,i+1\}$.
	
Since 
\[ s_i * a = s_i a w_0 s_i w_0, \]
and since
\[ s_i = (i,i+1) \]
and
\[ w_0 = (1,2n+1)(2,2n) \hdots (n,n+2), \]
if $s_i * a = a$, then in particular,
\[ (s_i a w_0 s_i w_0)(2n+2-i) = s_i a (2n+1-i) = a(2n+2-i), \]
and
\[ (s_i a w_0 s_i w_0)(2n+1-i) = s_i a (2n+2-i) = a(2n+1-i). \]

(Note that either of these implies the other.)  This happens if and only if $\{a(2n+1-i),a(2n+2-i)\} = \{i,i+1\}$, showing that $s_i$ is complex for $\mathcal{O}(v)$ if and only if $s_i * a \neq a$.

Now, since $\phi$ is a bijection, it follows from Corollary \ref{cor:weak_order_for_phi_injective} that if $m(s_i) * a \neq a$, as we have assumed, then $s_i \cdot \caO(v) \neq \caO(v)$, meaning that $s_i$ is either complex or non-compact imaginary for $\caO(v)$ (Proposition \ref{prop:nsc_for_orbit_lift}).  So if $s_i * a = a$, since $s_i$ is not complex, it must be non-compact imaginary.  The fact that it is type II follows from Proposition \ref{prop:twisted_inv_phi_equivariant}, combined with the fact that $s_i * a = a$.
\end{proof}

This description of $K \backslash G/B$ and the weak order allows for explicit computations on the level of twisted involutions, simply by starting at the bottom (the identity twisted involution, which corresponds to the unique closed orbit) and computing the $M$-action all the way up.

One can also parametrize the set of orbits by honest involutions in the Weyl group.  This parametrization is preferable, in part because it allows for a straightforward linear algebraic description of the orbits and their closures in terms of rank conditions on the form used to define $G$ (see below).

The translation between twisted involutions and honest involutions is made simply by multiplying on the right by $w_0$.  We know from Proposition \ref{prop:full_closure_twisted} that the full closure order on the set of $K \backslash G/B$ in this case corresponds precisely to the induced Bruhat order on $\mathcal{I}$, and since right multiplication by $w_0$ inverts this order, it follows that when we describe $K \backslash G/B$ by the set of honest involutions in $W$, the full closure order is given by the reverse Bruhat order.  So in this setting, $w_0$ corresponds to the unique closed orbit.

Further, if $a \in \mathcal{I}$ is a twisted involution, then checking whether $s_i * a = a$, i.e. whether 
\[ s_i a w_0 s_i w_0 = a, \]
obviously amounts to checking whether the involution $b=aw_0$ is fixed under conjugation by $s_i$:
\[ s_i a w_0 s_i w_0 = a \Leftrightarrow s_i (aw_0) s_i = aw_0 \Leftrightarrow s_i b s_i = b. \]
 
In light of these easy observations and Proposition \ref{prop:weak_order_so_odd} above, we have the following:

\begin{corollary}\label{cor:weak_order_involutions}
When $K \backslash G/B$ is parametrized by the set of involutions in $W$, the unique closed orbit corresponds to $w_0$.  The weak order can be generated inductively, starting at $w_0$ and moving up, as follows:  Given an involution $b \in W$ and a simple reflection $s_i$ such that $l(s_ib) < l(b)$, we have one of the following two scenarios (recall our conventions on black/blue edges, described immediately after the proof of Proposition \ref{prop:nsc_for_orbit_lift}):
	\begin{enumerate}
		\item $s_i b s_i \neq b$, in which case $b <_i s_i b s_i$, and the edge in the weak order graph is black.
		\item $s_i b s_i = b$, in which case $b <_i s_i b$, and the edge in the weak order graph is blue.
	\end{enumerate}
\end{corollary}

We remark that the material of this subsection either appears explicitly in \cite{Richardson-Springer-90}, or follows easily from the content of that paper.  Further, the parametrization and description of the weak order given here is easily seen to be equivalent to that of \cite{Matsuki-Oshima-90}.  However,  the question of black/blue edges is not directly addressed in either paper.

The parametrization of $K \backslash G/B$ by involutions is convenient because an involution $b \in W$ encodes a linear algebraic description of the orbit corresponding to $b$ in a straightforward way.  Namely, given an involution $b$, define, for any $i$ and $j$,
\[ r_b(i,j) := \#\{ k \leq i \ \vert \ b(k) \leq j \}. \]
Let $V = \C^{2n+1}$, and let $\gamma: V \otimes V \rightarrow \C$ denote the orthogonal form with isometry group $K$.  For any flag $F_{\bullet} = (F_1 \subset \hdots \subset F_{2n+1}) \in X$, denote by $\gamma|_{F_i \times F_j}$ the restriction of $\gamma$ to pairs of the form $(v,w)$ with $v \in F_i$ and $w \in F_j$.  The claim is that if $b \in W$ is an involution, then the set
\[ \caO_b := \{ F_{\bullet} \in X \ \vert \ \text{rank}(\gamma|_{F_i \times F_j}) = r_b(i,j) \text{ for all } i,j \} \]
is a $K$-orbit on $G/B$, and that the association $b \mapsto \caO_b$ defines a bijection between involutions in $W$ and $K$-orbits.

Since any permutation $w$ is uniquely determined by the set of numbers $r_w(i,j)$, it is clear by definition that the $\caO_b$ are mutually disjoint.  It is also clear that each set $\caO_b$, if non-empty, is stable under $K$ and hence is at least a union of $K$-orbits.  If we can see that every $\caO_b$ is non-empty, it will follow that each must be a \textit{single} $K$-orbit.  Indeed, by the results of \cite{Richardson-Springer-90}, we know that the orbits are in bijection with the involutions of $W$.  If each $\caO_b$ is non-empty, then it is impossible for any one of them to be anything other than a single $K$-orbit, for then there would be more $K$-orbits than involutions.

Thus we show that each set $\caO_b$ is non-empty by producing an explicit representative satisfying the appropriate rank conditions.  It suffices to produce a basis $\{v_1,\hdots,v_{2n+1}\}$ for $\C^{2n+1}$ such that the matrix for the form $\gamma$ relative to this basis is a monomial matrix (that is, a matrix such that each row and column has exactly one non-zero entry) whose image in $W$ is $b$.  Then we can simply take our flag $F_{\bullet}$ to be $\left\langle v_1,\hdots,v_{2n+1} \right\rangle$.

We choose such a basis as follows.  (Recall that the form $\gamma$ is defined by $\left\langle e_i,e_j \right\rangle = \delta_{i,2n+2-j}$.)  First, for each $i$ such that $b(i) \neq i$, choose $v_i$ and $v_{b(i)}$ to be $e_k$ and $e_{2n+2-k}$ for some $k \neq n+1$.  (Of course, we should choose a different such $k$ for each such $i$.)   There are an odd number of $i$ such that $b(i) = i$ --- for one such $i$, choose $v_i$ to be $e_{n+1}$, and for all other pairs $i_1,i_2$ of such $i$, choose $v_{i_1}$ to be $e_k + e_{2n+2-k}$ for some $k \neq n+1$ (and not yet used in the first step above), and choose $v_{i_2}$ to be $e_k - e_{2n+2-k}$ for the same $k$.  (We should choose a different such $k$ for each such pair $i_1,i_2$.)

\begin{prop}
With $v_1,\hdots,v_{2n+1}$ defined as above, the flag $\left\langle v_1,\hdots,v_{2n+1} \right\rangle$ lies in $\caO_b$.
\end{prop}
\begin{proof}
We first note that the matrix for the form $\gamma$ relative to this basis is indeed a monomial matrix whose image in $W$ is $b$.  This means precisely that for each $i$, $\left\langle v_i,v_j \right\rangle$ is non-zero if and only if $j = b(i)$.

For any $i$ with $b(i) \neq i$, this is clear.  Indeed, $v_i = e_k$ for some $k$, while $v_{b(i)} = e_{2n+2-k}$.  Meanwhile, $e_k$ appears with coefficient $0$ in all other $v_i$ by design.  Since $\left\langle e_i,e_j \right\rangle = \delta_{i,2n+2-j}$, we see that $\left\langle v_i,v_j \right\rangle = \delta_{j,b(i)}$.

Now suppose that $b(i) = i$.  Then either $v_i = e_{n+1}$, or $v_i = e_k  \pm e_{2n+2-k}$ for some $k \neq n+1$ (and not equal to any $k$ used to define $v_i$ with $b(i) \neq i$).  In the former case, we have 
\[ \left\langle v_i,v_{b(i)} \right\rangle = \left\langle v_i,v_i \right\rangle = \left\langle e_{n+1},e_{n+1} \right\rangle = 1, \]
while $\left\langle v_i,v_j \right\rangle = 0$ for any other $j$, since $e_{n+1}$ appears with coefficient $0$ in all other $v_i$.  In the latter case, supposing that $v_i = e_k  + e_{2n+2-k}$, we have
\[ \left\langle v_i,v_{b(i)} \right\rangle = \left\langle v_i,v_i \right\rangle = \left\langle e_k,e_{2n+2-k} \right\rangle + \left\langle e_{2n+2-k},e_k \right\rangle) = 2. \]
If $v_i = e_k  - e_{2n+2-k}$, the corresponding computation shows that $\left\langle v_i,v_{b(i)} \right\rangle = -2$.

For $j \neq b(i)$, either $e_k$ appears with coefficient $0$ in $v_j$ (in which case $\left\langle v_i,v_j \right\rangle = 0$), or $v_j = e_k \mp e_{2n+2-k}$, and in that case,
\[ \left\langle v_i,v_j \right\rangle = \left\langle e_k,e_{2n+2-k} \right\rangle - \left\langle e_{2n+2-k},e_k \right\rangle) = 0. \]

This establishes that the matrix for $\gamma$ relative to the basis $\{v_i\}$ is indeed monomial, with image $b$ in $W$.  Now, note that if $F_{\bullet} = \left\langle v_1,\hdots,v_{2n+1} \right\rangle$, then $\text{rank}(\gamma|_{F_i \times F_j})$ is, by definition, the rank of the upper-left $i \times j$ rectangle of this matrix.  For any monomial matrix with image $b$ in $W$, the rank of the upper-left $i \times j$ rectangle is precisely $r_b(i,j)$.  This proves the claim.
\end{proof}

We illustrate with two examples.  Suppose $n = 2$, so we are dealing with $G = SL(5,\C)$, $K = SO(5,\C)$.  First consider the involution $b=(2,4)$.  Since $b$ moves $2$ and $4$, we first choose $v_2 = e_1$ and $v_4 = e_5$.  Since $b$ fixes $1$, $3$, and $5$, we first choose $v_1 = e_3$, then we choose $v_3 = e_2+e_4$ and $v_5 = e_2 - e_4$.  Our ordered basis is thus
\[ \{ e_3,e_1,e_2+e_4,e_5,e_2-e_4 \}. \]
Relative to this ordered basis, the form $\gamma$ has matrix
\[ \begin{pmatrix}
1 & 0 & 0 & 0 & 0 \\
0 & 0 & 0 & 1 & 0 \\
0 & 0 & 2 & 0 & 0 \\
0 & 1 & 0 & 0 & 0 \\
0 & 0 & 0 & 0 & -2 \end{pmatrix}, \]
and one checks that if $F_{\bullet} = \left\langle v_1,\hdots,v_5 \right\rangle$, then the rank conditions specified by $b=(2,4)$ are satisfied.

Next, consider $b=(1,3)(2,5)$.  We first choose $v_1 = e_1$, $v_3 = e_5$, $v_2 = e_2$, and $v_5 = e_4$.  Finally, since $b$ fixes only $4$, we choose $v_4 = e_3$.  So our ordered basis is $\left\langle e_1,e_2,e_5,e_3,e_4 \right\rangle$, and the form $\gamma$, relative to this basis, has matrix
\[ \begin{pmatrix}
0 & 0 & 1 & 0 & 0 \\
0 & 0 & 0 & 0 & 1 \\
1 & 0 & 0 & 0 & 0 \\
0 & 0 & 0 & 1 & 0 \\
0 & 1 & 0 & 0 & 0 \end{pmatrix}. \]
One again checks easily that if $F_{\bullet} = \left\langle v_1,\hdots,v_5 \right\rangle$, the rank conditions encoded by $b = (1,3)(2,5)$ are satisfied.

\subsection{Example}
We now work out a very small example in detail, the case $n = 1$.  (So $G = SL(3,\C)$, $K = SO(3,\C)$.)  Here, there are 4 involutions, and hence 4 orbits.

We start from the minimal element $w_0 = (1,3)$, and work our way up as described in Corollary \ref{cor:weak_order_involutions}.  Since
\[ s_1 w_0 s_1 = (1,2)(1,3)(1,2) = (2,3) \neq w_0, \]
we have $w_0 <_1 (2,3)$, and the edge is black.

Similarly,
\[ s_2 w_0 s_2 = (2,3)(1,3)(2,3) = (1,2) \neq w_0, \]
so $w_0 <_2 (1,2)$, and again the edge is black.

Now, we move up to the orbits corresponding to $(1,2)$ and $(2,3)$.  Start with $(2,3) = s_2$.  Since $l(s_1 s_2) > l(s_2)$, there is no edge originating at $(2,3)$ with label $1$ and going up.  However, $l(s_2^2) < l(s_2)$, so we will obviously have that $s_2 <_2 id$.  (Since $id$ is the only involution left, this is the only possibility.)  Let us check this directly:
\[ s_2 s_2 s_2 = s_2, \]
so instead of conjugating by $s_2$, we multiply on the left by it to get $s_2^2 = id$, and in this case the edge is blue.

The situation with $(1,2) = s_1$ is similar.  Since $l(s_2 (1,2)) = l(s_2 s_1) > l(s_1)$, so there is no edge originating at $s_1$ with label $2$ and going up.  But $l(s_1^2) < l(s_2)$, so we will have that $s_1 <_1 id$.  Again, $s_1$ is fixed by conjugation by $s_1$, and so we multiply on the left to get $id$, and the edge connecting $s_1$ and $id$ has label $1$ and is blue.  The weak order graph appears as Figure \ref{fig:type-a-orthogonal-1} of Appendix B.

With this complete, we now determine formulas for the $S$-equivariant classes of all orbit closures.  By Proposition \ref{prop:formula_for_SO_odd} above, the class of the closed orbit corresponding to $w_0$ is given by the formula $[Q] = -2(y_1+y_2)(y_2+y_3)$.  The class $[Y_{(2,3)}]$ is given by
\[ [Y_{(2,3)}] = \partial_1([Q]), \]
and since 
\[ \partial_1(f(x,y_1,y_2,y_3)) = \dfrac{f - f(x,y_2,y_1,y_3)}{y_1-y_2}, \]
we have $[Y_{(2,3)}] = 2(y_1+y_2)$.  Similarly, $[Y_{(1,2)}] = \partial_2([Q]) = -2(y_2+y_3)$.

Finally, we can compute $[Y_{id}]$ either as $\frac{1}{2}\partial_2([Y_{(2,3)}])$, or as $\frac{1}{2}\partial_1([Y_{(1,2)}])$.  Using either formula, we get that $[Y_{id}] = 1$, as expected.

The results are summarized in Table \ref{tab:type-a-so3} of Appendix B.  The weak order graph and the list of formulas for the larger case $n = 5$ appear in Figure \ref{fig:type-a-orthogonal-2} and Table \ref{tab:type-a-so5}.  (In that case, there are 26 orbits.)

\section{$K \cong SO(2n,\C)$}
We now treat the case of the even special orthogonal group.

As before, we realize $K$ as the subgroup of $SL(2n,\C)$ preserving the orthogonal form given by the antidiagonal matrix $J=J_{2n}$.  

Once again, if $Y_i$ ($i=1,\hdots,2n$) are coordinates on $\frt$, restriction to $\frs$ is given by $\rho(Y_i) = X_i$ and $\rho(Y_{2n+1-i}) = -X_i$ for each $i = 1,\hdots,n$.

The roots of $K$ are 
\[ \Phi_K  = \{\pm(X_i \pm X_j) \mid i<j\}. \]

In this case, the Weyl group $W_K$ of $K$ acts on torus characters by signed permutations which change an \textit{even} number of signs.  The inclusion of $W_K$ into $W$ is as signed elements $\sigma \in S_{2n}$ (defined in Subsection \ref{ssec:notation}) having the further property that
\begin{equation}
	\sigma(i) > n \text{ for an even number of } i=1,\hdots,n.
\end{equation}

\subsection{Formulas for the closed orbits}
With conventions fixed as in the previous subsection, we have the following:

\begin{prop}\label{prop:closed-orbits-so-2n}
There are precisely $2$ closed $K$-orbits on $G/B$.  The first, $Q_1$, is the orbit $K \cdot 1B$.  The $S$-fixed points contained in $Q_1$ are signed elements of $S_{2n}$, and correspond to elements of $W_K$ under the embedding $W_K \subseteq W$ described above.

The second, $Q_2$, is the orbit $K \cdot s_nB$, where $s_n$ is the simple transposition $(n,n+1)$.  Fixed points in this orbit are signed elements of $S_{2n}$ that correspond to signed permutations of $\{1,\hdots,n\}$ changing an \textit{odd} number of signs.
\end{prop}
\begin{proof}
As in the previous case, the map induced on $W$ by $\theta$ is $w \mapsto w_0 w w_0$.  So as in that case, we see using Proposition \ref{prop:num-closed-orbits} that $K \cdot wB$ is closed if and only if $w$ is a signed element of $S_{2n}$.

These fixed points correspond to elements of $S_{2n}$ which constitute $2$ (left) $W_K$-cosets --- namely, $W_K \cdot 1$ and $W_K \cdot s_n$.  Thus there are two closed orbits, as follows from the discussion immediately following the statement of Proposition \ref{prop:num-closed-orbits}.
\end{proof}

With the closed orbits determined, we now give a formula for the $S$-equivariant class of each:

\begin{prop}\label{prop:formula_for_SO_even}
With $Q_1$ and $Q_2$ as in the previous proposition, $[Q_1]$ is represented by the polynomial $P_1(x,y)$, and $[Q_2]$ by the polynomial $P_2(x,y)$, where
\[ P_1(x,y) = 2^{n-1} (x_1 \hdots x_n + y_1 \hdots y_n) \displaystyle\prod_{1 \leq i < j \leq n}(y_i + y_j)(y_i + y_{2n+1-j}); \]
and
\[ P_2(x,y) = -2^{n-1} (x_1 \hdots x_n - y_1 \hdots y_n) \displaystyle\prod_{1 \leq i < j \leq n}(y_i + y_j)(y_i + y_{2n+1-j}). \]
\end{prop}
\begin{proof}
We demonstrate the correctness of the formula for $[Q_1]$. The argument is similar to that given in the previous case for the lone closed orbit of the odd orthogonal group.

As stated, $Q_1$ consists of those $S$-fixed points corresponding to elements of $W_K$ --- that is, signed permutations with an even number of sign changes.  Take $w \in Q_1$ to be such a fixed point.  We use Proposition \ref{prop:restriction-of-closed-orbit} to compute the restriction $[Q_1]|_w$.  As in the previous example, we first determine the restriction of the positive roots $\Phi^+$ to $\frs$, then apply the signed permutation $w$ to that set of weights.

Restricting the positive roots $\{Y_i - Y_j \ \vert \ 1 \leq i < j \leq 2n\}$ to $\frs$, we get the following set of weights:

\begin{enumerate}
	\item $X_i - X_j$, $1 \leq i < j \leq n$, each with multiplicity 2 (one is the restriction of $Y_i - Y_j$, the other the restriction of $Y_{2n+1-j} - Y_{2n+1-i}$)
	\item $X_i + X_j$, $1 \leq i < j \leq n$, each with multiplicity 2 (one is the restriction of $Y_i - Y_{2n+1-j}$, the other the restriction of $Y_j - Y_{2n+1-i}$)
	\item $2X_i$, $1 \leq i \leq n$, each with multiplicity 1 (the restriction of $Y_i - Y_{2n+1-i}$)
\end{enumerate}

Now, consider applying a signed permutation $w \in W_K$ to this set of weights.  The resulting set of weights will be

\begin{enumerate}
	\item $\pm (X_i - X_j)$, $1 \leq i < j \leq n$, each occurring with either a plus or minus sign, and with multiplicity 2 (these weights come from applying $w$ to weights of either type (1) or (2) above)
	\item $\pm (X_i + X_j)$, $1 \leq i < j \leq n$, each occurring with either a plus or minus sign, and with multiplicity 2 (these weights also come from applying $w$ to weights of either type (1) or (2) above)
	\item $\pm 2X_i$, $1 \leq i \leq n$, each ocurring with either a plus or minus sign, and with multiplicity 1 (these weights come from applying $w$ to weights of type (3) above)
\end{enumerate}

Subtracting roots of $K$, we are left with the following weights:

\begin{enumerate}
	\item $\pm (X_i - X_j)$, $1 \leq i < j \leq n$, each occurring with either a plus or minus sign, and with multiplicity 1
	\item $\pm (X_i + X_j)$, $1 \leq i < j \leq n$, each occurring with either a plus or minus sign, and with multiplicity 1
	\item $\pm 2X_i$, $1 \leq i \leq n$, each occurring with either a plus or minus sign, and with multiplicity 1
\end{enumerate}

The number of weights of the form $-2X_i$ is even, since $w$ changes an even number of signs.  So in computing the restriction, to get the sign right, we need only concern ourselves with the signs of the weights of types (1) and (2) above.

We may argue just as in the proof of Proposition \ref{prop:formula_for_SO_odd} that the number of $X_i \pm X_j$ ($i < j$) occurring with a negative sign is congruent mod $2$ to $l(|w|)$.  As such, if $w \in Q_1$ is an $S$-fixed point, then

\[ [Q_1]|_w = F(X) := (-1)^{l(|w|)} 2^n X_1 \hdots X_n \displaystyle\prod_{1 \leq i < j \leq n} (X_i + X_j)(X_i-X_j). \]

So we seek a polynomial in $x_1,\hdots,x_n,y_1,\hdots,y_{2n}$, say $f$, with the property that

\[ f(X,\rho(wY)) = 
\begin{cases}
F(X) & \text{ if $w \in W_K$} \\
0 & \text{ otherwise.}
\end{cases} \]

It is straightforward to verify that $P_1$ has these properties.  Indeed, first take $w \in W_K$.  Then applying $w$ to the term $x_1 \hdots x_n + y_1 \hdots y_n$ gives $2X_1 \hdots X_n$, since $w$ permutes the $Y_i$ with an even number of sign changes, and each restricts to the corresponding $X_i$.  Multiplying this by $2^{n-1}$ gives us the $2^n X_1 \hdots X_n$ part of $F$.  The terms $y_i + y_j$ and $y_i + y_{2n+1-j}$ give, up to sign, all terms of the form $X_i + X_j$ and $X_i - X_j$ ($i<j$).  Rewriting each such term as either $+1$ or $-1$ times a positive root by factoring out negative signs as necessary, we effectively introduce the sign of $(-1)^{l(|w|)}$, as required.

On the other hand, if $w \notin W_K$, then there are two possibilities:

\textit{Case 1:  $w$ is a signed element of $S_{2n}$ corresponding to a signed permutation with an \textit{odd} number of sign changes.}

In this case, $w$ clearly kills the term $x_1 \hdots x_n + y_1 \hdots y_n$, and hence $f(X,\rho(wY)) = 0$.

\textit{Case 2:  $w$ is not a signed element of $S_{2n}$.}

In this case, then $w(2n+1-i) \neq 2n+1-w(i)$ for some $1 \leq i \leq n$.  Let $j = 2n+1-w(i)$, and let $k = w^{-1}(j)$.  Clearly, $k \neq i$ or $2n+1-i$.  So the factor $y_i + y_k$ appears in $P_1$.  Applying $w$ to this factor gives $Y_{w(i)} + Y_{2n+1-w(i)}$, which then restricts to zero.

We see that in either case, $f(X,\rho(wY)) = 0$.  This proves that $P_1(x,y)$ represents $[Q_1]$.

The verification of the formula for $[Q_2]$ is similar.  The orbit $Q_2$ consists of those $S$-fixed points corresponding to signed elements of $S_{2n}$ corresponding to signed permutations which change an \textit{odd} number of signs.  Take $w \in Q_2$ to be such a fixed point.  We compute the restriction $[Q_2]_w$.  The set $\rho(w\Phi^+))$ is as follows:

\begin{enumerate}
	\item For each $1 \leq i < j \leq n$, exactly one of $\pm (X_i - X_j)$, occurring with multiplicity 2 
	\item For each $1 \leq i < j \leq n$, exactly one of $\pm (X_i + X_j)$ occurring with multiplicity 2 
	\item For each $i=1,\hdots,n$, exactly one of $\pm 2X_i$, with multiplicity 1
\end{enumerate}

Subtracting roots of $K$, we are left with
\begin{enumerate}
	\item $\pm (X_i - X_j)$, $1 \leq i < j \leq n$, each occurring with either a plus or minus sign, and with multiplicity 1
	\item $\pm (X_i + X_j)$, $1 \leq i < j \leq n$, each occurring with either a plus or minus sign, and with multiplicity 1
	\item $\pm 2X_i$, $1 \leq i \leq n$, each occurring with either a plus or minus sign, and with multiplicity 1
\end{enumerate}

The number of weights of the form $-2x_i$ is odd, since $w$ changes an odd number of signs.  The number of roots of types (1) or (2) occurring with a minus sign is again congruent mod $2$ to $l(|w|)$ (the argument is identical).  So in this case, we see that

\[ [Q_2]|_w = F(X) := (-1)^{l(|w|)+1} 2^n X_1 \hdots X_n \displaystyle\prod_{1 \leq i < j \leq n} (X_i + X_j)(X_i-X_j). \]

So now we seek a polynomial $f$ in the variables $x_1,\hdots,x_n,y_1,\hdots,y_{2n}$, such that

\[ f(X,\rho(wY)) = 
\begin{cases}
F(X) & \text{ if $w$ is signed and changes an odd number of signs} \\
0 & \text{ otherwise.}
\end{cases} \]

Again, the verification that $P_2$ fits the bill is routine.  Given $w \in Q_2$, it clearly sends $X_1 \hdots X_n - Y_1 \hdots Y_n$ to $2X_1 \hdots X_n$, and multiplication by $-2^{n-1}$ gives the first part of $F$ above.  And just as was the case for $Q_1$, the terms $y_i + y_j$ and $y_i + y_{2n+1-j}$ give the second part, with the appropriate sign of $(-1)^{l(|w|)}$.

For $w \notin Q_2$, again there are two possibilities:

\textit{Case 1:  $w$ is signed and and changes an even number of signs (i.e. $w \in W_K$)}

Then clearly $w$ kills the term $x_1 \hdots x_n - y_1 \hdots y_n$.

\textit{Case 2:  $w$ is not a signed permutation}

Then we may argue just as in the case of $Q_1$ that $w$ kills a factor of the form $y_i + y_j$ or $y_i + y_{2n+1-j}$.

So in either case, $f(X,\rho(wY)) = 0$.  Thus $P_2$ represents $[Q_2]$, as claimed.
\end{proof}

\begin{remark}
Note that the representatives for $[Q_1]$ and $[Q_2]$ involve both the $x$ and $y$ variables.  Unlike the odd case, there don't seem to be representatives involving only the $y$-variables (at least not that the author was able to find).  However, note that if we consider the lone closed orbit of $O(2n,\C)$ on $X$ (the union of $Q_1$ and $Q_2$), its class (being the sum of $[Q_1]$ and $[Q_2]$) involves only the $y$-variables.  This reflects the fact that the fundamental classes of degeneracy loci parametrized by $O(2n,\C)$-orbit closures are expressible in the Chern classes of a flag of vector subbundles of a given vector bundle $V$ over a variety $X$.  By contrast, the fundamental classes of the \textit{irreducible components} of such loci, parametrized by $SO(2n,\C)$-orbit closures, are only expressible in these Chern classes together with the \textit{Euler class} of the bundle $V$, see Subsection \ref{ssec:type-a-other-deg-loci}.
\end{remark}

\subsection{Parametrization of $K \backslash G/B$ and the weak order}\label{ssec:so2n_param}
Again we refer to \cite[Examples 10.2,10.3]{Richardson-Springer-90}.  Let $\phi$ be the Richardson-Springer map for the symmetric pair $(SL(2n,\C),SO(2n,\C))$, and $\phi'$ the Richardson-Springer map for the symmetric pair $(GL(2n,\C),O(2n,\C))$.  As mentioned in Subsection \ref{ssec:kgb_param_so_odd}, the analysis given there for the odd orthogonal group carries over to the even orthogonal group:  $\phi'$ is still a bijection, $O(2n,\C)$-orbits are parametrized by involutions in $W$, representatives of orbits are found in the same way, and the combinatorics of the weak order are described identically.  However, when one considers the $SO(2n,\C)$-orbits on $G/B$, things are a bit more complicated.  The map $\phi$ is still surjective, but is no longer injective.  When twisted involutions are identified with honest involutions (i.e. when $\phi$ is followed by right multiplication by $w_0$), the result is as follows:
\begin{enumerate}
	\item If $a \in W$ is an involution with fixed points, then $\phi^{-1}(a)$ is a single $SO(2n,\C)$-orbit, which coincides with $(\phi')^{-1}(a)$.
	\item If $a \in W$ is an involution \textit{without} fixed points, then $\phi^{-1}(a)$ consists of two $SO(2n,\C)$-orbits, which are components of the single $O(2n,\C)$-orbit $(\phi')^{-1}(a)$.
\end{enumerate}
Thus some of the $O(2n,\C)$-orbits split as a union of two $SO(2n,\C)$-orbits, while others do not.

If $b$ is an involution with fixed points, then one can determine a representative of the $SO(2n,\C)$-orbit $\caO_b$ just as described in Subsection \ref{ssec:kgb_param_so_odd}.  If $b$ is fixed point-free, then one can determine a representative of the $O(2n,\C)$-orbit corresponding to $b$ using the same procedure.  This gives a representative of one of the two $SO(2n,\C)$-orbits which correspond to $b$.  Note that this representative is always a $T$-fixed flag, corresponding to a permutation in $S_{2n}$.  To get a representative of the other $SO(2n,\C)$-orbit corresponding to $b$, one can multiply this permutation by the transposition $(n,n+1)$ and take the $T$-fixed flag corresponding to the resulting element of $S_{2n}$.

The two closed orbits are particular examples of this.  Indeed, the closed orbits are the two components of the $O(2n,\C)$-orbit corresponding to the involution $w_0$, which is fixed point-free.  To get a representative of one component, one follows the procedure of Subsection \ref{ssec:kgb_param_so_odd} to obtain the standard flag $\left\langle e_1,\hdots,e_{2n} \right\rangle$.  Then, to get a representative of the other component, we apply the permutation $(n,n+1)$ (note that this is the signed element of $S_{2n}$ corresponding to the signed permutation which interchanges $n$ and $-n$) to obtain $\left\langle e_1,\hdots,e_{n-1},e_{n+1},e_n,e_{n+2},\hdots,e_{2n} \right\rangle$.

The weak closure order on $SO(2n,\C)$-orbits, as well as whether edges of the weak order graph are black or blue, require a bit more care to get right when dealing with orbits which are components of $O(2n,\C)$-orbits.  Given two $O(2n,\C)$-orbits $\caO_1 <_i \caO_2$, supposing that either $\caO_1$ or $\caO_2$ (or both) splits as a union of two $SO(2n,\C)$-orbits, how does one describe the weak closure order on the components?

Note that there are two possible ways this can occur:  Either $\caO_1$ and $\caO_2$ \textit{both} split, or $\caO_1$ splits and $\caO_2$ does not.  Each possibility can occur, as we see in the case $n = 2$.  Indeed, when considering $O(4,\C)$-orbits, parametrized by involutions, we have $w_0 <_1 (1,3)(2,4)$, each of which is fixed point-free.  Thus both of these orbits split.  We also have $w_0 <_2 (1,4)$, and $(1,4)$ has fixed points, so it does not split.  The third ``possibility", where $\caO_1$ does not split while $\caO_2$ does, is clearly not possible either from a geometric or a combinatorial standpoint.  Indeed, it cannot happen that two different components of $\caO_2$ are \textit{both} dense in $\pi_{\ga}^{-1}(\pi_{\ga}(\caO_1))$ for $\ga = \ga_i \in \Delta$.  This is reflected combinatorially by the the fact that if $b \in W$ is an involution with fixed points, and if $l(s_ib) < l(b)$, then both of the following must hold:
\begin{enumerate}
	\item $s_i b s_i$ has fixed points.  Indeed, if $b$ fixes any value other than $i$ or $i+1$, then $s_i b s_i$ fixes that same value.  Otherwise, if $b(i) = i$, then $s_i b s_i(i+1) = i+1$, and if $b(i+1) = i+1$, then $s_i b s_i(i) = i$.
	\item If $s_ibs_i = b$, then $s_i b$ also has fixed points.  Indeed, if $s_i b s_i = b$, then $b$ must preserve the set $\{i,i+1\}$, as well as its complement.  If $b$ fails to fix any value other than $i$ and $i+1$, then it must fix both $i$ and $i+1$, since $b$ is assumed to have fixed points.  But in this case, we have $l(s_ib) > l(b)$, since $s_ib$ has one more inversion than $b$, namely $(i,i+1) \mapsto (i+1,i)$.  This contradicts our assumption that $l(s_ib) < l(b)$, thus $b$ must fix some value outside of $\{i,i+1\}$.  Then $s_ib$ necessarily fixes the same value.
\end{enumerate}

Let us consider the two possible cases.  Take first the case when $\caO_1$ splits while $\caO_2$ does not.  Then by the results of Subsection \ref{ssec:kgb_param_so_odd}, in the weak order graph for $O(2n,\C)$-orbits, any edge joining $\caO_1$ to $\caO_2$ must be blue.  Indeed, if the involution corresponding to $\caO_1$ is fixed point-free, then $s_i b s_i$ is also fixed point-free.  Since $\caO_2$ does not split, it corresponds to an involution with fixed points, which obviously cannot be $s_i b s_i$.  The only conclusion is that $s_i b s_i = b$, and that the involution corresponding to $\caO_2$ is $s_i b$.  This implies that any edge joining $\caO_1$ to $\caO_2$ is blue.

Let $\caO_1'$ and $\caO_1''$ denote the two $SO(2n,\C)$-orbits whose union is $\caO_1$.  Let $\phi(\caO_1') = \phi(\caO_1'') = a$, and $\phi(\caO_2) = b$.  We have $a <_i b$, since $s_i \cdot \caO_1 = \caO_2$ as $O(2n,\C)$-orbits.  So by Proposition \ref{prop:twisted_inv_map_phi_weak_order}, we must have either $s_i \cdot \caO_1' = \caO_2$ or $s_i \cdot \caO_1' = \caO_1'$.  Likewise, either $s_i \cdot \caO_1'' = \caO_2$, or $s_i \cdot \caO_1'' = \caO_1''$.  Since $\caO_1'$ and $\caO_1''$ are the only $SO(2n,\C)$-orbits mapping to $a$, it follows moreover from Proposition \ref{prop:twisted_inv_paths} that at least one of the following is true:  
\begin{enumerate}
	\item $s_i \cdot \caO_1' = \caO_2$, or
	\item $s_i \cdot \caO_1'' = \caO_2$.
\end{enumerate}
But in fact, both of these must be true.  To see this, recall how we determine whether a simple reflection is complex or non-compact imaginary for a given orbit:  We take a representative of the orbit $gB$ having the property that $g^{-1}\theta(g) \in N_G(T)$, so that $T' = gTg^{-1}$ is a $\theta$-stable maximal torus, and we examine the properties of the root $\text{int}(g)(\ga_i)$ of $T'$ relative to $\theta$.  This is independent of the choice of representative $gB$.  Take any such representative $gB$ of the $O(2n,\C)$-orbit $\caO_1$.  Suppose it lies in $\caO_1'$.  We obtain a different representative of $\caO_1$ lying in $\caO_1''$ by multiplying by the permutation matrix $\pi$ corresponding to the transposition $(n,n+1)$.  But if $g^{-1}\theta(g) \in N_G(T)$, $(\pi g)^{-1} \theta(\pi g) \in N_G(T)$ as well, so for the $O(2n,\C)$-orbit $\caO_1$, the analysis described above may be carried out with either representative.  Since $\ga_i$ is non-compact imaginary for $\caO_1$, carrying out the analysis for $\caO_1'$ using the representative $gB$ shows that $\ga_i$ is non-compact imaginary for $\caO_1'$, and carrying out the analysis for $\caO_1''$ using the representative $\pi gB$ shows that $\ga_i$ is non-compact imaginary for $\caO_1''$.

Now, it should be clear that $\ga_i$ is in fact a non-compact imaginary root of type I for both $\caO_1'$ and $\caO_1''$, whereas it was a type II root for $\caO_1$.  Indeed, since $s_i \cdot \caO_1' = \caO_2$ and $s_i \cdot \caO_1'' = \caO_2$, both $\caO_1'$ and $\caO_1''$ cover their images in $G/P_{\ga_i}$ either 1-to-1 or 2-to-1.  Since $\caO_1 = \caO_1' \cup \caO_1''$ covers its image 2-to-1, the only possibility is that each of $\caO_1'$ and $\caO_1''$ covers its image 1-to-1.  Thus whereas the edge joining $\caO_1$ to $\caO_2$ in the weak order graph for $O(2n,\C)$-orbits was blue, now the edges joining $\caO_1'$ and $\caO_1''$ to $\caO_2$ are each black.  (Note that this says in particular that while $\caO_1$ was fixed by the cross action of $s_{\ga_i}$, the orbits $\caO_1'$ and $\caO_1''$ are interchanged by it.)

The geometry here is simple:  The restriction of the map $\pi_{\ga_i}: G/B \rightarrow G/P_{\ga_i}$ to $\overline{\caO_1}$ is generically $2$-to-$1$.  Over a generic point $gP_{\ga_i}$ in the image, one of the two preimage points will lie in $\caO_1'$, and the other will lie in $\caO_1''$.  Thus the further restriction of $\pi_{\ga_i}$ to either component of $\overline{\caO_1}$ is birational.

Now consider the second case, where both $\caO_1$ and $\caO_2$ split (say as $\caO_1'$, $\caO_1''$ and $\caO_2'$, $\caO_2''$).  In this case, we can see combinatorially that any edge joining $\caO_1$ to $\caO_2$ must be black.  Indeed, $\caO_1$ corresponds to a fixed point-free involution $b$, while $\caO_2$ corresponds to a fixed point-free involution $c=m(s_i)*b$ for some $s_i$.  If $s_ibs_i = b$, then $s_i b$ must have fixed points.  Since $c$ is assumed not to have fixed points, we must have that $s_ibs_i = c$.  Thus any edge joining $\caO_1$ to $\caO_2$ is black.

It follows from Proposition \ref{prop:twisted_inv_paths} that we should have one of the following two cases:
\begin{enumerate}
	\item $\caO_1' <_i \caO_2'$ and $\caO_1'' <_i \caO_2''$ (both edges black)
	\item $\caO_1' <_i \caO_2''$ and $\caO_1'' <_i \caO_2'$ (both edges black).
\end{enumerate}

However, it is not obvious (at least to the author) how to tell which is the case once we have fixed our choices of $\caO_1'$, $\caO_1''$, $\caO_2'$, and $\caO_2''$.  As a simple example, consider the case $n=2$, with $\caO_1$ the bottom orbit corresponding to $w_0$, and $\caO_2$ the orbit corresponding to $(1,3)(2,4)$.  As noted above, we have $\caO_1 <_1 \caO_2$.  It is also the case that $s_3 w_0 s_3 = (1,3)(2,4)$, so $\caO_1 <_3 \caO_2$ as well.  If we declare, say, that $\caO_1'$, $\caO_1''$, $\caO_2'$, and $\caO_2''$ are represented by $\left\langle e_1,e_2,e_3,e_4 \right\rangle$, $\left\langle e_1,e_3,e_2,e_4 \right\rangle$, $\left\langle e_1,e_2,e_4,e_3 \right\rangle$, and $\left\langle e_1,e_3,e_4,e_2 \right\rangle$, respectively, how does one know which of the following four sets of closure relations is correct?
\begin{enumerate}
	\item $\caO_1' <_1 \caO_2'$, $\caO_1' <_3 \caO_2'$, $\caO_1'' <_1 \caO_2''$, $\caO_1'' <_3 \caO_2''$
	\item  $\caO_1' <_1 \caO_2'$, $\caO_1' <_3 \caO_2''$, $\caO_1'' <_1 \caO_2''$, $\caO_1'' <_3 \caO_2'$
	\item $\caO_1' <_1 \caO_2''$, $\caO_1' <_3 \caO_2''$, $\caO_1'' <_1 \caO_2'$, $\caO_1'' <_3 \caO_2'$
	\item $\caO_1' <_1 \caO_2''$, $\caO_1' <_3 \caO_2'$, $\caO_1'' <_1 \caO_2'$, $\caO_1'' <_3 \caO_2''$
\end{enumerate}

Ultimately, we can answer this question by examining the formulas for the equivariant fundamental classes of these orbit closures and computing their restrictions at $S$-fixed points contained in one orbit closure or another.  In the example given above, we know that the orbit $\caO_1'$ is represented by the polynomial $2(x_1x_2 + y_1y_2)(y_1+y_2)(y_1+y_3)$.  Applying $\partial_1$ to this polynomial, we get $2(x_1x_2+y_1y_2)(y_1+y_2)$.  This polynomial must represent either $[\caO_2']$ or $[\caO_2'']$.  As chosen above, $\caO_2'$ is represented by the $S$-fixed point corresponding to $1243$, while $\caO_2''$ is represented by the $S$-fixed point corresponding to $1342$.   Computing the restriction of the class $\partial_1([\caO_1'])$ at the fixed point $1243$, we get
\[ 2(X_1X_2+X_1X_2)(X_1+X_2) = 4X_1X_2(X_1+X_2). \]
On the other hand, when we compute the restriction of the class $\partial_1([\caO_1'])$ at the fixed point $1342$, we get
\[ 2(X_1X_2 + X_1X_3)(X_1+X_3) = 2(X_1X_2 - X_1X_2)(X_1-X_2) = 0. \]
This tells us that we must have $\caO_1' <_1 \caO_2'$ (and hence also $\caO_1'' <_1 \caO_2''$).  Indeed, the computation shows that the $S$-fixed point $1243$ must be contained in the closure of the orbit $s_1 \cdot \caO_1'$, or else the restriction of $[s_1 \cdot \caO_1']$ at $1243$ would necessarily be zero.  This says $s_1 \cdot \caO_1' = \caO_2'$.  A similar computation involving $\partial_3([\caO_1'])$ shows also that $\caO_1' <_3 \caO_2'$ and $\caO_1'' <_3 \caO_2''$.  Thus option (1) above is the correct one.

\subsection{Example}
We give the results of the remainder of the computation for the case $n=2$, some of which was worked out in the previous subsection to enhance the clarity of the exposition there.  (We treat both the cases $G=GL(4,\C),K=O(4,\C)$ and $G=SL(4,\C),K=SO(4,\C)$.)  There are 10 involutions in $W$:
\[ id; (1,2); (1,3); (1,4); (2,3); (2,4); (3,4); (1,2)(3,4); (1,3)(2,4); (1,4)(2,3). \]

The weak order graph for $O(4,\C)$-orbits on $X$ is given in Figure \ref{fig:type-a-orthogonal-3} of Appendix B, with formulas shown in Table \ref{tab:type-a-o4}.  The only comment we offer on that computation is simply to point out that the formula for the bottom orbit corresponding to $w_0$ is obtained by adding the formulas for the classes of the two irreducible components, those being the two closed $SO(4,\C)$-orbits.

The weak order graph for $SO(4,\C)$-orbits on $X$ is given in Figure \ref{fig:type-a-orthogonal-4}, with formulas shown in Table \ref{tab:type-a-so4}.  All the ideas required for the computation are discussed in the previous subsection, so we offer no further comment here.

\section{$K \cong Sp(2n,\C)$}
The final $K$ to consider in type $A$ is $K=Sp(2n,\C)$, which corresponds to the real form $G_\R = SL(n,\QU)$ of $SL(2n,\C)$.  ($\QU$ denotes the quaternions.)  We realize $K$ as the isometry group of the skew form given by $J_{n,n}$ (cf. Subsection \ref{ssec:notation}) --- that is, $K$ is the fixed point subgroup of the involution
\[ \theta(g) = J_{n,n} (g^{-1})^t J_{n,n}. \]

As was the case with the orthogonal groups, one checks easily that given this realization of $K$, the diagonal elements $S = K \cap T$ are a maximal torus of $K$, and the lower-triangular elements $B' = B \cap K$ are a Borel subgroup of $K$.  Also as with the orthogonal groups, we have $\text{rank}(K) < \text{rank}(G)$, so we have a proper inclusion of tori $S \subsetneq T$, and we work over $S$-equivariant cohomology $H_S^*(X)$.  If $Y_1,\hdots,Y_{2n} \in \frt^*$ are coordinates on $\frt$, restriction to $\frs$ is given by $\rho(Y_i) = X_i$, $\rho(Y_{2n+1-i}) = -X_i$ for $i=1,\hdots,n$.

The roots of $K$ are the following:
\[ \Phi_K = \{\pm (X_i \pm X_j) \mid 1 \leq i < j \leq n \} \cup \{\pm 2X_i \mid i=1,\hdots,n\}. \]

The Weyl group $W_K$ acts on $\frs^*$ as signed permutations of the coordinate functions $\{X_1,\hdots,X_n\}$ with any number of sign changes.  As we have seen in previous examples, $W_K$ embeds into $W$ as the signed elements of $S_{2n}$.

\subsection{A formula for the closed orbit}
As was the case with $K = SO(2n+1,\C)$, here there is only one closed orbit:

\begin{prop}
There is precisely $1$ closed $K$-orbit $Q$ on $G/B$ - namely, $Q=K \cdot 1B$, the orbit of the $S$-fixed point corresponding to the identity of $W$.  The $S$-fixed points contained in $Q$ correspond to the images of elements of $W_K$ in $W$, i.e. to signed elements of $S_{2n}$.
\end{prop}
\begin{proof}
The exact same proof given for the case $K=SO(2n+1,\C)$ goes through here, since the induced map on $W$ is once again $w \mapsto w_0 w w_0$.
\end{proof}

$Q$ being the only closed $K$-orbit, we give a formula for its $S$-equivariant class.

\begin{prop}\label{prop:formula_for_closed_sp_orbit}
Let $Q$ be the closed $K$-orbit of the previous proposition.  Then $[Q]$ is represented by
\[ P(x,y) := \displaystyle\prod_{1 \leq i < j \leq n}(y_i + y_j)(y_i + y_{2n+1-j}). \]
\end{prop}
\begin{proof}
The proof here is very similar to the orthogonal cases, except a bit simpler.  We start by computing $[Q]|_w$ for a fixed point $w \in Q$.  The set $\rho(w \Phi^+)$ is as follows:

\begin{enumerate}
	\item For each $1 \leq i < j \leq n$, exactly one of $\pm (X_i - X_j)$, with multiplicity 2
	\item For each $1 \leq i < j \leq n$, exactly one of $\pm (X_i + X_j)$, with multiplicity 2
	\item For each $i=1,\hdots,n$, exactly one of $\pm 2X_i$, with multiplicity 1
\end{enumerate}

Subtracting roots of $K$, here we are left with the following weights:

\begin{enumerate}
	\item For each $1 \leq i < j \leq n$, exactly one of $\pm (X_i - X_j)$, with multiplicity 1
	\item For each $1 \leq i < j \leq n$, exactly one of $\pm (X_i + X_j)$, with multiplicity 1
\end{enumerate}

Just as in the orthogonal cases, the number of $X_i \pm X_j$ ($i < j$) occurring with a negative sign is congruent mod $2$ to $l(|w|)$.  We conclude that if $w \in Q$ is an $S$-fixed point, then

\[ [Q]|_w = F(X) := (-1)^{l(|w|)} \displaystyle\prod_{1 \leq i < j \leq n} (X_i + X_j)(X_i-X_j). \]

So we seek a polynomial in $x_1,\hdots,x_n,y_1,\hdots,y_{2n}$, say $f$, with the property that

\[ f(X,\rho(wY)) = 
\begin{cases}
F(X) & \text{ if $w \in W_K$} \\
0 & \text{ otherwise.}
\end{cases} \]

The arguments needed to show that $P$ has these properties have already been made in the case of the even orthogonal group.  Indeed, there we saw that whether $w \in W_K$ changed an even or an odd number of signs (the cases were considered separately for two orbits $Q_1$ and $Q_2$), restricting the terms $y_i + y_j$ and $y_i + y_{2n+1-j}$ to $X_i + X_j$ and $X_i - X_j$, then applying the signed permutation $w$, we got 
\[ (-1)^{l(|w|)} \displaystyle\prod_{1 \leq i < j \leq n} (X_i + X_j)(X_i-X_j), \]

as required.  On the other hand, it was also argued in the case of the even orthogonal group that if $w$ is not a signed permutation (in this case, this is equivalent to the statement that $w \notin W_K$), then restricting and applying $w$ to $P$ gives zero.  This completes the proof.
\end{proof}

\subsection{Parametrization of $K \backslash G/B$ and the weak order}
We are again able to describe the orbit set and the weak order on the level of twisted involutions.  We refer the reader to \cite[Example 10.4]{Richardson-Springer-90}.

In this case, the map $\phi$ is not bijective, but it is injective.  Thus the orbits are in one-to-one correspondence with twisted involutions in the image $\phi(K \backslash G/B)$.  Moreover, it follows once again from the facts stated in Subsection \ref{ssec:twisted_involutions} (namely, from Corollary \ref{cor:weak_order_for_phi_injective} and Proposition \ref{prop:full_closure_twisted}) that the weak (respectively, full) closure order on $K \backslash G/B$ is given precisely by the restriction of the weak (respectively, full) Bruhat order on $\caI$ to the set $\phi(K \backslash G/B)$.

Observe that as with the orthogonal groups, the map induced by $\theta$ on $W$ is $\theta(w) = w_0 w w_0^{-1}$; hence the twisted involutions $\caI$, and the weak Bruhat order on them, are the same as in that case.  In particular, elements of $\caI$ correspond, after right multiplication by $w_0$, to the honest involutions in $W$.  In the case at hand, the image $\phi(K \backslash G/B)$ is then the set of all twisted involutions $a \in \caI$ such that $aw_0$ is a \textit{fixed point-free} involution.  Because the full closure order is given by the reverse Bruhat order on such involutions, and because $w_0$ is fixed point-free, we see that $w_0$ again corresponds to the unique closed orbit.  (Note, however, that $1$ no longer corresponds to the dense orbit, since $1$ has fixed points.  Indeed, the dense orbit in this case corresponds to the involution $(1,2)(3,4)\hdots,(2n-1,2n)$.)

Recall from our discussion of the orthogonal case that the weak Bruhat order on $\caI$ is generated by inductively applying the following rules to an involution $b \in W$, starting with the unique minimal element $w_0$ (assuming that $l(s_ib) < l(b)$):
\begin{enumerate}
	\item If $s_i b s_i \neq b$, then $b <_i s_ibs_i$.
	\item If $s_i b s_i = b$, then $b <_i s_ib$.
\end{enumerate}

Suppose here, though, that the condition of (1) fails, i.e. assume that $s_i b s_i = b$.  As noted in Subsection \ref{ssec:so2n_param}, this implies that $b$ preserves the set $\{i,i+1\}$.  If $b$ is assumed fixed point-free, then it must interchange $i$ and $i+1$.  Thus we see that $s_i b$ has fixed points (namely, $s_ib(i) = i$ and $s_ib(i+1)=i+1$).  This means that the $M$-action of $s_i$ on $b$ actually takes us out of the set $\phi(K \backslash G/B)$.  In this event, there is simply no edge originating at $b$ with label $i$, since the weak order on $K \backslash G/B$ coincides with the restriction of the weak order on $\caI$ to $\phi(K \backslash G/B)$.

Finally, as for the issue of black or blue edges, one easily checks that in this case, the map on the roots $\Phi(G,T)$ induced by $\theta$ is the same as in the orthogonal cases --- namely, it is defined by $\theta(Y_i) = -Y_{2n+1-i}$.  This means that deciding whether a root is complex or non-compact imaginary for a given orbit works just as in those cases.  (We refer the reader back to the proof of Proposition \ref{prop:weak_order_so_odd}.)  Recall that with the orthogonal groups, in case (1) above, the root $\ga_i$ was complex for the orbit corresponding to $b$, and so the edge originating at $b$ with label $i$ was black.  In case (2), the root was non-compact imaginary type II, and so the corresponding edge was blue, but as we have just noted, case (2) does not occur when we restrict our attention to involutions which actually correspond to $Sp(2n,\C)$-orbits.  The upshot is that in the present case, all roots are complex.  In particular, there are no blue edges in the weak order graph.

We sum up the discussion as follows:
\begin{prop}
The weak order on $K \backslash G/B$ corresponds to the order on fixed point-free involutions given inductively as follows:  Starting with $w_0$, for any fixed point-free involution $b$, and for any $s_i$ such that $l(s_ib) < l(b)$, $b <_i s_i b s_i$ if and only if $s_i b s_i \neq b$.  All roots are complex, and hence all edges are black.
\end{prop}

The parametrization of $K \backslash G/B$ by fixed point-free involutions encodes precisely the same linear algebraic description of the orbits in this case as it does in the case of the orthogonal groups.  Namely, letting $\gamma$ denote the symplectic form with isometry group $K$, if we define $\caO_b$ to be
\[ \{ F_{\bullet} \in X \ \vert \ \text{rank}(\gamma|_{F_i \times F_j}) = r_b(i,j) \text{ for all } i,j \},\]
then $\caO_b$ is a single $K$-orbit on $G/B$, and the association $b \mapsto \caO_b$ defines a bijection between the set of fixed point-free involutions and $K \backslash G/B$.

This can be seen in the same way as in the orthogonal case, and a representative of each orbit may be produced by the same procedure.  Because the argument is identical, we omit the details.

\subsection{Example}
We give the details of the computation in the very small case $n = 2$ (so $(G,K)=(SL(4,\C),Sp(4,\C))$).  Here, there are 3 fixed point-free involutions, and hence 3 orbits.  The involutions are $(1,2)(3,4)$, $(1,3)(2,4)$, and $(1,4)(2,3)$.

We start at $w_0=(1,4)(2,3)$ and work upward, applying the rule of the previous subsection:
\[ s_1 w_0 s_1 = (1,3)(2,4), \]
\[ s_2 w_0 s_2 = w_0, \]
\[ s_3 w_0 s_3 = (1,3)(2,4), \]

so $w_0 <_1 (1,3)(2,4)$ and $w_0 <_3 (1,3)(2,4)$.  Next, we move up to $(1,3)(2,4)$, noting that we only need to compute the $M$-action of $s_2$:
\[ s_2 (1,3)(2,4) s_2 = (1,2)(3,4), \]

and we are done.  The weak order graph appears as Figure \ref{fig:type-a-symplectic-1} of Appendix B.

By Proposition \ref{prop:formula_for_closed_sp_orbit}, the formula for $[Y_{w_0}]$ is $(y_1+y_2)(y_1+y_3)$.  We obtain $[Y_{(1,3)(2,4)}]$ by applying either $\partial_1$ or $\partial_3$.  In either case, the result is $[Y_{(1,3)(2,4)}] = y_1+y_2$.  Finally, we obtain $[Y_{(1,2)(3,4)}]$ by applying $\partial_2$ to $[Y_{(1,3)(2,4)}]$, and of course the result is $[Y_{(1,2)(3,4)}] = 1$.  These formulas appear in Table \ref{tab:type-a-sp4}.

The weak order graph and formulas for the larger example $n=3$ appear in Figure \ref{fig:type-a-symplectic-2} and Table \ref{tab:type-a-sp6}, respectively.  (In that case, there are 15 orbits.)

%% file: ch3-Type_B.tex
\chapter{Examples in Type $B$}
As was our preference in the type $A$ case $(G,K) = (SL(2n+1,\C),SO(2n+1,\C))$, here we realize $SO(2n+1,\C)$ as the subgroup of $SL(2n+1,\C)$ preserving the orthogonal form given by the antidiagonal matrix $J = J_{2n+1}$.  That is, 

\[ SO(2n+1,\C) =
\left\{ g \in SL(2n+1,\C) \ \vert \ g J g^t = J \right\}. \]

Fix a maximal torus $T$ of $G$, and let $Y_i$ ($i=1,\hdots,n$) denote coordinates on $\frt$.  In this case, the roots are
\[ \Phi = \{\pm (Y_i \pm Y_j) \ \vert \ 1 \leq i < j \leq n \} \cup \{\pm Y_i \ \vert \ 1 \leq i \leq n \}. \]
We choose the ``standard" positive system 
\[ \Phi^+ = \{Y_i \pm Y_j \ \vert \ i < j\} \cup \{Y_i \ \vert \ i = 1,\hdots,n\}, \]
and take $B$ to be the Borel subgroup containing $T$ and corresponding to the negative roots.  Again, to make things concrete, one may take $T$ to be the diagonal elements of $G$, let 
\[ Y_i(\text{diag}(a_1,\hdots,a_n,0,-a_n,\hdots,-a_1) = a_i, \]
and take $B$ to be the lower-triangular elements of $G$.

Let $X=G/B$ be the flag variety.  $X$ identifies with the set of flags which are isotropic with respect to the quadratic form
\[ \left\langle x,y \right\rangle = \displaystyle\sum_{i = 1}^{2n+1} x_i y_{2n+2-i}. \]
Thus a point of $X$ can be thought of as a partial flag of the form
\[ \{0\} = V_0 \subseteq V_1 \subseteq \hdots \subseteq V_n, \]
with $\dim V_i = i$ and each $V_i$ isotropic with respect to $\left\langle , \right\rangle$.  Such a flag is canonically extended to a complete flag $V_0 \subset V_1 \subset \hdots \subset V_{2n+1}$ by defining $V_{2n+1-i} = V_i^{\perp}$ for $i=0,\hdots,n$.

The Weyl group $W$ acts on the $Y_i$ as the $2^{n} n!$ signed permutations of $\{1,\hdots,n\}$ which change any number of signs.  The $T$-fixed points of $X$ then correspond to such permutations, as usual.

\section{$K \cong S(O(2p,\C) \times O(2q+1,\C))$}\label{ssec:type_b_ex_1}
In type $B$, there is only one family of symmetric subgroups $K$ to consider, up to conjugacy.  Suppose $p+q=n$, and take $\theta$ to be the involution
\[ \theta(g) = I_{p,2q+1,p} g I_{p,2q+1,p}. \]

One checks easily that $G$ is stable under $\theta$, and that 
\[ K = G^{\theta} = 
\left\{
k =  
\begin{pmatrix}
K_{11} & 0 & K_{13} \\
0 & K_{22} & 0 \\
K_{31} & 0 & K_{33}
\end{pmatrix}
\ \middle\vert \ 
\begin{array}{l}
K_{11}, K_{13}, K_{31}, K_{33} \in \text{Mat}(p,p) \\
\begin{pmatrix}
K_{11} & K_{13} \\
K_{31} & K_{33}
\end{pmatrix} \in O(2p,\C) \\
K_{22} \in O(2q+1,\C) \\
\det(k) = 1
\end{array}
\right\} \]
\[ \cong S(O(2p,\C) \times O(2q+1,\C)). \]

This choice of $K$ corresponds to the real form $G_\R = SO(2p,2q+1)$ of $G$.

Let $S \subseteq K$ be a maximal torus of $K$ contained in $T$.  This is an equal rank case, so in fact $S = T$.  We formally distinguish coordinates on $\frs$ (labeled by $X$ variables) from those on $\frt$ (labelled by $Y$ variables), with restriction given by $\rho(Y_i) = X_i$.

This is the first time that we have encountered a $K$ which is not connected.  We handle this by considering the connected components of the closed orbits separately.  As we will see, each closed orbit has two components.  Each is a single $K^0$-orbit, with $K^0 = SO(2p,\C) \times SO(2q+1,\C)$ the identity component of $K$.  These $K^0$-orbits coincide with the closed $\wt{K}$-orbits, with $\wt{K} = S(Pin(2p,\C) \times Pin(2q+1,\C))$ the corresponding (connected) symmetric subgroup of the simply connected cover $\wt{G} = Spin(2n+1,\C)$ of $G$.  Since $S \subset K^0$, each such component is stable under $S$, and hence has a $S$-equivariant class.  We apply our usual method to find formulas for these $S$-equivariant classes.  Having done so, we next identify exactly how the closed $K$-orbits break up as unions of these components.  We find a formula for each closed $K$-orbit by simply adding the formulas for the two components.  Finally, we parametrize the $K$-orbits by $(2p,2q+1)$-clans satisfying a certain additional combinatorial property, and describe the weak closure order on $K \backslash G/B$ in terms of this parametrization.  This allows us to perform the rest of the computation as in the type $A$ cases.

Identifying the connected components of \textit{all} $K$-orbits would amount to a parametrization of the $\wt{K}$-orbits on $X$.  Given such a parametrization, along with a description of the weak closure order, one could find formulas for the classes of all $\wt{K}$-orbit closures, which would then give formulas for the classes of irreducible components of $K$-orbit closures.  Formulas for $K$-orbit closures would follow by adding the formulas for the irreducible components.  One might say that this would be a more complete solution to the problem at hand.

We justify our approach as follows:  First, the $K$-orbits on $X$ are simpler to parametrize, since they are precisely the intersections of $K' = GL(2p,\C) \times GL(2q+1,\C)$-orbits on the type $A$ flag variety $X'$ with the smaller flag variety $X$.  Such intersections need not be single $\wt{K}$-orbits on $X$; some are, while others (e.g. the closed orbits) are unions of two $\wt{K}$-orbits.  It is not completely obvious how to determine precisely which intersections are single $\wt{K}$-orbits, and which are not.

Second, due to the fact that the $K$-orbits are intersections of $K'$-orbits on $X'$ with $X$, formulas for the classes of their closures pull back to Chern class formulas for degeneracy loci which admit identical linear algebraic descriptions to those in the type $A$ case, but which involve a vector bundle equipped with a quadratic form and an \textit{isotropic} flag of subbundles.  So from this standpoint, our type $B$ formulas have similar applications to those obtained in the type $A$ case, where the symmetric subgroup in question was connected.

The author acknowledges, however, that it would be nice to have the $\wt{K}$-orbit picture sorted out, since formulas for these classes would pull back to formulas for irreducible components of such degeneracy loci, giving more refined information.  Thus describing the combinatorics of those orbits would likely be a worthwhile question to consider.  However, we do not attempt to solve the problem in this generality here.

With all of these preliminary comments made, we turn our attention now to the ($S$-stable) connected components of closed $K$-orbits on $X$.  As stated, these coincide with the closed $\wt{K}$-orbits on $X$.  We will denote the Weyl group for $\wt{K}$ (alternatively, for $K^0$, or for $\text{Lie}(K)$) by $\wt{W_K}$.  $\wt{W_K}$ embeds in $W$ as those signed permutations of $\{1,\hdots,n\}$ which act separately as signed permutations of $\{1,\hdots p\}$ (changing an \textit{even} number of signs) and $\{p+1,\hdots,n\}$ (changing \textit{any} number of signs).  There are $2^{n-1} p! q!$ such permutations.

\subsection{Formulas for the closed $\wt{K}$-orbits}\label{ssec:type_b_ex_1_closed_twiddle_orbits}
Because we are in an equal rank case, and because $\wt{K}$ is connected, it follows from Corollary \ref{cor:num-closed-orbits-equal-rank} that the number of closed orbits is $|W/\wt{W_K}| = 2 \binom{n}{p}$, each containing $|\wt{W_K}|$ $S$-fixed points.

Define a function $f$ on $W$ by
\[ f(w) = \#\{i \in \{1,\hdots,n\} \ \vert \ w(i) < 0, |w(i)| \leq p\}. \]

(So, for example, if
\[ n = 5, p = 3, w = \overline{2}\overline{4}13\overline{5}, \]
then $f(w) = 1$.)

Recall also the definition of the function $l_p$, given in Subsection \ref{ssec:closed-orbits-supq}.

Then we have the following proposition:

\begin{prop}\label{prop:formula_for_SO(2p)_cross_SO(2q+1)}
Let $Q \in \wt{K} \backslash G / B$ be a closed $\wt{K}$-orbit, and suppose it contains the $S$-fixed point $w$.  Then $[Q]$ is represented by $(-1)^{f(w) + l_p(|w|)} P(x,y)$, where
\[ P(x,y) = \frac{1}{2} (x_1 \hdots x_p + y_{w^{-1}(1)} \hdots y_{w^{-1}(p)}) \displaystyle\prod_{i \leq p < j}(x_i - y_{w^{-1}(j)})(x_i + y_{w^{-1}(j)}). \]
\end{prop}
\begin{proof}
Before giving the proof, we first clarify that for a signed permutation $w \in W$, the notation $y_{w^{-1}(j)}$ means $y_{|w|^{-1}(j)}$ if $w^{-1}(j) > 0$, and $-y_{|w|^{-1}(j)}$ if $w^{-1}(j) < 0$.  We prefer this notation to, say, $w^{-1}(y_j)$, because it is more compact, and also because it is consistent with the notation used in the type $A$ cases.

We verify first that this formula is independent of the choice of $w$.  First, we note that the function $f$ is constant modulo 2 on right cosets $\wt{W_K} w$, since elements of $\wt{W_K}$ permute $\{1,\hdots,p\}$ with an even number of sign changes.  Considering $l_p(|w|)$, note that if $w' = w_k w$, then $|w'| = |w_k| |w|$, and $|w_k|$ is a permutation of $\{1,\hdots,n\}$ which acts separately on $\{1,\hdots,p\}$ and $\{p+1,\hdots,n\}$.  This says that $l_p(|w'|) = l_p(|w|)$.

Next, consider the term 

\[ x_1 \hdots x_p + y_{w^{-1}(1)} \hdots y_{w^{-1}(p)}. \]

Replacing $w$ by $w_k w$, we get 

\[ x_1 \hdots x_p + y_{w^{-1}(w_k^{-1}(1))} \hdots y_{w^{-1}(w_k^{-1}(p))} = x_1 \hdots x_p + y_{w^{-1}(1)} \hdots y_{w^{-1}(p)}, \]

since $w_k$ permutes $\{1,\hdots,p\}$ with an even number of sign changes.  Finally, to see that the product 

\[ \displaystyle\prod_{i \leq p < j}(x_i - y_{w^{-1}(j)})(x_i + y_{w^{-1}(j)}) \]

also does not depend on the choice of $w$, it is perhaps easiest to note that this expression is unchanged if we replace $w$ by $|w|$.  This reduces matters to the same type argument given in the $(SL(n,\C),S(GL(p,\C) \times GL(q,\C)))$ case.

With that established, we now apply Proposition \ref{prop:restriction-of-closed-orbit} to compute the restriction $[Q]|_w$.  Applying $w$ to positive roots of the form $Y_i$, we obtain $Y_{w(i)} = \pm Y_j$ for $j=1,\hdots,n$.  Applying $w$ to $Y_i \pm Y_j$, we obtain, for each $k < l$, exactly one of $\pm (Y_k + Y_l)$, and exactly one of $\pm (Y_k - Y_l)$.  Restricting to $\frs$ (i.e. replacing $Y$'s with $X$'s) and eliminating roots of $K$, we are left with $\pm X_j$ with $j \leq p$, along with, for each $k \leq p < l$, exactly one of $\pm (X_k + X_l)$, and exactly one of $\pm (X_k - X_l)$.

The number of weights of the form $\pm X_j$ occurring with a negative sign is clearly $f(w)$.  We may argue as in the proof of Proposition \ref{prop:formula_for_SO_even} to determine (modulo $2$) how many roots of the latter types occur with a negative sign.  The only difference between that case and this one is that here, we are only concerned with the inversions of pairs $i < j$ where $|w(j)| \leq p < |w(i)|$.  That is, the number of such roots occurring with a minus sign is congruent mod $2$ to $l_p(|w|)$ (as opposed to $l(|w|)$, as we saw in the proof of Proposition \ref{prop:formula_for_SO_even}).

The upshot is that for any $S$-fixed point $w \in Q$,
\[ [Q]|_w = F(X) := (-1)^{f(w)+l_p(|w|)} X_1 \hdots X_p \displaystyle\prod_{i \leq p < j}(X_i + X_j)(X_i - X_j). \]

(In particular, we see again that the class $[Q]$ actually restricts identically at every $S$-fixed point of $Q$.)

Thus we must prove that 
\[ P(X,\sigma X) = 
\begin{cases}
	F(X) & \text{ if $\sigma = w'w$ for some $w' \in \wt{W_K}$}, \\
	0 & \text{ if $\sigma w^{-1} \notin \wt{W_K}$.}
\end{cases}
\]

First, we establish this when $Q$ is the orbit containing the $S$-fixed point corresponding to the identity.  The general case follows easily.  Suppose first that $w \in \wt{W_K}$.  Since $w$ permutes $\{1,\hdots,p\}$ with an even number of sign changes, we have
\[ X_{w(1)} \hdots X_{w(p)} = X_1 \hdots X_p. \]

Further, since $w$ also permutes $\{p+1,\hdots,n\}$, we see that
\[ \displaystyle\prod_{i \leq p < j}(X_i + X_{w(j)})(X_i - X_{w(j)}) = \displaystyle\prod_{i \leq p < j}(X_i + X_j)(X_i - X_j). \]

(Note that $w$ \textit{can} change any number of signs, but this is taken care of by the presence of both $X_i + X_{w(j)}$ and $X_i - X_{w(j)}$ in the expression.)  All this is to say that
\[ P(X,wX) = X_1 \hdots X_p \displaystyle\prod_{i \leq p < j}(X_i + X_j)(X_i - X_j) \]
for $w \in \wt{W_K}$.

Now, suppose $w \notin \wt{W_K}$.  Then one of two things is true:  Either $w$ is separately a signed permutation of $\{1,\hdots,p\}$ and $\{p+1,\hdots,n\}$, but permutes $\{1,\hdots,p\}$ with an \textit{odd} number of sign changes, or $w$ is not separately a signed permutation of $\{1,\hdots,p\}$ and $\{p+1,\hdots,n\}$, in which case $w$ sends some $j > p$ to $\pm i$ for some $i \leq p$.  In the former case, we see that 
\[ X_1 \hdots X_p + X_{w(1)} \hdots X_{w(p)} = 0, \]
while in the latter case, either $X_i + X_{w(j)} = 0$, or $X_i - X_{w(j)} = 0$, whence
\[ \displaystyle\prod_{i \leq p < j}(X_i + X_{w(j)})(X_i - X_{w(j)}) = 0. \]

Together, these two facts say that 
\[ P(X,wX) = 0 \]
whenever $w \notin \wt{W_K}$.  We conclude that $P(x,y)$ represents $[Q]$.

Now, suppose that $\wt{Q}$ is another closed $K$-orbit, containing the $S$-fixed point $w \notin \wt{W_K}$.  All $S$-fixed points contained in $\wt{Q}$ are then of the form $w'w$ for $w' \in \wt{W_K}$.  So for any $w'w \in \wt{Q}$, we have
\[ P(X,w'w(X)) = \frac{1}{2} (X_1 \hdots X_p + X_{w'(1)} \hdots X_{w'(p)}) \displaystyle\prod_{i \leq p < j}(X_i - X_{w'(j)})(X_i + X_{w'(j)}) = \]
\[ X_1 \hdots X_p \displaystyle\prod_{i \leq p < j}(X_i + X_j)(X_i - X_j), \]
by our previous argument, since $w' \in \wt{W_K}$.  Noting that this is precisely what $P(X,w'w(X))$ is to be up to sign, and noting that we have corrected the sign by the appropriate factor of $(-1)^{f(w) + l_p(|w|)}$ in the statement of the proposition, we see that our proposed expression restricts correctly at $S$-fixed points contained in $\wt{Q}$.

On the other hand, for any $S$-fixed point $\wt{w}$ \textit{not} contained in $\wt{Q}$, we may write $\widetilde{w} = w'w$ for $w' \notin \wt{W_K}$.  Then
\[ P(X,\wt{w}(X)) = P(X,w'w(X)) = \]
\[ \frac{1}{2} (X_1 \hdots X_p + X_{w'(1)} \hdots X_{w'(p)}) \displaystyle\prod_{i \leq p < j}(X_i - X_{w'(j)})(X_i + X_{w'(j)}) = 0, \]
again by our previous argument, since $w' \notin \wt{W_K}$.

This completes the proof.
\end{proof}

\subsection{Parametrization of $K \backslash X$ and the weak order}\label{ssec:type-b-weak-order}
Recall the definition of a symmetric clan, given in Subsection \ref{ssec:clans}.  The following fact can be found in \cite{Matsuki-Oshima-90}:

\begin{fact}
The $K$-orbits on $X$ are parametrized by the set of all symmetric $(2p,2q+1)$-clans.
\end{fact}

In fact, as stated in Theorem \ref{thm:k-orbit-intersections}, the $K$-orbits on $X$ are precisely the intersections of the $K' = GL(2p,\C) \times GL(2q+1,\C)$-orbits on the type $A$ flag variety which meet $X$ non-trivially.  These $K'$-orbits are precisely those whose clans are symmetric.  Proofs are given in Appendix A.

We now describe the weak order on $K \backslash X$ in terms of its parametrization by symmetric $(2p,2q+1)$-clans.  The reference is \cite{Matsuki-Oshima-90}.

Order the simple roots as follows:  $\ga_i = Y_i - Y_{i+1}$ for $i=1,\hdots,n-1$, and $\ga_n = Y_n$.

Suppose that $\gamma=(c_1,\hdots,c_{2n+1})$ is a symmetric $(2p,2q+1)$-clan.  Then for $i=1,\hdots,n-1$, $\ga_i$ is complex for $Q_{\gamma}$ (and $s_{\ga_i} \cdot Q_{\gamma} \neq Q_{\gamma}$) if and only if one of the following holds:
\begin{enumerate}
	\item $c_i$ is a sign, $c_{i+1}$ is a number, and the mate of $c_{i+1}$ occurs to the right of $c_{i+1}$.
	\item $c_i$ is a number, $c_{i+1}$ is a sign, and the mate of $c_i$ occurs to the left of $c_i$.
	\item $c_i$ and $c_{i+1}$ are unequal natural numbers, the mate of $c_i$ occurs to the left of the mate of $c_{i+1}$, \textit{and} $(c_i,c_{i+1}) \neq (c_{2n-i},c_{2n-i+1})$.
\end{enumerate}

In the above cases, $s_{\ga_i} \cdot Q_{\gamma} = Q_{\gamma'}$, where $\gamma'$ is the clan obtained from $\gamma$ by interchanging $c_i$ and $c_{i+1}$, and also $c_{2n-i}$ and $c_{2n-i+1}$.

On the other hand, $\ga_i$ ($i < n$) is non-compact imaginary for $Q_{\gamma}$ if and only if one of the following two conditions holds:
\begin{enumerate}
	\item $c_i$ and $c_{i+1}$ (and, by symmetry, $c_{2n-i}$ and $c_{2n-i+1}$) are opposite signs.
	\item $c_i$ and $c_{i+1}$ are unequal natural numbers, with $(c_i,c_{i+1}) = (c_{2n-i},c_{2n-i+1})$.
\end{enumerate}

In case (1), $s_{\ga_i} \cdot Q_{\gamma} = Q_{\gamma''}$, where $\gamma''$ is obtained from $\gamma$ by replacing the signs in positions $i,i+1$ by a pair of matching natural numbers, and the signs in positions $2n-i,2n-i+1$ by a second pair of matching natural numbers.  In case (2), $s_{\ga_i} \cdot Q_{\gamma} = Q_{\gamma'''}$, where $\gamma'''$ is obtained from $\gamma$ by interchanging $c_i$ and $c_{i+1}$ (but \textit{not} $c_{2n-i}$ and $c_{2n-i+1}$).

Note that in case (1), $\ga_i$ is a type I root, since the cross action of $s_{\ga_i}$ is to interchange the opposite signs in positions $i,i+1$ and in positions $2n-i,2n-i+1$, so that $s_{\ga_i} \times Q_{\gamma} \neq Q_{\gamma}$.  On the other hand, in case (2), $\ga_i$ is a type II root, since the cross action of $s_{\ga_i}$ is to interchange the numbers in positions $i,i+1$ and the numbers in positions $2n-i,2n-i+1$.  Since $(c_i,c_{i+1}) = (c_{2n-i},c_{2n-i+1})$, this does not change the clan $\gamma$.  Thus $s_{\ga_i} \times Q_{\gamma} = Q_{\gamma}$, and $\ga_i$ is type II.

Now, consider $\ga_n$.  This root is complex for $Q_{\gamma}$ (and $s_{\ga_n} \cdot Q_{\gamma} \neq Q_{\gamma}$) if and only if $c_n$ and $c_{n+2}$ are unequal natural numbers, with the mate for $c_n$ occurring to the left of the mate for $c_{n+2}$.  (Note that by symmetry, this implies that the mate for $c_n$ occurs to the left of $c_n$, and the mate for $c_{n+2}$ occurs to the right of $c_{n+2}$.)  In this event, $s_{\ga_n} \cdot Q_{\gamma} = Q_{\delta}$, where $\delta$ is obtained from $\gamma$ by interchanging $c_n$ and $c_{n+2}$ (but \textit{not} their mates).  For example, 
\[ s_{\ga_3} \cdot (1,+,1,-,2,+,2) = (1,+,2,-,1,+,2). \]

The root $\ga_n$ is non-compact imaginary for $Q_{\gamma}$ if and only $c_n$ and $c_{n+1}$ are opposite signs.  (By symmetry, this says that $(c_n,c_{n+1},c_{n+2})$ is $(+,-,+)$ or $(-,+,-)$.)  Then $s_{\ga_n} \cdot Q_{\gamma} = Q_{\delta'}$, where $\delta'$ is obtained from $\gamma$ by replacing the matching signs in positions $n,n+2$ by a pair of matching natural numbers, and flipping the sign in position $n+1$.  In this case, $\ga_n$ is a type II root, since the cross-action of $s_{\ga_n}$ on a $(2p,2q+1)$-clan is to interchange the characters in positions $(n,n+2)$.  Because these characters are matching signs, $s_{\ga_n} \times Q_{\gamma} = Q_{\gamma}$, so $\ga_n$ is type II.

\subsection{Formulas for the closed $K$-orbits}\label{ssec:type-b-closed-k-orbits}
Having described a parametrization of the $K$-orbits, we now describe the closed $K$-orbits as unions of closed $\wt{K}$-orbits.  This will enable us to give formulas for the classes of closed $K$-orbits by simply adding the formulas for the appropriate closed $\wt{K}$-orbits, which were obtained in Subsection \ref{ssec:type_b_ex_1_closed_twiddle_orbits}.

The closed $K$-orbits correspond to the symmetric $(2p,2q+1)$-clans consisting only of $+$'s and $-$'s.  These are symmetric clans consisting of $2p$ plus signs and $2q+1$ minus signs.  Note that any such clan has a minus sign in position $n+1$, and is completely determined by the first $n$ characters.  Among these first $n$ characters, $p$ are plus signs, while $q$ are minus signs.  Thus there are $\binom{n}{p}$ closed $K$-orbits.  Since there are $2 \binom{n}{p}$ closed $\wt{K}$-orbits, this justifies our earlier claim that each closed $K$-orbit is a union of two connected components, each of which is a closed $\wt{K}$-orbit.

We now identify the components of closed $K$-orbits.  Recall the algorithm of Subsection \ref{ssec:orbits_supq} for producing representatives of $K' = GL(2p,\C) \times GL(2q+1,\C)$-orbits on $GL(2n+1,\C)/B$.  For a closed $K'$-orbit, it is easy to see that this algorithm produces an \textit{isotropic} flag precisely when the permutation $\sigma$ is chosen to be a signed element of $S_{2n+1}$, cf.  Subsection \ref{ssec:notation}.  That this representative is an $S$-fixed point is clear given how the algorithm is defined.  Based on this, one sees that for a symmetric $(2p,2q+1)$-clan $\gamma$ consisting only of $+$'s and $-$'s, the $S$-fixed points contained in $Q_{\gamma}$ correspond to all signed permutations of $\{1,\hdots,n\}$ which can be assigned to $\gamma$ in the following way:  Considering only the first $n$ characters of $\gamma$ , one assigns either $\pm j$ ($j=1,\hdots,p$) to the positions of the $p$ plus signs, and either $\pm k$ ($k = p+1,\hdots,n$) to the positions of the $q$ minus signs.

Now, recall that if $w \in W$ is an $S$-fixed point, the $S$-fixed points contained in the closed orbit $\wt{K} \cdot wB$ are $\wt{W_K} w$, and that $\wt{W_K}$ consists of signed permutations which act separately on $\{1,\hdots,p\}$ and $\{p+1,\hdots,n\}$, with an \textit{even} number of sign changes on the first set.  On the other hand, the above characterization of $S$-fixed points contained in a closed $K$-orbit says that they are of the form $\sigma w$, where $\sigma$ is a signed permutation which acts separately on $\{1,\hdots,p\}$ and $\{p+1,\hdots,n\}$, changing \textit{any} number of signs on either set.  The conclusion is that the closed $K$-orbit $K \cdot wB$ is the union of $\wt{K} \cdot wB$ and $\wt{K} \cdot \pi wB$, where $\pi \in W$ is the signed permutation $\overline{1}2 \hdots n$.  (More generally, $\pi$ could be taken to be any signed permutation which acts separately on $\{1,\hdots,p\}$ and $\{p+1,\hdots,n\}$, and which changes an odd number of signs on the first set.  This particular choice of $\pi$ seems to the author to be the simplest such choice.)  

With this observed, we have the following corollary of Proposition \ref{prop:formula_for_SO(2p)_cross_SO(2q+1)}:

\begin{cor}
Suppose $Q$ is a closed $K$-orbit containing the $S$-fixed point $w$.  Then
\[ [Q] = (-1)^{l_p(|w|)} y_{|w|^{-1}(1)} \hdots y_{|w|^{-1}(p)} \displaystyle\prod_{i \leq p < j}(x_i - y_{|w|^{-1}(j)})(x_i + y_{|w|^{-1}(j)}). \]
\end{cor}
\begin{proof}
This follows from Proposition \ref{prop:formula_for_SO(2p)_cross_SO(2q+1)}, simply by adding the formulas for $\wt{K} \cdot wB$ and $\wt{K} \cdot \pi wB$.  This sum simplifies to the above expression when one makes the following easy observations:
\begin{enumerate}
	\item $f(\pi w) = f(w) + 1$
	\item $l_p(|w|) = l_p(|\pi w|)$
	\item $y_{w^{-1}(1)} \hdots y_{w^{-1}(p)}$ = $-y_{(\pi w)^{-1}(1)} \hdots y_{(\pi w)^{-1}(p)}$
	\item $(-1)^{f(w)} y_{w^{-1}(1)} \hdots y_{w^{-1}(p)}$ = $y_{|w|^{-1}(1)} \hdots y_{|w|^{-1}(p)}$
\end{enumerate}
\end{proof}

\subsection{Example}\label{ssec:type_b_ex_1_example}
With a parametrization of $K \backslash G/B$ and a description of the weak order in hand, we give the weak order graph and table of formulas for the case $p=2,q=1$ ($(G,K) = (SO(7,\C),S(O(4,\C) \times O(3,\C))$) in Figure \ref{fig:type-b-graph} and Table \ref{tab:type-b-table} of Appendix B.  There are $25$ orbits.

%% file: ch4-Type_C.tex
\chapter{Examples in Type $C$}
Next, consider the group $G=Sp(2n,\C)$.  We realize $G$ as the group of matrices which preserve the exterior form
\[ x_1 \wedge x_{2n} + x_2 \wedge x_{2n-1} + \hdots + x_n \wedge x_{n+1}. \]

That is, 
\[ Sp(2n,\C) =
\left\{ g \in GL(2n,\C) \ \vert \ g^t J_{n,n} g = J_{n,n} \right\}. \]

Let $T$ be a maximal torus of $G$, and let $Y_i$ denote coordinates on $\frt = \text{Lie}(T)$.  The roots $\Phi$ are of the form
\[ \Phi = \{\pm (Y_i \pm Y_j) \ \vert \ 1 \leq i < j \leq n \} \cup \{\pm 2Y_i \ \vert \ i=1,\hdots,n\}. \] 
Take the positive roots to be those of the form $Y_i \pm Y_j (i < j)$ and $2Y_i (1 \leq i \leq n)$, and take $B \supset T$ to be the Borel subgroup such that the roots of $\frb = \text{Lie}(B)$ are negative.  For example, we can take $T$ to be the diagonal elements of $G$, define $Y_i(\text{diag}(a_1,\hdots,a_n,-a_n,\hdots,-a_1)) = a_i$, and take $B$ to be the lower-triangular elements of $G$.

The $T$-fixed points of $G/B$ are parametrized by the elements of $W$, which once again can be thought of as the $2^n n!$ signed permutations of $\{1,\hdots,n\}$, changing any number of signs.  (This is the action of $W$ on the coordinates $Y_i$.)

\section{$K = Sp(2p,\C) \times Sp(2q,\C)$}
Let $p+q=n$.  Consider the involution $\theta=\text{int}(I_{p,2q,p})$ of $G$.  (Refer to Subsection \ref{ssec:notation} for this notation.)

One checks that 
\[ K = G^{\theta} = 
\left\{ 
\begin{pmatrix}
K_{11} & 0 & K_{13} \\
0 & K_{22} & 0 \\
K_{31} & 0 & K_{33}
\end{pmatrix}
\ \middle\vert \ 
\begin{array}{l}
K_{11}, K_{13}, K_{31}, K_{33} \in \text{Mat}(p,p) \\
\begin{pmatrix}
K_{11} & K_{13} \\
K_{31} & K_{33}
\end{pmatrix} \in Sp(2p,\C) \\
K_{22} \in Sp(2q,\C)
\end{array}
\right\} \]
\[ \cong Sp(2p,\C) \times Sp(2q,\C). \]

In the notation of the introduction, this choice of $K$ corresponds to the real form $G_\R = Sp(p,q)$ of $G$.

Since this is another equal rank case, $S=T$.  We label coordinates on $\frs$ by $X_i$ ($i=1,\hdots,n$), with restriction $\frt \rightarrow \frs$ given by $\rho(Y_i) = X_i$.

The roots of $K$ are as follows:
\[ \Phi_K = \{\pm 2X_i\} \cup \{\pm (X_i \pm X_j) \mid i<j \leq p \text{ or } p<i<j\}. \]
 
Note that the Weyl group $W_K$ embeds in $W$ as those signed permutations of $\{1,\hdots,n\}$ which act separately on $\{1,\hdots,p\}$ and $\{p+1,\hdots,n\}$, with any number of sign changes on each set.  There are $2^n p! q!$ such permutations.

\subsection{Formulas for the closed orbits}
By Corollary \ref{cor:num-closed-orbits-equal-rank}, there are $|W/W_K| = \binom{n}{p}$ closed $K$-orbits, each containing $|W_K| = 2^n p! q!$ $S$-fixed points.  These $S$-fixed points correspond to the elements of some left coset $W_K \cdot w$.

We have the following formulas for the equivariant classes of closed $K$-orbits.

\begin{prop}\label{prop:sp(p,q)-closed-formulas}
Let $Q \in K \backslash G / B$ be a closed $K$-orbit, and suppose it contains the $S$-fixed point $w$.  Then
\[ [Q] = P(x,y) = (-1)^{l_p(|w|)} \displaystyle\prod_{i \leq p < j}(x_i - y_{w^{-1}(j)})(x_i + y_{w^{-1}(j)}). \]
\end{prop}
\begin{proof}
The proof is very similar to the proof of Proposition \ref{prop:formula_for_SO(2p)_cross_SO(2q+1)}.  The arguments needed to see that this formula is independent of the choice of $w$ have already been given there.

Using Proposition \ref{prop:restriction-of-closed-orbit} and arguing as in the proof of Proposition  \ref{prop:formula_for_SO(2p)_cross_SO(2q+1)}, we find that the restriction $[Q]|_w$ is as follows:
\[ [Q]|_w = F(X) = (-1)^{l_p(|w|)} \displaystyle\prod_{i \leq p < j}(X_i - X_j)(X_i + X_j). \]

Again, we remark that $[Q]$ restricts identically at all $S$-fixed points contained in the orbit $Q$.

So we are seeking a polynomial $f(x,y)$ such that $f(X, \sigma wX) = F(X)$ whenever $\sigma \in W_K$, and such that $f(X,w'X) = 0$ whenever $w'w^{-1} \notin W_K$.  It is straightforward to verify that $P$ has these properties.  For $\sigma \in W_K$, we have that

\[ f(X,\sigma wX) = (-1)^{l_p(|w|)} \displaystyle\prod_{i \leq p < j}(X_i - X_{\sigma w w^{-1}(j)})(X_i + X_{\sigma w w^{-1}(j)}) = \]
\[ (-1)^{l_p(|w|)} \displaystyle\prod_{i \leq p < j}(X_i - X_{\sigma(j)})(X_i + X_{\sigma(j)}) = F(X), \]
since $\sigma$, being an element of $W_K$, preserves the sets $\{\pm 1,\hdots,\pm p\}$ and $\{\pm (p+1),\hdots,\pm n\}$.

On the other hand, if $w'w^{-1} \notin W_K$, then

\[ f(X,w'X) = (-1)^{l_p(|w|)} \displaystyle\prod_{i \leq p < j}(X_i - X_{w'w^{-1}(j)})(X_i + X_{w'w^{-1}(j)}) = 0, \]

since $w'w^{-1}$, not being an element of $W_K$, necessarily sends some $j > p$ to either $\pm i$ for some $i \leq p$, resulting in either the $X_i - X_{w'w^{-1}(j)}$ or the $X_i + X_{w'w^{-1}(j)}$ factor being zero.

We conclude that $P(x,y)$ represents $[Q]$.
\end{proof}

\subsection{Parametrization of $K \backslash X$ and the weak order}\label{ssec:param_typec_case1}
The following parametrization of $K \backslash X$ is described in \cite{Matsuki-Oshima-90}:
\begin{fact}
	The $K$-orbits on $X$ are parametrized by symmetric $(2p,2q)$-clans $\gamma$ having the following additional property:  If $\gamma=(c_1,\hdots,c_{2n})$, and if $c_i \in \N$, then $c_i \neq c_{2n+1-i}$.
\end{fact}

As indicated by Theorem \ref{thm:k-orbit-intersections}, the $K$-orbits on $X$ are precisely the nonempty intersections of the $GL(2p,\C) \times GL(2q,\C)$-orbits on the type $A$ flag variety with $X$.  See Appendix A for the proof.

We now describe the weak order on $K \backslash X$ combinatorially in terms of this parametrization.  References are \cite{Yamamoto-97,McGovern-Trapa-09,Matsuki-Oshima-90}.

Order the simple roots as follows:  $\ga_i = Y_i - Y_{i+1}$ for $i=1,\hdots,n-1$, and $\ga_n = 2Y_n$.  Let $\gamma = (c_1,\hdots,c_{2n})$ be a $(2p,2q)$-clan satisfying the conditions listed above.  Then for $i=1,\hdots,n-1$, $\ga_i$ is complex for $Q_{\gamma}$ (and $s_{\ga_i} \cdot Q_{\gamma} \neq Q_{\gamma}$) if and only if one of the following holds:
\begin{enumerate}
	\item $c_i$ is a sign, $c_{i+1}$ is a natural number, and the mate for $c_{i+1}$ occurs to the right of $c_{i+1}$.
	\item $c_i$ is a number, $c_{i+1}$ is a sign, and the mate for $c_i$ occurs to the left of $c_i$.
	\item $c_i$ and $c_{i+1}$ are unequal natural numbers, the mate for $c_i$ occurs to the left of the mate for $c_{i+1}$, \textit{and} $(c_i,c_{i+1}) \neq (c_{2n-i},c_{2n-i+1})$.
\end{enumerate}

In these cases, $s_{\ga_i} \cdot Q_{\gamma} = Q_{\gamma'}$, where $\gamma'$ is obtained from $\gamma$ by interchanging $c_i$ and $c_{i+1}$, \textit{and} $c_{2n-i}$ and $c_{2n-i+1}$.

On the other hand, $\ga_i$ is non-compact imaginary for $Q_{\gamma}$ if and only if $c_i$ and $c_{i+1}$ (and, by symmetry, $c_{2n-i}$ and $c_{2n-i+1}$) are opposite signs.  In this event, $s_{\ga_i} \cdot Q_{\gamma} = Q_{\gamma''}$, where $\gamma''$ is obtained from $\gamma$ by replacing the signs in positions $i$ and $i+1$ by matching natural numbers, and the signs in positions $2n-i$ and $2n-i+1$ by a second pair of matching natural numbers.

The root $\ga_n = 2Y_n$ is complex for $Q_{\gamma}$ (and $s_{\ga_n} \cdot Q_{\gamma} \neq Q_{\gamma}$) if and only if $c_n$ and $c_{n+1}$ are unequal natural numbers, with the mate for $c_n$ occurring to the left of the mate for $c_{n+1}$.  In this event, $s_{\ga_n} \cdot Q_{\gamma} = Q_{\gamma'''}$, where $\gamma'''$ is obtained from $\gamma$ by interchanging $c_n$ and $c_{n+1}$.  $\ga_n$ is never a non-compact imaginary root.

Finally, we note that all bonds in the weak order graph are single.  Indeed, the only time $\ga_i$ is non-compact imaginary for $Q_{\gamma}$ is when $\gamma$ has opposite signs in the $(i,i+1)$ and $(2n-i,2n-i+1)$ positions.  The cross action of $s_{\ga_i}$ on $Q_{\gamma}$ in that case is to switch each pair of signs.  Thus whenever $\ga_i$ is non-compact imaginary for $Q_{\gamma}$, we have $s_{\ga_i} \times Q_{\gamma} \neq Q_{\gamma}$, meaning that all non-compact imaginary roots are of type I.

\subsection{Example}
With this parametrization and combinatorial description of the weak order in hand, we now give an example.  Take $p=2,q=1$, so $(G,K) = (Sp(6,\C),Sp(4,\C) \times Sp(2,\C))$.  There are 9 orbits in this case.  The weak order graph is given as Figure \ref{fig:type-c-graph-1} of Appendix B.

To obtain a representative of each closed orbit, we use the method of \cite[Theorem 4.3.12]{Yamamoto-97}.  In the case of closed orbits, whose clans consist only of signs, this amounts to the following:  First, fix a permutation $\sigma'' \in S_n$ such that 
\[ 1 \leq \sigma''(i) \leq p \text{ if the $i$th character is $+$}, \]
\[ p+1 \leq \sigma''(i) \leq n \text{ if the $i$th character is $-$}. \]

Next, define a permutation $\sigma' \in S_{2n}$ with the property that
\[ \sigma'(i) = 
\begin{cases}
	\sigma''(i) & \text{ if $i \leq n$} \\
	2n+1-\sigma''(2n+1-i) & \text{ if $i > n$.}
\end{cases}
\]

Finally, take the representative $F_{\bullet} = \left\langle v_1,\hdots,v_{2n} \right\rangle$, with $v_i = e_{\sigma'(i)}$ for each $i$.  Note that any flag so obtained is $S$-fixed, so it is straightforward to apply Proposition \ref{prop:sp(p,q)-closed-formulas} once a representative is chosen in this way.  Divided difference operators then give the remaining formulas.  The results are given in Table \ref{tab:type-c-table-1} of Appendix B.

\section{$K \cong GL(n,\C)$}\label{ssec:type_c_ex_2}
Now, let $\theta = \text{int}(i \cdot I_{n,n})$.  One checks that
\[ K = G^{\theta} = 
\left\{ 
\begin{pmatrix}
g & 0 \\
0 & J_n \ (g^t)^{-1} J_n
\end{pmatrix}
\ \middle\vert \ 
g \in GL(n,\C) \right\} \cong GL(n,\C). \]

This symmetric subgroup corresponds to the real form $G_\R = Sp(2n,\R)$ of $G$.

Let $S$ be a maximal torus of $K$ contained in $T$.  Again, $S = T$, but we again denote coordinates on $\frs$ by $X_i$, with restriction $\frt \rightarrow \frs$ given by $\rho(Y_i) = X_i$.

Roots of $K$ are as follows:
\[ \Phi_K = \{\pm(X_i - X_j) \ \vert \ 1 \leq i < j \leq n\}. \]

$W_K \cong S_n$ is embedded in $W$ as ordinary permutations of $\{1,\hdots,n\}$ --- that is, signed permutations of $\{1,\hdots,n\}$ which change no signs.

\subsection{Formulas for the closed orbits}\label{ssec:closed-orbit-formulas-type-c-ex-2}
Again, we are in an equal rank case, so by Corollary \ref{cor:num-closed-orbits-equal-rank}, the $K$-orbit of every $S$-fixed point is closed.  There are $|W/W_K| = 2^n$ such orbits, each containing $|W_K| = n!$ $S$-fixed points.

Given $w \in W$, define
\[ \Delta_n(x,y,w) := \det(c_{n+1+j-2i}), \]
where
\[ c_k = e_k(x_1,\hdots,x_n) + e_k(y_{w^{-1}(1)},\hdots,y_{w^{-1}(n)}), \]
$e_k$ denoting the $k$th elementary symmetric function in the inputs.  Again, we remind the reader that the elements of $W$ are signed permutations, and that in the event that $w^{-1}(i) < 0$, the notation $y_{w^{-1}(i)}$ means $-y_{|w^{-1}(i)|}$.

Next, for each $w \in W$, define the set 
\[ \text{Neg}(w) := \{i \ \vert \ w(i) < 0 \}, \] 
and define $\N$-valued functions $f, g$ on $W$ by 
\[ f(w) = \#\text{Neg}(w), \]
and
\[ g(w) = \displaystyle\sum_{i \in \text{Neg}(w)}(n - i). \]

Then formulas for the equivariant classes of closed $K$-orbits are as follows:
\begin{prop}\label{prop:closed-gln-formulas-type-c}
	Let $Q = K \cdot wB$ be a closed $K$-orbit on $G/B$ represented by the $S$-fixed point $w$.  Then $[Q]$ is represented by the polynomial 
	\[ P(x,y) := (-1)^{f(w)+g(w)} \Delta_n(x,y,w). \]
\end{prop}
\begin{proof}
We start by noting that $P(x,y)$ is independent of the choice of fixed point $w$ representing $Q$.  Because fixed points contained in $Q$ are all (left) $W_K$ translates of $w$, they all have the same ``sign pattern" (i.e. the set $\text{Neg}(w)$ is the same regardless of the choice of $w$).  This follows because elements of $W_K$ change no signs.  To illustrate, note that when $n = 2$, the $4$ closed $K$-orbits on $G/B$ contain $S$-fixed points $\{12,21\}$, $\{\overline{1}2,\overline{2}1\}$, $\{1\overline{2},2\overline{1}\}$, and $\{\overline{1} \overline{2}, \overline{2} \overline{1}\}$.

It follows that the functions $f,g$ are constant on $W_K w$, so that $(-1)^{f(w) + g(w)}$ is independent  of the choice of $w$.  It is also easy to see that $\Delta_n(x,y,w)$ is independent of this choice.  Indeed, replacing $w$ by $w'w$ for $w' \in W_K$, each $c_k$ becomes
\[ e_k(x_1,\hdots,x_n) + e_k(y_{w^{-1}(w'^{-1}(1))},\hdots,y_{w^{-1}(w'^{-1}(n))}). \]
But because $w'$ is just an ordinary permutation of $\{1,\hdots,n\}$, the effect is simply to permute the $y_{w^{-1}(i)}$, and because $e_k$ is invariant under permutation of the variables $y_i$, each $c_k$ is unchanged.

With this established, we now apply Proposition \ref{prop:restriction-of-closed-orbit} to compute the restriction $[Q]|_w$.  Applying $w$ to positive roots of the form $2Y_i$ and restricting to $\frs$, we get weights of the form $2X_{w(i)} = \pm 2X_j$.  The number of such weights occurring with a minus sign is $f(w)$.

Applying $w$ to positive roots of the form $Y_i \pm Y_j$ ($i<j$) and then restricting, we get weights of the form $X_{w(i)} \pm X_{w(j)}$, and these two weights together are of the form $\pm X_k \pm X_l$, $\pm X_k \mp X_l$, for some $k,l$.  Those of the latter form $\pm X_k \mp X_l$ are roots of $K$, while those of the former are not.  The number of such roots which are negative (i.e. of the form $-X_k - X_l$) is precisely $g(w)$.  To see this, note that if $w(i)$ is positive, then applying $w$ to any pair of roots $Y_i+Y_j,Y_i-Y_j$ with $i<j$ and then restricting is going to necessarily give a positive root of the form $X_k+X_l$, where $k=w(i)$, $l = |w(j)|$.  If $w(i)$ is negative, then applying $w$ to any such pair will necessarily give a negative root of the form $-X_k-X_l$.  For any fixed $i$, the number of pairs $\{Y_i \pm Y_j\}$ with $i<j$ is precisely $n-i$.  So for each $i$ with $w(i)$ negative, $n-i$ negative roots occur, for a total of $g(w)$ negative roots.

All of this adds up to the following conclusion:  For any $S$-fixed point $w \in Q$, we have
\[ [Q]|_w = F(X) := (-1)^{f(w)+g(w)} 2^n X_1 \hdots X_n \displaystyle\prod_{i < j}(X_i + X_j). \]

(Once again, we remark that $[Q]$ restricts identically at each $S$-fixed point contained in $Q$.)

So for any $u \in W$,
\[ [Q]|_u = 
\begin{cases}
	F(X) & \text{ if $uw^{-1} \in W_K$}, \\
	0 & \text{ otherwise.}
\end{cases}. \]

Thus the claim is that $P(X,uX)$ is $F(X)$ if $uw^{-1} \in W_K$, and $0$ otherwise.  Write $u = w'w$.  Noting that $\Delta_n(x,y,w) = \Delta(x,w^{-1}y,id)$, we have
\[ P(X,uX) =  (-1)^{f(w) + g(w)} \Delta_n(X,w'X,1). \]

If $uw^{-1} \in W_K$, then $w'$ is an ordinary permutation (a signed permutation with no sign changes), whereas if $uw^{-1} \notin W_K$, $w'$ has at least one sign change.  So our claim that $P(x,y)$ represents $[Q]$ amounts to the claim that $\Delta_n(x,y,id)$ has the following two properties:
\begin{enumerate}
	\item It is invariant under permutations of the $x_i$, $y_i$.
	\item If $\epsilon_i = \pm 1$, then
	\[ \Delta_n((X_1,\hdots,X_n),(\epsilon_1X_1,\hdots,\epsilon_nX_n),id) \]
	is zero unless all $\epsilon_i$ are equal to $1$, in which case it is equal to 
	\[ 2^n X_1 \hdots X_n \displaystyle\prod_{i < j}(X_i + X_j). \]
\end{enumerate}
That $\Delta_n(x,y,id)$ has these properties is proved directly in \cite[\S 3]{Fulton-96_1}.
\end{proof}

\begin{remark}
It may seem surprising that formulas for classes of closed $K$-orbits are so closely related to Fulton's formula for the smallest Schubert locus in the flag bundle.  (Fulton's representative is a polynomial of higher degree, but has the determinant $\Delta_n(x,y,id)$ as a factor.)  In fact, this is not an accident.  As has been observed by the author (\cite{Wyser-11b}), a number of the orbit closures in this case are \textit{Richardson varieties} --- intersections of Schubert varieties with opposite Schubert varieties.  In particular, all of the closed $K$-orbits are Richardson varieties.  See \cite{Wyser-11b} for an application of this fact to type $C$ Schubert calculus.
\end{remark}

\subsection{Parametrization of $K \backslash X$ and the weak order}\label{ssec:type-c-case-2-param}
Recall the definition of a skew-symmetric $(n,n)$-clan (Definition \ref{def:skew-symmetric-clan}).  The following parametrization of $K \backslash X$ is described in \cite{Matsuki-Oshima-90}:
\begin{fact}\label{prop:orbit-param-type-c}
	$K$-orbits on $X$ are parametrized by the set of skew-symmetric $(n,n)$-clans.
\end{fact}

As indicated by Theorem \ref{thm:k-orbit-intersections}, these orbits are precisely the nonempty intersections of the $GL(n,\C) \times GL(n,\C)$-orbits on the type $A$ flag variety with $X$.  See Appendix A.

We give a combinatorial description of the weak order on $K \backslash G/B$ in terms of this parametrization.  References are \cite{Matsuki-Oshima-90,McGovern-Trapa-09,Yamamoto-97}.

Order the simple roots as we did in the case of $(Sp(2n,\C),Sp(2p,\C) \times Sp(2q,\C))$:  $\ga_i = Y_i - Y_{i+1}$ for $i=1,\hdots,n-1$, and $\ga_n = 2Y_n$.  Let $\gamma = (c_1,\hdots,c_{2n})$ be a skew-symmetric $(n,n)$-clan.

The situation for the roots $\ga_1,\hdots,\ga_{n-1}$ is exactly the same as was described in the case of the pair $(G,K) = (SO(2n+1,\C),S(O(2p,\C) \times O(2q+1,\C)))$.  Rather than repeat that description verbatim here, we refer the reader back to Subsection \ref{ssec:type-b-weak-order}.

Thus we need only describe the situation for $\ga_n$.  This root is complex for $Q_{\gamma}$ (and $s_{\ga_n} \cdot Q_{\gamma} \neq Q_{\gamma})$ if and only if $c_n$ and $c_{n+1}$ are unequal natural numbers, with the mate for $c_n$ to the left of the mate for $c_{n+1}$.  In this case, $s_{\ga_n} \cdot Q_{\gamma} = Q_{\delta}$, where $\delta$ is obtained from $\gamma$ by interchanging $c_n$ and $c_{n+1}$.

On the other hand, $\ga_n$ is non-compact imaginary for $Q_{\gamma}$ if and only if $c_n$ and $c_{n+1}$ are opposite signs.  In this case, $s_{\ga_n} \cdot Q_{\gamma} = Q_{\delta'}$, where $\delta'$ is obtained from $\gamma$ by replacing $c_n$ and $c_{n+1}$ by a pair of matching natural numbers.  In this case, $\ga_n$ is of type I, since the cross action of $s_{\ga_n}$ is to interchange the opposite signs in positions $n,n+1$, so that $s_{\ga_n} \times Q_{\gamma} \neq Q_{\gamma}$.

\subsection{Example}
With the parametrization and ordering spelled out, consider the example $n=2$.  There are $11$ orbits.  The weak order graph appears as Figure \ref{fig:type-c-graph-2} of Appendix B.

To obtain a representative of each closed orbit, we use the method of \cite[Theorem 3.2.11]{Yamamoto-97}.  In the case of closed orbits, whose clans once again consist only of signs, this amounts to the following:  Letting $\gamma = (c_1,\hdots,c_{2n})$, choose a permutation $\sigma \in S_{2n}$ with the following properties:
\begin{enumerate}
	\item If $i \leq n$ and $c_i = +$, $\sigma(i) \leq n$.
	\item If $i \leq n$ and $c_i = -$, $\sigma(i) > n$.
	\item For $i=1,\hdots,n$, $\sigma(2n+1-i) = 2n+1-\sigma(i)$.
\end{enumerate}

Having chosen such a $\sigma$, the flag $F_{\bullet} = \left\langle v_1,\hdots,v_{2n} \right\rangle$, with $v_i = e_{\sigma(i)}$, is a representative of $Q_{\gamma}$.  Note that any representative so obtained is $S$-fixed, so it is straightforward to apply Proposition \ref{prop:closed-gln-formulas-type-c} to compute the class $[Q_{\gamma}]$.  Divided difference operators (scaled by factors of $\frac{1}{2}$ where appropriate) then give the remaining formulas.  The results are given in Table \ref{tab:type-c-table-2} of Appendix B.

%% file: ch5-Type_D.tex
\chapter{Examples in Type $D$}
Now let $G = SO(2n,\C)$.  We realize $G$ as the subgroup of $SL(2n,\C)$ preserving the orthogonal form given by the antidiagonal matrix $J_{2n}$, as we did in type $B$.  That is, 

\[ SO(2n,\C) =
\left\{ g \in SL(2n,\C) \ \vert \ g^t J_{2n} g = J_{2n} \right\}. \]

Let $T$ be a maximal torus of $G$, and let $Y_i$ denote coordinates on $\frt = \text{Lie}(T)$.  The roots $\Phi$ are of the form
\[ \Phi = \{\pm (Y_i \pm Y_j) \ \vert \ 1 \leq i < j \leq n \}. \]
We take the positive roots to be those of the form $Y_i \pm Y_j (i < j)$, and take $B \supset T$ to be such that the roots of $\frb = \text{Lie}(B)$ are negative.  Concretely, given our chosen realization of $G$, we can take $T$ to be the diagonal elements of $G$, define $Y_i(\text{diag}(a_1,\hdots,a_n,-a_n,\hdots,-a_1)) = a_i$, and take $B$ to be the lower-triangular elements of $G$.

Consider the variety $V$ of flags on $\C^{2n}$ which are isotropic with respect to the quadratic form whose matrix is $J_{2n}$:
\[ \left\langle x,y \right\rangle = \displaystyle\sum_{i = 1}^{2n} x_i y_{2n+1-i}. \]

Unlike in types $B$ and $C$, here $V$ is not a homogeneous space for $G$.  Indeed, $V$ is disconnected, with two isomorphic components, each of which is a single $SO(2n,\C)$-orbit.  ($V$ \textit{is} a homogeneous space for $O(2n,\C)$.)  To obtain a homogeneous space for $G$, we must choose one of these two components.  We choose the component containing the ``standard" isotropic flag $E_{\bullet} = \left\langle e_1,\hdots,e_{2n} \right\rangle$.  Then $X=G/B$ can be identified with the set of flags $F_{\bullet}$ on $\C^{2n}$ having the following properties:

\begin{enumerate}
	\item $\dim F_i = i$;
	\item $F_1,\hdots,F_n$ are isotropic subspaces with respect to $\langle \cdot, \cdot \rangle$;
	\item $F_{2n-i} = F_i^{\perp}$ for $i = 0,1,\hdots,n$.
	\item $\dim(F_n \cap E_n) \equiv n \pmod{2}$.
\end{enumerate}

Note that conditions (1)-(3) above simply say that $F_{\bullet}$ is isotropic, while condition (4) is needed to guarantee that $F_{\bullet}$ lies in the correct component of $V$.  See \cite{Edidin-Graham-95} for more details.

Let $W$ be the Weyl group for $\text{Lie}(G)$.  We think of $W$ as the $2^{n-1} n!$ signed permutations of $\{1,\hdots,n\}$ which change an even number of signs.  This is the action of $W$ on the coordinates $Y_i$.  The $T$-fixed points on $G/B$ are then in one-to-one correspondence with $W$, as usual.

\section{$K \cong S(O(2p,\C) \times O(2q,\C))$}\label{ssec:type_d_ex_1}
Let $\theta = \text{int}(I_{p,2q,p})$.  Then $G$ is stable under $\theta$, and
\[ K = G^{\theta} = 
\left\{ k = 
\begin{pmatrix}
K_{11} & 0 & K_{13} \\
0 & K_{22} & 0 \\
K_{31} & 0 & K_{33}
\end{pmatrix}
\ \middle\vert \ 
\begin{array}{l}
K_{11}, K_{13}, K_{31}, K_{33} \in \text{Mat}(p,p) \\
\begin{bmatrix}
K_{11} & K_{13} \\
K_{31} & K_{33}
\end{bmatrix} \in O(2p,\C) \\
K_{22} \in O(2q,\C) \\
\det(k) = 1
\end{array}
\right\} \]
\[ \cong S(O(2p,\C) \times O(2q,\C)). \]

This choice of $K$ corresponds to the real form $G_{\R} = SO(2p,2q)$ of $G$.

Let $S=T$ be the maximal torus of $K$, with $X_1,\hdots,X_n$ coordinates on $\frs$ and restriction $\frt \rightarrow \frs$ given by $\rho(Y_i) = X_i$.

As in the type $B$ case, this symmetric subgroup is disconnected.  We handle this issue just as we did in that case, by analyzing the components of each closed $K$-orbit separately.  We then obtain formulas for the closed $K$-orbits by adding the formulas for each component.  Since the components of the closed $K$-orbits are stable under $K^0 = SO(2p,\C) \times SO(2q,\C)$, and since $S \subseteq K^0$, it makes sense to talk about the $S$-equivariant class of the components of the closed $K$-orbits.  The $K^0$-stable components of the closed $K$-orbits are again closed $\wt{K}$-orbits, where $\wt{K} = S(Pin(2p,\C) \times Pin(2q,\C))$ is the corresponding (connected) symmetric subgroup of the simply connected cover $\wt{G} = Spin(2n,\C)$ of $G$.

Let $W_K$ be the Weyl group for $\wt{K}$, (equivalently, for $K^0$, or for $\text{Lie}(K)$).  $W_K$ embeds into $W$ as signed permutations which act separately on $\{1,\hdots,p\}$ and $\{p+1,\hdots,n\}$, changing an even number of signs on each set.  There are $2^{n-2} p! q!$ such permutations.

\subsection{Formulas for the closed $\widetilde{K}$-orbits}
We are once again in an equal rank case, so by Corollary \ref{cor:num-closed-orbits-equal-rank}, there are $|W/W_K| = 2 \binom{n}{p}$ closed $\wt{K}$-orbits, each containing $|W_K|$ $S$-fixed points.

We have the following formulas for the closed orbits.  We omit the proof, since it is virtually identical to the corresponding proofs in the cases $(SO(2n+1,\C),S(O(2p,\C) \times O(2q+1,\C)))$ and $(Sp(2n,\C),Sp(2p,\C) \times Sp(2q,\C)))$, the only difference being that it is simpler.

\begin{prop}\label{prop:closed-so2p-so2q-formulas}
Let $Q$ be a closed $\widetilde{K}$-orbit, containing the $S$-fixed point $w$.  Then
\[ [Q] = P(x,y) = (-1)^{l_p(|w|)} \displaystyle\prod_{i \leq p < j}(x_i - y_{w^{-1}(j)})(x_i + y_{w^{-1}(j)}). \]
\end{prop}

\subsection{Parametrization of $K \backslash X$}\label{ssec:type-d-case-1}
The following parametrization of $K \backslash X$ is described in \cite{Matsuki-Oshima-90}:

\begin{fact}\label{prop:so2p2q-orbit-params}
$K \backslash X$ is parametrized by the set of symmetric $(2p,2q)$-clans.
\end{fact}

Indeed, as indicated by Theorem \ref{thm:k-orbit-intersections}, these orbits are precisely the nonempty intersections of the $GL(2p,\C) \times GL(2q,\C)$-orbits on the type $A$ flag variety with $X$.  See Appendix A for more details.

\subsection{Formulas for closed $K$-orbits}
The argument here proceeds nearly identically to that given in the type $B$ case (cf. Subsection \ref{ssec:type-b-closed-k-orbits}).  The closed $K$-orbits once again correspond to symmetric $(2p,2q)$-clans consisting only of signs.  Such a clan is determined by its first $n$ symbols, of which $p$ are $+$ signs and $q$ are $-$ signs.  There are $\binom{n}{p}$ such clans, and thus $\binom{n}{p}$ closed $K$-orbits.  Since there are $2 \binom{n}{p}$ closed $\wt{K}$-orbits, we see that each closed $K$-orbit is a union of two closed $\wt{K}$-orbits.

We identify the components of closed $K$-orbits.  Once again, using the algorithm of \cite{Yamamoto-97} (cf. Subsection \ref{ssec:orbits_supq}), we can determine an $S$-fixed isotropic representative $wB$ of the $GL(2p,\C) \times GL(2q,\C)$-orbit corresponding to a symmetric $(2p,2q)$-clan $\gamma$ by taking the permutation $\sigma$ to be a signed element of $S_{2n}$, meaning that $\sigma(2n+1-i) = 2n+1-\sigma(i)$ for $i=1,\hdots,n$.  Then the $S$-fixed points of the closed orbit $K \cdot wB$ corresponding to $\gamma=(c_1,\hdots,c_{2n})$ are signed permutations of the form $w'w$, where $w'$ assigns $\pm j$ ($j=1,\hdots,p$) to the positions of the $p$ plus signs among $(c_1,\hdots,c_n)$, and $\pm k$ ($k = p+1,\hdots,n$) to the positions of the $q$ minus signs among $(c_1,\hdots,c_n)$.  Note that each $w'$ must change an even number of signs in total, but could change either an even or an odd number of signs on each of the individual sets $\{1,\hdots,p\}$ and $\{p+1,\hdots,n\}$.  By contrast, $S$-fixed points contained in $\wt{K} \cdot wB$ are of the form $w''w$, where $w'' \in \wt{W_K}$, meaning that $w''$ is required to change an even number of signs on \textit{both} the set $\{1,\hdots,p\}$ \textit{and} the set $\{p+1,\hdots,n\}$.  The conclusion is that the closed orbit $K \cdot wB$ is the union of $\wt{K} \cdot wB$ and $\wt{K} \cdot \pi wB$, where $\pi \in W$ is the signed permutation which interchanges $1$ and $-1$, and $p+1$ and $-(p+1)$.  (Actually, $\pi$ could be taken to be any signed permutation which acts separately on $\{1,\hdots,p\}$ and $\{p+1,\hdots,n\}$, and changes an odd number of signs on each set.  Our particular choice of $\pi$ strikes us as the simplest such permutation.)  

We then have the following corollary of Proposition \ref{prop:closed-so2p-so2q-formulas}:

\begin{cor}
Suppose $Q$ is a closed $K$-orbit containing the $S$-fixed point $w$.  Then
\[ [Q] = (-1)^{l_p(|w|)} \cdot 2 \displaystyle\prod_{i \leq p < j}(x_i - y_{w^{-1}(j)})(x_i + y_{w^{-1}(j)}). \]
\end{cor}
\begin{proof}
This follows from Proposition \ref{prop:closed-so2p-so2q-formulas}, simply by adding the formulas for $\wt{K} \cdot wB$ and $\wt{K} \cdot \pi wB$.  This sum simplifies to the above expression when one makes the easy observation that $l_p(|w|) = l_p(|\pi w|)$.
\end{proof}

\subsection{The weak order}\label{ssec:type-d-example-1-weak-order}
We describe the weak closure order on $K \backslash G/B$ in terms of its parametrization by symmetric $(2p,2q)$-clans.  The reference is \cite{Matsuki-Oshima-90}.

Ordering the simple roots $\ga_i = Y_i - Y_{i+1}$ for $i=1,\hdots,n-1$, and $\ga_n = Y_{n-1} + Y_n$, the situation for the roots $\ga_i$ with $i < n$ is identical to that described in the type $B$ case.  Rather than repeat that description here, we refer the reader back to Subsection \ref{ssec:type-b-weak-order}.

Thus we need only focus on $\ga_n$.  If $\gamma = (c_1,\hdots,c_{2n})$ is a symmetric $(2p,2q)$-clan, there is a fairly long list of possibilities which define when $\ga_n$ is complex for $Q_{\gamma}$ (and $s_{\ga_n} \cdot Q_{\gamma} \neq Q_{\gamma}$).  They are as follows:
\begin{enumerate}
	\item $(c_{n-1},c_n,c_{n+1},c_{n+2})$ form the pattern $(\pm,1,1,\pm)$.	
	\item $(c_{n-1},c_n,c_{n+1},c_{n+2})$ form the pattern $(\pm,1,2,\pm)$; the mate for $c_n$ occurs to the left of $c_{n-1}$; and the mate for $c_{n+1}$ occurs to the right of $c_{n+2}$.
	\item $(c_{n-1},c_n,c_{n+1},c_{n+2})$ form the pattern $(1,\pm,\pm,2)$; the mate for $c_{n-1}$ occurs to the left of $c_{n-1}$; and the mate for $c_{n+2}$ occurs to the right of $c_{n+2}$.
	\item $(c_{n-1},c_n,c_{n+1},c_{n+2})$ form the pattern $(1,2,2,3)$; the mate for $c_{n-1}$ occurs to the left of $c_{n-1}$; and the mate for $c_{n+2}$ occurs to the right of $c_{n+2}$.
	\item $(c_{n-1},c_n,c_{n+1},c_{n+2})$ form the pattern $(1,2,3,1)$; the mate for $c_n$ occurs to the left of $c_{n-1}$; and the mate for $c_{n+1}$ occurs to the right of $c_{n+2}$.
	\item $(c_{n-1},c_n,c_{n+1},c_{n+2})$ form the pattern $(1,2,3,4)$, the mates for $c_{n-1}$ and $c_n$ each occur to the left of $c_{n-1}$; and the mates for $c_{n+1}$ and $c_{n+2}$ each occur to the right of $c_{n+2}$.
	\item $(c_{n-1},c_n,c_{n+1},c_{n+2})$ form the pattern $(1,2,3,4)$; the mates for $c_{n-1}$ and $c_{n+1}$ each occur to the left of $c_{n-1}$; the mates for $c_n$ and $c_{n+2}$ each occur to the right of $c_{n+2}$; and the mate for $c_{n-1}$ occurs further from the center of the clan than the mate for $c_n$.
	\item $(c_{n-1},c_n,c_{n+1},c_{n+2})$ form the pattern $(1,2,3,4)$; the mates for $c_n$ and $c_{n+2}$ each occur to the left of $c_{n-1}$; the mates for $c_{n-1}$ and $c_{n+1}$ each occur to the right of $c_{n+2}$; and the mate for $c_{n-1}$ occurs closer to the center of the clan than the mate for $c_n$.
\end{enumerate}

In each case, $s_{\ga_n} \cdot Q_{\gamma} = Q_{\gamma'}$, where $\gamma'$ is obtained from $\gamma$ by interchanging $c_{n-1}$ with $c_{n+1}$, and $c_n$ with $c_{n+2}$.

On the other hand, $\ga_n$ is non-compact imaginary if and only if one of the following holds:
\begin{enumerate}
	\item $(c_{n-1},c_n,c_{n+1},c_{n+2}) = (+,-,-,+)$
	\item $(c_{n-1},c_n,c_{n+1},c_{n+2}) = (-,+,+,-)$
	\item $(c_{n-1},c_n,c_{n+1},c_{n+2})$ form the pattern $(1,1,2,2)$.
\end{enumerate}

In cases (1) and (2) above, $s_{\ga_n} \cdot Q_{\gamma} = Q_{\gamma''}$, where $\gamma''$ is obtained from $\gamma$ by replacing $(c_{n-1},c_{n+1})$ by a pair of matching natural numbers, and $(c_n,c_{n+2})$ by a second pair of natural numbers.  The effect is to replace $(c_{n-1},c_n,c_{n+1},c_{n+2})$ by the pattern $(1,2,1,2)$.  In these cases, $\ga_n$ is a type I root, since the cross action of $s_{\ga_n}$ is to interchange $c_{n-1}$ with $c_{n+1}$, and $c_n$ with $c_{n+2}$, which sends $(+,-,-,+)$ to $(-,+,+,-)$, and $(-,+,+,-)$ to $(+,-,-,+)$.  Thus $s_{\ga_n} \times Q_{\gamma} \neq Q_{\gamma}$, meaning $\ga_n$ is of type I.

In case (3), $s_{\ga_n} \cdot Q_{\gamma} = Q_{\gamma'''}$, where $\gamma'''$ is obtained from $\gamma$ by interchanging $c_{n-1}$ and $c_{n+1}$ (but \textit{not} $c_n$ and $c_{n+2}$).  The effect is to replace the pattern $(c_{n-1},c_n,c_{n+1},c_{n+2}) = (1,1,2,2)$ by the pattern $(1,2,2,1)$.  In this case, $\ga_n$ is a type II root, since the cross action of $s_{\ga_n}$ interchanges the natural numbers in positions $c_{n-1}$ and $c_{n+1}$, and the natural numbers in positions $c_n$ and $c_{n+2}$.  This sends the pattern $(c_{n-1},c_n,c_{n+1},c_{n+2}) = (1,1,2,2)$ to the equivalent pattern $(2,2,1,1)$, and so does not change the clan $\gamma$.  Thus $s_{\ga_n} \times Q_{\gamma} = Q_{\gamma}$, and $\ga_n$ is type II.
 
\subsection{Example}
With formulas for the closed orbits in hand, along with a combinatorial parametrization of the orbits and the above description of the weak order, consider the case $n=3$, $p=2$, $q=1$, i.e. the symmetric pair $(SO(6,\C), S(O(4,\C) \times O(2,\C)))$.  There are $12$ orbits.  The weak order graph appears in Figure \ref{fig:type-d-graph-1} of Appendix B.

$S$-fixed representatives of a closed orbit, corresponding to a clan $\gamma$ consisting of only $+$'s and $-$'s, can once again be produced by choosing a signed permutation which sends the coordinates of the $+$ signs among the first $n$ characters of $\gamma$ to $1,\hdots,p$, and the coordinates of the $-$ signs among the first $n$ characters of $\gamma$ to $p+1,\hdots,n$.  For example, the closed orbit corresponding to $\gamma = (+,+,-,-,+,+)$ contains the standard $S$-fixed flag $\left\langle e_1,\hdots,e_6 \right\rangle$, which corresponds to $w = 1$.  Thus $[Y_{(+,+,-,-,+,+)}]$ is represented by the polynomial $(x_1-y_3)(x_1+y_3)(x_2-y_3)(x_2+y_3)$.  Similarly, $Q_{(+,-,+,+,-,+)}$ is represented by the $S$-fixed flag corresponding to $w = 132$, so $[Y_{(+,-,+,+,-,+)}]$ is represented by $-(x_1-y_2)(x_1+y_2)(x_2-y_2)(x_2+y_2)$.  The final closed orbit, corresponding to $(-,+,+,+,+,-)$, contains the $S$-fixed point corresponding to $w=312$, so $[Y_{(-,+,+,+,+,-)}]$ is represented by $(x_1-y_1)(x_1+y_1)(x_2-y_1)(x_2+y_1)$.

Formulas for the remaining orbit closures are found using divided difference operators, as usual.  The complete list of formulas can be found in Table \ref{tab:type-d-table-1} of Appendix B.

\section{$K \cong GL(n,\C)$}\label{ssec:type_d_ex_2}
This case is very similar to that of the type $C$ pair $(G,K) = (Sp(2n,\C),GL(n,\C))$. Let $\theta = \text{int}(i \cdot I_{n,n})$.  One checks that
\[ K = G^{\theta} = 
\left\{ 
\begin{pmatrix}
g & 0 \\
0 & J_n \ (g^t)^{-1} J_n
\end{pmatrix}
\ \middle\vert \ 
g \in GL(n,\C) \right\} \cong GL(n,\C). \]

This choice of $K$ corresponds to the real form $G_{\R} = SO^*(2n)$ of $G$.

Once again, we are in an equal rank case, so $S=T$.  We label coordinates on $S$ as $X_1,\hdots,X_n$, with restriction $\frt \rightarrow \frs$ given by $\rho(Y_i) = X_i$.

Roots of $K$ are as follows:
\[ \Phi_K = \{\pm(X_i - X_j) \ \vert \ 1 \leq i < j \leq n\}. \]

As was the case with $(Sp(2n,\C),GL(n,\C))$, $W_K$ embeds into $W$ as ordinary permutations, i.e. as signed permutations of $\{1,\hdots,n\}$ which change no signs.

\subsection{Formulas for the closed orbits}
It once again follows from Corollary \ref{cor:num-closed-orbits-equal-rank} that there are $|W/W_K| = 2^{n-1}$ closed orbits, each containing $|W_K| = n!$ $S$-fixed points.

For $w \in W$, define the $(n-1) \times (n-1)$ determinant
\[ \Delta_{n-1}(x,y,w) := \det(c_{n+j-2i}), \]
where
\[ c_k = \frac{1}{2}(e_k(x_1,\hdots,x_n) + e_k(y_{w^{-1}(1)},\hdots,y_{w^{-1}(n)})), \]
$e_k$ denoting the $k$th elementary symmetric function in the inputs.  As usual, if $w^{-1}(i) < 0$, $y_{w^{-1}(i)}$ means $-y_{|w^{-1}(i)|}$.

As in Subsection \ref{ssec:closed-orbit-formulas-type-c-ex-2}, for each $w \in W$, define  
\[ \text{Neg}(w) := \{i \ \vert \ w(i) < 0 \}, \] 
and define the function $g: W \rightarrow \N$ by 
\[ g(w) = \displaystyle\sum_{i \in \text{Neg}(w)}(n - i). \]

Then formulas for the classes of closed $K$-orbits are as follows:
\begin{prop}
	Let $Q$ be a closed $K$-orbit on $G/B$ represented by the $S$-fixed point $w$.  Then $[Q]$ is represented by the polynomial 
	\[ P(x,y) := (-1)^{g(w)} \Delta_{n-1}(x,y,w). \]
\end{prop}
\begin{proof}
The proof is very similar to that for Proposition \ref{prop:closed-gln-formulas-type-c}, except it is simpler, since we no longer have roots of the form $2Y_i$ to deal with.

That $P(x,y)$ is independent of the choice of $w$ is argued identically.  Also by a nearly identical (but simpler) argument to the one given in the proof of Proposition \ref{prop:closed-gln-formulas-type-c}, we see that for any $S$-fixed point $w \in Q$, 
\[ [Q]|_w = F(X) := (-1)^{g(w)} \displaystyle\prod_{i < j}(X_i + X_j). \]

Then for any $u \in W$,
\[ [Q]|_u = 
\begin{cases}
	F(X) & \text{ if $uw^{-1} \in W_K$}, \\
	0 & \text{ otherwise.}
\end{cases}. \]

Arguing as in the proof of Proposition \ref{prop:closed-gln-formulas-type-c}, the claim that $P(x,y)$ represents $[Q]$ can then be seen to amount to the following claim regarding $\Delta_{n-1}(x,y,id)$:
\begin{enumerate}
	\item It is invariant under permutations of the $x_i$, $y_i$.
	\item If $\epsilon_i = \pm 1$, then
	\[ \Delta_{n-1}((X_1,\hdots,X_n),(\epsilon_1X_1,\hdots,\epsilon_nX_n),id) \]
	is zero unless all $\epsilon_i$ are equal to $1$, in which case it is equal to 
	\[ \displaystyle\prod_{i < j}(X_i + X_j). \]
\end{enumerate}
That $\Delta_{n-1}(x,y,id)$ has these properties is noted in \cite{Fulton-96_1}.
\end{proof}

\begin{remark}
We remark once again that the similarity between these formulas and the formulas of \cite{Fulton-96_1} for classes of Schubert loci in flag bundles is not a coincidence.  Indeed, as in the type $C$ case, it is observed in \cite{Wyser-11b} that certain of the $K$-orbit closures in this case, including all of the closed $K$-orbits, are Richardson varieties.  This is applied in \cite{Wyser-11b} to give some limited information on type $D$ Schubert calculus.
\end{remark}

\subsection{Parametrization of $K \backslash X$ and the weak order}\label{ssec:gln-orbit-param}
The following parametrization of $K \backslash X$ is described in \cite{McGovern-Trapa-09,Matsuki-Oshima-90}:
\begin{fact}\label{prop:type-d-case-2-params}
$K \backslash X$ is parametrized by the set of all skew-symmetric $(n,n)$-clans having the following two further properties:  If $\gamma=(c_1,\hdots,c_{2n})$ is such a clan, then we require
\begin{enumerate}
	\item $c_i \neq c_{2n+1-i}$ whenever $c_i \in \N$.
	\item Among $(c_1,\hdots,c_n)$, the total number of $-$ signs and pairs of equal natural numbers is even.
\end{enumerate}
\end{fact}

Indeed, as indicated by Theorem \ref{thm:k-orbit-intersections}, these orbits are precisely the nonempty intersections of the $GL(n,\C) \times GL(n,\C)$-orbits on the type $A$ flag variety with $X$.  Proofs are given in Appendix A.

We now describe the weak order on $K \backslash X$ in terms of this parametrization.  For this, we refer again to \cite{McGovern-Trapa-09,Matsuki-Oshima-90}.

Order the simple roots as follows:  $\ga_i = Y_i - Y_{i+1}$ for $i=1,\hdots,n-1$, and $\ga_n = Y_{n-1} + Y_n$.  Let $\gamma=(c_1,\hdots,c_{2n})$ be an $(n,n)$-clan having properties (1)-(3) of the previous proposition.  The situation for the simple roots $\ga_1,\hdots,\ga_{n-1}$ is identical to that described for the type $C$ pair $(Sp(2n,\C),Sp(2p,\C) \times Sp(2q,\C))$.  Rather than repeat that description verbatim, we refer the reader back to Subsection \ref{ssec:param_typec_case1}.

Thus we need only consider the root $\ga_n$.  The most concise way to describe the monoidal action of $s_{\ga_n}$ on $Q_{\gamma}$ is as follows:  Let $\text{Flip}(\gamma)$ denote the clan obtained from $\gamma$ by interchanging the characters in positions $n,n+1$.  Then $s_{\ga_n} \cdot Q_{\gamma} = Q_{\gamma'}$, where
\[ \gamma' = \text{Flip}(s_{\ga_{n-1}} \cdot \text{Flip}(\gamma)). \]

When $n = 3$, we have the following examples:
\begin{enumerate}
	\item $s_3 \cdot (+,+,+,-,-,-) = (+,1,2,1,2,+)$.  We apply Flip to obtain $(+,+,-,+,-,-)$, act by $s_2$ on the result to obtain $(+,1,1,2,2,+)$, and finally apply Flip once more to obtain $(+,1,2,1,2,+)$.
	\item $s_3 \cdot (1,-,1,2,+,2) = (1,2,+,-,1,2)$.  We apply Flip to obtain $(1,-,2,1,+,2)$, apply $s_2$ to obtain $(1,2,-,+,1,2)$, and apply Flip again to obtain $(1,2,+,-,1,2)$.
	\item $s_3 \cdot (-,1,1,2,2,+) = (-,1,1,2,2,+)$.  We Flip to obtain $(-,1,2,1,2,+)$, apply $s_2$ to the result (which does nothing), and Flip again, which returns us to the clan we started with.
\end{enumerate}

In terms of complex and non-compact imaginary roots, this amounts to the following:  $\ga_n$ is complex for $Q_{\gamma}$ if and only if $(c_{n-1},c_n,c_{n+1},c_{n+2})$ satisfy one of the following:
\begin{enumerate}
	\item $c_{n-1}$ is a number, $c_{n+2}$ is a (different) number, $c_n$ and $c_{n+1}$ are opposite signs, and the mate for $c_{n-1}$ lies to the left of $c_{n-1}$ (implying, by skew-symmetry, that the mate for $c_{n+2}$ lies to the right of $c_{n+2}$).
	\item $c_{n-1}$ and $c_{n+2}$ are opposite signs, $c_n$ is a number, $c_{n+1}$ is a (different) number, and the mate for $c_n$ lies to the left of $c_n$ (implying, by skew-symmetry, that the mate for $c_{n+1}$ lies to the right of $c_{n+1}$).
	\item $c_{n-1},c_n,c_{n+1},c_{n+2}$ are $4$ distinct numbers, with the mate of $c_{n-1}$ lying to the left of the mate of $c_{n+1}$ (implying, by skew-symmetry, that the mate of $c_n$ lies to the left of the mate of $c_{n+2}$).
\end{enumerate}

On the other hand, $\ga_n$ is non-compact imaginary for $Q_{\gamma}$ if and only if $(c_{n-1},c_n,c_{n+1},c_{n+2}) = (+,+,-,-)$ or $(-,-,+,+)$.

From this, we can see once again that all edges in the weak order graph must be black.  Indeed, we can argue just as in Subsection \ref{ssec:param_typec_case1} that for $i < n$, if $\ga_i$ is a non-compact imaginary root, it must be of type I.  And since the cross action of $s_{\ga_n}$ is to interchange $c_{n-1}$ with $c_{n+2}$, and $c_n$ with $c_{n+1}$, $s_{\ga_n}$ reverses one of the two above strings of four consecutive signs in the event that $\ga_n$ is non-compact imaginary for $Q_{\gamma}$.  Thus $s_{\ga_n} \times Q_{\gamma} \neq Q_{\gamma}$, and so all non-compact imaginary roots are of type I.

\subsection{Example}
With this combinatorial description of the orbit structure and the weak ordering in hand, consider the example $n=3$.  There are 10 orbits.  See Figure \ref{fig:type-d-graph-2} of Appendix B for the weak order graph.

As usual, the closed orbits are parametrized by the clans consisting only of signs.  To obtain an $S$-fixed representative of each, we simply take $w \in S_{2n}$ to be the permutation which assigns $\{1,\hdots,n\}$, in ascending order, to the coordinates of the $+$ signs, and $\{n+1,\hdots,2n\}$, also in ascending order, to the coordinates of the $-$ signs.  The skew-symmetry of the clan dictates that this gives a signed element of $S_{2n}$, which corresponds to flag $\left\langle e_{w(1)},\hdots,e_{w(2n)} \right\rangle \in X$.  We then take the signed permutation in $W$ which corresponds to this signed element of $S_{2n}$.  This signed permutation is the one which assigns, for $i=1,\hdots,n$, $i \mapsto \pm i$, depending on whether the sign in position $i$ is a $+$ or a $-$.

Since our formulas for classes of closed orbits are a bit complicated, we give a couple of examples.  For the orbit $(+,+,+,-,-,-)$, take $w=id$.  Since $g(w)=0$,
\[ [Q_{(+,+,+,-,-,-)}] = 
\left| \begin{array}{cc}
c_2 & c_3 \\
c_0 & c_1 \\
\end{array} \right|. \]

In this case, we have that
\[ c_2 = \dfrac{1}{2}(x_1x_2+x_1x_3+x_2x_3+y_1y_2+y_1y_3+y_2y_3), \]
\[ c_3 = \dfrac{1}{2}(x_1x_2x_3+y_1y_2y_3), \]
\[ c_0 = 1, \text{ and}\]
\[ c_1 = \dfrac{1}{2}(x_1+x_2+x_3+y_1+y_2+y_3). \]

Thus we conclude that
\[ [Q_{(+,+,+,-,-,-)}] = \] 
\[ \dfrac{1}{4}(x_1x_2+x_1x_3+x_2x_3+y_1y_2+y_1y_3+y_2y_3)(x_1+x_2+x_3+y_1+y_2+y_3)-\dfrac{1}{2}(x_1x_2x_3+y_1y_2y_3). \]

In the case of $[Q_{(-,-,+,-,+,+)}]$, taking $w = \overline{3}\overline{2}1$, we have $g(w) = 3$.  Thus
\[ [Q_{(-,-,+,-,+,+)}] = -
\left| \begin{array}{cc}
c_2 & c_3 \\
c_0 & c_1 \\
\end{array} \right|. \]

Here,
\[ c_2 = \dfrac{1}{2}(x_1x_2+x_1x_3+x_2x_3+y_1y_2-y_1y_3-y_2y_3), \]
\[ c_3 = \dfrac{1}{2}(x_1x_2x_3+y_1y_2y_3), \]
\[ c_0 = 1, \text{ and}\]
\[ c_1 = \dfrac{1}{2}(x_1+x_2+x_3-y_1-y_2+y_3). \]

Thus
\[ [Q_{(-,-,+,-,+,+)}] = \]
\[ -\dfrac{1}{4}(x_1x_2+x_1x_3+x_2x_3+y_1y_2-y_1y_3-y_2y_3)(x_1+x_2+x_3-y_1-y_2+y_3)+\dfrac{1}{2}(x_1x_2x_3+y_1y_2y_3). \]

Formulas for the other two closed orbits are found similarly, and formulas for the higher orbit closures are found by applying divided difference operators, as usual.  All formulas appear in Table \ref{tab:type-d-table-2} of Appendix B.

\section{$K \cong S(O(2p+1,\C) \times O(2q-1,\C))$}\label{ssec:type_d_ex_3}
We come now to our final example.  For this case, we change our realization of $G$.  We now take $G=SO(2n,\C)$ to be the subgroup of $SL(2n,\C)$ which preserves the standard (diagonal) quadratic form on $\C^{2n}$ given by
\[ \left\langle x,y \right\rangle = \displaystyle\sum_{i=1}^{2n} x_i y_i. \]
Thus $G$ is now the set of determinant $1$ matrices $g$ such that $gg^t = I_{2n}$.

With this realization of $G$, the diagonal elements no longer form a maximal torus.  We take $T \subseteq G$ to be the maximal torus of $G$ such that $\text{Lie}(T) = \frt$ consists of matrices of the following form:

\[
	\begin{pmatrix}
		\begin{tabular}[t]{c|c|}
			$0$ & $a_1$ \\ \hline
			$-a_1$ & $0$ \\ \hline
		\end{tabular}
		& & & \multirow{2}{*}{\Large 0} \\
		& 
		\begin{tabular}[b]{|c|c|}
			\hline
			$0$ & $a_2$ \\ \hline
			$-a_2$ & $0$ \\ \hline
		\end{tabular}
		& & \\
		\multirow{2}{*}{\Large 0} & & \ddots & 
		\\
		& & & 
		\begin{tabular}{|c|c}
			\hline
			$0$ & $a_n$ \\ \hline
			$-a_n$ & $0$
		\end{tabular}
	\end{pmatrix}
\]

Let $Y_i \in \frt^*$ be the function defined by $Y_i(t) = a_i$, with $t$ a matrix of the above form.  As before, take the positive roots to be 
\[ \Phi^+ = \{Y_i \pm Y_j \mid (i < j)\}, \]
and let $B \subseteq G$ be chosen so that the roots of $\text{Lie}(B)$ are negative.

Let $X=G/B$ be the flag variety, now one component of the variety of flags which are isotropic with respect to the diagonal form $\langle \cdot, \cdot \rangle$.

We take $K=G^{\theta}$ to be the fixed points of the involution
\[ \theta(g) = I_{2p+1,2q-1} g I_{2p+1,2q-1}. \]
One checks easily that our chosen realization of $G$ is stable under $\theta$, that $T$ is stable under $\theta$, and that
\[ K =  
\left\{
k = 
\begin{bmatrix}
A & 0 \\
0 & B
\end{bmatrix}
\ \middle\vert \
A \in O(2p+1,\C),
B \in O(2q-1,\C),
\det(k) = 1 \right\} \]
\[ \cong S(O(2p+1,\C) \times O(2q-1,\C)). \]

This choice of $K$ corresponds to the real form $G_{\R} = SO(2p+1,2q-1)$ of $G$.  

Note here that we are in an unequal rank case, with $\text{rank}(K) = n-1$.  We take $S \subseteq T$ to be the maximal torus of $K$ such that $\frs = \text{Lie}(S)$ consists of matrices of the form

\[
	\begin{pmatrix}
		\begin{tabular}{c|c|}
			$0$ & $a_1$ \\ \hline
			$-a_1$ & $0$ \\ \hline
		\end{tabular}
		& & & & & & \\
		& \ddots & & & & \multirow{2}{*}{\Huge 0} & \\
		& & 
		\begin{tabular}{|c|c|}
			\hline
			$0$ & $a_p$ \\ \hline
			$-a_p$ & $0$ \\ \hline
		\end{tabular}
		& & & & \\
		& & & 
		\begin{tabular}{|c|c|}
			\hline
			$0$ & $0$ \\ \hline
			$0$ & $0$ \\ \hline
		\end{tabular}
		& & & \\
		& & & &
		\begin{tabular}{|c|c|}
			\hline
			$0$ & $a_{p+2}$ \\ \hline
			$-a_{p+2}$ & $0$ \\ \hline
		\end{tabular}
		& & \\
		& \multirow{2}{*}{\Huge 0} & & & & \ddots & \\
		& & & & & &
		\begin{tabular}{|c|c}
			\hline
			$0$ & $a_n$ \\ \hline
			$-a_n$ & $0$
		\end{tabular}
	\end{pmatrix}
\]

One checks easily that $S$ is also stable under $\theta$.  We label coordinates on $\frs$ as 
\[ X_1,\hdots,X_p,X_{p+2},\hdots,X_n, \]
with $X_i(s) = a_i$ when $s$ is a matrix of the above block form.  With this choice of labelling, the restriction map $\rho: \frt^* \rightarrow \frs^*$ is given by $\rho(Y_i) = X_i$ for $i \neq p+1$, and $\rho(Y_{p+1}) = 0$.

The roots of $K$ are as follows:
\[ \Phi_K = \{\pm X_i \mid i \neq p+1\} \cup \{\pm(X_i \pm X_j) \mid i < j \leq p \text{ or } p+1 < i < j\}. \]

Just as in the examples of Subsections \ref{ssec:type_b_ex_1} and \ref{ssec:type_d_ex_1}, this $K$ is disconnected.  However, unlike in those cases, this time the closed $K$-orbits are nonetheless connected.  Thus it will turn out that there is no need to concern ourselves with closed $K$-orbits versus closed $\wt{K} = S(Pin(2p+1,\C) \times Pin(2q-1,\C))$-orbits, as here they coincide.  However, because we have not yet proved this, we concern ourselves first with the closed $\wt{K}$-orbits.  Each is a closed $K^0$-orbit, with $K^0 = SO(2p+1,\C) \times SO(2q-1,\C)$ the identity component of $K$.  Since $S \subseteq K^0$, each is stable under $S$, and so has an $S$-equivariant class.  We use our usual methods to find formulas for these classes.  We then concern ourselves with parametrizing the $K$-orbits, at which point we will see that the closed $K$-orbits coincide with the closed $K^0$-orbits, or the closed $\wt{K}$-orbits.

Let $W_K$ be the Weyl group of $\wt{K}$ (or of $K^0$, or of $\text{Lie}(K)$).  $W_K$ embeds in $W$ as those signed permutations of $\{1,\hdots,n\}$ which act separately on the first $p$ elements $\{1,\hdots,p\}$ and the last $q-1$ elements $\{p+2,\hdots,n\}$, changing any number of signs on each set, and which either fix $p+1$ or send it to its negative, whichever is necessary to guarantee that the resulting signed permutation changes an \textit{even} number of signs.  There are $2^{n-1} p!(q-1)!$ such signed permutations.

\subsection{Formulas for the closed orbits}
Since this is an unequal rank case, there will not be $|W/W_K|$ closed orbits.  We first use Proposition \ref{prop:num-closed-orbits-2} to determine how many closed $\wt{K}$-orbits there are, and which $S$-fixed points they contain.

\begin{prop}\label{prop:type-d-ex-3-closed-orbits}
Let $wB$ be an $S$-fixed point, with $w \in W$.  Then $\wt{K} \cdot wB$ is closed if and only if $w(n) = \pm(p+1)$.  There are $\binom{n-1}{p}$ closed $\wt{K}$-orbits.
\end{prop}
\begin{proof}
We use the characterization of closed orbits given in Proposition \ref{prop:num-closed-orbits-2}.  Since we have chosen $B$ to be the negative Borel, the condition that $wBw^{-1}$ be $\theta$-stable is equivalent to the condition that $w \Phi^-$ is a $\theta$-stable subset of $\Phi$.  One checks easily that the action of $\theta$ on $\Phi$ is defined by $\theta(Y_i) = Y_i$ for $i \neq p+1$, and $\theta(Y_{p+1}) = -Y_{p+1}$.  Any positive system contains, for each $i<j$, exactly one of $Y_i + Y_j$ and $-Y_i - Y_j$, and exactly one of $Y_i - Y_j$ and $-Y_i + Y_j$.  For $i,j \neq p+1$, all such roots are fixed by $\theta$.  Thus for $\theta$-stability, it suffices to focus on roots of the form $\pm Y_i \pm Y_{p+1}$, with $i \neq p+1$.  It is easy to check that a positive system is $\theta$-stable if and only if it contains either $\{Y_i-Y_{p+1}, Y_i+Y_{p+1}\}$ or $\{-Y_i + Y_{p+1}, -Y_i-Y_{p+1}\}$ for each $i \neq p+1$.

This holds if and only if $w(n) = \pm(p+1)$.  Recall that $w\Phi^- = \{-wY_i \pm wY_j \mid i < j\}$.  Suppose that $w(n) = \pm(p+1)$.  Let $i \neq p+1$ be given, with $k = |w|^{-1}(i)$.  Then $-wY_k \pm wY_n$ is either the set $\{Y_i + Y_{p+1}, Y_i - Y_{p+1}\}$ or $\{-Y_i + Y_{p+1}, -Y_i - Y_{p+1}\}$, as required.  Conversely, suppose that $|w(n)| = j \neq p+1$.  Let $k = |w|^{-1}(p+1)$.  Then $-wY_k \pm wY_n$ is either the set $\{-Y_{p+1} + Y_j, -Y_{p+1} - Y_j\}$ or $\{Y_{p+1} + Y_j, Y_{p+1} - Y_j\}$, and thus $w \Phi^-$ is not $\theta$-stable.  This establishes the first claim.

To establish the claim on the number of closed orbits, note that any element $u \in W$ such that $u(n) = \pm(p+1)$ is in the same left $W_K$-coset as a unique element $w \in W$ having the following properties:
\begin{enumerate}
	\item $w$ changes no signs.
	\item $w(n) = p+1$.
	\item $w^{-1}(1) < w^{-1}(2) < \hdots < w^{-1}(p)$.
	\item $w^{-1}(p+2) < w^{-1}(p+3) < \hdots < w^{-1}(n)$.
\end{enumerate}

Recall that all elements of $W_K$ are separately signed permutations of $\{1,\hdots,p\}$ and $\{p+2,\hdots,n\}$, which either fix $p+1$ or send it to its negative so as to ensure that the entire signed permutation changes an even number of signs.  Supposing that, in the one-line notation for $u$, the values $1,\hdots,p$ (possibly with signs) are ``scrambled", then there is precisely one signed permutation of $\{1,\hdots,p\}$ which will unscramble them and remove all negative signs, and likewise for the set $\{p+2,\hdots,n\}$.  Taking $w' \in W_K$ to be the unique element which separately acts on $\{1,\hdots,p\}$ and $\{p+2,\hdots,n\}$ as required, we have that $w'u = w$.

As an example, suppose that $p=q=3$, and let $u$ be the signed permutation $\overline{3} 1 6 2 5\overline{4}$.  To unscramble the $\overline{3}12$, we must multiply on the left by $1 \mapsto 2$, $2 \mapsto 3$, $3 \mapsto \overline{1}$, and to unscramble the $65$ we must multiply on the left by $5 \mapsto 6, 6 \mapsto 5$.  Thus we multiply $u$ on the left by $w' = 23\overline{1}\overline{4}65$ to get $w'u = w = 125364$.

Note that a permutation $w$ having the properties above is completely determined by the positions (in the one-line notation) of $1,\hdots,p$ among the first $n-1$ spots, which can be chosen freely.  Thus there are $\binom{n-1}{p}$ such $w$, and hence $\binom{n-1}{p}$ closed $\wt{K}$-orbits, as claimed.  
\end{proof}

\begin{definition}
Let $Q \in \wt{K} \backslash X$ be a closed orbit.  Call the flag $wB \in Q$, where $w$ has the properties listed in the proof of Proposition \ref{prop:type-d-ex-3-closed-orbits}, the \textbf{standard representative} of $Q$.
\end{definition}

For $w \in W$ such that $wB$ is the standard representative of some closed orbit $Q$, define
\[ I_w := \{ i \in \{1,\hdots,n-1\} \mid w(i) > p+1\}. \]
For each $i \in I_w$, define
\[ C(i) := \#\{j \mid i < j \leq n-1, w(j) \leq p\}. \]
Finally, define
\[ f(w) := \displaystyle\sum_{i \in I_w} C(i). \]

Then we have the following formula for the $S$-equivariant class of the closed orbit $Q$:
\begin{prop}\label{prop:type-d-ex-3-closed-orbit-formulas}
Let $Q=\wt{K} \cdot wB$ be any closed orbit, with $wB$ the standard representative.  With $f(w)$ defined as above, $[Q]$ is represented by the polynomial
\[ P(x,y) := (-1)^{f(w)} y_1 \hdots y_{n-1} \displaystyle\prod_{i \leq p < p+1<j}(x_i+y_{w^{-1}(j)})(x_i - y_{w^{-1}(j)}). \]
\end{prop}
\begin{proof}
First consider $\rho(w \Phi^+)$, the elements of $\frs^*$ obtained by first applying the standard representative $w$ to the positive roots, then restricting to $\frs$.  They are as follows:
\begin{itemize}
	\item $X_i$ ($i \neq p+1$), with multiplicity $2$.  (One is the restriction of $w(Y_i + Y_n) = Y_{w(i)} + Y_{p+1}$, the other the restriction of $w(Y_i-Y_n) = Y_{w(i)} - Y_{p+1}$.)
	\item $X_i + X_j$ ($i<j$, $i,j\neq p+1$), with multiplicity $1$.
	\item For each $i<j$ with $i,j \neq p+1$, exactly one of $\pm(X_i-X_j)$, with multiplicity $1$.
\end{itemize}

Removing roots of $K$, we have the following set of weights:
\begin{itemize}
	\item $X_i$ ($i \neq p+1$), with multiplicity $1$.
	\item $X_i + X_j$ ($i \leq p < p+1 < j$), with multiplicity $1$.
	\item For each $i<j$ with $i \leq p < p+1 < j$, exactly one of $\pm(X_i-X_j)$, with multiplicity $1$.
\end{itemize}

Recall that $w$ is an honest permutation, with no sign changes.  This means that the only way to get a weight of the form $-(X_i-X_j)$ by the action of $w$ is to apply $w$ to some $Y_k - Y_l$ ($k<l$) with $w(k) > w(l)$.  (Clearly, we want $k,l \neq n$.)  For this root to remain after discarding roots of $K$, it must be the case that $w(k) > p+1$, while $w(l) \leq p$.  Thus for each $k < n$ such that $w(k) > p+1$ (this says that $k \in I_w$), we count the number of $l$ with $k < l < n-1$ such that $w(l) \leq p$ (this says that $l \in C(k)$).  Adding up the total number of such pairs as we let $k$ range over $I_w$, we arrive at $f(w)$.  This says that the number of weights of the form $-(Y_i-Y_j)$ contained in $\rho(w \Phi^+) \setminus (\rho(w \Phi^+) \cap \Phi_K)$ is $f(w)$.
 
Now we consider the set $\rho(w'w \Phi^+) \setminus (\rho(w'w \Phi^+) \cap \Phi_K)$ with $w' \in W_K$, and compute the restriction $[Q]|_{w'w}$ at an arbitrary $S$-fixed point.  Since the action of $w'$ on $\frt$ commutes with restriction to $\frs$, and since $w'$ acts on the roots of $K$ (and hence also on $\rho(\Phi) \setminus \Phi_K$), we can simply apply $w'$ to the set of weights described in the previous paragraph.  We temporarily forget that some of those roots are of the form $-(X_i-X_j)$ ($i<j$), and add the sign of $(-1)^{f(w)}$ back in at the end.  So consider the action of $w' \in W_K$ on the following set of weights, each with multiplicity $1$:
\begin{itemize}
	\item $X_i$ ($i \neq p+1$)
	\item $X_i \pm X_j$ ($i \leq p < p+1 < j$)
\end{itemize}

Since $w'$ acts separately as signed permutations on $\{1,\hdots,p\}$ and $\{p+2,\hdots,n\}$, it clearly sends the set of weights $X_i \pm X_j$ to itself, except possibly with some sign changes.  We observe that the number of sign changes must be even.  Suppose first that $w'(X_i + X_j)$ is a negative root.  Then it is either of the form $-X_k-X_l$ or $-X_k + X_l$, with $k=|w(i)|$ and $l=|w(j)|$.  In the former case, $w'(X_i - X_j) = -X_k + X_l$, also a negative root.  In the latter, $w'(X_i - X_j) = -X_k-X_l$, again a negative root.  Likewise, if $w'(X_i - X_j)$ is a negative root of the form $-X_k-X_l$ or $-X_k+X_l$, then $w'(X_i+X_j)$ is also a negative root, equal to $-X_k+X_l$ in the former case, and $-X_k-X_l$ in the latter.  Thus the negative roots arising from the action of $w'$ on roots of the form $X_i \pm X_j$ occur in pairs.

Now consider roots of the form $X_i$, $i \neq p+1$.  The action of $w'$ again preserves this set of roots, except possibly with some sign changes.  The number of sign changes could be either even or odd.  (Recall that $w'$ acts with any number of sign changes on $\{1,\hdots,p\}$ and $\{p+2,\hdots,n\}$, and sends $p+1$ either to itself or to $-(p+1)$, whichever ensures that the total number of sign changes for $w'$ is even.)

This discussion all adds up to the following.  The product of the weights $\rho(w'w \Phi^+) \setminus (\rho(w'w \Phi^+) \cap \Phi_K)$ is
\[ [Q]|_{w'w} = (-1)^{f(w) + \text{Neg}(w')} \displaystyle\prod_{i \neq p+1} X_i \displaystyle\prod_{i \leq p < p+1 < j} (X_i+X_j)(X_i-X_j), \]
where $\text{Neg}(w')$ denotes the number of sign changes of $w'$ on the set $\{1,\hdots,p,p+2,\hdots,n\}$.

Thus we wish to prove that the polynomial $P(x,y)$ has the properties that $P(X,\rho(w'w(Y)))$ is equal to this restriction for all $w' \in W_K$, and that $P(X,\rho(uw(Y))) = 0$ whenever $u \notin W_K$.

Consider first the action of $w'w$ on $P(x,y)$ for $w' \in W_K$.  Since $w$ sends the set $\{1,\hdots,n-1\}$ to the set $\{1,\hdots,p,p+2,\hdots,n\}$ with no sign changes, the action of $w'w$ on $y_1 \hdots y_{n-1}$ is clearly to send it to $(-1)^{\text{Neg}(w')} \displaystyle\prod_{i \neq p+1} Y_i$.  Thus applying $w'w$ to $(-1)^{f(w)} y_1 \hdots y_{n-1}$ gives us the portion
\[ (-1)^{f(w) + \text{Neg}(w')} \displaystyle\prod_{i \neq p+1} X_i \]
of the required restriction.  Now consider the action of $w'w$ on the term 
\[ \displaystyle\prod_{i \leq p < p+1<j}(x_i+y_{w^{-1}(j)})(x_i - y_{w^{-1}(j)}). \]
We get 
\[ \displaystyle\prod_{i \leq p < p+1<j}(X_i+X_{w'(j)})(X_i - X_{w'(j)}). \]
Since $w'$ acts as a signed permutation on $\{p+2,\hdots,n\}$, this is clearly the same as
\[ \displaystyle\prod_{i \leq p < p+1<j}(X_i+X_j)(X_i - X_j), \]
giving us the remaining part of the required restriction.

Now, consider the action of $uw$ on $P(x,y)$ for $u \notin W_K$.  Suppose first that $u(p+1) \neq \pm(p+1)$.  Then $u(i) = \pm(p+1)$ for some $i \neq p+1$.  Let $j = w^{-1}(i)$.  Then the action of $uw$ sends the term $y_j$ to $\pm Y_{p+1}$, which restricts to zero.  Now suppose that $u(p+1) = \pm(p+1)$.  Then since $u \notin W_K$, $u$ must send some $j > p+1$ to $\pm i$ for some $i \leq p$.  If it sends $j$ to $i$, then $uw$ applied to the term $x_i-y_{w^{-1}(j)}$ is zero.  If it sends $j$ to $-i$, then $uw$ applied to the term $x_i+y_{w^{-1}(j)}$ is zero.  This shows that $P(X,\rho(uw(Y))) = 0$ for $u \notin W_K$, and completes the proof.
\end{proof}

\subsection{Parametrization of $K \backslash X$ and the weak order}
The following parametrization of $K \backslash X$ is described in \cite{Matsuki-Oshima-90}:
\begin{fact}
The $K$-orbits on $X$ are parametrized by the set of all symmetric $(2p+1,2q-1)$-clans.
\end{fact}

Indeed, as indicated by Theorem \ref{thm:k-orbit-intersections}, these orbits are precisely the nonempty intersections of the $K' = GL(2p+1,\C) \times GL(2q-1,\C)$-orbits on the type $A$ flag variety $X'$ with $X$.  See Appendix A.

Consider the closed $K$-orbits.  In all cases prior to this one outside of type $A$, the closed orbits have been parametrized by clans (satisfying some further combinatorial conditions) consisting only of signs.  This said that the closed orbits in those cases were the intersections of closed $K'$-orbits on $X'$ with $X$.  Note here, though, that there are no symmetric $(2p+1,2q-1)$-clans consisting only of signs.  Thus no closed $K'$-orbits on $X'$ intersect $X$.  The lowest orbits in the closure order on $K' \backslash X'$ which intersect $X$ lie one step above the closed orbits in the order, and correspond to symmetric $(2p+1,2q-1)$-clans of the form $(c_1,\hdots,c_{n-1},1,1,c_{n+2},\hdots,c_{2n})$, with $c_1,\hdots,c_{n-1}$ consisting of $p$ $+$'s and $q-1$ $-$'s.  The closed $K$-orbits on $X$ are parametrized by symmetric clans of this form.  Note that there are $\binom{n-1}{p}$ such clans, thus $\binom{n-1}{p}$ closed $K$-orbits.  This number is the same as the number of closed $\wt{K}$-orbits (see Proposition \ref{prop:type-d-ex-3-closed-orbits}).  This establishes our earlier claim that the closed $K$-orbits coincide with the closed $\wt{K}$-orbits.  Thus there is no need in this case to describe closed $K$-orbits as unions of $\wt{K}$-orbits and add the appropriate formulas, as we have done in other cases.  Summarizing, we have the following result:
\begin{prop}
The closed $K$-orbits on $X$ coincide with the closed $\wt{K}$-orbits on $X$.  Thus formulas for the $S$-equivariant fundamental classes of closed $K$-orbits are given by Proposition \ref{prop:type-d-ex-3-closed-orbit-formulas}.
\end{prop}

The weak order on $K$-orbits in this case corresponds to the weak order on symmetric $(2p+1,2q-1)$-clans described in \cite{Matsuki-Oshima-90}.  The combinatorics of this order are exactly the same as those described in Subsection \ref{ssec:type-d-example-1-weak-order}.  We refer the reader back to that section, rather than repeat the description here.

\subsection{Example}
Consider now the example $p=1$, $q=2$, $n=3$.  The corresponding symmetric pair is $(SO(6,\C),S(O(3,\C) \times O(3,\C)))$.  There are two closed orbits, corresponding to the clans $(+,-,1,1,-,+)$ and $(-,+,1,1,+,-)$.  One checks (for general $p,q$) that the standard representative of the closed orbit corresponding to $\gamma$ is $wB$, where $w$ is the permutation which assigns $1,\hdots,p$, in order, to the positions of the $+$ signs among the first $n$ characters of $\gamma$; $p+2,\hdots,n$, in order, to the positions of the $-$ signs among the first $n$ characters of $\gamma$; and $p+1$ to position $n$.  Thus the standard representatives of the closed orbits correspond to the following permutations:
\begin{itemize}
	\item $(+,-,1,1,-,+)$:  $132$
	\item $(-,+,1,1,+,-)$:  $312$
\end{itemize}

By Proposition \ref{prop:type-d-ex-3-closed-orbit-formulas}, formulas for the closed orbits are as follows:
\begin{itemize}
	\item $[Q_{(+,-,1,1,-,+)}] =  y_1y_2(x_1+y_2)(x_1-y_2)$
	\item $[Q_{(-,+,1,1,+,-)}] = -y_1y_2(x_1+y_1)(x_1-y_1)$
\end{itemize}

There are $13$ orbits in all.  The weak order graph appears as Figure \ref{fig:type-d-graph-3} of Appendix B.  The formulas for the remaining orbit closures, obtained using divided difference operators, are given in Table \ref{tab:type-d-table-3}.

%% file: ch6-Degeneracy_Loci.tex
\chapter{$K$-orbit Closures as Universal Degeneracy Loci}
In this chapter, we describe our main application of the formulas of Chapters 2-5.  Namely, in the type $A$ cases, we realize the $K$-orbit closures as universal degeneracy loci of a certain type determined by $K$.  We describe a translation between our formulas for equivariant fundamental classes of $K$-orbit closures and Chern class formulas for the fundamental classes of such degeneracy loci.  Lastly, we indicate that similar results should hold for the symmetric pairs considered in types $BCD$, given explicit linear algebraic descriptions of $K$-orbit closures in those cases.

Before handling the specifics of each case, we first describe the general setup.  Denote by $E$ a contractible space with a free action of $G$.  Then $E$ also has a free action of $B$, and of $K$, by restriction of the $G$-action.  We shall use the same space $E=EG=EB=EK$ as the total space of a universal principal $G$, $B$, or $K$-bundle, as appropriate.  Denote by $BG$, $BB$, and $BK$ the quotients of $E$ by the actions of $G$, $B$, and $K$, respectively.  These are classifying spaces for the respective groups.

The reason we have worked in $S$-equivariant cohomology $H_S^*(G/B)$ throughout is to take advantage of the localization theorem.  However, the equivariant fundamental classes of $K$-orbit closures in fact live in $K$-equivariant cohomology $H_K^*(G/B)$.  (In the event that $K$ is disconnected, this should be interpreted as $H_{K^0}^*(G/B)$, where $K^0$ denotes the identity component of $K$.)  Indeed, for a $K$-orbit closure $Y$, the $S$-equivariant class $[Y]_S$ is simply the image $\pi^*([Y]_K)$ under the pullback by the natural map
\[ \pi:  E \times^S (G/B) \rightarrow E \times^K (G/B). \]
It is a basic fact about equivariant cohomology that this pullback is injective, and embeds $H_K^*(G/B)$ in $H_S^*(G/B)$ as the $W_K$-invariants (\cite{Brion-98_i}).  Thus $H_K^*(G/B)$ is a subring of $H_S^*(G/B)$, and the $S$-equivariant fundamental classes of $K$-orbit closures live in this subring.  

Now, $H_K^*(G/B)$ is, by definition, the cohomology of the space $E \times^K (G/B)$, and this space is easily seen to be isomorphic to the fiber product $BK \times_{BG} BB$.  (The argument is identical to that given in the proof of Proposition \ref{prop:eqvt-cohom-flag-var} to show that $E \times^S (G/B) \cong BS \times_{BG} BB$ --- simply replace $S$ by $K$.)

Now, suppose that $X$ is a scheme, and that $V \rightarrow X$ is a complex vector bundle of rank $n$.  In type $A$, no further structure on $V$ is presumed, while in types $BCD$, $V$ is assumed to be equipped with an orthogonal ($BD$) or symplectic ($C$) form.  In any event, we have a classifying map $X \stackrel{\rho}{\longrightarrow} BG$ such that $V$ is the pullback $\rho^*(\mathcal{V})$, where $\mathcal{V} = E \times^G \C^n$ is a universal vector bundle over $BG$, with $\C^n$ carrying the natural representation of $G$.

For any closed subgroup $H$ of $G$, $BH \rightarrow BG$ is a fiber bundle with fiber isomorphic to $G/H$.  A lift of the classifying map $\rho$ to $BH$ corresponds to a reduction of structure group to $H$ of the bundle $V$.  Such a reduction of structure group can often be seen to amount to some additional structure on $V$.  For instance, in type $A$, reduction of the structure group of $V$ from $GL(n,\C)$ to the Borel subgroup $B$ of upper-triangular matrices is well-known to be equivalent to $V$ being equipped with a complete flag of subbundles.  (In Types $BCD$, this flag is required to be isotropic with respect to the form on $V$.)

We will be concerned with certain structures on $V$ which amount to a reduction of structure group to $K$.  Such a reduction gives us a lift of the classifying map $\rho$ to $BK$.  Suppose that we know what this structure is, and that $V$ possesses this structure, along with a single flag of subbundles $E_{\bullet}$ (assumed isotropic in types $BCD$).  Then we have two separate lifts of $\rho$, one to $BK$, and one to $BB$.  Taken together, these two lifts give us a map
\[ X \stackrel{\phi}{\longrightarrow} BK \times_{BG} BB. \]

Our general thought is to consider a subscheme $D$ of $X$ which is defined as a set by linear algebraic conditions imposed on fibers over points in $X$.  These linear algebraic conditions describe the ``relative position" of a flag of subbundles of $V$ and the additional structure on $V$ amounting to the lift of the classifying map to $BK$.  The varieties we consider are precisely those which are set-theoretic inverse images under $\phi$ of (isomorphic images of) $K$-orbit closures in $BK \times_{BG} BB \cong E \times^K (G/B)$.  The linear algebraic descriptions of such a subscheme $D$ come directly from similar linear algebraic descriptions of a corresponding $K$-orbit closure $Y$.  We also realize various bundles on $X$ as pullbacks by $\phi$ of certain tautological bundles on the universal space, so that the Chern classes of the various bundles on $X$ are pullbacks of $S$-equivariant classes represented by the variables $x_i$ and $y_i$ (or perhaps polynomials in these classes), which we worked with in Chapters 2-5.

As explained in \cite{Fulton-92,Fulton-Pragacz}, $D$ can be given a scheme structure, simply as the \textit{scheme}-theoretic inverse image under the map $\phi$ above.  When the setup is ``suitably generic", we have
\begin{equation}\label{eqn:pullback}
	[D] = [\phi^{-1}(Y)] = \phi^*([Y]),
\end{equation}
and so our equivariant formula for $[Y]$ gives us, in the end, a formula for $[D]$ in terms of the Chern classes of the bundles involved.  The phrase ``suitably generic" should be thought of as a requirement that the various structures on $V$ be in general position with respect to one another.  See \cite{Fulton-92,Fulton-Pragacz} for more details on the intersection-theoretic arguments regarding precisely when (\ref{eqn:pullback}) holds.

With the general picture painted, we now proceed to our specific examples.

\section{Examples in type $A$}
\subsection{$K = S(GL(p,\C) \times GL(q,\C))$}
Suppose that we are given an $n$-dimensional vector space $V$, a complete flag
\[ E_{\bullet} = \{E_0 \subset E_1 \subset \hdots \subset E_n\} \]
of subspaces of $V$, and a splitting of $V$ as a direct sum of subspaces of dimensions $p$ and $q$, i.e. $V = V' \oplus V''$.  Let $\pi: V \rightarrow V'$ be the projection onto $V'$.  Let $\gamma = (c_1,\hdots,c_n)$ be a $(p,q)$-clan.  Recalling the notation and results of Subsection \ref{ssec:orbits_supq}, we make the following definitions:

\begin{definition}\label{def:rel-pos}
Let $V = V' \oplus V''$ be as described.  We say that a flag $E_{\bullet}$ on $V$ is \textbf{in position $\gamma$ relative to the splitting $V' \oplus V''$} if the following three conditions hold for all $i,j$:
	
	\begin{enumerate}
		\item $\dim(E_i \cap V') = \gamma(i; +)$
		\item $\dim(E_i \cap V'') = \gamma(i; -)$
		\item $\dim (\pi(E_i) + E_j) =  j + \gamma(i; j)$ 
	\end{enumerate}
	
We say furthermore that a flag $E_{\bullet}$ on $V$ is \textbf{in position at most $\gamma$ relative to the splitting $V' \oplus V''$} if the following three inequalities hold for all $i,j$:
	\begin{enumerate}
		\item $\dim(E_i \cap V') \geq \gamma(i; +)$
		\item $\dim(E_i \cap V'') \geq \gamma(i; -)$
		\item $\dim (\pi(E_i) + E_j) \leq  j + \gamma(i; j)$
	\end{enumerate}
\end{definition}

Recall (Subsection \ref{ssec:orbits_supq}) that given a $(p,q)$-clan $\gamma$, the corresponding orbit $Q_{\gamma}$ is precisely
\[ Q_{\gamma} = \{E_{\bullet} \mid E_{\bullet} \textit{ is in position $\gamma$ relative to } \C\left\langle e_1,\hdots,e_p \right\rangle \oplus \C \left\langle e_{p+1},\hdots,e_n \right\rangle \}. \]

We conjecture here, without proof, the following set-theoretic description of the orbit \textit{closure} $Y_{\gamma} = \overline{Q_{\gamma}}$:

\begin{conj}\label{conj:pq-closure-description}
Given a $(p,q)$-clan $\gamma$, the orbit closure $Y_{\gamma}$ is precisely
\[ Y_{\gamma} = \{E_{\bullet} \mid E_{\bullet} \textit{ is in position at most $\gamma$ relative to } \C\left\langle e_1,\hdots,e_p \right\rangle \oplus \C \left\langle e_{p+1},\hdots,e_n \right\rangle \}. \]
\end{conj}

\begin{remark}
Conjecture \ref{conj:pq-closure-description} has been verified to be true using Sage through $p+q=8$.  (The method used is to actually build the full Bruhat order graph consisting of $(p,q)$-clans, and check whether relation of clans $\gamma_1, \gamma_2$ in this Bruhat order graph is equivalent to the conditions on the numbers $\gamma_1(i; +)$, $\gamma_2(i; +)$, etc. which amount to Conjecture \ref{conj:pq-closure-description} being true.)
\end{remark}

Assuming Conjecture \ref{conj:pq-closure-description}, we now wish to define a set of degeneracy loci occurring as set-theoretic inverse images of such $K$-orbit closures.  The setup in this case involves a scheme $X$ equipped with a vector bundle $V$ carrying a complete flag of subbundles and a splitting as a direct sum of subbundles of ranks $p$ and $q$, i.e. $V = V' \oplus V''$.  The latter structure is relevant because it amounts to a reduction of the structure group of $V$ from $G$ to $K$:

\begin{prop}
Suppose $X$ is a scheme and suppose that $V \rightarrow X$ is a vector bundle.  The classifying map $X \stackrel{\rho}{\longrightarrow} BSL(n,\C)$ lifts to $BK$ if and only if $V$ splits as a direct sum of subbundles of ranks $p$ and $q$.
\end{prop}
\begin{proof}
$(\Leftarrow)$:  Suppose that $V = V' \oplus V''$, with $V'$ of rank $p$, and $V''$ of rank $q$.  Let $\{U_{\ga}\}$ be an open cover of $X$ over which both $V'$ and $V''$ are locally trivial (say by taking the common refinement of the open covers associated to atlases of $V'$ and $V''$).  Over each $U_{\ga}$, we can choose a basis of sections $s_{1,\ga},\hdots,s_{p,\ga}$ for $V'$ and a basis of sections $s_{p+1,\ga},\hdots,s_{n,\ga}$ for $V''$, with $s_{i,\ga}(x) = (x,e_i)$ for $x \in U_{\ga}$ ($e_1,\hdots,e_n$ the standard basis for $\C^n$).  Then $s_{i,\ga}$ for $i=1,\hdots,n$ are a basis of sections for $V$.  The gluing data for $V'$ and $V''$ dictates that for $x \in U_{\ga} \cap U_{\beta}$, 
\[ s_{i,\ga}(x) = \displaystyle\sum_{j=1}^p \lambda_{i,j} s_{j,\beta}(x) \]
for $i=1,\hdots,p$, and 
\[ s_{i,\ga}(x) = \displaystyle\sum_{j=p+1}^n \lambda_{i,j} s_{j,\beta}(x) \]
for $i=p+1,\hdots,n$.
This defines a family of transition functions for $V$, associating to $x \in U_{\ga} \cap U_{\beta}$ the matrix $(\lambda_{i,j}) \in GL(p,\C) \times GL(q,\C)$.  Thus the classifying map for $V = V' \oplus V''$ lifts to $BK$.

$(\Rightarrow)$:  Conversely, suppose that $V$ admits a reduction of structure group to $K$.  Let $\{U_{\ga},h_{\ga}\}$ be an atlas for $V$ whose transition functions take values in $K$.  Then there are sections $s_{1,\ga},\hdots,s_{p,\ga}$ and $s_{p+1,\ga},\hdots,s_{n,\ga}$ satisfying linear relations of the above form, more or less by definition.  Taking the sections $s_{1,\ga},\hdots,s_{p,\ga}$, together with gluing information determined by composing the transition functions $\tau_{\ga,\beta}$ of $V$ with projection to $GL(p,\C)$, we have the data of a rank $p$ subbundle of $V$.  Likewise, taking the sections $s_{p+1,\ga},\hdots,s_{n,\ga}$ together with gluing information determined by composing the $\tau_{\ga,\beta}$ with projection to $GL(q,\C)$, we have the data of a rank $q$ subbundle $V''$.  Clearly, $V'$ and $V''$ are in direct sum, by construction.
\end{proof}

With this established, let $V \rightarrow X$ be a vector bundle possessing a complete flag of subbundles $F_{\bullet}$ and a splitting $V = V' \oplus V''$ as a direct sum of subbundles of ranks $p$ and $q$.  Let $\gamma$ be a $(p,q)$-clan, and define
\[ D_{\gamma} := \{ x \in X \ \vert \ F_{\bullet}(x) \text{ is in position at most $\gamma$ relative to the splitting } V'(x) \oplus V''(x) \}. \]

We now describe how to use the formula for the equivariant class $[Y_{\gamma}]$ to obtain a formula for the fundamental class $[D_{\gamma}]$ of this locus in terms of the Chern classes of $V'$, $V''$, and $F_i/F_{i-1}$ ($i = 1,\hdots,n$).

As described in the introduction to this chapter, the splitting of $V$ as a direct sum of subbundles, together with the complete flag of subbundles, gives us a map
\[ X \stackrel{\phi}{\longrightarrow} BK \times_{BG} BB. \]

Our first task is to see that $D_{\gamma}$ is precisely $\phi^{-1}(\tilde{Y_{\gamma}})$, where $\tilde{Y_{\gamma}}$ denotes the isomorphic image of $E \times^K Y_{\gamma}$ in $BK \times_{BG} BB$.

First, note that $G/K$ can naturally be identified with the space of splittings of $\C^n$ as a direct sum of subspaces of dimensions $p$ and $q$, respectively.  Indeed, $G$ acts transitively on the space of such splittings, and $K$ is precisely the isotropy group of the ``standard" splitting of $\C^n$ as $\C \left\langle e_1,\hdots,e_p \right\rangle \oplus \C \left\langle e_{p+1},\hdots,e_n \right\rangle$.  Now $BK$ is a $G/K$-bundle over $BG$, and a point of $BK$ lying over $eG \in BG$ should be thought of as a splitting of the fiber $\caV_{eG}$, where $\caV = E \times^G \C^n$ is the universal rank $n$ vector bundle over $BG$.  Specifically, the point $egK \in BK$ over $eG \in BG$ is the splitting of $\caV_{eG}$ as the direct sum
\[ \C \left\langle [e,g \cdot e_1],\hdots,[e, g \cdot e_p] \right\rangle \oplus \C \left\langle [e,g \cdot e_{p+1}],\hdots,[e,g \cdot e_n] \right\rangle. \]
Note that $BK$ carries two tautological bundles, say $\caS'$ and $\caS''$, of ranks $p$ and $q$ respectively, which sum directly to $\pi_K^* \caV$, where $\pi_K$ is the projection $BK \rightarrow BB$.  The fiber of $\caS'$ (resp. $\caS''$) over a point $egK$ is the $p$-dimensional (resp. $q$-dimensional) summand of the splitting of $\caV_{eG}$ determined by that point.

Similarly, $BB$ is a $G/B$-bundle over $BG$, with a point of $BB$ representing a complete flag on $\caV_{eG}$.  Specifically, the point $egB \in BB$ is the flag
\[ \left\langle [e,g \cdot e_1],\hdots,[e,g \cdot e_n] \right\rangle. \]
The space $BB$ carries a complete tautological flag of subbundles of $\pi_B^* \caV$ ($\pi_B$ the projection $BB \rightarrow BG$), say $\caT_{\bullet}$.  The fiber of $\caT_i$ over a point $egB \in BB$ is simply the $i$th subspace of the flag on $\caV_{eG}$ determined by that point.

Thus a point of $BK \times_{BG} BB$ should be thought of concretely as a pair consisting of a splitting and a flag of a fiber of $\caV$.  Now let $\gamma$ be a $(p,q)$-clan, with $Y_{\gamma}$ the corresponding $K$-orbit closure on $G/B$.  Assuming Conjecture \ref{conj:pq-closure-description}, we now note that the isomorphic image of $E \times^K Y_{\gamma}$ is precisely the set of points consisting of splittings and flags where the flag is in position at most $\gamma$ relative to the splitting.  Indeed, a point $[e, gB] \in E \times^K Y_{\gamma}$ (with the flag $gB = \left\langle g \cdot e_1,\hdots,g \cdot e_n \right\rangle$ in position at most $\gamma$ relative to the standard splitting of $\C^n$) is carried by the isomorphism $E \times^K G/B \rightarrow BK \times_{BG} BB$ to the point $(eK, egB)$.  This point represents the standard splitting of $\caV_{eG}$, along with the flag $gB$ on $\caV_{eG}$.  Thus the flag is in position at most $\gamma$ relative to the splitting, since $gB \in Y_{\gamma}$.  On the other hand, any such point in $BK \times_{BG} BB$ is of the form $(eK,egB)$ for some $e \in E$ and $g \in G$, which is then carried back to the point $[e,gB] \in E \times^K Y_{\gamma}$ by the inverse isomorphism.

Now, consider the map $\phi$.  If $\rho$ is the classifying map $X \rightarrow BG$ for $V$, denote by $\rho_K$ and $\rho_B$ the lifts of $\rho$ to $BK$ and $BB$, respectively.  The subbundles $V'$ and $V''$ are the pullbacks $\rho_K^*\caS'$, $\rho_K^*\caS''$ of the tautological bundles on $BK$ mentioned above.  Likewise, the flag $E_{\bullet}$ is $\rho_B^*\caT_{\bullet}$.  The map $\phi$ sends $x \in X$ to the pair
\[ (\caS'(\rho_K(x)) \oplus \caS''(\rho_K(x)), \caT_{\bullet}(\rho_B(x)) = (V'(x) \oplus V''(x), F_{\bullet}(x)). \]
In light of this, we see that $\phi(x) \in \tilde{Y_{\gamma}}$ if and only if $F_{\bullet}(x)$ is in position at most $\gamma$ relative to the splitting $V'(x) \oplus V''(x)$.  This says that $\phi^{-1}(\tilde{Y_{\gamma}})$ is precisely the locus $D_{\gamma}$ defined above.  Thus, assuming the situation is suitably generic, as described in the introduction to this chapter, we have that $[D_{\gamma}] = \phi^*([\tilde{Y_{\gamma}}])$.  Again we mention that the genericity should be thought of as requiring that our splitting and our flag of subbundles are in general position with respect to one another.

Our next task is to relate the classes $x_1,\hdots,x_n,y_1,\hdots,y_n$, in terms of which we have expressed the equivariant classes of $K$-orbit closures, to the Chern classes of the bundles $V'$, $V''$, and $F_i/F_{i-1}$ ($i=1,\hdots,n$) on $X$.  The space $(G/B)_K = E \times^K G/B$ carries two bundles $S'_K$ and $S''_K$ of ranks $p$ and $q$, respectively.  Explicitly, the bundle $S'_K$ is $(E \times^K \C \left\langle e_1,\hdots,e_p \right\rangle) \times G/B$, while the bundle $S''_K$ is $(E \times^K \C \left\langle e_{p+1},\hdots,e_n \right\rangle) \times G/B$.  When pulled back to $(G/B)_S$ via the natural map $(G/B)_S \rightarrow (G/B)_K$, these two bundles split as direct sums of line bundles.  $S'_K$ splits as a direct sum of $(E \times^S \C_{X_i}) \times G/B$ for $i=1,\hdots,p$, while $S''_K$ splits as a direct sum of $(E \times^S \C_{X_i}) \times G/B$ for $i=p+1,\hdots,n$.  Recall that the classes $x_i \in H_S^*(G/B)$ are the first Chern classes of these line bundles.  So the pullbacks of the Chern classes of $S'_K$ and $S''_K$ are the elementary symmetric polynomials in $x_1,\hdots,x_p$ and $x_{p+1},\hdots,x_n$, respectively.  Since the pullback is an injection, when we consider $H_K^*(G/B)$ as a subring of $H_S^*(G/B)$, the Chern classes $c_1(S'_K),\hdots,c_p(S'_K)$ are identically $e_1(x_1,\hdots,x_p),\hdots,e_p(x_1,\hdots,x_p)$, while the Chern classes $c_1(S''_K),\hdots,c_q(S''_K)$ are $e_1(x_{p+1},\hdots,x_n),\hdots,e_q(x_{p+1},\hdots,x_n)$.  The bundles $S'_K$ and $S''_K$ are identified with the bundles $\caS'$ and $\caS''$ on the isomorphic space $BK \times_{BG} BB$, and as we have noted, the latter two bundles pull back to $V'$ and $V''$, respectively.  Thus pulling back elementary symmetric polynomials in the $x_i$ to $X$ gives us the Chern classes of the bundles $V'$ and $V''$.

Now, consider the classes $y_i$.  $G/B$ has a tautological flag of bundles $T_{\bullet}$.  Each bundle in this flag is $K$-equivariant, so that we get a flag of bundles $(T_{\bullet})_K = E \times^K T_{\bullet}$ on $(G/B)_K$.  This flag pulls back to a tautological flag $(T_{\bullet})_S$ on $(G/B)_S$ whose subquotients $(T_i)_S / (T_{i-1})_S$ are the line bundles $E \times^S (G \times^B \C_{Y_i})$.  Recall that the classes $y_i$ are precisely the first Chern classes of the latter line bundles.  The bundles $(T_{\bullet})_K$ match up with the bundles $\caT_{\bullet}$ on $BK \times_{BG} BB$ via our isomorphism, and as we have noted, the latter bundles pull back to the flag $F_{\bullet}$ of bundles on $X$.  Thus when we pull back to $X$, the class $y_i$ is sent to $c_1(F_i / F_{i-1})$ for $i=1,\hdots,n$.

As an illustration, suppose we have a scheme $X$ and a rank $4$ vector bundle $V \rightarrow X$.  Suppose that $V$ splits as a direct sum of rank $2$ subbundles ($V = V' \oplus V''$), and suppose further that $V$ is equipped with a complete flag of subbundles ($F_1 \subset F_2 \subset F_3 \subset V$).  Let $z_1,z_2,z_3,z_4$ be $c_1(V')$, $c_2(V')$, $c_1(V'')$, $c_2(V'')$, respectively.  Let $y_i = c_1(F_i/F_{i-1})$ for $i = 1,2,3,4$.  For any $(2,2)$-clan $\gamma$, we can use the example of Subsection \ref{ssec:supq_example} to give Chern class formulas for the class of any locus $D_{\gamma}$ in terms of the $z_i$ and $y_i$.

For instance, consider the clan $\gamma = (+,+,-.-)$.  The formula for $[Y_{\gamma}]$, when expanded and regrouped conveniently, gives
\[ (x_1x_2)^2 - (x_1+x_2)(x_1x_2)(y_3+y_4) + (x_1x_2)(y_3+y_4)^2-(x_1+x_2)(y_3y_4)(y_3+y_4) + (x_1^2 + x_2^2)(y_3y_4) + y_3^2y_4^2. \]
We have seen that, through all our identifications, $x_1 + x_2$ pulls back to $z_1$, and $x_1x_2$ pulls back to $z_2$.  Thus the conclusion is that
\[ [D_{(+,+,-,-)}] = z_2^2 -z_1z_2(y_3+y_4) + z_2(y_3+y_4)^2 - z_1y_3y_4(y_3+y_4)+(z_1^2-2z_2)(y_3y_4) + y_3^2y_4^2. \]
One checks that this factors as
\[ [D_{(+,+,-,-)}] = (z_1y_4 - z_2 - y_4^2)(z_1y_3 - z_2 - y_3^2). \]

\begin{remark}
The author wonders whether loci of this type occur ``in nature".  That is, are there interesting varieties which can be realized as loci of the type we have described here?  In particular, is condition (3) of Definition \ref{def:rel-pos} an interesting geometric condition to place on a degeneracy locus of this type?

There are some instances in which condition (3) turns out to be redundant.  Indeed, in \cite{Wyser-11a}, it is noted that a number of the $K$-orbit closures can be described without need of condition (3).  Such orbit closures are \textit{Richardson varieties}, intersections of Schubert varieties with opposite Schubert varieties.  In such cases, formulas for the corresponding loci can actually be deduced from the results of \cite{Fulton-92}, since they are (proper, reduced) intersections of two degeneracy loci treated by the results of that paper.  We note, however, that the Chern class formulas one gets from doing the computation that way are different from those we obtain here using our $K$-orbit formulas.

In cases where condition (3) \textit{is} needed, the author sees no apparent way to deduce formulas for the corresponding locus from Fulton's results, since condition (3) is not really a ``Schubert-like" condition.
\end{remark}

\subsection{Other symmetric subgroups in type $A$}\label{ssec:type-a-other-deg-loci}
Here, we treat the remaining cases in type $A$.  Because they are all so similar, we describe them here together rather than giving each example its own subsection.

We start first with the full orthogonal group.  Recall (Subsections \ref{ssec:kgb_param_so_odd}, \ref{ssec:so2n_param}) the parametrization of $K$-orbits in this case.  Namely, the orbits are parametrized by involutions in $S_n$.  Moreover, if $\gamma$ is the quadratic form for which $K$ is the isometry group, and if $b \in S_n$ is an involution, then the orbit $\caO_b$ admits the following linear algebraic description:
\[ \caO_b := \{ F_{\bullet} \in G/B \mid \text{rank}(\gamma|_{F_i \times F_j}) = r_b(i,j) \text{ for all } i,j \}. \]

We now give linear algebraic descriptions of the orbit \textit{closures}.  Recall (Proposition \ref{prop:full_closure_twisted}, Remark \ref{rmk:bruhat-order-on-I}) that when the Richardson-Springer map is injective, the full closure order on the set of twisted involutions is precisely the restricted Bruhat order.  Recall also (Subsection \ref{ssec:kgb_param_so_odd}) that one passes from the set of twisted involutions to the set of honest involutions via multiplication by the long element $w_0$, which inverts the Bruhat order.  From this it follows that when $K \backslash G/B$ is identified with the set of involutions in $S_n$, its closure order is precisely given by the \textit{reverse} Bruhat order on these involutions.  Given this, it is easy to see that $\overline{\caO_b}$ is precisely 
\begin{equation}\label{eqn:closure-equations}
	\overline{\caO_b} := \{ F_{\bullet} \in G/B \mid \text{rank}(\gamma|_{F_i \times F_j}) \leq r_b(i,j) \text{ for all } i,j \}.
\end{equation}

Indeed, one need only use the definition of the Bruhat order on $S_n$ given in \cite[\S 10.5]{Fulton-YoungTableaux}, formulated in terms of the rank numbers $r_b(i,j)$.  This definition is easily seen to be equivalent to other, more ``standard" definitions of the Bruhat order (\cite[\S 10.5, Exercises 8-9]{Fulton-YoungTableaux}).

For the sake of brevity, given a form $\gamma$ on a vector space $V$, together with a flag $F_{\bullet}$ on $V$, we say that $\gamma$ ``has rank at most $b$ on the flag $F_{\bullet}$" if the flag satisfies the conditions of (\ref{eqn:closure-equations}) relative to $\gamma$.

The space $BK$ is a $G/K$-bundle over $BG$, with $G/K$ the space of all nondegenerate, symmetric bilinear forms on $\C^n$.  This correspondence associates to the coset $gK \in G/K$ the form $g \cdot \gamma$, with
\[ g \cdot \gamma(v,w) = \gamma(g^{-1}v,g^{-1}w). \]
The form $\gamma$ is the one associated to the coset $1K$, and is defined by
\[ \gamma(e_i,e_j) = \delta_{i,n+1-j} \]
where $e_1,\hdots,e_n$ is the standard basis for $\C^n$.  Then a point $eK \in BK$ can naturally be identified with a quadratic form on the fiber $\caV_{eG}$ in the following way:  Let $v_1,\hdots,v_n = [e,e_1],\hdots,[e,e_n] \in E \times^G \C^n$ be a basis for $\caV_{eG}$, and define the form associated to $eK$ by
\[ \left\langle v_i, v_j \right\rangle = \delta_{i,n+1-j}. \]

It is a standard fact that a vector bundle $V \rightarrow X$ of rank $n$ admits a reduction of structure group to $O(n,\C)$ if and only if the bundle carries a nondegenerate quadratic form.  By this we mean a bundle map $\text{Sym}^2(V) \rightarrow X \times \C$ which restricts to a nondegenerate quadratic form on every fiber.  (We will always assume our forms take values in the trivial line bundle.)  If $\rho: X \rightarrow BG$ is a classifying map for the bundle $V$, then the lift of $\rho$ to $BK$ sends $x \in X$ to the point of $BK$ which represents the form $\gamma|_{V_x} = \gamma|_{\caV_{\rho(x)}}$ on the fiber $\caV_{\rho(x)}$.  Then $\gamma$ is effectively pulled back from a corresponding ``tautological" form $\tau$ on $\pi^* \caV \rightarrow BK$ ($\pi$ the projection $BK \rightarrow BG$), whose values on the fiber of $\pi^* \caV$ over every point of $BK$ are identified by the point itself.

A lift of $\rho$ to $BB$ is equivalent, as in the last subsection, to a flag $E_{\bullet}$ of subbundles of the bundle $V$.  Thus we see that given a vector bundle $V$ (with classifying map $\rho$) equipped with a quadratic form $\gamma$ and a complete flag of subbundles $E_{\bullet}$, we get a map $\phi: X \rightarrow BK \times_{BG} BB$ which sends $x \in X$ to the point $(\tau|_{\rho(x)}, (\caT_{\bullet})_{\rho(x)}) = (\gamma|_{V_x}, (E_{\bullet})_x)$.

We now note that if $Y_b = \overline{\caO_b} \subseteq G/B$ is a $K$-orbit closure, then the isomorphism between $E \times^K (G/B)$ and $BK \times_{BG} BB$ carries $E \times^K Y_b$ to the set of all (Form, Flag) pairs where the form has rank at most $b$ on the flag.  Indeed, given $gB \in Y_b$, the point $[e,gB] \in E \times^K Y_b$ is carried to the point $(eK,egB) \in BK \times_{BG} BB$.  This point represents the antidiagonal form on $\caV_{eG}$ relative to the basis $[e,e_1],\hdots,[e,e_n]$, together with the flag $gB$ on $\caV_{eG}$ relative to that same basis.  Then the form has rank at most $b$ on the flag, by choice of $gB$.  On the other hand, any point $(eK,egB) \in BK \times_{BG} BB$ where the antidiagonal form on $\caV_{eG}$ has rank at most $b$ on the flag $gB$ is matched up with the point $[e,gB]$, clearly an element of $E \times^K Y_b$.

Given this, along with our description of the map $\phi$, we see that given a vector bundle $V$ over $X$ with a form and a flag, and an involution $b$, the locus
\begin{equation}\label{eqn:deg-locus}
	D_b = \{x \in X \mid \gamma|_{V_x} \text{ has rank at most $b$ on } (F_{\bullet})_x\}
\end{equation}
is precisely $\phi^{-1}(\widetilde{Y_b})$, with $\wt{Y_b}$ the isomorphic image of $E \times^K Y_b$ in $BK \times_{BG} BB$.  Thus generically, the class of such a locus is given by $[D_b] = \phi^*(\widetilde{Y_b})$.  As explained in the previous subsection, the classes $y_i \in H_K^*(G/B)$ pull back through $\phi$ to the Chern classes $c_1(F_i/F_{i-1})$.  Thus a formula for the equivariant classes of the $K$-orbit closure $Y_b$, which we note involves \textit{only the $y$ variables}, can be viewed as giving a formula for $[D_b]$ in terms of the Chern classes $c_1(F_i/F_{i-1})$.

Note that the above analysis applies to the case $G = GL(n,\C)$, $K = O(n,\C)$.  The case $G = SL(n,\C)$, $K = SO(n,\C)$ is identical in the event that $n$ is odd, but a bit different in the case that $n$ is even.  We address this in a moment.  First, we point out that the above analysis applies equally well to the case of $G = SL(2n,\C)$, $K = Sp(2n,\C)$, with only very minor modifications.  The orbit closures in that case are parametrized by \textit{fixed point-free} involutions, and descriptions of their closures are identical to those of (\ref{eqn:closure-equations}) when $\gamma$ is taken to be the \textit{skew} form for which $K$ is the isometry group.  A lift of the classifying map to $BK$ then amounts to a nondegenerate \textit{skew} form on the bundle $V$, by which we mean a bundle map $\bigwedge^2(V) \rightarrow X \times \C$ which restricts to a nondegenerate skew form on each fiber.  Given such a form, along with a flag of subbundles of $V$, one can define a degeneracy locus $D_b \subseteq X$ associated to a fixed point-free involution $b$ just as in (\ref{eqn:deg-locus}) above.  And just as above, our formulas for the equivariant classes of $K$-orbit closures (which again involve only the $y$ variables) pull back to a formula for $[D_b]$ in the Chern classes of the subquotients of the flag.

We now address the case of $(SL(2n,\C), SO(2n,\C))$.  In the even case, each $O(2n,\C)$-orbit on $GL(n,\C)/B$ associated to a fixed point-free involution splits as a union of two $SO(2n,\C)$-orbits, so that each $O(2n,\C)$-orbit \textit{closure} has two irreducible components, each the closure of a distinct $SO(2n,\C)$-orbit.  Thus a formula for the class of an $SO(2n,\C)$-orbit closure associated to a fixed point-free involution $b$ should pull back to a formula for an irreducible component of the locus $D_b$, defined as in (\ref{eqn:deg-locus}).  Note (see, e.g., Table \ref{tab:type-a-so4}) that our formulas for equivariant classes of $SO(2n,\C)$-orbit closures associated to involutions with fixed points involve the $y$-variables only, but the formulas for equivariant classes of orbit closures associated to fixed point-free involutions typically also involve the class $x_1 \hdots x_n$.  We now identify this class as pulling back to an ``Euler class" $e \in H^*(X)$ associated to our bundle with quadratic form.

The Euler class of a rank $2n$ complex vector bundle $V \rightarrow X$ with nondegenerate quadratic form is a class $e \in H^{2n}(X)$ which is uniquely defined up to sign by the following property:  If $W \rightarrow Y$ is any rank $2n$ complex vector bundle with nondegenerate quadratic form, possessing a maximal (rank $n$) isotropic subbundle $E$, and if $\rho: Y \rightarrow X$ is a map for which $W = \rho^*V$, then $\rho^*(e) = \pm c_n(E)$.  In particular, the space $BK$ carries the bundle $\caV$ (omitting the pullback notation), equipped with a ``tautological" nondegenerate quadratic form, as we have already noted, so there is an associated Euler class in $H^{2n}(BK)$.  (For the interested reader, we mention that this class is the Euler class --- in the sense of \cite[\S 9]{Milnor-Stasheff} --- of a rank $2n$ \textit{real} bundle on $BK$ whose complexification is $\caV$.  The real bundle in question is pulled back, through a homotopy equivalence $BSO(2n,\C) \rightarrow BSO(2n,\R)$, from the canonical rank $2n$ real bundle $\caV_{\R}$ on the latter classifying space.)  The Euler class of $V \rightarrow X$ is the pullback of this class in $H^{2n}(BK)$ through the classifying map.  Note that it exists even in cases where $V$ does not carry a maximal isotropic subbundle.  This class is \textit{not} a polynomial in the Chern classes of $V$.  (This could indicate that the equivariant classes of $SO(2n,\C)$-orbit closures on $G/B$ associated to fixed point-free involutions are \textit{not} expressible in the $y$-variables alone.)  These facts are explained further in \cite{Edidin-Graham-95} where, among other results, the existence of an \textit{algebraic} Euler class of a Zariski-locally trivial bundle with quadratic form is established.

Now, note that the class $x_1 \hdots x_n \in H_S^*(G/B)$ is (the pullback to $H_S^*(G/B)$ of) $c_n(\bigoplus_{i=1}^n \caL_{X_i})$ in the notation of Subsection \ref{ssec:eqvt_cohom}, Proposition \ref{prop:eqvt-cohom-flag-var} (again omitting pullback notation).  The bundle $\bigoplus_{i=1}^n \caL_{X_i}$ is a maximal isotropic subbundle of the pullback of $\caV$ to $BS$ through the projection $BS \rightarrow BK$.  Thus $x_1 \hdots x_n$, viewed as a class in $H_K^*(G/B)$, is an Euler class for $\caV$.  Pulling all the way back to $X$ through the classifying map, we see that $\phi^*(x_1 \hdots x_n)$ is an Euler class for the bundle $V \rightarrow X$.

Summarizing, our formulas for the equivariant classes of $SO(2n,\C)$-orbit closures can be interpreted as formulas for the fundamental classes of irreducible components of degeneracy loci $D_b$ ($b$ a fixed point-free involution) defined as above, expressed in the first Chern classes of the subquotients of the flag of subbundles, together with an Euler class for the bundle with quadratic form.

\section{Notes on other types}
Each of the symmetric pairs we have considered in types $BCD$ should give similar degeneracy locus formulas to those we have described above in type $A$.  The setup should be roughly as follows:  One starts with a vector bundle $V$ over a scheme $X$, equipped with a non-degenerate quadratic (types $BD$) or skew (type $C$) form, along with a flag of subbundles which is isotropic/Lagrangian with respect to that form.  The form amounts to $V$ having structure group $G=SO(n,\C)$ or $Sp(2n,\C)$, as we have discussed, while the flag corresponds to a lift of the classifying map for the bundle to $BB$.  One should then determine the additional structure on the bundle which amounts to a lift of the classifying map to $BK$.  In the cases where $K$ is $S(O(p,\C) \times O(q,\C))$ or $Sp(2p,\C) \times Sp(2q,\C)$, this should be a splitting of the bundle as a direct sum of two subbundles of the appropriate ranks such that the form restricts to each summand non-degenerately.  In the cases where $K = GL(n,\C)$, it should be a splitting of the bundle as a direct sum of two rank $n$ subbundles which are orthogonal complements with respect to the form.

Given such a setup, one should be able to parametrize subvarieties of $X$ determined by imposing linear algebraic conditions on the fibers of $V$ relative to all of these structures, as we have just described in type $A$.  The linear algebraic conditions one must impose should correspond to the linear algebraic conditions defining $K$-orbit closures.  We do not carry this out explicitly here, since it is not clear at this time exactly what linear algebraic conditions define the $K$-orbit closures in the cases outside of type $A$.  As we have noted, in all cases outside of type $A$, the $K$-orbits are the intersections of $GL(p,\C) \times GL(q,\C)$-orbits on the type $A$ flag variety with a smaller flag variety of type $BCD$ (for some appropriate choice of $p,q$).  One would hope that this carries over to orbit \textit{closures} --- i.e.

\begin{question}\label{question:orbit-closure}
For $K$ a symmetric subgroup in types $BCD$, are the $K$-orbit \textit{closures} intersections of the corresponding $GL(p,\C) \times GL(q,\C)$-orbit \textit{closures} with the smaller flag variety?
\end{question}

The answer to this question is not obvious.  In fact, if one considers the analogous question for Schubert varieties, the answer is ``yes" in types $BC$, but ``no" in type $D$.  Combinatorially, we are asking the following question:

\begin{question}\label{question:orbit-poset-restrictions}
For $K$ a symmetric subgroup in types $BCD$, is the poset $K \backslash G/B$ (equipped with the full closure order) poset-isomorphic to the corresponding subposet of $GL(p,\C) \times GL(q,\C)$-orbits?
\end{question}

Approaching the matter from this combinatorial perspective would be one possible approach to answering Question \ref{question:orbit-closure}.  Question \ref{question:orbit-poset-restrictions} should be relatively easy to answer in the affirmative in cases where the \textit{weak} closure order on $K \backslash G/B$ is visibly the restriction of the \textit{weak} closure order on $GL(p,\C) \times GL(q,\C)$-orbits.

However, this is not the case in all examples.  Indeed, consider the pair $(G,K) = (SO(2n,\C), GL(n,\C))$.  The $K$-orbits corresponding to $(1,1,-,+,2,2)$ and $(1,+,2,1,-,2)$ are related in the weak order.  Indeed, the simple reflection $s_{\ga_3}$ raises $(1,1,-,+,2,2)$ to $(1,+,2,1,-,2)$.  On the other hand, the $GL(3,\C) \times GL(3,\C)$-orbits corresponding to these $(3,3)$-clans are \textit{not} related in the weak order on $GL(3,\C) \times GL(3,\C)$-orbits.  It \textit{is} the case, however, that $(1,1,-,+,2,2)$ is below $(1,+,2,1,-,2)$ in the \textit{full} closure order, so this does \textit{not} provide a negative answer to Question \ref{question:orbit-poset-restrictions}.  Symmetric pairs where the weak order on $K \backslash G/B$ does not correspond to the restricted weak order on $GL(p,\C) \times GL(q,\C)$-orbits may be a bit more difficult to analyze combinatorially.

In cases where the answer to Question \ref{question:orbit-closure} turns out to be ``yes", then assuming Conjecture \ref{conj:pq-closure-description} is true, the $K$-orbit closures would be described exactly by the linear algebraic conditions of Conjecture \ref{conj:pq-closure-description}, simply restricting attention to flags which were isotropic or Lagrangian with respect to the appropriate form.  In cases where the answer to Question \ref{question:orbit-closure} is ``no" (if, indeed, there are any such cases), then one would hope to be able to give some alternative linear algebraic description of the orbit closures as sets of isotropic/Lagrangian flags, and then describe degeneracy loci by compatible linear algebraic conditions on fibers of a vector bundle over a scheme.  Assuming such cases even exist, it is not at all clear what these linear algebraic descriptions might be.

We leave these questions open for now, and hope to address them in future work.

%% file: Appendix-A.tex
\chapter{Proofs of the correctness of the orbit parametrizations in types $BCD$}
In this appendix, we give a case-by-case proof of the correctness of the parametrizations of orbit sets described in all cases outside of type $A$.  This includes the proof of Theorem \ref{thm:k-orbit-intersections}.  In each of these cases, we have made the claim that the $K$-orbits on $G/B$ are parametrized by some subset of the $(p,q)$-clans (for some appropriate $p,q$) possessing one or more special combinatorial properties.  We now indicate how this can be proved.

In all of these cases, the involution $\theta$ on $G$ for which $K = G^{\theta}$ is the restriction of an involution $\theta'$ on $G' = GL(n,\C)$ for some $n$, for which $K' = (G')^{\theta'} \cong GL(p,\C) \times GL(q,\C)$ for some $p,q$.  Then $K = G \cap K'$, so that the intersection of a $K'$-orbit on $X' = G'/B'$ with $X$, if non-empty, is clearly stable under $K$ and hence \textit{a priori} is a union of $K$-orbits.  The $K'$-orbits on $X'$ being parametrized by $(p,q)$-clans, one may ask for combinatorial conditions on the $(p,q)$-clan $\gamma$ which amount to $Q_{\gamma}' \cap X \neq \emptyset$.  Once one has determined such combinatorial conditions, then the next question is, given a clan $\gamma$ satisfying these combinatorial conditions, is the $K$-stable set $Q_{\gamma} \cap X$ a \textit{single} $K$-orbit, or a union of multiple $K$-orbits?  If it is always a single $K$-orbit, then the $K$-orbits are clearly in 1-to-1 correspondence with the $K'$-orbits on $X'$ which meet $X$, and thus are parametrized by the $(p,q)$-clans satisfying the appropriate combinatorial conditions.  As indicated in the introduction, we have chosen our symmetric pairs $(G,K)$ precisely so that this is always the case.

By what we have said so far, then, it is clear that to prove the correctness of the parametrizations of $K \backslash X$ in any case in type $BCD$, we must do the following two things:
\begin{enumerate}
	\item Prove that a $K'$-orbit $Q_{\gamma}$ intersects $X$ if and only if the corresponding clan $\gamma$ has the specified combinatorial properties.
	\item Prove that each such non-empty intersection is a \textit{single} $K$-orbit.
\end{enumerate}

We first establish (1) on a case-by-case basis, then we deal with (2).

\begin{prop}\label{prop:appendix-combinatorial-criteria}
In each case outside of type $A$, the $K'$-orbit $Q_{\gamma}$ corresponding to the clan $\gamma$ intersects $X$ if and only if $\gamma$ has the combinatorial properties specified in all of our parametrizations.
\end{prop}

\section{Case-by-case proof of Proposition \ref{prop:appendix-combinatorial-criteria}}
We indicate the details of each case.  We start with the type $C$ pairs, since adequate proofs for those cases have already appeared in the literature.  For these, we simply indicate the appropriate reference.  We then move on to the type $B$ case, which is a bit complicated, and then handle the type $D$ cases, each of which is very similar to either a type $C$ case or the type $B$ case.

\subsection{$(Sp(2n,\C), Sp(2p,\C) \times Sp(2q,\C))$}\label{ssec:sp2p2q-clans}
Here, the goal is to prove that the $GL(2p,\C) \times GL(2q,\C)$-orbit $Q_{\gamma}$ corresponding to the $(2p,2q)$-clan $\gamma=(c_1,\hdots,c_{2n})$ meets $X$ if and only if 
\begin{enumerate}
	\item $\gamma$ is symmetric; and
	\item $c_i \neq c_{2n+1-i}$ whenever $c_i \in \N$.
\end{enumerate}
The ``only if" portion is \cite[Proposition 4.3.2]{Yamamoto-97}, while the ``if" portion is \cite[Theorem 4.3.12]{Yamamoto-97}.  The statement of the latter result also spells out how to find a representative of $Q_{\gamma} \cap X$.  It amounts to choosing a representative of $Q_{\gamma}$, using the algorithm described in Subsection \ref{ssec:orbits_supq}, in a certain way so as to always produce a Lagrangian flag.

\subsection{$(Sp(2n,\C), GL(n,\C))$}
Here, we must prove that the $GL(n,\C) \times GL(n,\C)$-orbit $Q_{\gamma}$ corresponding to the $(n,n)$-clan $\gamma$ meets $X$ if and only if $\gamma$ is skew-symmetric.

The ``only if" portion is \cite[Proposition 3.2.2]{Yamamoto-97}.  The ``if" portion follows from \cite[Theorem 3.2.11]{Yamamoto-97}.  Again, the latter result indicates how one can choose a representative of $Q_{\gamma}$, using the algorithm of Subsection \ref{ssec:orbits_supq}, so as to always produce a Lagrangian flag.

\subsection{$(SO(2n+1,\C), S(O(2p,\C) \times O(2q+1,\C)))$}
Here, we must prove that the $K' = GL(2p,\C) \times GL(2q+1,\C)$-orbit $Q_{\gamma}$ corresponding to the $(2p,2q+1)$-clan $\gamma$ meets $X$ if and only if $\gamma$ is symmetric.  In this case, $K'$ should be realized as the fixed points of the involution $\text{int}(I_{p,2q+1,p})$.

The proof that $\gamma$ must be symmetric is identical to the one given in \cite[Proposition 4.3.2]{Yamamoto-97} for the case $(G,K) = (Sp(2n,\C),Sp(2p,\C) \times Sp(2q,\C))$, alluded to above.  Of course, in the proof of that proposition, the $(2p,2q)$-clan $\gamma = (c_1, \hdots, c_{2n})$ representing a $GL(2p,\C) \times GL(2q,\C)$-orbit is also shown to have the additional property that $c_i \neq c_{2n+1-i}$ whenever $c_i \in \N$.  This does not hold in the present case; in fact, \textit{all} symmetric $(2p,2q+1)$-clans correspond to orbits which meet $X$, even those containing matching natural numbers in positions $(i,2n+2-i)$ for some (or for many) $i$.  We see this presently.

The other implication involves finding an isotropic representative of the $K'$-orbit $Q_{\gamma}$ in the event that $\gamma$ is symmetric.  One might naively hope that it is possible, as in other cases, to choose an isotropic representative of $Q_{\gamma}$ using the algorithm of \cite{Yamamoto-97} described in Subsection \ref{ssec:orbits_supq}.  However, this is not the case, as one can see even in very small examples.  Thus we must describe a way to take one of these representatives and ``move" it by an element of $K'$ to produce a flag that \textit{is} isotropic.  This is easier to do in the more typical setting, where $G=SO(2n+1,\C)$ is realized as the isometry group of the \textit{diagonal} form - that is,
\begin{equation}\label{eqn:so-odd-eqn}
G = \{ g \in SL(2n+1,\C) \ \vert \ gg^t = Id \}.
\end{equation}

So until further notice, let 
\[ G' = GL(2n+1,\C); \theta' = \text{int}(I_{2p,2q+1}); K' = (G')^{\theta'} = GL(2p,\C) \times GL(2q+1,\C) \]
(as in Subsection \ref{sect:type_a_supq}); and let
\[ G = SO(2n+1,\C); \theta = \theta'|_G; K = G^{\theta} = G \cap K' = S(O(2p,\C) \times O(2q+1,\C)), \]
with $SO(2n+1,\C)$ realized as in (\ref{eqn:so-odd-eqn}) above.  We will describe how to move one of the representatives of \cite{Yamamoto-97} by $K'$ to produce an isotropic flag (with respect to the diagonal form).  After doing so, we will describe how to conjugate everything back to our preferred realization.

Suppose that we are given a symmetric $(2p,2q+1)$-clan $\gamma$, and suppose that from that clan we produce the representative
\[ V_{\bullet} = \left\langle v_1,\hdots,v_{2n+1} \right\rangle \]
using Yamamoto's algorithm.  We shall say how to modify each vector $v_i$ so as to produce an isotropic representative of the orbit $Q_{\gamma}$.  Our modification of each $v_i$ is accomplished by specifying a way to send each standard basis vector $e_j$ to a linear combination of basis vectors
\[ e_j \mapsto \displaystyle\sum_{k=1}^{2n+1} \lambda_k e_k, \]
where $\lambda_k = 0$ for $k > 2p$ if $j \leq 2p$, and where $\lambda_k = 0$ for $k \leq 2p$ if $j > 2p$.  (Note that this simply specifies an element of $K'$ by which to act on the flag $V_{\bullet}$ to produce an isotropic flag.)

Where to send each vector $v_i$ will depend upon the character in the $i$th position of the clan.  Further, if that character is a natural number, it will also depend upon the position in which the matching natural number appears.  We break this down by cases:

\textit{Case 1:  $c_i = \pm$.}
First, note that because the clan $\gamma$ is symmetric, either a + or a - must appear in position $n+1$.  So, in the event that $i=n+1$, we know that $v_i = e_j$ for some $j$.  In this case, we simply leave $e_j$ alone:  $e_j \mapsto e_j$.

If $i \neq n+1$, then again we have $v_i = e_j$ for some $j$, and also that $v_{2n+2-i} = e_k$ for some $k$.  If $c_i = +$, then $j,k \leq 2p$, and if $c_i = -$, then $j,k \geq 2p+1$.  In either event, we should send 
\[ e_j \mapsto e_j + ie_k, \text{ and} \]
\[ e_k \mapsto e_j - ie_k. \]

Thus, if $i \neq n+1$, $v_i \mapsto e_j + ie_k$, and $v_{2n+2-i} \mapsto e_j - ie_k$.

\textit{Case 2:  $c_i \in \N$, and $c_{2n+2-i} = c_i$.}
In this case, we know that $v_i = e_j + e_k$, and $v_{2n+2-i} = e_j - e_k$ for some $j \leq 2p$ and $k \geq 2p+1$.  Then we should send 
\[ e_j \mapsto e_j, \text{ and} \]
\[ e_k \mapsto ie_k, \] 
hence sending 
\[ v_i \mapsto e_j + ie_k, \text{ and} \]
\[ v_{2n+2-i} \mapsto e_j - ie_k. \]

\textit{Case 3:  $c_i \in \N$, and $c_{2n+2-i} \neq c_i$.}
We may assume that $i < j$ for whichever $j$ is such that $c_i = c_j$.  Let's say that $c_i = a \in \N$.  Then, because the clan in question is symmetric, there is a different natural number, say $b$, in position $2n+2-i$.  Further, if the other occurrence of $a$ occurs in position $j$, then the other occurrence of $b$ occurs in position $2n+2-j$.  We know that $v_i = e_k + e_l$ and $v_j = e_k - e_l$ for some $k \leq 2p$, $j \geq 2p+1$.  We also know that $v_{2n+2-j}$ and $v_{2n+2-i}$ are $e_r + e_s$ and $e_r - e_s$, respectively, for some $r \leq 2p$, $s \geq 2p+1$.  Then we should send
\[ e_k \mapsto e_k + ie_r; \]
\[ e_l \mapsto e_l + ie_s; \]
\[ e_r \mapsto e_r + ie_k; \text{ and} \]
\[ e_s \mapsto -e_s - ie_l. \]

It is clear that the flag so obtained is isotropic with respect to the chosen bilinear form.  Indeed, our form is characterized by the fact that $\left\langle e_i,e_j \right\rangle = \delta_{i,j}$.  This being the case, it is obvious by construction that the vectors $\{v_i\}_{i=1}^n$ are all isotropic and pairwise orthogonal.  It is also clear that $v_{n+i}$ is orthogonal to each of $v_1,\hdots,v_{n+1-i}$ for each $i=1,\hdots,n$.  This says that the flag is isotropic.

Having obtained an isotropic representative of the $K'$-orbit on $X'$ corresponding to each symmetric clan, we now describe how to translate this back to our chosen setting, where $SO(2n+1,\C)$ is realized as the group of linear automorphisms of $\C^{2n+1}$ preserving the \textit{anti-diagonal} form.  For clarity, let us now say that $G_1$ is the realization of $SO(2n+1,\C)$ given by the anti-diagonal form (those matrices with $gJg^t = J$), and that $G_2$ is the realization of $SO(2n+1,\C)$ given by the diagonal form (those matrices with $gg^t = \text{Id}$).  Let $K_1 \subseteq G_1$ be the fixed points of the involution $\text{int}(I_{p,2q+1,p})$, and let $K_2 \subseteq G_2$ be the fixed points of $\text{int}(I_{2p,2q+1})$.

Consider the matrix $g$ given as follows:
\[
g_{j,j} = 
\begin{cases}
	1 & \text{ if $j = n+1$}, \\
	\dfrac{1+i}{2} & \text{ otherwise.}
\end{cases}
\]

\[
g_{j,2n+2-j} =
\begin{cases}
	1 & \text{ if $j = n+1$}, \\
	\dfrac{1-i}{2} & \text{ otherwise.}
\end{cases}
\]

\[ g_{j,k} = 0 \text{ if $k \neq j, 2n+2-j$}. \]

Then $g$ is a symmetric square root of the anti-diagonal matrix $J$.  This means that it conjugates $G_2$ to $G_1$, since if $hh^t = \text{Id}$, we have
\[ (ghg^{-1})^t J (ghg^{-1}) = g^{-1} h^t (gJg) hg^{-1} = g^{-2} = J. \]

However, the problem with this $g$ is that while it conjugates $G_2$ to $G_1$, it does \textit{not} conjugate 
$K_2$ to $K_1$.  To remedy this, we modify $g$ a bit.  If $p$ is even, then let $\pi$ be the symmetric permutation matrix which corresponds to the involution given in cycle notation by 
\[ (p+1,2n+1)(p+2,2n) \hdots (2p,2n+2-p). \]
If $p$ is odd, then let $\pi$ be the negative of this permutation matrix.  Then $\pi \in G_2$.  It is clear, then, that $g\pi$ still conjugates $G_2$ to $G_1$.  However, we claim that $g\pi$ also conjugates $K_2$ to $K_1$.  It is actually a bit easier to see (equivalently) that $\pi g^{-1}$ conjugates $K_1$ to $K_2$.  Suppose that $k \in K_1$, so that
\[ I_{p,2q+1,p}kI_{p,2q+1,p} = k. \]
Then since $g$ commutes with $I_{p,2q+1,p}$, we have
\[ (gI_{p,2q+1,p}g^{-1})k(gI_{p,2q+1,p}g^{-1}) = k, \]
so
\[ I_{p,2q+1,p}(g^{-1}kg)I_{p,2q+1,p} = g^{-1}kg. \]
Now, since $\pi I_{p,2q+1,p} \pi = I_{2p,2q+1}$, and since $\pi^2 = \text{Id}$, we have
\[ (\pi I_{p,2q+1,p} \pi)(\pi g^{-1} k g \pi)(\pi I_{p,2q+1,p} \pi) = \pi g^{-1} k g \pi, \]
so 
\[ I_{2p,2q+1} (\pi  g^{-1} k g \pi) I_{2p,2q+1} = \pi g^{-1} k g \pi. \]
This says that $\pi g^{-1} k g \pi \in K_2$.  Thus $g \pi$ conjugates $K_2$ to $K_1$.

Now, with that established, given a representative $F_{\bullet}$ of the $K_2$-orbit on $X$ given by some symmetric $(2p,2q+1)$-clan, to get a representative of the $K_1$-orbit corresponding to that same clan, we just act on the flag $F_{\bullet}$ by the matrix $g \pi$ to get the new flag $F_{\bullet}' = g \pi F_{\bullet}$.  The flag $F_{\bullet}$ is isotropic with respect to the diagonal form, so the flag $F_{\bullet}'$ is isotropic with respect to the anti-diagonal form.

Let us look at a small example which illustrates the method just described for finding an isotropic representative of the $K'$-orbit $Q_{\gamma}$ corresponding to a symmetric $(2p,2q+1)$-clan $\gamma$.  Take $p = q = 1$, so that $n = 2$, and so that we are dealing with $G = SO(5,\C)$, $K = S(O(2,\C) \times O(3,\C))$.  Take the symmetric $(2,3)$-clan $\gamma = (1,-,+,-,1)$.  None of the possible representatives produced by the algorithm of \cite{Yamamoto-97} are isotropic.  To produce an isotropic representative of this orbit by the method just described, though, we do take one of these representatives as a starting point.  So assign $\pm$ signatures to the $1$'s as follows:

\[ (1_+,-,+,-,1_-). \]

Choose the permutation $\sigma = 13245$.  This gives us the following representative of $Q_\gamma \in K' \backslash X'$:
\[ \left\langle e_1+e_5,e_3,e_2,e_4,e_1-e_5 \right\rangle. \]

We next move this representative by $K'$ to obtain a flag isotropic with respect to the diagonal form.  The result is the flag

\[ F_{\bullet} = \left\langle e_1 + ie_5, e_3 + ie_4, e_2, e_3 - ie_4, e_1 - ie_5 \right\rangle. \]

This flag is isotropic with respect to the diagonal form.  Now, we must move this flag by the matrix $g \pi$, where
\[ g = 
\begin{pmatrix}
\frac{1+i}{2} & 0 & 0 & 0 & \frac{1-i}{2} \\
0 & \frac{1+i}{2} & 0 & \frac{1-i}{2} & 0 \\
0 & 0 & 1 & 0 & 0 \\
0 & \frac{1-i}{2} & 0 & \frac{1+i}{2} & 0 \\
\frac{1-i}{2} & 0 & 0 & 0 & \frac{1+i}{2}
\end{pmatrix}, \]
and 
\[ \pi =
\begin{pmatrix}
1 & 0 & 0 & 0 & 0 \\
0 & 0 & 0 & 0 & 1 \\
0 & 0 & 1 & 0 & 0 \\
0 & 0 & 0 & 1 & 0 \\
0 & 1 & 0 & 0 & 0
\end{pmatrix}.
\]

Thus
\[ g \pi = 
\begin{pmatrix}
\frac{1+i}{2} & \frac{1-i}{2} & 0 & 0 & 0 \\
0 & 0 & 0 & \frac{1-i}{2} & \frac{1+i}{2} \\
0 & 0 & 1 & 0 & 0 \\
0 & 0 & 0 & \frac{1+i}{2} & \frac{1-i}{2} \\
\frac{1-i}{2} & \frac{1+i}{2} & 0 & 0 & 0
\end{pmatrix}. \]

Applying this matrix to the flag $F_{\bullet}$ above, and multiplying all coefficients by $2$ just to clean things up, we get the flag
\[ F_{\bullet}' = \left\langle 
\begin{pmatrix}
1+i \\
-(1-i) \\
0 \\
1+i \\
1-i
\end{pmatrix},
\begin{pmatrix}
0 \\
1+i \\
2 \\
-(1-i) \\
0
\end{pmatrix},
\begin{pmatrix}
1-i \\
0 \\
0 \\
0 \\
1+i
\end{pmatrix},
\begin{pmatrix}
0 \\
-(1+i) \\
2 \\
1-i \\
0
\end{pmatrix}
\begin{pmatrix}
1+i \\
1-i \\
0 \\
-(1+i) \\
1-i
\end{pmatrix}
\right\rangle. \]
One checks that this flag is isotropic with respect to the anti-diagonal form, and by construction it lies in the $K'$-orbit on $X'$ corresponding to the clan $(1,-,+,-,1)$.

\subsection{$(SO(2n,\C), S(O(2p,\C) \times O(2q,\C)))$}
We must show that the $K'=GL(2p,\C) \times GL(2q,\C)$-orbit $Q_{\gamma}$ corresponding to the $(2p,2q)$-clan $\gamma$ meets $X$ if and only if $\gamma$ is symmetric.  ($K'$ should be realized in this case as the fixed points of the involution $\text{int}(I_{p,2q,p})$ on $GL(2n,\C)$.)

The proof is virtually the same as the one given in the type $B$ case above.  There is only one issue which bears mentioning, and that is that it isn't clear that the representative produced by the procedure described above should necessarily produce a flag lying in the correct component of the variety of isotropic flags.  Indeed, it may not.  However, $K'$ does contain elements of the determinant $-1$ component of $O(2n,\C)$.  For instance, it contains the permutation matrix corresponding to the transposition $(n,n+1)$.  Thus if the representative obtained by the procedure described above for the type $B$ case does not live in the correct component, it can then be moved to the other component by the action of such an element of $K'$.  This simply says that every $K'$-orbit on $X'$ which intersects the variety of isotropic flags intersects \textit{both} components of it, so in particular it intersects $X$.

\subsection{$(SO(2n,\C), GL(n,\C))$}\label{ssec:appendix-type-d-ex-2}
Here, the goal is to prove that the $K' = GL(n,\C) \times GL(n,\C)$-orbit $Q_{\gamma}$ corresponding to the $(n,n)$-clan $\gamma = (c_1,\hdots,c_{2n})$ meets $X$ if and only if 
\begin{enumerate}
	\item $\gamma$ is skew-symmetric;
	\item $c_i \neq c_{2n+1-i}$ for any $c_i \in \N$; and
	\item Among $c_1,\hdots,c_n$, the total number of $-$ signs and pairs of equal natural numbers is even.
\end{enumerate}
In this case, the group $K'$ should be realized as the fixed points of the involution $\text{int}(I_{n,n})$.

The proof that $\gamma$ should be skew-symmetric is word-for-word the same as that given in \cite{Yamamoto-97} for the case $(Sp(2n,\C),GL(n,\C))$ (\cite[Proposition 3.2.2]{Yamamoto-97}).  The proof that we cannot have $(c_i,c_{2n+1-i}) = (a,a)$ for $a \in \N$ is nearly identical to the corresponding proof in the case of $(Sp(2n,\C),Sp(2p,\C) \times Sp(2q,\C))$.  (See part (3) of the proof of \cite[Proposition 4.3.2]{Yamamoto-97}.)  And an isotropic representative of $Q_{\gamma}$ can be produced by the same method described in \cite[Theorem 3.2.11]{Yamamoto-97}.

As in the previous example, though, we must consider the question of whether the representative so obtained lies in our chosen component of the variety of isotropic flags.  This explains the need for the parity condition (3).  Whereas in the previous example, any $K'$-orbit on $X'$ meeting the variety of isotropic flags met both components of it, here any $K'$-orbit on $X'$ containing an isotropic flag meets one component or the other, but not both.  This follows from the fact that when $K'$ is realized as the fixed points of $\text{int}(I_{n,n})$, any element of $K' \cap O(2n,\C)$ has determinant $1$, which is an easy computation.  Thus one cannot pass from one component to the other by the action of an element of $K'$.  Condition (3), then, is what guarantees that an isotropic representative of $Q_{\gamma}$ lives in $X$, rather than in the opposite component of the variety of isotropic flags.

\subsection{$(SO(2n,\C), S(O(2p+1,\C) \times O(2q-1,\C)))$}
For the last case, we must prove that a $K' = GL(2p+1,\C) \times GL(2q-1,\C)$-orbit $Q_{\gamma}$ corresponding to the $(2p+1,2q-1)$-clan $\gamma$ meets $X$ if and only if $\gamma$ is symmetric.  The proof here is exactly the same as that for the pairs $(SO(2n+1,\C),S(O(2p,\C) \times O(2q+1,\C)))$ and $(SO(2n,\C), S(O(2p,\C) \times O(2q,\C)))$.  Note, though, that once we find a representative for $Q_{\gamma}$ which is isotropic with respect to the diagonal form, we are done, since in Subsection \ref{ssec:type_d_ex_3}, we chose to realize $SO(2n,\C)$ as the isometry group of the diagonal form.  Thus there is no need to conjugate over to another realization of $SO(2n,\C)$.

As with the pair $(SO(2n,\C), S(O(2p,\C) \times O(2q,\C)))$, we remark that if the isotropic flag produced by the method described in the type $B$ case does not lie in our chosen component of the variety of isotropic flags, it can be moved by $K'$ to the correct component, since $K'$ does contain determinant $-1$ elements of $O(2n,\C)$.  This says once again that every $K'$-orbit on $X'$ which intersects the variety of isotropic flags intersects both components of it.

\section{Proof of Theorem \ref{thm:k-orbit-intersections}}
Having established the combinatorial conditions on clans which amount to $K'$-orbits on $X'$ meeting $X$, in each example we now have a surjective map
\[ K \backslash X \rightarrow \{\text{Clans satisfying some combinatorial conditions}\}. \]
Indeed, if $\gamma$ is a clan in the target, with associated $K'$-orbit $Q_{\gamma}$, the fiber over $\gamma$ is the collection of $K$-orbits whose union is the nonempty, $K$-stable subset $Q_{\gamma} \cap X$.  However, we want this map to be a bijection.  This amounts to the fact that $Q_{\gamma} \cap X$ is in fact a \textit{single} $K$-orbit, and not a union of multiple $K$-orbits.

As alluded to in Subsection \ref{ssec:clans}, one way to establish this involves a fairly intricate counting argument, the principles of which were explained to the author by Peter Trapa.  Before making this argument in each of our examples, we describe the general setup.

\subsection{The one-sided parameter space $\caX$}\label{ssec:one_sided_param}
For any complex reductive algebraic group $G$, consider the exact sequence
\[ 1 \rightarrow \text{Int}(G) \rightarrow \text{Aut}(G) \rightarrow \text{Out}(G) \rightarrow 1, \]
where $\text{Int}(G) = G/Z(G)$ is the group of inner automorphisms, and $\text{Out}(G)$ is the quotient.  Two automorphisms $f_1,f_2$ are said to be in the same \textit{inner class} if they have the same image in $\text{Out}(G)$.  By an \textit{inner class of involutions}, we shall mean the set of all involutions in a given inner class.

It is a fact (see \cite{Adams-DuCloux-09}) that for any inner class of involutions, there exists a ``distinguished" representative which fixes a chosen pinning of the group $G$.  (A pinning is the data of a maximal torus $T$, a Borel subgroup $B$ containing $T$, and a choice of positive root vectors $X_{\alpha}$ corresponding to the positive system defined by $B$.)  This distinguished involution is the image of the chosen outer automorphism $\gamma$ under a canonical splitting of the exact sequence above.

Associated to the inner class of any involution $\theta$ is the so-called \textit{one-sided parameter space}, which we denote by $\caX$.  For our purposes, we define it as follows:  Let $\theta_1,\hdots,\theta_n$ be the inner class of involutions containing $\theta$, and let $K_i = G^{\theta_i}$.  Then we define
\[ \caX := \coprod_{i=1}^n K_i \backslash G/B. \]

The set $\caX$ plays a prominent role in an algorithm (implemented in the software known as ATLAS) which computes (among other things) the space of admissible representations of a given real reductive group $G_{\R}$.  See \cite{Adams-DuCloux-09,Adams-08} for details.  We remark that in those references, a different definition of $\caX$ is given, after which it is established as a theorem that $\caX$ is in bijection with the set above.  However, because we are only concerned here with counting $K$-orbits, and not with the deeper representation-theoretic significance of the set $\caX$, it is more convenient for our purposes to simply take this as our definition.

The key feature of $\caX$ from our point of view is that it comes equipped with a map to the set $\mathcal{I}$ of twisted involutions.  In fact, this map is simply the Richardson-Springer map $\phi$ (cf. Subsection \ref{ssec:closed_orbits}) ``spread out" to the $K_i=G^{\theta_i}$-orbits as $\theta_i$ runs over an entire inner class of involutions.  While the map $\phi$ is \textit{not} surjective in general when we restrict attention to one $K$ at a time, the map from $\caX$ that we get when considering at once all $K$ associated to a given inner class of involutions \textit{is} surjective.  Moreover, for any given $\tau \in \mathcal{I}$, the cardinality of the fiber over $\tau$ (which we denote $\caX_{\tau}$) is explicitly computable.  Indeed, letting $T \subseteq G$ be our fixed $\theta$-stable maximal torus, define

\[ T_{\tau} := \{t \in T \ \vert \ t \tau(t) \in Z(G) \}, \]

and

\[ T^{-\tau} := \{t \in T \ \vert \ t \tau(t) = 1 \}. \]

In the above definitions, the action of $\tau$ on $T$ is twisted by the distinguished involution $\theta$.  That is, 
\[ \tau(t) = \tau . \theta(t), \]
where $\tau.t$ denotes the usual action of $W$ on $T$.  In all cases of interest to us save one (the non-equal rank case in type $D$), the distinguished involution is simply the identity, and so the action of $\tau$ is the usual one.  At any rate, with these definitions given, we have the following result:

\begin{prop}
With notation as above, 
\[ |\caX_{\tau}| = |T_{\tau} / T_0^{-\tau}|. \]
\end{prop}
For a proof, see \cite[Proposition 11.2 and Remark 11.5]{Adams-DuCloux-09} or \cite[Proposition 2.4]{DuCloux-05}.

In our examples, this result allows us to compute the cardinality of $\caX$, and then compare it to the total number of clans which correspond to $K_i'$-orbits intersecting $X$, where $K_i = G \cap K_i'$, the groups $K_i$ are the fixed point subgroups of an entire inner class of involutions, and each $K_i'$ is isomorphic to an appropriate $GL(p,\C) \times GL(q,\C)$.  Since we have already established combinatorial descriptions of such clans in the previous subsection, the latter number is computable.

If these two counts turn out to be equal, then for each $K_i$, it is impossible for the intersection of any $K_i'$-orbit on $X'$ with $X$ to split as a union of multiple $K_i$-orbits --- if it did, the cardinality of $\caX$ would necessarily be greater than the clan count.  Making this counting argument thus establishes in one fell swoop that for \textit{any} $K_i$ in the inner class, each $K_i$-orbit is precisely the intersection of a $K_i'$-orbit with $X$.

Before proceeding to make this argument explicit in each of our examples, we offer some basic comments regarding the computation of inner classes of involutions, and the corresponding family of symmetric subgroups.  An excellent reference for these facts is \cite{Adams-09}.  First, it is a fact \cite[Lemma 4.9]{Adams-09} that for a semisimple algebraic group $G$, the group $\text{Out}(G)$ is a subgroup of the automorphism group of its Dynkin diagram.  These two groups are equal if $G$ is simply connected or adjoint.  This says already that $\text{Out}(G) = \{1\}$ if $G$ is of type $B$ or $C$.  Thus for those two groups, there is only one inner class of involutions, namely the inner involutions
\[ \{\text{int}(g) \mid g^2 \in Z(G) \}. \]
The inner class of involutions of this form is referred to as the ``compact" inner class, so named because it contains the identity involution, which corresponds to the compact real form of the complex group.  Describing the compact inner class for the groups $SO(2n+1,\C)$, $Sp(2n,\C)$, and $SO(2n,\C)$ is an elementary matrix computation in each case; the answers can all be found in \cite{Adams-09}.  As we have already noted, in types $B$ and $C$, the compact inner class is the \textit{only} inner class that there is to consider.  In type $D$, there is the compact inner class, along with one additional inner class, called the ``unequal rank" inner class.  (Note that the automorphism group of the Dynkin diagram of type $D_n$ is $\Z/2\Z$ if $n \geq 5$, and $S_3$ if $n=4$.  Thus one should expect for there to be more than one inner class of involutions.)  This inner class consists of all involutions whose fixed groups are of the form $S(O(2p+1,\C) \times O(2q-1,\C))$, as $p,q$ range over all possibilities with $p \geq 0$, $q \geq 1$, and $p+q=n$.  See \cite[Exercise 4.12]{Adams-09}.

It is noteworthy that the family of symmetric subgroups associated to an inner class of involutions consists of \textit{conjugacy} classes of symmetric subgroups, as opposed to \textit{isomorphism} classes.  For instance, it is easy to see that in type $A$, the compact inner class contains all inner involutions whose fixed groups are of the form $GL(p,\C) \times GL(q,\C)$ as $p,q \geq 0$ run over all possibilities with $p+q=n$.  From this perspective, the symmetric subgroups $GL(2,\C) \times GL(3,\C)$ and $GL(3,\C) \times GL(2,\C)$ (for example) are considered different; these groups are isomorphic, but they are \textit{not} conjugate.

We now proceed to make the counting argument described above in each of our examples.

\subsection{Type $B$, compact inner class}
As alluded to above, since $\text{Out}(SO(2n+1,\C)) = \{1\}$, there is only one inner class of involutions in type $B$, which consists of all inner involutions.  The symmetric subgroups $K_i$ corresponding to this inner class are $S(O(2p,\C) \times O(2q+1,\C))$ as $p,q \geq 0$ run over \textit{all} possibilities with $p+q=n$.  The distinguished representative of this inner class is the identity, so the the twisted involutions in this case are honest involutions in $W$, and action of a (twisted) involution on $T$ is the usual one, induced by the $W$-action on $\widehat{T}$.

Now, note that $Z(SO(2n+1,\C))$ is trivial, so in the notation of the previous subsection, $T_{\tau} = T^{-\tau}$.  Thus $T_{\tau}/T^{-\tau}_0$ is simply the component group of $T_{\tau}$.

We wish to prove, then, that the number of symmetric $(2p,2q+1)$-clans (as $p,q \geq 0$ run over all possibilities with $p+q=n$) is equal to the cardinality of the one-sided parameter space $\caX$.  To see this, first recall that we have a bijection
\[ \mathcal{I} := \text{Involutions in $W$} \longleftrightarrow \]
\[ \mathcal{J} := \text{Involutions $\sigma \in S_{2n+1}$ such that $\sigma(2n+2-i) = 2n+2-\sigma(i)$} \]
via the embedding of $W$ into $S_{2n+1}$ as signed elements, cf. Subsection \ref{ssec:notation}.  Moreover, each symmetric $(2p,2q+1)$-clan is clearly associated to precisely one element of $\mathcal{J}$ in a natural way:  The positions of matching numbers in such a clan give transpositions, while the positions of signs give fixed points.  (The symmetry property of such a clan is precisely what guarantees that the associated involution is a signed element of $S_{2n+1}$.)  So, for example, if $p = q = 1$, the symmetric $(2,3)$-clan $(1,-,+,-,1)$ gives the involution $(1,5)$, while the clan $(1,2,-,1,2)$ gives the involution $(1,4)(2,5)$.

With this noted, we must simply see that for each element $\sigma$ of $\mathcal{J}$, the total number of symmetric $(2p,2q+1)$-clans ($p,q$ running over all possibilities) corresponding to $\sigma$ is equal to the cardinality of the fiber $\caX_{\sigma'}$ of $\caX$ over $\sigma'$, the element of $\mathcal{I}$ which corresponds to $\sigma$.

So let $\sigma \in \mathcal{J}$ be given.  Any $(2p,2q+1)$-clan associated to $\sigma$ has the positions of its matching natural numbers prescribed by the transpositions of $\sigma$.  Thus the only choice one has in constructing such a clan is in assigning $+$ and $-$ signs to the fixed points of $\sigma$.  It is clear that for a given $n$, once the cycle structure of $\sigma$ is decided upon, the middle sign (that in position $n+1$) for any clan associated to $\sigma$ is completely determined.  Further, since we are restricting attention to those clans which are symmetric, one only has free choice of signs occupying the fixed points of $\sigma$ in positions up to $n$, which (by symmetry) determine the signs assigned to the fixed points of $\sigma$ in positions beyond $n+2$.  This says that if
\[ k = \#\{i \in \{1,\hdots,n\} \ \vert \ \sigma(i) = i\}, \]
then the number of symmetric clans associated to $\sigma$ is $2^k$.

Now, consider $|\caX_{\sigma'}|$.  As noted above, this is equal to the number of components of 
\[ T_{\sigma'} = \{t \in T \ \vert \ t\sigma'(t) = 1\}. \]
Given $t = \text{diag}(a_1,\hdots,a_n,1,a_n^{-1},\hdots,a_1^{-1})$, one has that
\[ t\sigma'(t) = \text{diag}(a_1a_{\sigma(1)},\hdots,a_na_{\sigma(n)},1,(a_na_{\sigma(n)})^{-1},\hdots,(a_1a_{\sigma(1)})^{-1}), \]
where $\sigma$ is the permutation in $S_{2n+1}$ associated to $\sigma'$.  This is equal to the identity matrix precisely when, for any $i = 1,\hdots,n$, 
\[
a_i = 
\begin{cases}
	\pm 1 & \text{ if $\sigma(i) = i$}, \\
	a_{\sigma(i)}^{-1} & \text{ otherwise.}
\end{cases}
\]

Thus $T_{\sigma'}$ is a product of $\C^*$'s and $\Z/2\Z$'s, with one $\Z/2\Z$ for each $i \in \{1,\hdots,n\}$ such that $\sigma(i) = i$.  As such, it has $2^k$ components.

This completes the proof that $|\caX|$ is equal to the total number of symmetric $(2p,2q+1)$-clans as $p,q \geq 0$ run over all possibilities with $p+q=n$.  Thus we have established that it is impossible for the intersection of a $K'$-orbit with $X$ to be anything more than a single $K$-orbit on $X$.

\subsection{Type $C$, compact inner class}
Once again, we have only one inner class of involutions, consisting solely of inner involutions.  The corresponding symmetric subgroups are those of the form $Sp(2p,\C) \times Sp(2q,\C)$ ($p,q \geq 0$ running over all possibilities with $p+q=n$), along with the one additional group $GL(n,\C)$.
Thus here, we must compare $|\caX|$ to the total number of symmetric $(2p,2q)$-clans ($p,q$ running over all possibilities) satisfying the additional  ``anti-reflexive" condition of Subsection \ref{ssec:sp2p2q-clans}, \textit{plus} the number of skew-symmetric $(n,n)$-clans.

As before, we perform the count of $|\caX|$ fiber-by-fiber.  As in Type $B$, the twisted involutions $\mathcal{I}$ amount to honest involutions in $W$.  Using the embedding of $W$ into $S_{2n}$ as signed elements, these are once again in bijection with the set of all involutions in $S_{2n}$ which are signed elements.

Let $\tau \in S_{2n}$ be such an involution.  We first compute the total number of clans naturally associated to $\tau$.  Then we compute $|\caX_{\tau}|$.  There are two cases:  Either $\tau$ switches $i$ and $2n+1-i$ for some $i$, or not.

In the former case, there will be no $(2p,2q)$-clans corresponding to $Sp(2p,\C) \times Sp(2q,\C)$-orbits associated to $\tau$.  This is because any clan associated to such a $\tau$ cannot have the anti-reflexive property.  Thus we need only count the skew-symmetric $(n,n)$-clans.  The positions of matching natural numbers in a clan corresponding to $\tau$ being entirely determined by $\tau$, our only freedom is in assigning $\pm$ signs to the fixed points of $\tau$.  This amounts to assigning any configuration of $\pm$ signs to those $i \in \{1,\hdots,n\}$ fixed by $\tau$.  The signs on the fixed points of $\tau$ from $\{n+1,\hdots,2n\}$ are then determined by skew-symmetry.  Thus if $k$ is the number of elements of $\{1,\hdots,n\}$ fixed by $\tau$, then the number of clans associated to $\tau$ is $2^k$.

On the other hand, if $\tau$ does not switch $i$ and $2n+1-i$ for any $i$, then there \textit{will} be $(2p,2q)$-clans corresponding to $Sp(2p,\C) \times Sp(2q,\C)$-orbits associated to $\tau$, for such a clan would have the anti-reflexive property.  The total number of such clans is $2^k$ ($k$ as in the previous paragraph), since we assign $\pm$ signs to the fixed points of $\tau$ from $\{1,\hdots,n\}$, while the remaining signs are determined by symmetry.  Counting the skew-symmetric $(n,n)$-clans associated to $\tau$, we again get $2^k$, by the same argument as before.  Thus in this case, the total number of clans associated to $\tau$ is $2\cdot 2^k=2^{k+1}$.

Now, we must see that $|\mathcal{X}_{\tau}| = 2^k$ or $2^{k+1}$, depending on which of the two cases we are in.  Recall that 
\[ |\mathcal{X}_{\tau}| = |T_{\tau} / T^{-\tau}_0|, \]
where 
\[ T_{\tau} = \{ t \in T \ \vert \ t \tau(t) \in Z(G) \}, \]
and
\[ T^{-\tau} = \{t \in T \ \vert \ t \tau(t) = 1\}. \]
In type $B$, $Z(G)$ was trivial, so this amounted simply to counting the components of $T_{\tau}$.  However, here, $Z(G) = \{ \pm 1\}$, so this computation is a bit different.  First, suppose that $\tau(i) \neq 2n+1-i$ for any $i$.  Then if $t = \text{diag}(a_1,\hdots,a_n,a_{n+1},\hdots,a_{2n})$ (with $a_{2n+1-i} = a_i^{-1}$), we have that
\[ 
	t \tau(t) = \text{diag}(a_1 a_{\tau'(1)},\hdots,a_n a_{\tau'(n)},a_{n+1} a_{\tau'(n+1)},\hdots,a_{2n} a_{\tau'(2n)}).
\]
For each $i = 1,\hdots,n$ such that $\tau'(i) = i$, this gives $a_i^2$ in the $i$th diagonal position (and $a_i^{-2}$ in the $(2n+1-i)$th diagonal position).  The other positions simply have $a_i a_{\tau'(i)}$.

Now, if this is to be equal to $1$, then we have a choice of $\pm 1$ for each position $i=1,\hdots,n$ such that $\tau'(i) = i$.  (The number in position $2n+1-i$ is then determined.)  We also have a completely free choice of the other $a_i$, which then determine the $a_{\tau'(i)}$.

On the other hand, if $t \tau(t)$ were to equal $-1$, then we would have a choice of $\pm \sqrt{-1}$ for each $a_i$ such that $\tau'(i) = i$.  (Again, this choice determines the value of $a_{2n+1-i}$).  We again have a completely free choice of the other $a_i$, which in turn determine the $a_{\tau'(i)}$.

Thus $T_{\tau}$ can be thought of like this:  First, decide whether $t \tau(t)$ is going to be $1$ or $-1$.  Having decided that, choose the $a_i$ on the fixed points $i$ to be either $\pm 1$ or $\pm \sqrt{-1}$ (depending on the first choice that was made), and choose arbitrary values in $\C^*$ for the remaining $a_i$.  This shows that  
\[ T_{\tau} \cong \Z/2\Z \times ((\Z/2\Z)^k \times (\C^*)^{n-k}) , \]
where $k$ represents the number of values fixed by $\tau$ (or the number of $i = 1,\hdots,n$ fixed by $\tau'$) as above.

Now, we turn to $T^{-\tau}_0$.  Here, $t \tau(t)$ takes the same form, but now we insist that $t \tau(t) = 1$.  Thus $T^{-\tau}$ (by the same analysis of the previous paragraph) is of the form
\[ (\Z/2\Z)^k \times (\C^*)^{n-k}. \]
However, in considering only the identity component, we get rid of the factors of $\Z/2\Z$.  The upshot is that $T^{-\tau}_0$ is made up of just the $\C^*$ factors of $T_{\tau}$:
\[ T^{-\tau}_0 = (C^*)^{n-k}. \]

Then it is clear that the order of the quotient $T_{\tau}/T^{-\tau}_0$ is $2^{k+1}$, which is the same as the number of total clans associated to $\tau$, calculated above.

Now, suppose that there \textit{is} some $i$ for which $\tau(i) = 2n+1-i$.  Consider $T_{\tau}$.  For any $t = \text{diag}(a_1,\hdots,a_n,a_n^{-1},\hdots,a_1^{-1})$, we now have that the value in the $i$th position of $t \tau(t)$ is a $1$.  This removes the option of $t \tau(t)$ being $-1$, so in this case we see that
\[ T_{\tau} \cong (\Z/2\Z)^k \times (\C^*)^{n-k}, \]
by the same analysis given for the previous case.  And $T^{-\tau}_0$ is $(\C^*)^{n-k}$, as it was in the previous case, so here we see that the quotient $|T_{\tau} / T^{-\tau}_0|$ has order $2^k$, which again matches the number of clans associated to $\tau$.  This completes the argument.

\subsection{Type $D$, compact inner class}
In type $D$, $\text{Out}(G)$ has two elements, so there are actually two inner classes of involutions to consider.  The first is the compact inner class, consisting of inner involutions, and is similar to what we saw in type $C$.  The corresponding symmetric subgroups are of the form $K \cong S(O(2p,\C) \times O(2q,\C))$ ($p,q \geq 0$ running over all possibilities with $p+q=n$), plus \textit{two} symmetric subgroups isomorphic to $GL(n,\C)$.  One of these is the group considered in Subsection \ref{ssec:type_d_ex_2}, the fixed points of $\theta = \text{int}(i \cdot I_{n,n})$.  The other, which we will call $K^-$, is the fixed point subgroup of the involution 
\[ \theta^- = \text{int}(\text{diag}(\underbrace{i,\hdots,i}_{n-1},-i,i,\underbrace{-i,\hdots,-i}_{n-1}). \]
The group $K^-$ is also isomorphic to $GL(n,\C)$, but is \textit{not} conjugate to $K$.  Thus its orbits must be considered separately to get an accurate count of $\caX$.  Its orbits on $X$ are in bijection with those of $K$, however, so it suffices to simply double the count of clans associated to the $K$-orbits to account for the extra $K^-$-orbits comprising $\caX$.

With all of that said, the argument here is very similar to the type $C$ case.  The twisted involutions $\caI$ are once again simply honest involutions of $W$.  These can be thought of either as signed permutations of $\{1,\hdots,n\}$ changing an even number of signs, or as signed elements of $S_{2n}$ with an even number of $i \leq n$ sent to some $j > n$.  Each symmetric $(2p,2q)$-clan ($p+q=n$) and each skew-symmetric $(n,n)$-clan (satisfying the further conditions described in Subsection \ref{ssec:appendix-type-d-ex-2}) is naturally associated to such a signed element of $S_{2n}$, with the positions of matching natural numbers defining $2$-cycles, and with the signs ascribed to fixed points.  Letting $\tau \in S_{2n}$ be such an element, there two possible cases:  Either $\tau$ interchanges $i$ and $2n+1-i$ for some $i$, or not.  In the former case, there are \textit{no} skew-symmetric $(n,n)$-clans satisfying the anti-reflexive condition.  If $k$ is the number of fixed points of $\tau$ among $\{1,\hdots,n\}$, then there are $2^k$ associated symmetric $(2p,2q)$-clans, each obtained by assigning $\pm$ signs to those fixed points (the signs on the remaining fixed points being determined by symmetry).    In the latter case, again letting $k$ be the number of fixed points of $\tau$ among $\{1,\hdots,n\}$, there are again $2^k$ associated symmetric $(2p,2q)$-clans.  There are only $2^{k-1}$ skew-symmetric $(n,n)$-clans satisfying the conditions of Subsection \ref{ssec:appendix-type-d-ex-2}, due to the extra parity condition.  However, as indicated above, this number should be doubled to also account for the $K^-$-orbits, giving a total of $2^k$.  Combining the $2^k$ symmetric clans with the $2^k$ skew-symmetric clans, we have a total of $2^{k+1}$ associated clans.

To conclude, then, we need to see that if $\tau \in S_{2n}$ is an involution of the former type, then $|\caX_{\tau}| = 2^k$, and if it is an involution of the latter type, then $|\caX_{\tau}| = 2^{k+1}$.  That argument is identical to that already given in type $C$, so we do not repeat it here.

\subsection{Type $D$, unequal rank inner class}
The symmetric subgroups corresponding to the other inner class of involutions in type $D$ are those of the form $S(O(2p+1,\C) \times O(2q-1,\C))$, $p,q$ running over all possibilities with $p \geq 0$, $q \geq 1$, and $p+q=n$.

When dealing with this class of symmetric subgroups in Section \ref{ssec:type_d_ex_3}, we chose a realization of $SO(2n,\C)$ which had no diagonal elements.  However, it is simpler notationally to make the relevant counting argument when the torus $T$ consists of diagonal matrices.  For this reason, here we will assume that $SO(2n,\C)$ is the group of linear automorphisms of $\C^{2n}$ preserving the anti-diagonal form.  Take $T$ to be the diagonal elements of this group, which are of the form
\[ t = \text{diag}(a_1,\hdots,a_n,a_n^{-1},\hdots,a_1^{-1}). \]

This is the first (and only) case we consider where the distinguished representative $\theta$ of our inner class of involutions is \textit{not} the identity.  Indeed, the distinguished involution representing this inner class is $\theta(g) = I_{2n-1,1}gI_{2n-1,1}$.  As we have seen in type $B$, to pass from the usual realization of $G$ to the one we now wish to consider, we can conjugate by the element $\pi$, where
\[ \pi_{i,j} = 
\begin{cases}
	\frac{1+i}{2} & \text{ if } i=j \\
	\frac{1-i}{2} & \text{ if } i = 2n+1-j \\
	0 & \text{ otherwise.}
\end{cases}
\]

When we perform this conjugation, the distinguished involution $\theta$ then becomes
\[ g \mapsto \tilde{I}_{2n-1,1} g \tilde{I}_{2n-1,1}, \]
where
\[ \tilde{I}_{2n-1,1} = \pi I_{2n-1,1} \pi^{-1} = 
\begin{pmatrix}
	0 & \multicolumn{3}{c}{\hdots} & 4i \\
	\multirow{3}{*}{\vdots} & 1 &&& \multirow{3}{*}{\vdots} \\
	   && \ddots \\ 
	   &&& 1 \\
	4i & \multicolumn{3}{c}{\hdots} & 0
\end{pmatrix}
.\]

Since elements of $W$ are represented (for this choice of realization of $G$, and for this choice of maximal torus $T$) by monomial matrices corresponding to signed permutations changing an even number of signs, the map on $W$ induced by $\theta$ is visibly
\[ w \mapsto \nu w \nu, \]
where $\nu=(1,-1)$ is the signed permutation which interchanges $1$ and $-1$.  The set $\mathcal{I}$ of twisted involutions, then, is no longer the set of honest involutions, but rather is the set of elements of $w$ sent to their inverses by this map.  Note, though, that elements $w \in W$ having the property that 
\[ \nu w \nu = w^{-1} \]
correspond 1-to-1 with signed permutations which are (honest) involutions and which change an \textit{odd} number of signs.  Indeed, given such a signed permutation $\sigma$, let $w = \nu \sigma$.  Then $w$ is a signed permutation which changes an even number of signs, and 
\[ (\nu w \nu) w= (\nu w)(\nu w) = \sigma^2 = 1, \]
so $w$ is a twisted involution.  On the other hand, given a twisted involution $w \in W$, the permutation $\sigma = w \nu$ changes an odd number of signs, and is an involution, since 
\[ \sigma^2 = (w \nu)(w \nu) = w(\nu w \nu) = 1. \]

Now, let us note that each symmetric $(2p+1,2q-1)$-clan corresponds in a natural way to a unique signed involution $\sigma$ changing an odd number of signs.  Indeed, the correspondence here is the same as it has been all along:  The positions of matching natural numbers in such a clan determine the transpositions of such a permutation, while the signs determine the fixed points.  The only observation that needs to be made is that the involution corresponding to a symmetric $(2p+1,2q-1)$-clan in this way necessarily changes an \textit{odd} number of signs.  The number of sign changes of the signed involution corresponding to a symmetric $(2p+1,2q-1)$-clan $\gamma$ is clearly the number of natural numbers occurring in positions $1,\hdots,n$ of $\gamma$ whose mate occurs at or beyond position $n+1$.  The claim, then, is that this number must be odd.

Suppose first that $n$ is even.  Then $p-q$ is even as well.  Suppose by contradiction that there are an even number of natural numbers occurring in positions $1,\hdots,n$ whose mates occur at or beyond position $n+1$.  This means that among the first $n$ symbols of the clan, the combined number of $+$ and $-$ signs must be even.  Thus the total number of $+$ and $-$ signs comprising the entire clan $\gamma$ must divisible by $4$.  Let $a$ be the number of $+$ signs, and let $b$ be the number of $-$ signs.  Then we have that 
\[ a + b \equiv 0 \pmod{4}, \]
while
\[ a - b = 2p+1-(2q-1) = 2(p-q) + 2 \equiv 2 \pmod{4}. \]
But this implies that $a$ is odd.  However, $a$ cannot be odd if $\gamma$ is symmetric, and so we have a contradiction.

If $n$ is odd, a similar contradiction can be reached.  Here, $p-q$ is odd, and using similar reasoning we obtain the congruences
\[ a + b \equiv 2 \pmod{4} \]
and
\[ a - b \equiv 0 \pmod{4}, \]
again a contradiction.

The upshot is that here the cardinality of the fiber $\mathcal{X_{\tau}}$ ($\tau$ a twisted involution) can be compared to the number of symmetric $(2p+1,2q-1)$-clans corresponding to the signed involution $\tau \nu$.

So let $\tau \in W$ be any twisted involution.  Recall from the description of Subsection \ref{ssec:clans} that in the definitions of $T_{\tau}$ and $T^{-\tau}$, the action of $\tau$ on $T$ is twisted by the map $\theta$, so that
\[ \tau(t) = \tau.\theta(t),\]
where $\tau.t$ indicates the usual action of $\tau$ on $T$ (by permutation of the diagonal entries).  Given a diagonal element $t = \text{diag}(a_1,\hdots,a_n,a_n^{-1},\hdots,a_1^{-1})$, one checks easily that
\[ \theta(t) = \text{diag}(a_1^{-1},a_2,\hdots,a_n,a_n^{-1},\hdots,a_2^{-1},a_1). \]
This amounts to applying the signed permutation $\nu$ to the coordinates of $t$.  In light of this, it is clear that $\tau(t)$ amounts to acting on $t$ by the signed involution $\tau \nu$ corresponding to $\tau$.

Next, let us note that any signed involution $\sigma$ which changes an odd number of signs must send some $j$ to its negative.  Indeed, suppose that $j > 0$, and that $\sigma(j) < 0$ but that $\sigma(j) \neq -j$.  Then $\sigma(j) = -i$ for some $i > 0$, and so $\sigma(i) = \sigma(-(-i)) = -\sigma(-i) = -j$.  This says that $\sigma$ can have only an even number of sign changes of the form $\sigma(j) \neq -j$, and so if $\sigma$ has an odd number of sign changes, at least one (and indeed, an odd number) must be of the form $\sigma(j) = -j$.

The point of this observation is that when $t = \text{diag}(a_1,\hdots,a_n,a_n^{-1},\hdots,a_1^{-1}) \in T$, and when $\tau$ is a twisted involution, the action of $\tau$ on $t$ must interchange at least one $a_j$ with $a_j^{-1}$.  That is, there must be at least one $j$ between $1$ and $n$ such that the $j$th and $(2n+1-j)$th entries of $\tau(t)$ are $a_j^{-1}$ and $a_j$, respectively.  This being the case, we must have that $t \tau(t)$ has $1$'s in these positions.  The upshot is that while $Z(G) = \pm 1$, it can never be the case that $t \tau(t) = -1$.  Thus
\[ T_{\tau} = \{t \in T \ \vert \ t \tau(t) \in Z(G) \} = \{t \in T \ \vert \ t \tau(t) = 1 \} = T^{-\tau}. \]
So in computing the cardinality of the fiber $\mathcal{X}_{\tau}$ over $\tau$, we simply have to count the components of $T_{\tau}$.

So consider $t' = t \tau(t)$.  For each fixed point $j$ of the signed involution $\sigma = \tau \nu$, $t'$ has $a_j^2$ in the $j$th position (and $a_j^{-2}$ in the $(2n+1-j)$th).  For each $j$ such that $\sigma(j) = -j$, $t'$ has a $1$ in the $j$th and $(2n+1-j)$th positions.  And for any other $j$, $t'$ simply has $a_i a_l^{\pm 1}$ in the $i$th position for some $l$ (and $a_i^{-1} a_l^{\mp 1}$ in the $(2n+1-i)$th position).  Based on this, if we let $k$ be the number of fixed points of $\sigma$, it is clear that
\[ T_{\tau} \cong (\Z/2\Z)^k \times (\C^*)^a \]
for some $a$.  (It is easy to check that $a = \frac{1}{2}(n+d - k)$, where $d$ is the number of $j$ such that $\sigma(j) = -j$.)  Thus $|\mathcal{X}_{\tau}| = 2^k$.  On the other hand, the number of symmetric $(2p+1,2q-1)$-clans associated to the involution $\sigma$ is also $2^k$, since we have free choice of assigning $\pm$ signs to the fixed points of $\sigma$ between $1$ and $n$, while the signs in positions $n+1$ through $2n$ are then determined by symmetry.  This completes the argument.

%% file: Appendix-B.tex
\chapter{Figures and Tables}
\begin{figure}[h!]
	\caption{$(GL(4,\C),GL(2,\C) \times GL(2,\C))$}\label{fig:type-a-2-2}
	\centering
	\includegraphics[scale=0.5]{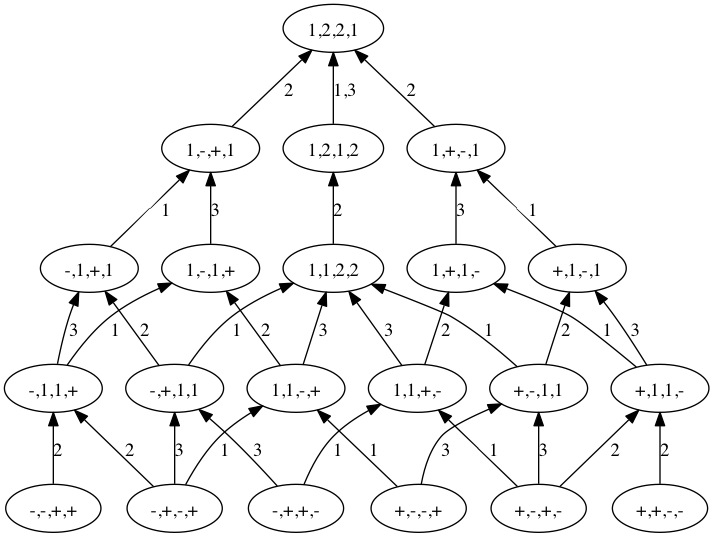}
\end{figure}

\begin{figure}[h!]
	\caption{$(SL(3,\C),SO(3,\C))$}\label{fig:type-a-orthogonal-1}
	\centering
	\includegraphics[scale=0.5]{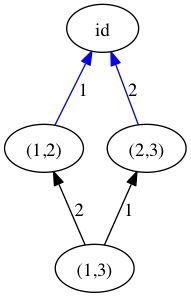}
\end{figure}

\begin{figure}[h!]
	\caption{$(SL(5,\C),SO(5,\C))$}\label{fig:type-a-orthogonal-2}
	\centering
	\includegraphics[scale=0.6]{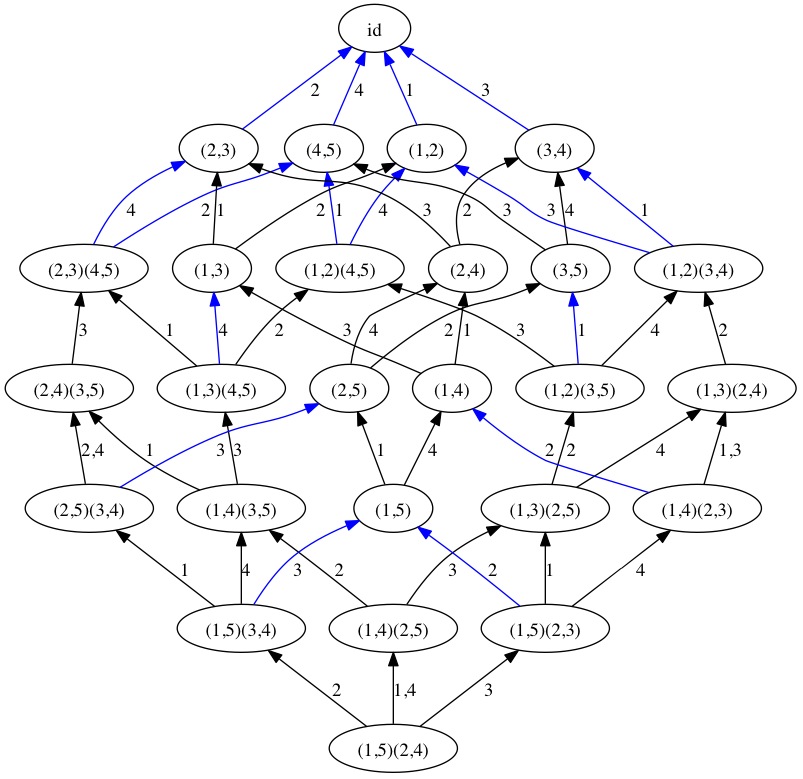}
\end{figure}

\begin{figure}[h!]
	\caption{$(GL(4,\C),O(4,\C))$}\label{fig:type-a-orthogonal-3}
	\centering
	\includegraphics[scale=0.6]{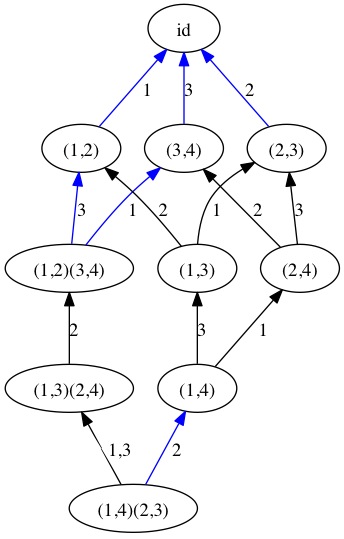}
\end{figure}

\begin{figure}[h!]
	\caption{$(SL(4,\C),SO(4,\C))$}\label{fig:type-a-orthogonal-4}
	\centering
	\includegraphics[scale=0.6]{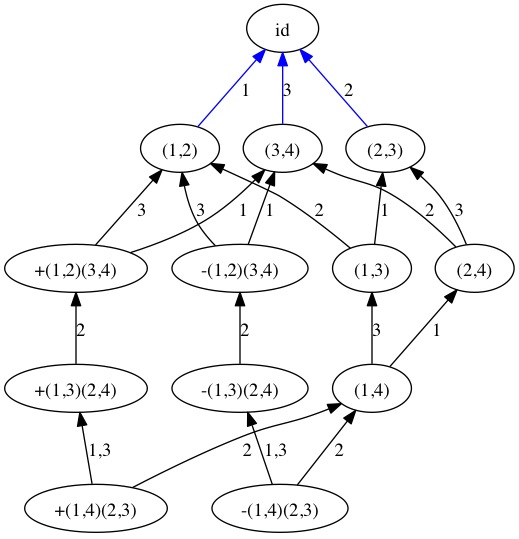}
\end{figure}

\begin{figure}[h!]
	\caption{$(SL(4,\C),Sp(4,\C))$}\label{fig:type-a-symplectic-1}
	\centering
	\includegraphics[scale=0.6]{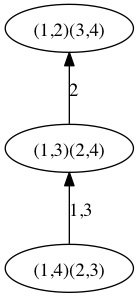}
\end{figure}

\begin{figure}[h!]
	\caption{$(SL(6,\C),Sp(6,\C))$}\label{fig:type-a-symplectic-2}
	\centering
	\includegraphics[scale=0.5]{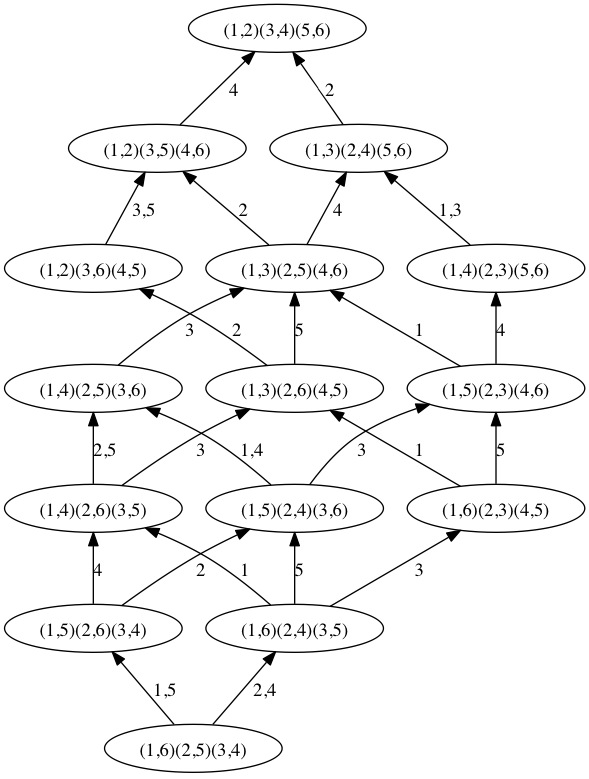}
\end{figure}

\begin{figure}[h!]
	\caption{$(SO(7,\C),S(O(4,\C) \times O(3,\C)))$}\label{fig:type-b-graph}
	\centering
	\includegraphics[scale=0.5]{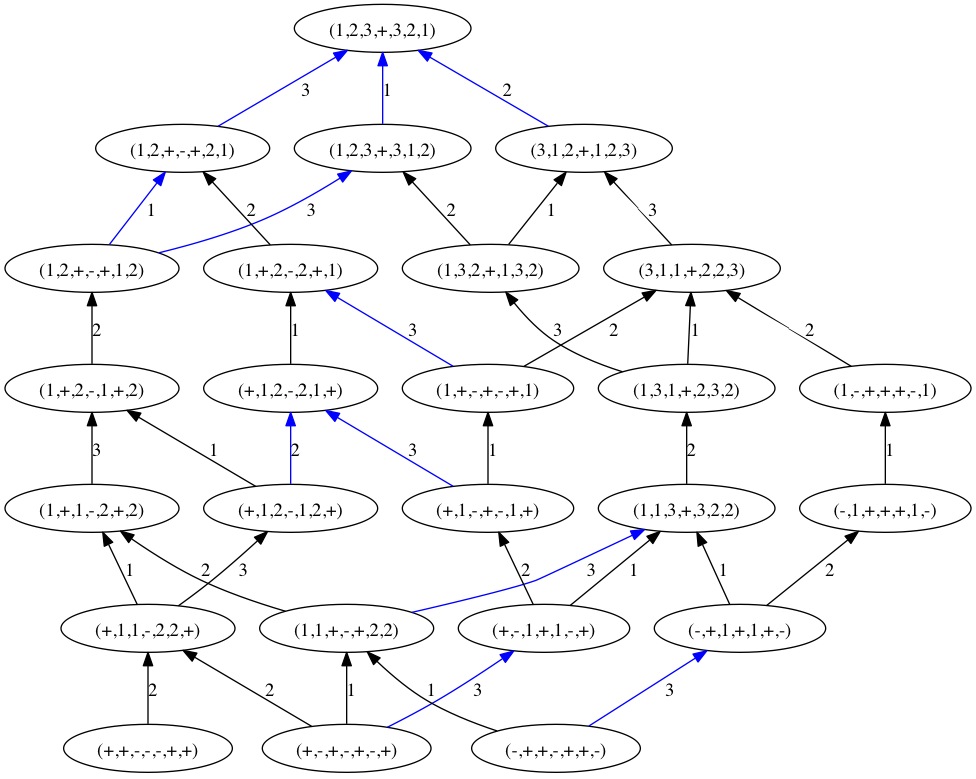}
\end{figure}

\begin{figure}[h!]
	\caption{$(Sp(6,\C),Sp(4,\C) \times Sp(2,\C))$}\label{fig:type-c-graph-1}
	\centering
	\includegraphics[scale=0.5]{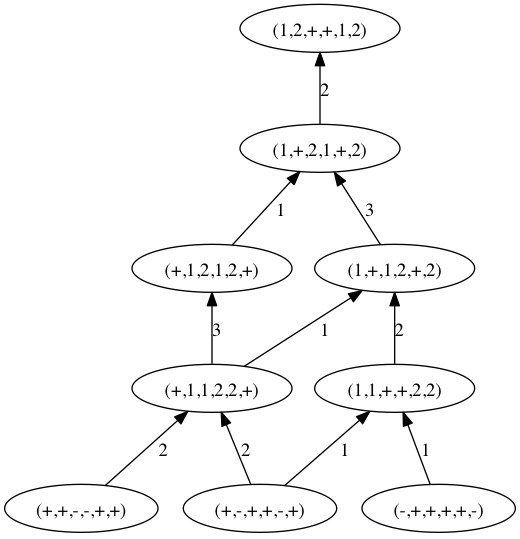}
\end{figure}

\begin{figure}[h!]
	\caption{$(Sp(4,\C),GL(2,\C))$}\label{fig:type-c-graph-2}
	\centering
	\includegraphics[scale=0.5]{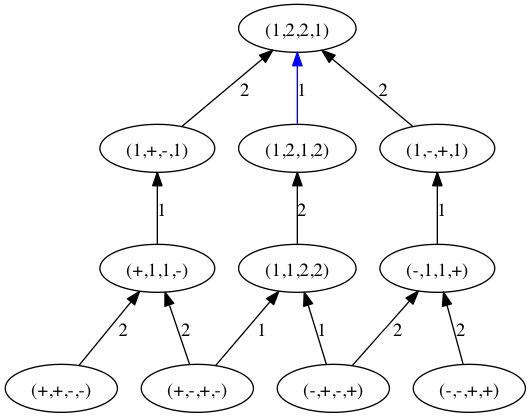}
\end{figure}

\begin{figure}[h!]
	\caption{$(SO(6,\C),S(O(4,\C) \times O(2,\C)))$}\label{fig:type-d-graph-1}
	\centering
	\includegraphics[scale=0.5]{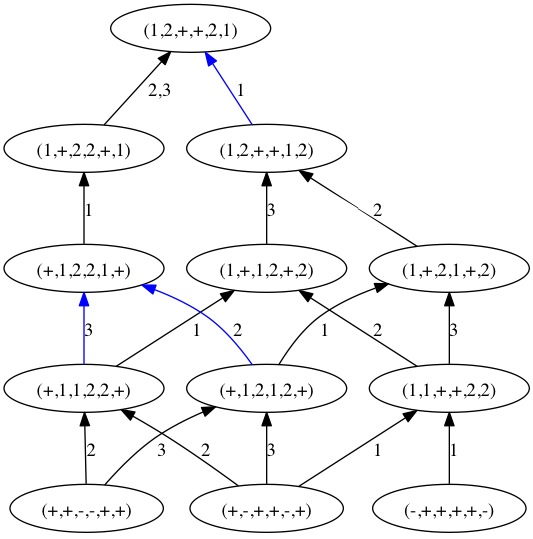}
\end{figure}

\begin{figure}[h!]
	\caption{$(SO(6,\C),GL(3,\C))$}\label{fig:type-d-graph-2}
	\centering
	\includegraphics[scale=0.5]{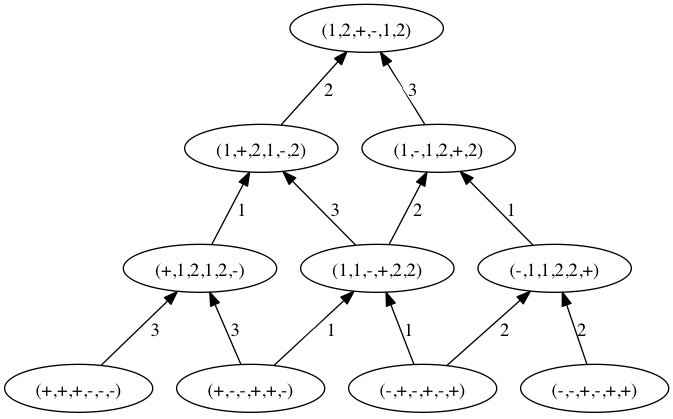}
\end{figure}

\begin{figure}[h!]
	\caption{$(SO(6,\C),S(O(3,\C) \times O(3,\C)))$}\label{fig:type-d-graph-3}
	\centering
	\includegraphics[scale=0.5]{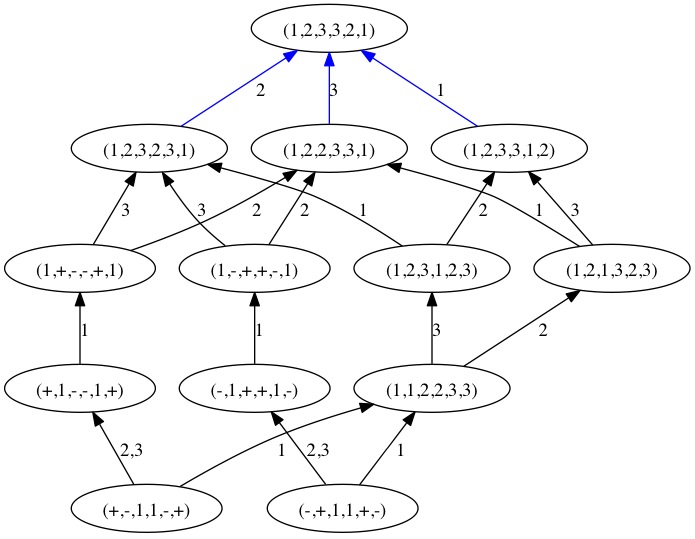}
\end{figure}

\begin{table}[h]
	\caption{Formulas for $(GL(4,\C),GL(2,\C) \times GL(2,\C))$}\label{tab:type-a-2-2}
	\resizebox{16cm}{!}{	
		\begin{tabular}{|l|l|}
			\hline
			$(2,2)$-clan $\gamma$ & Formula for $[Y_{\gamma}]$ \\ \hline
			$(+,+,-,-)$ & $(x_1-y_3)(x_1-y_4)(x_2-y_3)(x_2-y_4)$ \\ \hline
			$(+,-,+,-)$ & $-(x_1-y_2)(x_1-y_4)(x_2-y_2)(x_2-y_4)$ \\ \hline
			$(+,-,-,+)$ & $(x_1-y_2)(x_1-y_3)(x_2-y_2)(x_2-y_3)$ \\ \hline
			$(-,+,+,-)$ & $(x_1-y_1)(x_1-y_4)(x_2-y_1)(x_2-y_4)$ \\ \hline
			$(-,+,-,+)$ & $-(x_1-y_1)(x_1-y_3)(x_2-y_1)(x_2-y_3)$ \\ \hline
			$(-,-,+,+)$ & $(x_1-y_1)(x_1-y_2)(x_2-y_1)(x_2-y_2)$ \\ \hline
			$(+,1,1,-)$ & $(x_1-y_4)(x_2-y_4)(x_1+x_2-y_2-y_3)$ \\ \hline
			$(1,1,+,-)$ & $-(x_1-y_4)(x_2-y_4)(x_1+x_2-y_1-y_2)$ \\ \hline
			$(+,-,1,1)$ & $-(x_1-y_2)(x_2-y_2)(x_1+x_2-y_3-y_4)$ \\ \hline
			$(1,1,-,+)$ & $(x_1-y_3)(x_2-y_3)(x_1+x_2-y_1-y_2)$ \\ \hline
			$(-,+,1,1)$ & $(x_1-y_1)(x_2-y_1)(x_1+x_2-y_3-y_4)$ \\ \hline
			$(-,1,1,+)$ & $-(x_1-y_1)(x_2-y_1)(x_1+x_2-y_2-y_3)$ \\ \hline
			$(1,+,1,-)$ & $(x_1-y_4)(x_2-y_4)$ \\ \hline
			$(+,1,-,1)$ & $x_1^2+x_1x_2-x_1y_2-x_1y_3-x_1y_4+x_2^2-x_2y_2-x_2y_3-x_2y_4+y_2y_3+y_2y_4+y_3y_4$ \\ \hline
			$(1,1,2,2)$ & $-(x_1+x_2-y_3-y_4)(x_1+x_2-y_1-y_2)$ \\ \hline
			$(1,-,1,+)$ & $x_1^2+x_1x_2-x_1y_1-x_1y_2-x_1y_3+x_2^2-x_2y_1-x_2y_2-x_2y_3+y_1y_2+y_1y_3+y_2y_3$ \\ \hline
			$(-,1,+,1)$ & $(x_1-y_1)(x_2-y_1)$ \\ \hline
			$(1,+,-,1)$ & $x_1+x_2-y_3-y_4$ \\ \hline
			$(1,2,1,2)$ & $y_1-y_4$ \\ \hline
			$(1,-,+,1)$ & $-(x_1+x_2-y_1-y_2)$ \\ \hline
			$(1,2,2,1)$ & $1$ \\
			\hline
		\end{tabular}
	}
\end{table}

\begin{table}[h]
	\caption{Formulas for $(SL(3,\C),SO(3,\C))$}\label{tab:type-a-so3}
	\begin{tabular}{|l|l|}
		\hline
		Involution $\pi$ & Formula for $[Y_{\pi}]$ \\ \hline
		$(1,3)$ & $-2(y_1+y_2)(y_2+y_3)$ \\ \hline
		$(1,2)$ &  $-2(y_2+y_3)$ \\ \hline
		$(2,3)$ & $2(y_1+y_2)$ \\ \hline
		id & $1$ \\ 
		\hline
	\end{tabular}
\end{table}

\begin{table}[h]
	\caption{Formulas for $(SL(5,\C),SO(5,\C))$}\label{tab:type-a-so5}
	\resizebox{18cm}{7cm}{
		\begin{tabular}{|l|l|}
			\hline
			Involution $\pi$ & Formula for $[Y_{\pi}]$ \\ \hline
			$(1,5)(2,4)$ & $4(y_1+y_3)(y_3+y_5)(y_2+y_3)(y_3+y_4)(y_1+y_2)(y_1+y_4)$ \\ \hline
			$(1,5)(3,4)$ & $-4(y_1+y_2)(y_1+y_3)(y_1+y_4)(y_2+y_3)(y_2+y_3+y_4+y_5)$ \\ \hline
			$(1,4)(2,5)$ & $4(y_1+y_2)(y_1+y_3)(y_2+y_3)(y_3+y_4)(y_3+y_5)$ \\ \hline
			$(1,5)(2,3)$ & $4(y_1+y_2)(y_1+y_3)(y_1+y_4)(y_3+y_4)(y_2+y_3+y_4+y_5)$ \\ \hline
			$(2,5)(3,4)$ & $-4(y_1+y_2)(y_1+y_3)(y_2+y_3)(y_3+y_5)$ \\ \hline
			$(1,4)(3,5)$ & $-4(y_1+y_2)(y_1+y_3)(y_2+y_3)(y_2+y_3+y_4+y_5)$ \\ \hline
			$(1,5)$ & $-2(y_1+y_2)(y_1+y_3)(y_1+y_4)(y_2+y_3+y_4+y_5)$ \\ \hline
			$(1,3)(2,5)$ & $4(y_1+y_2)(y_3+y_4)(y_3^2+y_4^2 + \displaystyle\sum_{1 \leq i < j \leq 5} y_iy_j)$ \\ \hline
			$(1,4)(2,3)$ & $4(y_1+y_2)(y_1+y_3)(y_1+y_3+y_4+y_5)(y_2+y_3+y_4+y_5)$ \\ \hline
			$(2,4)(3,5)$ & $4(y_1+y_2)(y_1+y_3)(y_2+y_3)$ \\ \hline
			$(1,3)(4,5)$ & $-4(y_1+y_2)(y_1+y_2+y_3+y_4)(y_2+y_3+y_4+y_5)$ \\ \hline
			$(2,5)$ & $-2(y_1+y_2)(y_1y_2+y_1y_3+y_1y_4+y_1y_5+y_2y_3+y_2y_4+y_2y_5+y_3^2+y_3y_4+y_3y_5+y_4^2+y_4y_5)$ \\ \hline
			$(1,4)$ & $-2(y_1+y_2)(y_1+y_3)(y_2+y_3+y_4+y_5)$ \\ \hline
			$(1,2)(3,5)$ & $-4(y_2^2y_3+y_2y_3^2-y_2y_4^2-y_2y_4y_5-y_3y_4^2-y_3y_4y_5-y_4^3-y_4^2y_5 + (y_1^2+y_1y_2) \displaystyle\sum_{i=2}^5 y_i + y_1y_3 \displaystyle\sum_{i=3}^5 y_i)$ \\ \hline
			$(1,3)(2,4)$ & $4(y_1+y_2)(y_1+y_3+y_4+y_5)(y_2+y_3+y_4+y_5)$ \\ \hline
			$(2,3)(4,5)$ & $4(y_1+y_2)(y_1+y_2+y_3+y_4)$ \\ \hline
			$(1,3)$ & $-2(y_1+y_2)(y_2+y_3+y_4+y_5)$ 	\\ \hline
			$(1,2)(4,5)$ & $-4(y_1+y_2+y_3+y_4)(y_2+y_3+y_4+y_5)$ \\ \hline
			$(2,4)$ & $-2(y_1+y_2)(y_4+y_5)$ \\ \hline
			$(3,5)$ & $-2(y_4+y_5)(y_1+y_2+y_3+y_4)$ \\ \hline
			$(1,2)(3,4)$ & $4(y_4+y_5)(y_2+y_3+y_4+y_5)$ \\ \hline
			$(2,3)$ & $2(y_1+y_2)$ \\ \hline
			$(4,5)$ & $2(y_1+y_2+y_3+y_4)$ \\ \hline
			$(1,2)$ & $-2(y_2+y_3+y_4+y_5)$ \\ \hline
			$(3,4)$ & $-2(y_4+y_5)$ \\ \hline
			id & $1$ \\
			\hline
		\end{tabular}
	}
\end{table}

\begin{table}[h]
	\caption{Formulas for $(GL(4,\C),O(4,\C))$}\label{tab:type-a-o4}
	\begin{tabular}{|l|l|}
		\hline
		Involution $\pi$ & Formula for $[Y_{\pi}]$ \\ \hline
		$(1,4)(2,3)$ & $4y_1y_2(y_1+y_2)(y_1+y_3)$ \\ \hline
		$(1,3)(2,4)$ & $4y_1y_2(y_1+y_2)$ \\ \hline
		$(1,4)$ & $2y_1(y_1+y_2)(y_1+y_3)$ \\ \hline
		$(1,2)(3,4)$ & $4y_1(y_1+y_2+y_3)$ \\ \hline
		$(1,3)$ & $2y_1(y_1+y_2)$ \\ \hline
		$(2,4)$ & $2(y_1+y_2)(y_1+y_2+y_3)$ \\ \hline
		$(1,2)$ & $2y_1$ \\ \hline
		$(3,4)$ & $2(y_1+y_2+y_3)$ \\ \hline
		$(2,3)$ & $2(y_1+y_2)$ \\ \hline
		id & $1$ \\ 
		\hline
	\end{tabular}
\end{table}

\begin{table}[h]
	\caption{Formulas for $(SL(4,\C),SO(4,\C))$}\label{tab:type-a-so4}
	\begin{tabular}{|l|l|l|}
		\hline
		Parameter for $Q$ & Representative for $Q$ & Formula for $[Y]$ \\ \hline
		$+(1,4)(2,3)$ & $\left\langle e_1,e_2,e_3,e_4 \right\rangle$ & $2(x_1x_2+y_1y_2)(y_1+y_2)(y_1+y_3)$ \\ \hline
		$-(1,4)(2,3)$ &  $\left\langle e_1,e_3,e_2,e_4 \right\rangle$ & $-2(x_1x_2-y_1y_2)(y_1+y_2)(y_1+y_3)$ \\ \hline
		$+(1,3)(2,4)$ & $\left\langle e_1,e_2,e_4,e_3 \right\rangle$ & $2(x_1x_2+y_1y_2)(y_1+y_2)$ \\ \hline
		$-(1,3)(2,4)$ & $\left\langle e_1,e_3,e_4,e_2 \right\rangle$ & $-2(x_1x_2-y_1y_2)(y_1+y_2)$ \\ \hline
		$(1,4)$ & $\left\langle e_1,e_2+e_3,e_2-e_3,e_4 \right\rangle$ & $2y_1(y_1+y_2)(y_1+y_3)$ \\ \hline
		$+(1,2)(3,4)$ & $\left\langle e_1,e_4,e_2,e_3 \right\rangle$ & $2(x_1x_2+y_1^2+y_1y_2+y_1y_3)$ \\ \hline
		$-(1,2)(3,4)$ & $\left\langle e_1,e_4,e_3,e_2 \right\rangle$ & $-2(x_1x_2-y_1^2-y_1y_2-y_1y_3)$ \\ \hline
		$(1,3)$ & $\left\langle e_1,e_2+e_3,e_4,e_2-e_3 \right\rangle$ & $2y_1(y_1+y_2)$ \\ \hline
		$(2,4)$ & $\left\langle e_2+e_3,e_1,e_2-e_3,e_4 \right\rangle$ & $2(y_1+y_2)(y_1+y_2+y_3)$ \\ \hline
		$(1,2)$ & $\left\langle e_1,e_4,e_2+e_3,e_2-e_3 \right\rangle$ & $2y_1$ \\ \hline
		$(3,4)$ & $\left\langle e_2+e_3,e_2-e_3,e_1,e_4 \right\rangle$ & $2(y_1+y_2+y_3)$ \\ \hline
		$(2,3)$ & $\left\langle e_2+e_3,e_1,e_4,e_2-e_3 \right\rangle$ & $2(y_1+y_2)$ \\ \hline
		id & $\left\langle e_1+e_4,e_1-e_4,e_2+e_3,e_2-e_3 \right\rangle$ & $1$ \\
		\hline
	\end{tabular}
\end{table}

\begin{table}[h]
	\caption{Formulas for $(SL(4,\C),Sp(4,\C))$}\label{tab:type-a-sp4}
	\begin{tabular}{|c|l|}
		\hline
		Involution $\pi$ & Formula for $[Y_{\pi}]$ \\ \hline
		$(1,4)(2,3)$ & $(y_1+y_2)(y_1+y_3)$ \\ \hline
		$(1,3)(2,4)$ & $y_1+y_2$ \\ \hline
		$(1,2)(3,4)$ & $1$ \\
		\hline
	\end{tabular}
\end{table}

\begin{table}[h]
	\caption{Formulas for $(SL(6,\C),Sp(6,\C))$}\label{tab:type-a-sp6}
	\begin{tabular}{|c|l|}
		\hline
		Involution $\pi$ & Formula for $[Y_{\pi}]$ \\ \hline
		$(1,6)(2,4)(3,5)$ & $(y_1+y_2)(y_1+y_5)(y_1+y_3)(y_1+y_4)(y_2+y_3)(y_2+y_4)$ \\ \hline
		$(1,5)(2,6)(3,4)$ & $(y_1+y_2)(y_1+y_3)(y_1+y_4)(y_2+y_3)(y_2+y_4)$ \\ \hline
		$(1,6)(2,4)(3,5)$ & $(y_1+y_2)(y_1+y_5)(y_1+y_3)(y_1+y_4)(y_2+y_3)$ \\ \hline
		$(1,4)(2,6)(3,5)$ & $(y_1+y_2)(y_1+y_3)(y_2+y_3)(y_1+y_2+y_4+y_5)$ \\ \hline		
		$(1,5)(2,4)(3,6)$ & $(y_1+y_2)(y_1+y_3)(y_1+y_4)(y_2+y_3)$ \\ \hline
		$(1,6)(2,3)(4,5)$ & $(y_1+y_2)(y_1+y_5)(y_1+y_3)(y_1+y_4)$ \\ \hline
		$(1,4)(2,5)(3,6)$ & $(y_1+y_2)(y_1+y_3)(y_2+y_3)$ \\ \hline
		$(1,3)(2,6)(4,5)$ & $(y_1+y_2)(y_1^2+y_2^2+\displaystyle\sum_{1 \leq i < j \leq 5} y_iy_j)$ \\ \hline
		$(1,5)(2,3)(4,6)$ & $(y_1+y_2)(y_1+y_3)(y_1+y_4)$ \\ \hline		
		$(1,2)(3,6)(4,5)$ & $(y_1+y_2+y_3+y_4)(y_1+y_2+y_3+y_5)$ \\ \hline
		$(1,3)(2,5)(4,6)$ & $(y_1+y_2)(y_1+y_2+y_3+y_4)$ \\ \hline
		$(1,4)(2,3)(5,6)$ & $(y_1+y_2)(y_1+y_3)$ \\ \hline
		$(1,2)(3,5)(4,6)$ & $y_1+y_2+y_3+y_4$ \\ \hline
		$(1,3)(2,4)(5,6)$ & $y_1+y_2$ \\ \hline
		$(1,2)(3,4)(5,6)$ & $1$ \\
		\hline
	\end{tabular}
\end{table}

\begin{table}[h]
	\caption{Formulas for $(SO(7,\C),S(O(4,\C) \times O(3,\C)))$}\label{tab:type-b-table}
	\begin{tabular}{|c|l|}
		\hline
		Symmetric $(4,3)$-clan $\gamma$ & Formula for $[Y_{\gamma}]$ \\ \hline
		$(+,+,-,-,-,+,+)$ & $y_1y_2(x_1-y_3)(x_1+y_3)(x_2-y_3)(x_2+y_3)$ \\ \hline
		$(+,-,+,-,+,-,+)$ & $-y_1y_3(x_1-y_2)(x_1+y_2)(x_2-y_2)(x_2+y_2)$ \\ \hline
		$(-,+,+,-,+,+,-)$ & $y_2y_3(x_1-y_1)(x_1+y_1)(x_2-y_1)(x_2+y_1)$ \\ \hline
		$(+,1,1,-,2,2,+)$ & $y_1(x_1^2 x_2^2 + y_2y_3(x_1^2+x_2^2-y_2^2-y_2y_3-y_3^2))$  \\ \hline		
		$(1,1,+,-,+,2,2)$ & $-y_3(x_1^2 x_2^2 + y_1y_2(x_1^2+x_2^2-y_1^2-y_1y_2-y_2^2))$ \\ \hline
		$(+,-,1,+,1,-,+)$ & $-y_1(x_1+y_2)(x_1-y_2)(x_2+y_2)(x_2-y_2)$ \\ \hline
		$(-,+,1,+,1,+,-)$ & $y_2(x_1+y_1)(x_1-y_1)(x_2+y_1)(x_2-y_1)$ \\ \hline
		$(1,+,1,-,2,+,2)$ & $x_1^2x_2^2+y_1y_2(y_1y_3+y_2y_3+y_3^2)$ \\ \hline
		$(+,1,2,-,1,2,+)$ & $2y_1y_2(x_1^2+x_2^2-y_2^2-y_3^2)$ \\ \hline		
		$(+,1,-,+,-,1,+)$ & $y_1(y_2+y_3)(x_1^2+x_2^2-y_2^2-y_3^2)$ \\ \hline
		$(1,1,3,+,3,2,2)$ & $-(x_1^2x_2^2 + y_1y_2(x_1^2+x_2^2-y_1^2-y_1y_2-y_2^2))$ \\ \hline
		$(-,1,+,+,+,1,-)$ & $(x_1+y_1)(x_1-y_1)(x_2+y_1)(x_2-y_1)$ \\ \hline
		$(1,+,2,-,1,+,2)$ & $2y_1y_2(y_1+y_2)$ \\ \hline
		$(+,1,2,-,2,1,+)$ & $y_1(x_1^2+x_2^2-y_2^2-y_3^2)$ \\ \hline
		$(1,+,-,+,-,+,1)$ & $x_1^2y_3+x_2^2y_3+y_1^2y_2+y_1y_2^2+y_1y_2y_3-y_3^3$ \\ \hline
		$(1,3,1,+,2,3,2)$ & $-y_1(x_1^2+x_2^2-y_1^2-y_1y_2-y_1y_3-y_2^2-y_2y_3-y_3^2)$ \\ \hline
		$(1,-,+,+,+,-,1)$ & $-(y_1+y_2)(x_1^2+x_2^2-y_1^2-y_2^2)$ \\ \hline
		$(1,2,+,-,+,1,2)$ & $2y_1(y_1+y_2+y_3)$ \\ \hline
		$(1,+,2,-,2,+,1)$ & $x_1^2+x_2^2+y_1y_2-y_3^2$ \\ \hline
		$(1,3,2,+,1,3,2)$ & $2y_1(y_1+y_2)$ \\ \hline
		$(3,1,1,+,2,2,3)$ & $-(x_1^2+x_2^2-y_1^2-y_1y_2-y_1y_3-y_2^2-y_2y_3-y_3^2)$ \\ \hline
		$(1,2,+,-,+,2,1)$ & $y_1+y_2+y_3$ \\ \hline
		$(1,2,3,+,3,1,2)$ & $2y_1$ \\ \hline
		$(3,1,2,+,1,2,3)$ & $2(y_1+y_2)$ \\ \hline
		$(1,2,3,+,3,2,1)$ & $1$ \\
		\hline
	\end{tabular}
\end{table}

\begin{table}[h]
	\caption{Formulas for $(Sp(6,\C),Sp(4,\C) \times Sp(2,\C))$}\label{tab:type-c-table-1}
	\begin{tabular}{|l|l|}
		\hline
		Symmetric $(4,2)$-clan $\gamma$ & Formula for $[Y_{\gamma}]$ \\ \hline
		$(+,+,-,-,+,+)$ & $(x_1+y_3)(x_1-y_3)(x_2+y_3)(x_2-y_3)$ \\ \hline
		$(+,-,+,+,-,+)$ & $-(x_1+y_2)(x_1-y_2)(x_2+y_2)(x_2-y_2)$ \\ \hline
		$(-,+,+,+,+,-)$ & $(x_1+y_1)(x_1-y_1)(x_2+y_1)(x_2-y_1)$ \\ \hline
		$(+,1,1,2,2,+)$ & $(y_2+y_3)(x_1^2+x_2^2-y_2^2-y_3^2)$ \\ \hline
		$(1,1,+,+,2,2)$ & $-(y_1+y_2)(x_1^2+x_2^2-y_1^2-y_2^2)$ \\ \hline
		$(1,+,1,2,+,2)$ & $-(x_1^2+x_2^2-y_1^2-y_1y_2-y_1y_3-y_2^2-y_2y_3-y_3^2)$ \\ \hline
		$(+,1,2,1,2,+)$ & $x_1^2+x_2^2-y_2^2-y_3^2$ \\ \hline
		$(1,+,2,1,+,2)$ & $y_1+y_2$ \\ \hline
		$(1,2,+,+,1,2)$ & $1$ \\
		\hline
	\end{tabular}
\end{table}

\begin{table}[h]
	\caption{Formulas for $(Sp(4,\C),GL(2,\C))$}\label{tab:type-c-table-2}
	\begin{tabular}{|l|l|}
		\hline
		Skew-symmetric $(2,2)$-clan $\gamma$ & Formula for $[Y_{\gamma}]$ \\ \hline
		$(+,+,-,-)$ & $(x_1+x_2+y_1+y_2)(x_1x_2+y_1y_2)$ \\ \hline
		$(+,-,+,-)$ & $-(x_1+x_2+y_1-y_2)(x_1x_2-y_1y_2)$ \\ \hline
		$(-,+,-,+)$ & $(x_1+x_2-y_1+y_2)(x_1x_2-y_1y_2)$ \\ \hline
		$(-,-,+,+)$ & $-(x_1+x_2-y_1-y_2)(x_1x_2+y_1y_2))$ \\ \hline
		$(+,1,1,-)$ & $(x_1+y_1)(x_2+y_1)$ \\ \hline
		$(1,1,2,2)$ & $-2(x_1x_2-y_1y_2)$ \\ \hline
		$(-,1,1,+)$ & $(x_1-y_1)(x_2-y_1)$ \\ \hline
		$(1,+,-,1)$ & $x_1+x_2+y_1+y_2$ \\ \hline
		$(1,2,1,2)$ & $2y_1$ \\ \hline
		$(1,-,+,1)$ & $-(x_1+x_2-y_1-y_2)$ \\ \hline
		$(1,2,2,1)$ & $1$ \\
		\hline
	\end{tabular}
\end{table}

\begin{table}[h]
	\caption{Formulas for $(SO(6,\C),S(O(4,\C) \times O(2,\C)))$}\label{tab:type-d-table-1}
	\begin{tabular}{|l|l|}
		\hline
		Symmetric $(4,2)$-clan $\gamma$ & Formula for $[Y_{\gamma}]$ \\ \hline
		$(+,+,-,-,+,+)$ & $(x_1-y_3)(x_1+y_3)(x_2-y_3)(x_2+y_3)$ \\ \hline
		$(+,-,+,+,-,+)$ & $-(x_1-y_2)(x_1+y_2)(x_2-y_2)(x_2+y_2)$ \\ \hline
		$(-,+,+,+,+,-)$ & $(x_1-y_1)(x_1+y_1)(x_2-y_1)(x_2+y_1)$ \\ \hline
		$(+,1,1,2,2,+)$ & $(y_2+y_3)(x_1^2+x_2^2-y_2^2-y_3^2)$ \\ \hline
		$(+,1,2,1,2,+)$ & $(y_2-y_3)(x_1^2+x_2^2-y_2^2-y_3^2)$ \\ \hline
		$(1,1,+,+,2,2)$ & $-(y_1+y_2)(x_1^2+x_2^2-y_1^2-y_2^2)$ \\ \hline
		$(1,+,1,2,+,2)$ & $-(x_1^2+x_2^2-y_1^2-y_1y_2-y_1y_3-y_2^2-y_2y_3-y_3^2)$ \\ \hline
		$(+,1,2,2,1,+)$ & $x_1^2+x_2^2-y_2^2-y_3^2$ \\ \hline
		$(1,+,2,1,+,2)$ & $-(x_1^2+x_2^2-y_1^2-y_1y_2+y_1y_3-y_2^2+y_2y_3-y_3^2)$ \\ \hline
		$(1,2,+,+,1,2)$ & $2y_1$ \\ \hline
		$(1,+,2,2,+,1)$ & $y_1+y_2$ \\ \hline
		$(1,2,+,+,2,1)$ & $1$ \\ 
		\hline
	\end{tabular}
\end{table}

\begin{table}[h]
	\caption{Formulas for $(SO(6,\C),GL(3,\C))$}\label{tab:type-d-table-2}
	\resizebox{18cm}{3cm}{
		\begin{tabular}{|l|l|}
			\hline
			Skew-symmetric $(3,3)$-clan $\gamma$ & Formula for $[Y_{\gamma}]$ \\ \hline
			$(+,+,+,-,-,-)$ & $\Delta_2(x,y,id)$ \\ \hline
			$(-,-,+,-,+,+)$ & $-\Delta_2(x,y,(\overline{1}\ \overline{2}\ 3))$ \\ \hline
			$(-,+,-,+,-,+)$ & $\Delta_2(x,y,(\overline{1}\ 2\ \overline{3}))$ \\ \hline
			$(+,-,-,+,+,-)$ & $-\Delta_2(x,y,(1\ \overline{2}\ \overline{3}))$ \\ \hline
			$(+,1,2,1,2,-)$ & $\frac{1}{2}(x_1x_2+x_1x_3+x_1y_1+x_2x_3+x_2y_1+x_3y_1+y_1^2+y_2y_3)$ \\ \hline
			$(-,1,1,2,2,+)$ & $\frac{1}{2}(x_1x_2+x_1x_3-x_1y_1+x_2x_3-x_2y_1-x_3y_1+y_1^2-y_2y_3)$ \\ \hline
			$(1,1,-,+,2,2)$ & $-\frac{1}{2}(x_1x_2+x_1x_3-x_1y_3+x_2x_3-x_2y_3-x_3y_3-y_1y_2+y_3^2)$ \\ \hline
			$(1,+,2,1,-,2)$ & $\frac{1}{2}(x_1+x_2+x_3+y_1+y_2-y_3)$ \\ \hline
			$(1,-,1,2,+,2)$ & $-\frac{1}{2}(x_1+x_2+x_3-y_1-y_2-y_3)$ \\ \hline
			$(1,2,+,-,1,2)$ & $1$ \\
			\hline
		\end{tabular}
	}
\end{table}

\begin{table}[h]
	\caption{Formulas for $(SO(6,\C),S(O(3,\C) \times O(3,\C)))$}\label{tab:type-d-table-3}
	\begin{tabular}{|l|l|}
		\hline
		Symmetric $(3,3)$-clan $\gamma$ & Formula for $[Y_{\gamma}]$ \\ \hline
		$(+,-,1,1,-,+)$ &  $y_1y_2(x_1+y_2)(x_1-y_2)$ \\ \hline
		$(-,+,1,1,+,-)$ & $-y_1y_2(x_1+y_1)(x_1-y_1)$ \\ \hline
		$(1,1,2,2,3,3)$ & $y_1y_2(y_1+y_2)$ \\ \hline
		$(+,1,-,-,1,+)$ & $y_1(x_1^2-y_2^2-y_2y_3-y_3^2)$ \\ \hline
		$(-,1,+,+,1,-)$ & $-y_1(x_1-y_1)(x_1+y_1)$ \\ \hline
		$(1,2,1,3,2,3)$ & $y_1(y_1+y_2+y_3)$ \\ \hline
		$(1,2,3,1,2,3)$ & $y_1(y_1+y_2-y_3)$ \\ \hline
		$(1,+,-,-,+,1)$ & $x_1^2+y_1y_2-y_3^2$ \\ \hline
		$(1,-,+,+,-,1)$ & $-x_1^2+y_1^2+y_1y_2+y_2^2$ \\ \hline
		$(1,2,2,3,3,1)$ & $y_1+y_2+y_3$ \\ \hline
		$(1,2,3,3,1,2)$ & $2y_1$ \\ \hline
		$(1,2,3,2,3,1)$ & $y_1+y_2-y_3$ \\ \hline
		$(1,2,3,3,2,1)$ & $1$ \\
		\hline
	\end{tabular}
\end{table}

%% file: thesis.bbl
\begin{thebibliography}{{Wys}11b}

\bibitem[Ada08]{Adams-08}
Jeffrey Adams.
\newblock Guide to the {A}tlas software: computational representation theory of
  real reductive groups.
\newblock In {\em Representation theory of real reductive {L}ie groups}, volume
  472 of {\em Contemp. Math.}, pages 1--37. Amer. Math. Soc., Providence, RI,
  2008.

\bibitem[Ada09]{Adams-09}
Jefrey Adams.
\newblock The {A}tlas {A}lgorithm {W}orkshop, {S}alt {L}ake {C}ity, {J}uly
  20-24, 2009.
\newblock \url{http://www.liegroups.org/workshopNotes/adams.pdf}, jul 2009.

\bibitem[AdC09]{Adams-DuCloux-09}
Jeffrey Adams and Fokko du~Cloux.
\newblock Algorithms for representation theory of real reductive groups.
\newblock {\em J. Inst. Math. Jussieu}, 8(2):209--259, 2009.

\bibitem[Bri98]{Brion-98_i}
Michel Brion.
\newblock Equivariant cohomology and equivariant intersection theory.
\newblock In {\em Representation theories and algebraic geometry ({M}ontreal,
  {PQ}, 1997)}, volume 514 of {\em NATO Adv. Sci. Inst. Ser. C Math. Phys.
  Sci.}, pages 1--37. Kluwer Acad. Publ., Dordrecht, 1998.
\newblock Notes by Alvaro Rittatore.

\bibitem[Bri99]{Brion-99}
M.~Brion.
\newblock Rational smoothness and fixed points of torus actions.
\newblock {\em Transform. Groups}, 4(2-3):127--156, 1999.
\newblock Dedicated to the memory of Claude Chevalley.

\bibitem[Bri01]{Brion-01}
Michel Brion.
\newblock On orbit closures of spherical subgroups in flag varieties.
\newblock {\em Comment. Math. Helv.}, 76(2):263--299, 2001.

\bibitem[dC05]{DuCloux-05}
Fokko du~Cloux.
\newblock Combinatorics for the representation theory of real reductive groups.
\newblock \url{http://www.liegroups.org/papers/summer05/combinatorics.pdf}, jul
  2005.

\bibitem[EG95]{Edidin-Graham-95}
Dan Edidin and William Graham.
\newblock Characteristic classes and quadric bundles.
\newblock {\em Duke Math. J.}, 78(2):277--299, 1995.

\bibitem[FP98]{Fulton-Pragacz}
William Fulton and Piotr Pragacz.
\newblock {\em Schubert varieties and degeneracy loci}, volume 1689 of {\em
  Lecture Notes in Mathematics}.
\newblock Springer-Verlag, Berlin, 1998.
\newblock Appendix J by the authors in collaboration with I. Ciocan-Fontanine.

\bibitem[Ful92]{Fulton-92}
William Fulton.
\newblock Flags, {S}chubert polynomials, degeneracy loci, and determinantal
  formulas.
\newblock {\em Duke Math. J.}, 65(3):381--420, 1992.

\bibitem[Ful96a]{Fulton-96_2}
William Fulton.
\newblock Determinantal formulas for orthogonal and symplectic degeneracy loci.
\newblock {\em J. Differential Geom.}, 43(2):276--290, 1996.

\bibitem[Ful96b]{Fulton-96_1}
William Fulton.
\newblock Schubert varieties in flag bundles for the classical groups.
\newblock In {\em Proceedings of the {H}irzebruch 65 {C}onference on
  {A}lgebraic {G}eometry ({R}amat {G}an, 1993)}, volume~9 of {\em Israel Math.
  Conf. Proc.}, pages 241--262, Ramat Gan, 1996. Bar-Ilan Univ.

\bibitem[Ful97]{Fulton-YoungTableaux}
William Fulton.
\newblock {\em Young tableaux}, volume~35 of {\em London Mathematical Society
  Student Texts}.
\newblock Cambridge University Press, Cambridge, 1997.
\newblock With applications to representation theory and geometry.

\bibitem[Ful98]{Fulton-IT}
William Fulton.
\newblock {\em Intersection theory}, volume~2 of {\em Ergebnisse der Mathematik
  und ihrer Grenzgebiete. 3. Folge. A Series of Modern Surveys in Mathematics
  [Results in Mathematics and Related Areas. 3rd Series. A Series of Modern
  Surveys in Mathematics]}.
\newblock Springer-Verlag, Berlin, second edition, 1998.

\bibitem[Gra97]{Graham-97}
William Graham.
\newblock The class of the diagonal in flag bundles.
\newblock {\em J. Differential Geom.}, 45(3):471--487, 1997.

\bibitem[Hum75]{Humphreys-75}
James~E. Humphreys.
\newblock {\em Linear algebraic groups}.
\newblock Springer-Verlag, New York, 1975.
\newblock Graduate Texts in Mathematics, No. 21.

\bibitem[Mat79]{Matsuki-79}
Toshihiko Matsuki.
\newblock The orbits of affine symmetric spaces under the action of minimal
  parabolic subgroups.
\newblock {\em J. Math. Soc. Japan}, 31(2):331--357, 1979.

\bibitem[M{\=O}90]{Matsuki-Oshima-90}
Toshihiko Matsuki and Toshio {\=O}shima.
\newblock Embeddings of discrete series into principal series.
\newblock In {\em The orbit method in representation theory ({C}openhagen,
  1988)}, volume~82 of {\em Progr. Math.}, pages 147--175. Birkh\"auser Boston,
  Boston, MA, 1990.

\bibitem[MS74]{Milnor-Stasheff}
John~W. Milnor and James~D. Stasheff.
\newblock {\em Characteristic classes}.
\newblock Princeton University Press, Princeton, N. J., 1974.
\newblock Annals of Mathematics Studies, No. 76.

\bibitem[MT09]{McGovern-Trapa-09}
William~M. McGovern and Peter~E. Trapa.
\newblock Pattern avoidance and smoothness of closures for orbits of a
  symmetric subgroup in the flag variety.
\newblock {\em J. Algebra}, 322(8):2713--2730, 2009.

\bibitem[RS90]{Richardson-Springer-90}
R.~W. Richardson and T.~A. Springer.
\newblock The {B}ruhat order on symmetric varieties.
\newblock {\em Geom. Dedicata}, 35(1-3):389--436, 1990.

\bibitem[RS93]{Richardson-Springer-92}
R.~W. Richardson and T.~A. Springer.
\newblock Combinatorics and geometry of {$K$}-orbits on the flag manifold.
\newblock In {\em Linear algebraic groups and their representations ({L}os
  {A}ngeles, {CA}, 1992)}, volume 153 of {\em Contemp. Math.}, pages 109--142.
  Amer. Math. Soc., Providence, RI, 1993.

\bibitem[RS94]{Richardson-Springer-94}
R.~W. Richardson and T.~A. Springer.
\newblock Complements to: ``{T}he {B}ruhat order on symmetric varieties''
  [{G}eom. {D}edicata {\bf 35} (1990), no. 1-3, 389--436; {MR}1066573
  (92e:20032)].
\newblock {\em Geom. Dedicata}, 49(2):231--238, 1994.

\bibitem[Spr85]{Springer-85}
T.~A. Springer.
\newblock Some results on algebraic groups with involutions.
\newblock In {\em Algebraic groups and related topics ({K}yoto/{N}agoya,
  1983)}, volume~6 of {\em Adv. Stud. Pure Math.}, pages 525--543.
  North-Holland, Amsterdam, 1985.

\bibitem[Spr87]{Springer-87}
T.~A. Springer.
\newblock The classification of involutions of simple algebraic groups.
\newblock {\em J. Fac. Sci. Univ. Tokyo Sect. IA Math.}, 34(3):655--670, 1987.

\bibitem[{Wys}11a]{Wyser-11a}
Benjamin~J. {Wyser}.
\newblock {GL(p) x GL(q)-orbit closures on the flag variety and Schubert
  structure constants for (p,q)-pairs}.
\newblock {\em ArXiv e-prints}, sep 2011.
\newblock Submitted. {\href{http://arxiv.org/abs/1109.0336}{arXiv:1109.0336}.}

\bibitem[{Wys}11b]{Wyser-11b}
Benjamin~J. {Wyser}.
\newblock {K-orbits on G/B and Schubert constants for pairs of signed shuffles
  in types C and D}.
\newblock {\em ArXiv e-prints}, sep 2011.
\newblock Submitted. {\href{http://arxiv.org/abs/1109.2574}{arXiv:1109.2574}.}

\bibitem[Yam97]{Yamamoto-97}
Atsuko Yamamoto.
\newblock Orbits in the flag variety and images of the moment map for classical
  groups. {I}.
\newblock {\em Represent. Theory}, 1:329--404 (electronic), 1997.

\end{thebibliography}
